\documentclass[11pt]{article}
\usepackage{amssymb,latexsym,amsmath,amsbsy,amsthm,amsxtra,amsgen,graphicx,makeidx,dsfont,longtable}
\numberwithin{equation}{section}

\newfont{\msbm}{msbm10 at 11pt}
\newcommand {\R} {\mbox{\msbm R}}
\newcommand {\Z} {\mbox{\msbm Z}}

\newcommand {\N} {\mbox{\msbm N}}
\newcommand {\1} {\mathds{1}}

\newfont{\msbmsm}{msbm10 at 8pt}

\newcommand {\Zsm} {\mbox{\msbmsm Z}}

\newtheorem{Theo}{Theorem}[section]
\newtheorem{Lemma}[Theo]{Lemma}
\newtheorem{Cor}[Theo]{Corollary}
\newtheorem{Prop}[Theo]{Proposition}

\newtheorem{Rmk}[Theo]{Remark}

\def\eps{\varepsilon}
\def\Var{\textup{Var}}

\begin{document}
\title{Rigorous results for a population model with selection I: \\ evolution of the fitness distribution}
\author{Jason Schweinsberg\thanks{Supported in part by NSF Grant DMS-1206195} \\
University of California at San Diego}
\maketitle



\vspace{-.6in}
\begin{abstract}
We consider a model of a population of fixed size $N$ undergoing selection.  Each individual acquires beneficial mutations at rate $\mu_N$, and each beneficial mutation increases the individual's fitness by $s_N$.  Each individual dies at rate one, and when a death occurs, an individual is chosen with probability proportional to the individual's fitness to give birth.  Under certain conditions on the parameters $\mu_N$ and $s_N$, we obtain rigorous results for the rate at which mutations accumulate in the population and the distribution of the fitnesses of individuals in the population at a given time.  Our results confirm predictions of Desai and Fisher (2007).
\end{abstract}

\section{Introduction}

We consider the following model of a population undergoing selection.  We assume there are exactly $N$ individuals in the population at all times.  Each individual independently acquires mutations at times of a Poisson process with rate $\mu_N$, and all mutations are assumed to be beneficial.  Each individual is assigned a fitness, which depends on how many mutations the individual has acquired relative to the mean of the population.  More precisely, let $X_j(t)$ be the number of individuals with $j$ mutations at time $t$, and let $$M(t) = \frac{1}{N} \sum_{j=0}^{\infty} j X_j(t)$$ be the average number of mutations for the $N$ individuals in the population at time $t$.  Then the fitness of an individual with $j$ mutations at time $t$ is
\begin{equation}\label{fitness}
\max \big\{ 0, 1 + s_N(j - M(t)) \big\}.
\end{equation}
Each individual independently lives for a time which is exponentially distributed with mean one, then dies and gets replaced by a new individual.  The parent of the new individual is chosen at random from the population, and the probability that a particular individual is chosen as the parent is proportional to that individual's fitness.  The new individual inherits all of its parent's mutations.  Note that this model includes two parameters: the mutation rate $\mu_N$ and the selection parameter $s_N$.  

This model is of interest mostly because it is essentially the simplest possible model that allows for repeated beneficial mutations.  The model has appeared previously in the literature; see, for example, \cite{brw08, df07}.  An alternative to (\ref{fitness}), which was considered, for example, in \cite{beer07, dm11, rbw08}, is to assign a fitness of $(1 + s_N)^j$ to an individual with $j$ mutations.  However, assumption A3 below will ensure that for the range of parameters that we will consider, $s_N (j - M(t))$ is small and therefore the approximation $1 + s_N (j - M(t)) \approx (1 + s_N)^{(j - M(t))}$ is valid.  Consequently, the distinction between these two models is not important for our purposes.  A limitation to our model is that the selective advantage $s_N$ is assumed to be the same for every beneficial mutation.  Some authors have considered models in which the selective advantage resulting from a mutation is random (see \cite{dflmm10, gl98, pk07, psk10, wilke}), but we do not consider this complication here.

Here we will be interested in determining how rapidly the population acquires beneficial mutations, that is, how fast $M(t)$ grows as a function of $t$.  This growth rate is sometimes called the rate of adaptation or the speed of evolution.  Also, we will be interested in understanding the distribution of the fitnesses of individuals in the population at a given time.

\subsection{Previous work}\label{prevsec}

The behavior of the population in this model can vary considerably depending on the values of the parameters $\mu_N$ and $s_N$.  The simplest case to handle is when the mutation rate $\mu_N$ is small enough that there is only one beneficial mutation in the population at a time.  This occurs, for example, when $s_N = s > 0$ is a fixed constant and $\lim_{N \rightarrow \infty} \mu_N (N \log N) = 0$.  In this case, there is approximately an exponentially distributed waiting time until there is a so-called selective sweep, in which a beneficial mutation appears on one individual and then spreads to the entire population, followed by another exponentially distributed waiting time until another selective sweep occurs, and so on.  See, for example, Chapter 6 of \cite{durrett} for details.  However, the process becomes much more complicated as soon as mutations occur rapidly enough that there can be more than one beneficial mutation in the population at a time.

Another case that has been studied in detail is when $N \mu_N \rightarrow \alpha \in (0, \infty)$ and $N s_N \rightarrow \gamma \in (0, \infty)$ as $N \rightarrow \infty$.  That is, the mutation rate $\mu_N$ and the selection parameter $s_N$ are both of the order $1/N$.  In this case, one can describe the process using diffusion theory.  For a summary of results in this direction, see sections 7.2 and 8.1 of \cite{durrett} and chapter 10 of \cite{ek86}.

An important paper which establishes rigorous results is the work of Durrett and Mayberry \cite{dm11}, who were motivated by cancer modeling.  They considered the variation of the model in which the fitness of an individual with $j$ mutations is given by $(1 + s)^j$, where $s$ is a fixed constant not depending on $N$.  They also assumed that $\mu_N \sim N^{-\alpha}$, where $0 < \alpha < 1$.  They showed that if $T_j = \min\{t: X_j(t) \geq 1\}$ is the first time when an individual gets $j$ mutations, then $$\frac{s T_j}{\log(1/\mu_N)} \rightarrow_p t_j$$ for a certain deterministic sequence of constants $(t_j)_{j=1}^{\infty}$, where $\rightarrow_p$ denotes convergence in probability as $N \rightarrow \infty$.  They also obtained more precise results describing how the number of type $j$ individuals evolves over time.

Yu, Etheridge, and Cuthbertson \cite{yec10} considered very fast mutation rates, where $\mu_N = \mu > 0$ and $s_N = s > 0$ for all $N$.  That is, neither the mutation rate nor the selection parameter depends on $N$.  The model they considered is slightly different from the one presented here in that an individual's fitness affects its death rate as well as its birth rate.  They observed that the process that keeps track of the differences between the fitness of the individuals and the mean fitness of the population has a stationary distribution.  They proved that if the process starts from this stationary distribution, then for all $\delta > 0$, we have $$E[M(t) - M(t - 1)] \geq (\log N)^{1 - \delta}$$ if $N$ is sufficiently large, thus establishing a lower bound of $(\log N)^{1 - \delta}$ on the rate of adaptation.  Kelly \cite{kelly} considered the same model and obtained a corresponding upper bound by showing that if at time zero there are no mutations in the population, then $$\frac{E[M(t)]}{t} \leq \frac{C \log N}{(\log \log N)^2}$$ for $t \geq \log \log N$, where $C$ is a positive constant.  However, up to now, the precise asymptotic rate of adaptation has not been calculated rigorously in this case.

Although there are only a few rigorous results available for this model, there has been a considerable amount of previous nonrigorous work on this model and closely related models, mostly appearing in the Biology literature.  Of particular relevance for the present paper is the work of Desai and Fisher \cite{df07}, who carried out a precise and detailed analysis of this model.  They found, under certain conditions on the parameters $s_N$ and $\mu_N$, that the difference in the number of mutations between the fittest individual in the population and a typical individual in the population is approximately
\begin{equation}\label{dfwidth}
\frac{2 \log (N s_N)}{\log (s_N/\mu_N)}
\end{equation}
and that in the long-run, the number of mutations carried by a typical individual in the population increases at the rate of approximately
\begin{equation}\label{dfspeed}
\frac{2 s_N \log(N s_N)}{[\log (s_N/\mu_N)]^2}
\end{equation}
per unit time.  See the discussion around equations (4) and (5) on p. 1765 of \cite{df07} for a brief explanation, and see the discussion around (40) and (41) on p. 1774 of \cite{df07} for a more detailed analysis.  See also Brunet, Rouzine, and Wilke \cite{brw08} for further analysis of these results.  The heuristic arguments in \cite{df07} are discussed in more detail in section \ref{heursec} below, and are largely the basis for the rigorous results proved in this paper.

Rouzine, Brunet, and Wilke \cite{rbw08} studied the same problem using a different approach, building on earlier work of Rouzine, Wakeley, and Coffin \cite{rwc03}, and estimated the rate of increase in the number of mutations carried by a typical individual in the population to be approximately
\begin{equation}\label{rbwspeed}
\frac{2s_N \log (N \sqrt{s_N \mu_N})}{[\log((s_N/\mu_N) \log(N \sqrt{s_N \mu_N}))]^2},
\end{equation}
which will match (\ref{dfspeed}) asymptotically as long as the extra factors inside the logarithms can be ignored.  See equation (53) in \cite{rbw08}, and see section A.1 of \cite{rbw08} for a discussion of the assumptions required for (\ref{rbwspeed}) to be valid.

In addition to obtaining the estimates (\ref{dfspeed}) and (\ref{rbwspeed}) on the speed of evolution, these and other authors have considered the distribution of fitnesses of individuals in the population at a given time, coming to the conclusion that this distribution should be approximately Gaussian.  See, for example, the discussion at the top of p. 1775 in \cite{df07}, the mathematical appendix in \cite{rwc03}, and the discussion around (11) of \cite{rbw08}.  Other heuristic arguments for why the distribution of fitnesses should be approximately Gaussian are given in section 3 of \cite{yec10} and in the supporting information to \cite{beer07}.  Because the mean of this Gaussian distribution is increasing in time as the population evolves, the evolution of the fitness distribution in the population can be modeled as a Gaussian traveling wave.  This point of view is emphasized in \cite{beer07} and can be traced back at least to \cite{tlk96}.
It should be noted that Durrett and Mayberry \cite{dm11} rigorously obtained traveling wave behavior in their model.  However, for the low mutation rates that they considered, the number of values of $j$ for which $X_j(t) > 0$ does not tend to infinity as $N \rightarrow \infty$.  Consequently, they did not observe a traveling wave with a Gaussian shape, and indeed the Gaussian traveling wave picture has not been established rigorously for any range of parameter values.

The goal of this paper is to carry out a detailed, mathematically rigorous analysis of the model described above.  Under certain conditions on $s_N$ and $\mu_N$, we are able to confirm several of the most important nonrigorous predictions about the model.  We obtain rigorous results concerning the speed of evolution and the distribution of fitnesses of individuals in the population at a given time.  We present our assumptions in section \ref{asssec} and our main results in section \ref{ressec}.  In section \ref{heursec}, we explain the heuristics behind the results, most of which are adapted from the previous nonrigorous work mentioned above.  The rest of the paper is devoted to proving the main results.

This is the first in a series of two papers devoted to the study of this model.  In the follow-up paper \cite{schII}, we show that the genealogy of the population can be described by a process called the Bolthausen-Sznitman coalescent, confirming predictions of Desai, Walczak, and Fisher \cite{dwf13} and Neher and Hallatschek \cite{nh13}.  The paper \cite{schII} uses extensively the results and techniques developed here.

\subsection{Assumptions on the parameters}\label{asssec}

For deterministic sequences $(x_N)_{N=1}^{\infty}$ and $(y_N)_{N=1}^{\infty}$ depending on the population size $N$, we write $x_N \sim y_N$ if $\lim_{N \rightarrow \infty} x_N/y_N = 1$.  We write $x_N \ll y_N$ if $\lim_{N \rightarrow \infty} x_N/y_N = 0$ and $x_N \gg y_N$ if $\lim_{N \rightarrow \infty} x_N/y_N = \infty$.

For our main results, we will need the following assumptions on the parameters $s_N$ and $\mu_N$:

\bigskip
{\bf A1}:  We have ${\displaystyle \lim_{N \rightarrow \infty} \frac{\log N}{\log(s_N/\mu_N) \log (1/s_N)} = \infty}$.

\bigskip
{\bf A2}:  We have ${\displaystyle \lim_{N \rightarrow \infty} \frac{\log N}{[\log(s_N/\mu_N)]^2} \log \bigg( \frac{\log N}{\log(s_N/\mu_N)} \bigg) = 0.}$

\bigskip
{\bf A3}:  We have ${\displaystyle \lim_{N \rightarrow \infty} \frac{s_N \log N}{\log(s_N/\mu_N)} = 0}$.

\bigskip
\noindent The biological meaning of these assumptions, and the reason why they are needed for the main results, will be described later in section \ref{meaningass}.  Here we mention some of their consequences.  

Dividing A3 by A1, we see that the assumptions imply that $\lim_{N \rightarrow \infty} s_N \log(1/s_N) = 0$ and therefore
\begin{equation}\label{sto0}
\lim_{N \rightarrow \infty} s_N = 0.
\end{equation}
This result and A1 imply that
\begin{equation}\label{A1prime}
\lim_{N \rightarrow \infty} \frac{\log N}{\log(s_N/\mu_N)} = \infty,
\end{equation}
and combining this observation with A2 gives
\begin{equation}\label{A2prime}
\lim_{N \rightarrow \infty} \frac{\log N}{[\log(s_N/\mu_N)]^2} = 0.
\end{equation}
Dividing (\ref{A2prime}) by A1, we get $\lim_{N \rightarrow \infty} \log(1/s_N)/\log(s_N/\mu_N) = 0$.  Thus, $\log(1/\mu_N) \gg \log(1/s_N)$, which means that for all $a > 0$, we have
\begin{equation}\label{muspower}
\mu_N \ll s_N^a.
\end{equation}
That is, the mutation rate $\mu_N$ tends to zero faster than any power of $s_N$.  Another consequence of the fact that $\lim_{N \rightarrow \infty} \log(1/s_N)/\log(s_N/\mu_N) = 0$ is that $\log(s_N/\mu_N) \sim \log(1/\mu_N).$  In particular, (\ref{A1prime}) implies that $\log N \gg \log(1/\mu_N)$, which means that for all $a > 0$, we have
\begin{equation}\label{muNpower}
\mu_N \gg \frac{1}{N^a}.
\end{equation}
That is, the mutation rate tends to zero more slowly than any power of $1/N$.  Also, note that because $\log(s_N/\mu_N) \sim \log(1/\mu_N)$, the expression $\log(s_N/\mu_N)$ could be replaced by $\log(1/\mu_N)$ in any of the conditions A1, A2, and A3.  We state the conditions in their current form because $\log(s_N/\mu_N)$ arises more naturally, as we will see later.  We will always assume $N$ is large enough that $\mu_N < s_N$, so $\log(s_N/\mu_N) > 0$.

To illustrate how these assumptions can be satisfied, we observe that if $1/2 < b < 1$ and $0 < a < 1-b$, and if for all $N$ we have $$\mu_N = e^{-(\log N)^b}$$ and $$e^{-(\log N)^a} \leq s_N \leq \frac{1}{\sqrt{\log N}},$$ then assumptions A1-A3 hold.  

\subsection{Main results}\label{ressec}

Let
\begin{equation}\label{aNdef}
a_N = \frac{1}{s_N} \log \bigg( \frac{s_N}{\mu_N} \bigg).
\end{equation}
We will see later that, as was observed in \cite{dwf13}, the quantity $a_N$ is approximately the amount of time between when the first individual with $j$ mutations appears and when individuals in the population have $j$ mutations on average.  This is the time scale on which we will study the process.  
Also, define
\begin{equation}\label{kNdef}
k_N = \frac{\log N}{\log(s_N/\mu_N)},
\end{equation}
which we will see is the natural scale on which to consider the number of mutations.  For $t \geq 0$, let
\begin{equation}\label{Qdef}
Q(t) = \max\{j: X_j(t) > 0\} - M(t)
\end{equation}
be the difference between the number of mutations carried by the fittest individual in the population and the mean number of mutations in the population.  Our first theorem is an asymptotic result for this quantity.  Here and throughout the paper, the notation $\rightarrow_p$ denotes convergence in probability as $N \rightarrow \infty$.

\begin{Theo}\label{Qthm}
Assume A1-A3 hold.  There is a unique bounded function $q: [0, \infty) \rightarrow [0, \infty)$ such that
\begin{equation}\label{qdef}
q(t) = \left\{
\begin{array}{ll} e^t & \mbox{ if }0 \leq t < 1  \\
\int_{t-1}^t q(u) \: du & \mbox{ if }t \geq 1.
\end{array} \right.
\end{equation}
If $S$ is a compact subset of $(0,1) \cup (1, \infty)$, then
\begin{equation}\label{mainQres}
\sup_{t \in S} \bigg| \frac{Q(a_N t)}{k_N} - q(t) \bigg| \rightarrow_p 0.
\end{equation}
Furthermore, we have
\begin{equation}\label{qlim}
\lim_{t \rightarrow \infty} q(t) = 2.
\end{equation}
\end{Theo}
Note that Theorem \ref{Qthm} implies that for large $t$, we have
\begin{equation}\label{Qlarget}
Q(a_N t) \approx \frac{2 \log N}{\log(s_N/\mu_N)},
\end{equation}
which is consistent with Desai and Fisher's prediction (\ref{dfwidth}) because $|\log s_N| \ll \log N$ when A1-A3 hold.  Note also that the function $q$ is discontinuous at $1$, which is why we can not expect uniform convergence to hold over intervals containing 1.

The next result is our main theorem concerning the speed of evolution.  It shows how the mean number of mutations in the population changes over time.  

\begin{Theo}\label{speedthm}
Assume A1-A3 hold.  Let $m: [0, \infty) \rightarrow \R$ be the function defined by
\begin{equation}\label{mdef}
m(t) = \left\{
\begin{array}{ll} 0 & \mbox{ if }0 \leq t < 1  \\
1 + \int_0^{t-1} q(u) \: du & \mbox{ if }t \geq 1,
\end{array} \right.
\end{equation}
where $q$ is the function defined in (\ref{qdef}).  Then, if $S$ is a compact subset of $[0, 1) \cup (1, \infty)$, we have
\begin{equation}\label{mainMres}
\sup_{t \in S} \bigg| \frac{M(a_N t)}{k_N} - m(t) \bigg| \rightarrow_p 0.
\end{equation}
\end{Theo}
Note that the function $m$ is discontinuous at 1, so Theorem \ref{speedthm} implies that the average number of mutations of individuals in the population stays close to zero until time $a_N$, then rapidly increases to approximately $k_N$.  To see the long-run rate at which the population acquires beneficial mutations, note that (\ref{qlim}) implies that
\begin{equation}\label{mtlim}
\lim_{t \rightarrow \infty} \frac{m(t)}{t} = 2.
\end{equation}
Therefore, for large $t$, 
\begin{equation}\label{speed}
\frac{M(a_N t)}{a_N t} \approx \frac{m(t) k_N}{a_N t} = \frac{2 s_N \log N}{[\log(s_N/\mu_N)]^2}.
\end{equation}
The right-hand side of (\ref{speed}) can be viewed as the rate of adaptation, or the rate per unit time at which new mutations take hold in the population.  Because $|\log s_N| \ll \log N$ and $\log(1/\mu_N) \ll \log N$ when A1-A3 hold, as can be seen from (\ref{muspower}) and (\ref{muNpower}), and $\log \log N \ll \log(s_N/\mu_N)$ by (\ref{A2prime}), this result is consistent with the predictions (\ref{dfspeed}) and (\ref{rbwspeed}).

\begin{Rmk}\label{renewalrmk}
{\em The functions $q$ and $m$ have a renewal theory interpretation, which helps to explain (\ref{qlim}) and (\ref{mtlim}).  
Consider a renewal process in which the distribution of the time between renewals is uniform on $(0,1)$.  Let $N(t)$ be the number of renewals by time $t$, and let $U(t) = E[N(t)]$.  The renewal equation gives $$U(t) = (t \wedge1) + \int_{(t-1) \vee 0}^t U(x) \: dx.$$  Let $U'$ denote the right derivative of $U$.  If $0 \leq t < 1$, then $U'(t) = 1 + U(t)$, and since $U(0) = 0$, it follows that $U(t) = e^t - 1$ and thus $U'(t) = e^t$.  If $t \geq 1$, then $U'(t) = U(t) - U(t-1) = \int_{t-1}^t U'(u) \: du$.  It follows that $U'$ satisfies (\ref{qdef}), so $U'(t) = q(t)$ for all $t$.  Also, for $t \geq 1$, $$m(t) = 1 + \int_0^{t-1} U'(u) \: du = 1 + U(t-1).$$  For large $t$, because the uniform distribution on $(0,1)$ has mean $1/2$, we have $U(t) \approx 2(t-1)$, which explains (\ref{mtlim}).}
\end{Rmk}

Next we state our main result for the distribution of fitnesses of individuals in the population at a given time.  Let $\tau_0 = 0$ and for $j \in \N$, let
\begin{equation}\label{taujdef}
\tau_j = \inf\bigg\{t: X_{j-1}(t) \geq \frac{s_N}{\mu_N} \bigg\},
\end{equation}
which we will see later is approximately the time when some individuals with $j-1$ mutations start to acquire a $j$th mutation.  Also, let
\begin{equation}\label{gammajdef}
\gamma_j = \tau_j + a_N.
\end{equation}
We will see later that most individuals in the population between times $\gamma_j$ and $\gamma_{j+1}$ will have $j$ mutations.  For $t > 1$, let
\begin{equation}\label{jtdef}
j(t) = \max\{j: \gamma_j \leq a_N t\}.
\end{equation}
On the event that $\gamma_{j(t)} < \gamma_{j(t) + 1} < \infty$, which we will see later has probability close to one, let $d(t)$ be the number in $[-1/2, 1/2)$ such that
\begin{equation}\label{ddef}
a_Nt = (1/2 - d(t))\gamma_{j(t)} + (1/2 + d(t))\gamma_{j(t)+1}.
\end{equation}
Otherwise, let $d(t) = 0$.

\begin{Theo}\label{gaussthm}
Assume A1-A3 hold.  For each $\eta > 0$ and $t \in (1, 2) \cup (2, \infty)$, there exists $\theta > 0$ such that
\begin{align*}
&\lim_{N \rightarrow \infty} P \bigg( \bigg| \log \bigg( \frac{X_{j(t) + \ell}(a_Nt)}{X_{j(t)}(a_Nt)} \bigg) + \frac{[\log(s_N/\mu_N)]^2 (\ell^2 - 2 \ell d(t))}{2 q(t-1) \log N} \bigg| \leq \frac{\eta \ell^2 [\log(s_N/\mu_N)]^2}{\log N} \\
&\hspace{3.5in}\mbox{ for all }\ell \in [-\theta k_N, \theta k_N] \cap \Z \bigg) = 1.
\end{align*}
Furthermore, $\theta = \theta(\eta, t)$ can be chosen as a function of $\eta$ and $t$ such that for each fixed $\eta > 0$ and $a > 2$, we have
\begin{equation}\label{inftheta1}
\inf_{t \in [a, \infty)} \theta(\eta, t) > 0.
\end{equation}
\end{Theo}

Theorem \ref{gaussthm} compares the number of individuals with $j(t)$ mutations to the number of individuals with $j(t) + \ell$ mutations at time $a_N t$.  To see why this result is consistent with the conjecture from section \ref{prevsec} that the distribution of fitnesses is Gaussian, note that if $Z$ is a random variable having a Gaussian distribution with mean $x + d$ and variance $\sigma^2$, and $f$ is the probability density function of $Z$, then $$\log \bigg( \frac{f(x + \ell)}{f(x)} \bigg) = -\frac{1}{2 \sigma^2} [(\ell - d)^2 - d^2] = - \frac{\ell^2 - 2\ell d}{2 \sigma^2}.$$  Therefore, the result of Theorem \ref{gaussthm} suggests that, in some sense, the distribution of the fitnesses of individuals in the population at time $a_N t$ is approximately Gaussian with a mean of $j(t) + d(t)$ and a variance of $$\sigma_N^2(t) = \frac{q(t-1) \log N}{[\log(s_N/\mu_N)]^2}.$$  It should be noted, however, that (\ref{A2prime}) implies that $\lim_{N \rightarrow \infty} \sigma_N^2(t) = 0$.  Consequently, the distribution of fitnesses of individuals in the population at time $a_N t$ does not actually converge to a Gaussian distribution as $N \rightarrow \infty$.  Rather, the fraction of individuals in the population with exactly $j(t)$ mutations will be close to 1, unless $|d(t)|$ is very close to $1/2$.  Nevertheless, the appearance of $\ell^2 - 2 \ell d(t)$ in Theorem \ref{gaussthm} demonstrates Gaussian-like tail behavior.

\subsection{Notation}

We collect here for the convenience of the reader some of the most important notation used throughout the paper.  Because most of this notation has not yet been introduced, the reader is encouraged to skip this section for now and refer back to it as needed.

\begin{longtable}{ll}
$a_N$ & $(1/s) \log(s/\mu)$, natural time scale for the process \\
$b$ & Defined in (\ref{bdef}), used to determine which mutations are ``early" \\
$B_j(t)$ & Birth rate of a type $j$ individual at time $t$, see (\ref{Bjeq}) \\
$D_j(t)$ & Death rate of a type $j$ individual at time $t$, see (\ref{Djeq}) \\
$({\cal F}_t)_{t \geq 0}$ & Natural filtration of the process \\
$G_j(t)$ & $s(j - M(t)) - \mu$, growth rate of type $j$ population at time $t$ \\
$j(t)$ & $\max\{j: \gamma_j \leq a_N t\}$, corresponds to most common type at time $t$ \\
$J$ & $3k_N T + k^* + 1$, bound on number of types likely to appear by time $a_N T$ \\
$k_N$ & $\log N/\log(s/\mu)$, natural scale for the number of mutations \\
$k_N^-, k_N^+$ & numbers slightly smaller and larger than $k_N$, see (\ref{kNminus}) and (\ref{kNplus}) \\
$k^*$ & largest integer less than $k_N^+$ \\
$K$ & $\lfloor k_N/4 \rfloor$ \\
$L$ & $\lceil 17 k_N \rceil$ \\
$m(t)$ & Scaling limit of $(M(t), t \geq 0)$, defined in (\ref{mdef}) \\
$M(t)$ & Mean number of mutations in the population at time $t$ \\
${\bar M}(t)$ & Approximation to mean number of mutations at time $t$, defined in (\ref{Mbardef}) \\
$N$ & Population size \\
$Q(t)$ & Difference in number of mutations between fittest individual and average \\
$q(t)$ & Scaling limit of $(Q(t), t \geq 0)$, defined in (\ref{qdef}) \\
$q_j$ & Approximately the value of $Q(\tau_j)$, see (\ref{qjdef}) \\
$R(t)$ & Number of $\tau_j$ between $t - a_N$ and $t$, see (\ref{Rdef}) \\
$s = s_N$ & Selective advantage resulting from a mutation \\
$S_j(t)$ & Number of individuals with $j$ or fewer mutations \\
$t^*$ & Time before which individuals of types up to $k_N$ appear, defined in (\ref{tstardef}) \\
$T$ & Large positive number; the process is studied up to time $a_N T$ \\
$x_j(t)$ & Approximation to number of individuals with $j$ mutations at time $t$ for $t \leq t^*$ \\
$X_j(t)$ & Number of individuals with $j$ mutations at time $t$ \\
$X_{j,1}(t)$ & Number of type $j$ individuals at time $t$ descended from mutations before $\xi_j$ \\
$X_{j,2}(t)$ & $X_j(t) - X_{j,1}(t)$ \\
$Z_j(t)$ & Martingale associated with evolution of type $j$ individuals, see Proposition \ref{Zmart} \\
$\gamma_j$ & $\tau_j + a_N$, approximately when most individuals have acquired $j$ mutations \\
$\delta$ & Small positive number, bounds error in various approximations \\
$\eps$ & Small positive number, bounds probability that conclusions of results fail \\
$\zeta$ & First time that the conclusion of Proposition \ref{prop1}, \ref{prop2}, \ref{meanprop}, or \ref{tauprop} fails \\
$\mu = \mu_N$ & Mutation rate for each individual \\
$\xi_j$ & Time before which type $j$ mutations are early, see (\ref{xijdef}) \\
$\tau_j$ & $\inf\{t: X_{j-1}(t) \geq s/\mu\}$, approximately when type $j$ individuals appear \\
$\tau_j^*$ & $\tau_j + a_N/(4Tk_N)$ \\
\end{longtable}

\section{Heuristics}\label{heursec}

In this section, we discuss the key ideas behind the main results in the paper.  The goal is to explain to the reader, in just a few pages of calculations, why the main results are true.  Most of these heuristics have already appeared in the Biology literature, particularly in the work of Desai and Fisher \cite{df07}.  We postpone rigorous proofs of the results, and justification for the approximations used, until later sections, and in this section we assign no precise meaning to the approximation symbol $\approx$.  Here and throughout the rest of the paper, to lighten notation we write $\mu$ and $s$ in place of $\mu_N$ and $s_N$ respectively, even though these parameters depend on the population size $N$.

\subsection{The initial stage}\label{initialsec}

Consider first the initial stage of the process, when the average number of mutations in the population is close to zero.  For times $t$ in this range, we have $X_0(t) \approx N$ and $M(t) \approx 0$.  During this stage, we can approximate the process by a multitype branching process in which a type $j$ individual dies at rate 1, gives birth to another type $j$ individual at rate $1 + sj$, and mutates to type $j+1$ at rate $\mu$.  This means that the total rate at which type $j+1$ individuals appear due to mutations is $\mu X_{j-1}(t)$, and if such a mutation appears at time $u < t$, the expected number of descendants of this individual in the population at time $t$ is $e^{(sj - \mu)(t-u)} \approx e^{sj(t - u)}$, where the approximation is valid because $\mu$ is much smaller than $s$.  This leads to the approximation
\begin{equation}\label{smallexp1}
E[X_1(t)] \approx \int_0^t \mu E[X_0(t)] e^{s(t-u)} \: du \approx \int_0^t N \mu e^{s(t-u)} \: du = \frac{N \mu (e^{st} - 1)}{s}.
\end{equation}
Then an inductive argument gives
\begin{equation}\label{smallexpj}
E[X_j(t)] \approx \int_0^t \mu E[X_{j-1}(u)] e^{js(t - u)} \: du \approx \frac{N \mu^j}{s^j j!} (e^{st} - 1)^j.
\end{equation}

The approximation (\ref{smallexpj}) only holds when the mean number of mutations is close to zero, which can be true only when $X_1(t)$ is much smaller than $N$.  From (\ref{smallexp1}), we see that $X_1(t)$ will be of order $N$ when $e^{st}$ is comparable to $s/\mu$, which happens near time $$t = \frac{1}{s} \log \bigg( \frac{s}{\mu}\bigg) = a_N.$$  Before time $a_N$, the average number of mutations in the population will be close to zero, and the approximation (\ref{smallexpj}) will be valid.

For the approximation (\ref{smallexpj}) to be useful for understanding the evolution of the number of type $j$ individuals, we need to know that $X_j(t) \approx E[X_j(t)]$.  We will calculate, using a second moment argument, that this approximation holds for small times $t$ when $j \leq k_N$.  This is true essentially because, for $j \leq k_N$, type $j$ individuals appear in the population very quickly.  For larger values of $j$, however, it is not true that $X_j(t) \approx E[X_j(t)]$.  Rather, the expectation is dominated by rare events in which an individual acquires a $j$th mutation much earlier than usual, causing the number of type $j$ individuals at later times to be unusually large.  Therefore, for $j > k_N$, we can not approximate $X_j(t)$ by its expectation, and we need a different technique to understand the process $(X_j(t), t \geq 0)$. 

\subsection{Evolution of the number of type $j$ individuals}\label{evj}

We now consider the evolution of type $j$ individuals when $j > k_N$.  The key idea is to break the process into two stages: an initial stage in which the type $j$ population becomes established as a result of mutations experienced by type $j-1$ individuals, and a second stage in which these mutations are no longer important and the type $j$ population evolves essentially in a deterministic way.  This idea has been used in previous work on this model, and in particular many of the calculations in this section strongly resemble those in \cite{df07}.  Recall that $\tau_j$ is the first time when there are at least $s/\mu$ individuals of type $j-1$ in the population.  We will show using a first moment argument that with high probability, no type $j$ individuals will appear before time $\tau_j$.  The type $j$ population becomes established during the interval $[\tau_j, \tau_{j+1}]$, then evolves approximately deterministically after time $\tau_{j+1}$. 

After time $\tau_{j+1}$, we will see that mutations from type $j-1$ to type $j$ no longer have a significant impact on the size of the type $j$ population.  Consequently, at a time $u \geq \tau_{j+1}$, the number of type $j$ individuals will be growing approximately deterministically at the rate $s(j - M(u))$, which is the size of the selective advantage that a type $j$ individual has over an individual of average fitness.  That is, for $t \geq \tau_{j+1}$, we have
\begin{equation}\label{japprox}
X_j(t) \approx \frac{s}{\mu} e^{\int_{\tau_{j+1}}^t s(j - M(u)) \: du}.
\end{equation}

Consider next what happens between times $\tau_j$ and $\tau_{j+1}$, when the type $j$ population gets established.  We can use (\ref{japprox}) to approximate the number of type $j-1$ individuals shortly after time $\tau_j$.  As long as no type $j$ individual appears before time $\tau_j$, we have $(j-1) - M(\tau_j) = Q(\tau_j)$, so (\ref{japprox}) suggests the approximation
\begin{equation}\label{j-1approx}
X_{j-1}(t) \approx \frac{s}{\mu} e^{s Q(\tau_j) (t - \tau_j)}.
\end{equation}
As long as the average fitness of the population does not change much shortly after time $\tau_j$, a new type $j$ individual that appears because of a mutation at time $u$ will have on average $e^{s(Q(\tau_j) + 1)(t - u)}$ descendants at time $t$.  Thus, we have the approximation
\begin{align}\label{earlyXj}
X_j(t) &\approx \int_{\tau_j}^t \mu \cdot \frac{s}{\mu} e^{s Q(\tau_j) (u - \tau_j)} \cdot e^{s(Q(\tau_j) + 1)(t - u)} \: du \nonumber \\
&= s e^{s (Q(\tau_j) + 1) (t - \tau_j)} \int_{\tau_j}^t e^{-s (u - \tau_j)} \: du \nonumber \\
&\approx e^{s (Q(\tau_j) + 1) (t - \tau_j)},
\end{align}
where the last approximation requires $t - \tau_j \gg 1/s$.  Therefore, $\tau_{j+1}$ should occur approximately when the expression in (\ref{earlyXj}) equals $s/\mu$, which leads to
\begin{equation}\label{tauj}
\tau_{j+1} - \tau_j \approx \frac{1}{s(Q(\tau_j) + 1)} \log \bigg( \frac{s}{\mu} \bigg) \approx \frac{a_N}{Q(\tau_j)}.
\end{equation}

To estimate $Q(\tau_j)$, note that (\ref{j-1approx}) and (\ref{earlyXj}) lead to
$$\frac{X_j(t)}{X_{j-1}(t)} = \frac{e^{s (Q(\tau_j) + 1) (t - \tau_j)}}{(s/\mu) e^{s Q(\tau_j) (t - \tau_j)}} = \frac{\mu}{s} e^{s(t - \tau_j)},$$ which equals one when $$t - \tau_j = \frac{1}{s} \log \bigg( \frac{s}{\mu}\bigg) = a_N.$$  That is, the number of type $j$ individuals surpasses the number of type $j-1$ individuals approximately $a_N$ time units after type $j$ individuals first appear.  Around that time, there will be more type $j$ individuals than individuals of any other type, and the mean number of mutations in the population will be approximately $j$.  It follows that $M(\tau_j)$ will be approximately the type that first appeared roughly $a_N$ time units in the past, and $Q(\tau_j)$ will be approximately the number of new types that have appeared in the last $a_N$ time units.  Because the rate per unit time at which new types are appearing can be approximated by the reciprocal of the expression in (\ref{tauj}), we obtain for $t > 1$ the approximation
\begin{equation}\label{Qapprox}
Q(a_N t) \approx \int_{a_N(t - 1)}^{a_Nt} \frac{Q(u)}{a_N} \: du = \int_{t-1}^t Q(a_N v) \: dv.
\end{equation}
For $t < 1$, we know from the discussion in section \ref{initialsec} that $M(a_Nt) \approx 0$, so $Q(a_Nt)$ is approximately the number of types that have originated before time $a_N t$.  Since we know from the discussion in section \ref{initialsec} that $k_N$ types appear at very small times, we have for $t < 1$ the approximation $$Q(a_N t) \approx k_N + \int_0^{a_N t} \frac{Q(u)}{a_N} \: du = k_N + \int_0^t Q(a_N v) \: dv,$$ which implies $Q(a_N t) \approx k_N e^t$.  This result and (\ref{Qapprox}) lead to the approximation to $Q(a_N t)$ in Theorem \ref{Qthm}.  The result (\ref{qlim}) then follows from the renewal theory argument outlined in Remark \ref{renewalrmk}.

To understand Theorem \ref{speedthm}, recall again that $M(a_N t) \approx 0$ for $t < 1$.  For $t > 1$, we know from the discussion in the previous paragraph that $M(a_N t)$ is approximately the number of types that appear before time $a_N(t - 1)$.  Because $k_N$ types appear near time zero and the rate at which new types appear can be approximated by the reciprocal of the expression in (\ref{tauj}), we get for $t > 1$ the approximation $$M(a_N t) \approx k_N + \int_0^{a_N(t - 1)} \frac{Q(u)}{a_N} \: du = k_N + \int_0^{t-1} Q(a_N v) \: dv,$$ which leads to Theorem \ref{speedthm}.  To obtain the result of Theorem \ref{gaussthm}, we use the approximation (\ref{japprox}) to compare $X_{j(t)+\ell}$ and $X_{j(t)}$.  We refer the reader to the proof of Theorem \ref{gaussthm} in subsection \ref{tsec3} for the details of this calculation.

Although the main ideas discussed in this section come from \cite{df07}, it has been assumed in most previous work on this model such as \cite{df07, rbw08} that the population is already in equilibrium.  Then one can argue that this equilibrium is only possible when (\ref{Qlarget}) and (\ref{speed}) hold.  One of the contributions of the present work is to show how the process arrives at such a state, beginning from a population in which no mutations are present.

\subsection{Meaning of the assumptions}\label{meaningass}

We briefly discuss here the assumptions required for these results to be valid.  Note that (\ref{A1prime}) is equivalent to the condition $$\lim_{N \rightarrow \infty} k_N = \infty.$$  Since $Q(a_N t)$ is of the order $k_N$, assumption A1 implies that the number of different types in the population at a given time tends to infinity as $N \rightarrow \infty$.  This condition is not satisfied in the parameter regime considered by Durrett and Mayberry \cite{dm11}.  Assumption A1 also ensures $s_N$ is large enough for mutations to take hold in the population in the manner described above. 

For the heuristics described in section \ref{evj} to be valid, the type $j$ population must be growing approximately exponentially after time $\tau_{j+1}$, which will happen as long as additional mutations from type $j-1$ to type $j$ are no longer having a significant impact on the population size.  The contribution to the type $j$ population from mutations at different times can be seen from the integral in the second line of (\ref{earlyXj}).  The primary contribution to this integral comes when $u$ is comparable to $1/s$.  Consequently, we need $\tau_{j+1} - \tau_j \gg 1/s$ for the number of type $j$ individuals to be growing exponentially after time $\tau_{j+1}$.  In view of (\ref{tauj}) and the fact that $Q(\tau_j)$ is the same order of magnitude as $k_N$, this is equivalent to the condition $$\frac{1}{s k_N} \log \bigg( \frac{s}{\mu} \bigg) \gg \frac{1}{s},$$ which is equivalent to (\ref{A2prime}).  Thus, the role of assumption A2 is to ensure that the mutation rate $\mu$ is slow enough that we can ignore mutations from type $j-1$ to type $j$ after time $\tau_{j+1}$.  For technical reasons, assumption A2 is slightly stronger than (\ref{A2prime}), but we conjecture that the main results of the paper are still true if assumption A2 is replaced by (\ref{A2prime}).  It remains an open question to understand how the process evolves if the mutation rates are fast enough that (\ref{A2prime}) fails to hold.  

Assumption A3 is equivalent to the condition
\begin{equation}\label{skN}
\lim_{N \rightarrow \infty} s k_N = 0.
\end{equation}
Because the difference in fitness between the fittest individual and an individual of average fitness is of the order $s k_N$, assumption A3 implies that we are not considering very strong selection.
 
\section{Structure of the Proofs}\label{strucsec}

In this section, we state some intermediate results that will lead to the proofs of the main results.  Some of these intermediate results may also be of independent interest, as they provide some insight into how the number of individuals with $j$ mutations evolves over time.  Throughout the section, we will fix three positive numbers: $\eps$, $\delta$, and $T$.  We will use $\eps \in (0, 1)$ for the maximum allowable probability of some ``bad" event and
\begin{equation}\label{deltadef}
0 < \delta < \frac{1}{100}
\end{equation}
for the maximum allowable error in certain approximations.  We will study the process up to time $a_N T$, where $T > 1$.  
Throughout the paper, we will introduce some positive constants $C_n$.  These constants may depend on the three parameters $\eps$, $\delta$, and $T$, even though this dependence will not be specifically mentioned each time.

On numerous occasions throughout the paper, we will assert that a statement holds ``for sufficiently large $N$".  This means that there exists a positive integer $N_0$, depending on $\eps$, $\delta$, and $T$, such that the statement in question holds for $N \geq N_0$.  Often the statement in question will somehow involve the evolution of type $j$ individuals, where $j$ could take values in a certain range, typically $0 \leq j \leq k^*$ or $k^* + 1 \leq j \leq J$, where $k^*$ and $J$ are defined below.  The statement may also involve the time $t$, which may be permitted to take values in a certain range.  In such cases, the value of $N_0$ may not depend on $j$ or $t$.  That is, the same $N_0$ must work for all $j$ and $t$ in the indicated ranges.

\subsection{The process until time $t^*$}

We begin by considering the initial stage of the process.  Recall from subsection \ref{initialsec} that for $j \leq k_N = \log N/\log(s/\mu)$, we expect individuals of type $j$ to appear in the population very early, and we expect the number of type $j$ individuals to be well approximated by the right-hand side of (\ref{smallexpj}).  To state a precise result, define
\begin{equation}\label{kNminus}
k_N^- = \frac{\log N}{\log(s/\mu)} - \frac{\log N}{\log(s/\mu)^2} \log \bigg( \frac{\log N}{\log(s/\mu)} \bigg)
\end{equation}
and
\begin{equation}\label{kNplus}
k_N^+ = \frac{\log N}{\log(s/\mu)} + \frac{2 \log N}{\log(s/\mu)^2} \log \bigg( \frac{\log N}{\log(s/\mu)} \bigg).
\end{equation}
Also, let
\begin{equation}\label{kstar}
k^* = \max\{j \in \N: j < k_N^+\}
\end{equation}
be the largest integer less than $k_N^+$.  Assumption A2 implies that
\begin{equation}\label{kdiff}
\lim_{N \rightarrow \infty} (k_N^+ - k_N^-) = 0,
\end{equation}
so for sufficiently large $N$, the number of integers $j$ such that $k_N^- < j < k_N^+$ must be either zero or one.
Define the time 
\begin{equation}\label{tstardef}
t^* = \left\{
\begin{array}{ll} (4/s) \log k_N & \mbox{ if there exists an integer $j$ such that }k_N^- < j < k_N^+  \\
(2/s) \log k_N & \mbox{ otherwise}
\end{array} \right.
\end{equation}
The following proposition, which we prove in section \ref{earlysec}, describes how the process evolves before time $t^*$.

\begin{Prop}\label{earlyprop}
For all nonnegative integers $j$ and all $t \geq 0$, define
\begin{equation}\label{Xjbardef}
x_j(t) = \frac{N \mu^j (e^{st} - 1)^j}{s^j j!}.
\end{equation}
Then there exist positive constants $C_1$ and $C_2$ such that for sufficiently large $N$, the following four statements all hold with probability at least $1 - \eps/2$:
\begin{enumerate}
\item For all $j \leq k_N^-$, we have 
\begin{equation}\label{early1}
\sup_{t \in [0, t^*]} |X_j(t) - x_j(t)| \leq \delta x_j(t^*).
\end{equation}

\item For all $j \in (k_N^-, k_N^+)$, write
\begin{equation}\label{bjdef}
j = \frac{\log N}{\log(s/\mu)} + \frac{b_j \log N}{\log(s/\mu)^2} \log \bigg( \frac{\log N}{\log(s/\mu)} \bigg),
\end{equation}
where $-1 < b_j < 2$, and let $d_j = \max\{0, b_j\}$.  Then
\begin{equation}\label{earlypt2}
C_1 k_N^{-d_j} x_j(t^*) \leq X_j(t^*) \leq C_2 k_N^{-d_j} x_j(t^*).
\end{equation}

\item For all $t \in [0, t^*]$, we have $X_{k^*}(t) < s/\mu$.

\item For all $j \geq k_N^+$ and $t \in [0, t^*]$, we have $X_j(t) = 0$.
\end{enumerate}
\end{Prop}

\subsection{Evolution of type $j$ individuals}

In this subsection, we consider how the population evolves after time $t^*$.  Recall the definitions of $\tau_j$ and $\gamma_j$ from (\ref{taujdef}) and (\ref{gammajdef}).    For nonnegative integers $j$ and $t \geq 0$, define also
\begin{equation}\label{Gdef}
G_j(t) = s(j - M(t)) - \mu,
\end{equation}
which we can interpret as the rate of growth for the number of type $j$ individuals at time $t$.  We will also define the integers
$$K = \lfloor k_N/4 \rfloor, \hspace{.2in}L = \lceil 17 k_N \rceil.$$
The next proposition describes the evolution, after time $t^*$, of the number of individuals with $k^*$ or fewer mutations.  The first part of the proposition controls the evolution of the type $j$ individuals after time $t^*$.  The second and third parts provide upper bounds on the number of type $j$ individuals as these individuals get close to extinction.

\begin{Prop}\label{prop1}
For sufficiently large $N$ the following statements all hold with probability at least $1 - \eps$:
\begin{enumerate}
\item For all $j \leq k^*$ and $t \in [t^*, \gamma_{k^* + K}]$, we have
\begin{equation}\label{prop11}
(1 - \delta) X_j(t^*) \exp \bigg( \int_{t^*}^t G_j(v) \: dv \bigg) \leq X_j(t) \leq (1 + \delta) X_j(t^*) \exp \bigg( \int_{t^*}^t G_j(v) \: dv \bigg).
\end{equation}

\item For all $j \leq k^*$ and $t \in [\gamma_{k^* + K}, a_N T]$, we have
\begin{equation}\label{prop12}
X_j(t) \leq k_N^2 X_j(t^*) \exp \bigg( \int_{t^*}^t G_j(v) \: dv \bigg).
\end{equation}

\item On the event that $\gamma_{k^* + L} \leq a_N T$, we have $X_j(t) = 0$ for all $j \leq k^*$ and $t \geq \gamma_{k^* + L}$.
\end{enumerate}
\end{Prop}

We next consider the individuals of type $j$ for $j \geq k^* + 1$.  By part 4 of Proposition \ref{earlyprop}, individuals of these types typically do not appear until after time $t^*$, so we need to consider how these types originate.  Define the positive number
\begin{equation}\label{bdef}
b = \log \bigg(\frac{24000 \,T}{\delta^2 \eps} \bigg).
\end{equation}
For $j \geq k^* + 1$, let 
\begin{displaymath}
q_j^* = \left\{
\begin{array}{ll} j - k_N & \mbox{ if } a_N - 2a_N/k_N \leq \tau_j \leq a_N + 2a_N/k_N \\
j - M(\tau_j) & \mbox{ otherwise }
\end{array} \right.
\end{displaymath}
and
\begin{equation}\label{qjdef}
q_j = \max\{1, q_j^*\}.
\end{equation}
Then define
\begin{equation}\label{xijdef}
\xi_j = \max \bigg\{\tau_j, \: \tau_j + \frac{1}{sq_j} \log \bigg( \frac{1}{sq_j} \bigg) + \frac{b}{sq_j} \bigg\}.
\end{equation}

When an individual with $j-1$ mutations gets an additional mutation, we call this a type $j$ mutation.  Each type $j$ individual in the population at time $t$ has an ancestor that got a type $j$ mutation at some earlier time.  We call the individual an early type $j$ individual if this type $j$ mutation happened at or before time $\xi_j$.  Let $X_{j,1}(t)$ be the number of early type $j$ individuals at time $t$, and let $X_{j,2}(t)$ be the number of other type $j$ individuals at time $t$.  This means, of course, that $$X_j(t) = X_{j,1}(t) + X_{j,2}(t).$$  Also, define the time
\begin{equation}\label{taustar}
\tau_j^* = \tau_j + \frac{a_N}{4T k_N}.
\end{equation}

The result below describes the evolution of the type $j$ individuals for $j \geq k^* + 1$.  The first two parts of the proposition concern the evolution of the type $j$ individuals up to time $\tau_{j+1}$ and require classifying the type $j$ individuals as being early or not early.  The remaining three parts parallel the three parts of Proposition \ref{prop1}.

\begin{Prop}\label{prop2}
There exists a positive constant $C_3$ such that for sufficiently large $N$, the following statements all hold with probability at least $1 - \eps$:
\begin{enumerate}
\item For all $j \geq k^* + 1$ and all $t \in [\tau_j^*, \tau_{j+1}] \cap [0, a_N T]$, we have
\begin{equation}\label{prop21}
X_{j,1}(t) \leq C_3 \exp \bigg( \int_{\tau_j}^t G_j(v) \: dv \bigg).
\end{equation}
Also, $X_{j,1}(t) \leq s/2\mu$ for all $t \leq \tau_j^* \wedge a_N T$, and no early type $j$ individual acquires a type $j+1$ mutation until after time $\tau_{j+1} \wedge a_N T$.

\item For all $j \geq k^* + 1$ and all $t \in [\tau_j^*, \tau_{j+1}] \cap [0, a_N T]$, we have
\begin{equation}\label{prop22}
(1 - 4 \delta) \exp \bigg( \int_{\tau_j}^t G_j(v) \: dv \bigg) \leq X_{j,2}(t) \leq (1 + 4 \delta) \exp \bigg( \int_{\tau_j}^t G_j(v) \: dv \bigg).
\end{equation}
Moreover, the upper bound holds for all $t \in [\xi_j, \tau_{j+1}] \cap [0, a_N T]$.

\item For all $j \geq k^* + 1$ and all $t \in [\tau_{j+1}, \gamma_{j+K}] \cap [0, a_N T]$, we have
\begin{equation}\label{prop23}
\frac{(1 - \delta)s}{\mu} \exp \bigg( \int_{\tau_{j+1}}^t G_j(v) \: dv \bigg) \leq X_j(t) \leq \frac{(1 + \delta)s}{\mu} \exp \bigg( \int_{\tau_{j+1}}^t G_j(v) \: dv \bigg).
\end{equation}

\item For all $j \geq k^* + 1$ such that $\gamma_{j + K} \leq a_N T$, we have
\begin{equation}\label{prop24}
X_j(t) \leq \frac{k_N^2 s}{\mu} \exp \bigg( \int_{\tau_{j+1}}^t G_j(v) \: dv \bigg)
\end{equation}
for all $t \in [\gamma_{j + K}, a_N T]$.

\item For all $j \geq k^* + 1$ such that $\gamma_{j+L} \leq a_N T$, we have $X_j(t) = 0$ for all $t \geq \gamma_{j+L}$.
\end{enumerate}
\end{Prop}

\begin{Rmk}\label{tauorder}
{\em When the statement of part 1 of Proposition \ref{prop2} holds, the number of early type $j$ individuals can not reach $s/\mu$ until after time $\tau_j \wedge a_N T$, and because $\xi_j \geq \tau_j$ by definition, no other type $j$ individuals appear until after time $\tau_j$.  It follows that if $j \geq k^*+1$, then $\tau_{j+1} \geq \tau_j \wedge a_N T$.}
\end{Rmk}

The next proposition shows how the mean number of mutations in the population evolves over time.  Note that the mean number of mutations in the population is near zero before time $a_N$ and is near $j$ during the time interval $[\gamma_j, \gamma_{j+1})$.

\begin{Prop}\label{meanprop}
There exist positive constants $C_4$ and $C_5$ such that for sufficiently large $N$, the following statements all hold with probability at least $1 - \eps$:
\begin{enumerate}
\item For all $t \in (t^*, a_N]$, we have $M(t) < 3 e^{-s(a_N - t)}$.

\item For all $t \in (a_N, \gamma_{k^*+1})$, we have $M(t) < k_N + C_4$.  

\item For all $j \geq k^* + 1$ and $t \in [\gamma_j, \gamma_{j+1}) \cap [0, a_N T]$, we have
\begin{equation}\label{meaneq}
|M(t) - j| < C_5 (e^{-s(t - \gamma_j)} + e^{-s(\gamma_{j+1} - t)}).
\end{equation}

\item For all $j \geq k^* + 1$ and $t \in [\tau_j, \tau_{j+1})$, we have $M(t) < j - 1$.
\end{enumerate}
\end{Prop}

Next, we state a result concerning the differences $\tau_{j+1} - \tau_j$.  Here $q$ is the function defined in (\ref{qdef}).

\begin{Prop}\label{tauprop}
For $t \in [0, a_NT]$, let
\begin{equation}\label{Rdef}
R(t) = k^* \1_{\{t < a_N\}} + \#\{j \geq k^* + 1: t - a_N < \tau_j \leq t\},
\end{equation}
where $\#S$ denotes the cardinality of a set $S$.
For sufficiently large $N$, the following statements all hold with probability at least $1 - \eps$:
\begin{enumerate}
\item We have $\tau_{k^* + 1} \leq 2 a_N/k_N$.

\item We have $$\sup_{t \in [0, T]} \bigg| \frac{R(a_N t)}{k_N} - q(t) \bigg| < \delta.$$

\item For all $j \geq k^* + 1$ such that either $\tau_j + 2a_N/k_N \leq a_N T$ or $\tau_{j+1} \leq a_N T$, we have
\begin{equation}\label{tautight1}
\int_{\tau_j/a_N}^{\tau_{j+1}/a_N} q(t) \: dt \leq \frac{1 + 2 \delta}{k_N}
\end{equation}
and
\begin{equation}\label{tautight2}
\int_{\tau_j/a_N}^{\tau_{j+1}/a_N} (q(t) + \1_{\{t \in [1, \gamma_{k^* + 1}/a_N)\}}) \: dt \geq \frac{1- 2\delta}{k_N}.
\end{equation}
In particular,
\begin{equation}\label{tauspacing}
\frac{a_N}{3k_N} \leq \tau_{j+1} - \tau_j \leq \frac{2a_N}{k_N}.
\end{equation}
\end{enumerate}
\end{Prop}

\begin{Rmk}\label{JRmk}
{\em Let
\begin{equation}\label{Jdef}
J = 3 k_N T + k^* + 1.
\end{equation}
If the statement of part 3 of Proposition \ref{tauprop} holds, then (\ref{tauspacing}) implies that
\begin{equation}\label{JRemeq}
\tau_J > \tau_J - \tau_{k^* + 1} \geq \frac{a_N}{3k_N} (J - (k^* + 1)) \wedge a_N T = a_N T.
\end{equation}
Assuming, in addition, that the last statement of part 1 of Proposition \ref{prop2} holds, it follows that no individual of type $J+1$ or higher can appear until after time $a_N T$.  Consequently, throughout the paper, it will usually only be necessary to consider individuals of type $j$ for $0 \leq j \leq J$.}
\end{Rmk}

In section \ref{3propsec}, we show how Theorems \ref{Qthm}, \ref{speedthm}, and \ref{gaussthm} follow from Propositions \ref{prop1}, \ref{prop2}, \ref{meanprop}, and \ref{tauprop}. 

\subsection{Waiting for the time $\zeta$}

Although Proposition \ref{earlyprop} is proved in section \ref{earlysec} independently of the other results in this section, it does not seem to be possible to prove Propositions \ref{prop1}, \ref{prop2}, \ref{meanprop}, and \ref{tauprop} sequentially.  Proving Propositions \ref{prop1} and \ref{prop2} requires that we have some control over the quantities $M(t)$ and $\tau_{j+1} - \tau_j$, which are established in Propositions \ref{meanprop} and \ref{tauprop}.  On the other hand, to prove Propositions \ref{meanprop} and \ref{tauprop}, it will be necessary to have control over the quantities $X_j(t)$, as established by Propositions \ref{prop1} and \ref{prop2}.  Consequently, we will prove these propositions simultaneously by defining a random time $\zeta$ which will be the first time that one of the statements in the above propositions fails.  We will then show that $\zeta > a_N T$ with high probability.

Choose constants $C_1$ and $C_2$ as in Proposition \ref{earlyprop}.  Let
\begin{align*}
\zeta_0 &= \inf\{t \leq t^*: \mbox{ either }|X_j(t) - x_j(t)| > \delta x_j(t^*)\mbox{ for some }j \leq k_N^-, \\
&\hspace{.5in} t = t^* \mbox{ and } (\ref{earlypt2}) \mbox{ fails to hold for some } j \in (k_N^-, k_N^+), \\
&\hspace{1in} X_{k^*}(t) \geq s/\mu, \\
&\hspace{1.5in} \mbox{ or }X_j(t) > 0 \mbox{ for some }j \geq k_N^+\}.
\end{align*}
Note that $\zeta_0 = \infty$ if the four statements of Proposition \ref{earlyprop} all hold.

Next, for all nonnegative integers $j$, we will define a random time $\zeta_{1,j}$, which is essentially the first time that the behavior of the type $j$ individuals violates the conditions of Proposition \ref{prop1} or Proposition \ref{prop2}.  First consider $j \leq k^*$.  For $t \in [t^*, \gamma_{k^* + K}]$, let $A_j(t)$ be the event that (\ref{prop11}) fails to hold.  For $t \in (\gamma_{k^* + K}, \gamma_{k^* + L})$, let $A_j(t)$ be the event that (\ref{prop12}) fails to hold.  For $t \geq \gamma_{k^* + L}$, let $A_j(t)$ be the event that $X_j(t) > 0$.  Now consider $j \geq k^* + 1$.  Choose a constant $C_3$ as in Proposition \ref{prop2}.  For $t \geq t^*$, we say that $A_j(t)$ occurs if $t \in [\tau_j^*, \tau_{j+1}]$ and (\ref{prop21}) or (\ref{prop22}) fails to hold, if $t \in [\xi_j, \tau_{j+1}]$ and the upper bound in (\ref{prop22}) fails to hold, if $t \leq \tau_j^*$ and $X_{j,1}(t) > s/2\mu$, if $t \leq \tau_{j+1}$ and an early type $j$ individual acquires a type $j+1$ mutation at time $t$, if $t \in [\tau_{j+1}, \gamma_{j+K}]$ and (\ref{prop23}) fails to hold, if $t  \geq \gamma_{j+K}$ and (\ref{prop24}) fails to hold, or if $t \geq \gamma_{j+L}$ and $X_j(t) > 0$.  Then let $$\zeta_{1,j} = \inf\{t: A_i(t) \mbox{ occurs for some }i \leq j\}$$ and $$\zeta_1 = \inf\{\zeta_{1,j}: 0 \leq j \leq J\}.$$

Next, we will define $\zeta_2$ to be the first time when the result of Proposition \ref{meanprop} fails.  More precisely, choose $C_4$ and $C_5$ as in Proposition \ref{meanprop}, and let
\begin{align*}
\zeta_2 &= \inf\{t: \mbox{ either }t \in (t^*, a_N] \mbox{ and }M(t) \geq 3e^{-s(a_N - t)}, \\
&\hspace{.5in} t \in (a_N, \gamma_{k^* + 1}) \mbox{ and }M(t) \geq k_N + C_4, \\
&\hspace{1in} \mbox{ for some $j \geq k^* + 1$ we have }t \in [\gamma_j, \gamma_{j+1}) \mbox{ but (\ref{meaneq}) fails to hold}\}, \\
&\hspace{1.5in} \mbox{ or for some $j \geq k^* + 1$ we have }t \in [\tau_j, \tau_{j+1}) \mbox{ but }M(t) \geq j - 1\}.
\end{align*}
Also, let
\begin{align*}
\zeta_3 &= \inf\{t: \mbox{ either }t = 2a_N/k_N \mbox{ and }\tau_{k^* + 1} > 2a_N/k_N, \\
&\hspace{.2in} |R(t)/k_N - q(t/a_N)| \geq \delta, \\
&\hspace{.4in} \mbox{ there exists $j \geq k^* + 1$ such that }\tau_{j+1} \leq t \mbox{ but (\ref{tautight1}), (\ref{tautight2}), or (\ref{tauspacing}) fails to hold}, \\
&\hspace{.6in} \mbox{ or there exists $j \geq k^* + 1$ such that }\tau_{j+1} > t \mbox{ and }t = \tau_j + 2a_N/k_N\},
\end{align*}
which can be interpreted as the first time when Proposition \ref{tauprop} fails.  Finally, let
\begin{equation}\label{zetadef}
\zeta = \min\{\zeta_0, \zeta_1, \zeta_2, \zeta_3\}.
\end{equation}
Note that $\zeta$ depends on $\delta$ and depends also on $\eps$ and $T$ through the choice of $b$ in (\ref{bdef}).  Also, $\zeta$ depends on the constants $C_1, \dots, C_5$.  The constants $C_1$ and $C_2$ are chosen independently of the others in Proposition \ref{early2prop} below.  The constant $C_3$ is specified below in (\ref{C3def}).  The constants $C_4$ and $C_5$, which depend on $C_3$, are obtained below in Propositions \ref{tint2} and \ref{tint3} respectively.

\begin{Prop}\label{zetaprop}
There exist positive constants $C_1, \dots, C_5$ such that for sufficiently large $N$, the following hold:
\begin{enumerate}
\item On the event $\{\zeta_0 = \infty\}$, either $\zeta_2 \geq \zeta_1 \wedge \zeta_3$ or $\zeta_1 \wedge \zeta_2 \wedge \zeta_3 > a_N T$.

\item On the event $\{\zeta_0 = \infty\}$, either $\zeta_3 \geq \zeta_1 \wedge \zeta_2$, with strict inequality on the event $\{\zeta_2 < \zeta_1\}$, or else $\zeta_1 \wedge \zeta_2 \wedge \zeta_3 > a_N T$.

\item We have $$\sum_{j=0}^J P(\{\zeta_0 = \infty\} \cap \{\zeta_{1,j} \leq \zeta_2 \wedge \zeta_3 \wedge \zeta_{1,j-1} \wedge a_N T\}) < \frac{\eps}{2},$$ using the convention that $\zeta_{1,-1} = \infty$.
\end{enumerate}
\end{Prop}

We prove parts 1, 2, and 3 of Proposition \ref{zetaprop} in sections \ref{zetasec1}, \ref{zetasec2}, and \ref{zetasec3} respectively.  Here we show how Proposition \ref{zetaprop}, along with Proposition \ref{earlyprop}, implies Propositions \ref{prop1}, \ref{prop2}, \ref{meanprop}, and \ref{tauprop}.  Essentially, parts 1, 2, and 3 of Proposition \ref{zetaprop} show that $\zeta_2$, $\zeta_3$, and $\zeta_1$ respectively are unlikely to be the first of these three times to occur.  This forces $\zeta$ to be pushed beyond time $a_N T$ with high probability.  Note that a consequence of this result is that for sufficiently large $N$, the conclusions of Propositions \ref{earlyprop}, \ref{prop1}, \ref{prop2}, \ref{meanprop}, and \ref{tauprop} simultaneously hold with probability at least $1 - \eps$.

\begin{proof}[Proof of Propositions \ref{prop1}, \ref{prop2}, \ref{meanprop}, and \ref{tauprop}]
By Proposition \ref{earlyprop}, we have $P(\zeta_0 = \infty) > 1 - \eps/2$ for sufficiently large $N$.  By Proposition \ref{zetaprop}, we have $$P \bigg( \{\zeta_0 = \infty\} \cap \bigg( \bigcup_{j=0}^{J} \{\zeta_{1,j} \leq \zeta_2 \wedge \zeta_3 \wedge \zeta_{1,j-1} \wedge a_N T\} \bigg) \bigg) < \frac{\eps}{2}$$ for sufficiently large $N$.  Combining these results, we get
\begin{equation}\label{zetaevent}
P \bigg( \{\zeta_0 = \infty\} \cap \bigg( \bigcap_{j=0}^J \{ \zeta_{1,j} > \zeta_2 \wedge \zeta_3 \wedge \zeta_{1,j-1} \wedge a_N T\} \bigg) \bigg) > 1 - \eps
\end{equation}
for sufficiently large $N$.

We claim that for sufficiently large $N$, we must have $\zeta > a_N T$ on the event in (\ref{zetaevent}).  By part 1 of Proposition \ref{zetaprop}, for sufficiently large $N$, on $\{\zeta_0 = \infty\} \cap \{\zeta \leq a_N T\}$, we must have $\zeta_2 \geq \zeta_1 \wedge \zeta_3$, which implies either $\zeta = \zeta_1$ or $\zeta = \zeta_3$.  Likewise, part 2 of Proposition \ref{zetaprop} implies that for sufficiently large $N$, on $\{\zeta_0 = \infty\} \cap \{\zeta \leq a_N T\}$, either $\zeta = \zeta_1$ or $\zeta = \zeta_2$.  Since the strict inequality required by part 2 of Proposition \ref{zetaprop} rules out the possibility that $\zeta_2 = \zeta_3 < \zeta_1$, it follows that for sufficiently large $N$, on $\{\zeta_0 = \infty\} \cap \{\zeta \leq a_N T\}$, we must have $\zeta = \zeta_1$, and therefore $\zeta = \zeta_{1,j}$ for some $j$.  However, on the event in (\ref{zetaevent}), we see by induction on $j$ that we can not have $\zeta = \zeta_{1,j}$ for any $j \leq J$. 

Hence, for sufficiently large $N$, we have $\zeta > a_N T$ on the event in (\ref{zetaevent}).  Thus, by (\ref{zetaevent}), for such $N$ we have
\begin{equation}\label{probzeta}
P(\zeta > a_N T) > 1 - \eps.
\end{equation}
Propositions \ref{prop1}, \ref{prop2}, \ref{meanprop}, and \ref{tauprop} follow from (\ref{probzeta}).  Note that Remark \ref{JRmk} implies that on $\{\zeta > a_N T\}$, no individual of type $J+1$ or higher appears until after time $a_N T$, which is why it is only necessary to consider $\zeta_{1,j}$ for $0 \leq j \leq J$.
\end{proof}

\section{A useful martingale}

In this section, we introduce a martingale which will be useful throughout the paper for controlling the fluctuations of the number of type $j$ individuals in the population.

\subsection{Constructing the martingale}

We first record the birth and death rates for different types of individuals.  Let $F_j(t)$ be the fitness of a type $j$ individual at time $t$, which is $\max\{0, 1 + s(j - M(t))\}$, divided by the sum of the fitnesses of the $N$ individuals in the population.  Note that, if there is a birth event at time $t$, then $F_j(t-)$ is the probability that a particular type $j$ individual is the one chosen to give birth.  As long as every individual's fitness is strictly positive, the sum of the fitnesses of the $N$ individuals in the population is
\begin{align*}
\sum_{j=0}^{\infty} X_j(t)(1 + s(j - M(t))) &= \sum_{j=0}^{\infty} X_j(t) + s \sum_{j=0}^{\infty} j X_j(t) - s M(t) \sum_{j=0}^{\infty} X_j(t) \\
&= N + s M(t) N - s M(t) N \\
&= N,
\end{align*}
in which case $F_j(t) = (1 + s(j - M(t)))/N$.

There are three ways that the number of type $j$ individuals could change at time $t$:
\begin{enumerate}
\item If $j \geq 1$, a type $j-1$ individual could acquire a $j$th mutation at time $t$.  This event happens at rate $\mu X_{j-1}(t-)$.  So that our formulas hold also when $j = 0$, we adopt the convention that $X_{-1}(t) = 0$ for all $t \geq 0$.

\item The number of type $j$ individuals could increase by one because of a birth.  This happens if one of the $N - X_j(t-)$ individuals that is not type $j$ dies at time $t$, and the new individual born has type $j$.  Because each individual dies at rate $1$, and when a death occurs, the probability that a type $j$ individual is born is $X_j(t-) F_j(t-)$, the rate at which new type $j$ individuals are born is $(N - X_j(t-)) X_j(t-) F_j(t-)$.  We define
\begin{equation}\label{Bjeq}
B_j(t) = (N - X_j(t)) F_j(t),
\end{equation}
which can be interpreted as the rate at which a particular type $j$ individual gives birth following the death of an individual with a different type.

\item The number of type $j$ individuals could decrease by one because of a mutation or death.  The rate at which one of the type $j$ individuals acquires a $(j+1)$st mutation is $\mu X_j(t-)$.  The rate at which the number of type $j$ individuals decreases due to a death is given by $X_j(t-)(1 - X_j(t-)F_j(t-))$ because there are $X_j(t-)$ type $j$ individuals each dying at rate one, and when a death occurs, the probability that the new individual born is not a type $j$ individual is $1 - X_j(t-) F_j(t-)$.  Thus, the total rate of events that reduce the number of type $j$ individuals is $\mu X_j(t-) + X_j(t-)(1 - X_j(t-) F_j(t-))$.  We define
\begin{equation}\label{Djeq}
D_j(t) = \mu + 1 - X_j(t)F_j(t),
\end{equation}
which can be interpreted as the rate at which a particular type $j$ individual either acquires a mutation or dies and gets replaced by an individual with a different type.
\end{enumerate}
Let $X_j^b(t)$ be the number of times in $[0, t]$ that the number of type $j$ individuals increases by one.  Let $X_j^d(t)$ be the number of times in $[0, t]$ that the number of type $j$ individuals decreases by one.  Then $X_0(t) = N + X_0^b(t) - X_0^d(t)$ for all $t \geq 0$, and $X_j(t) = X_j^b(t) - X_j^d(t)$ for all $j \in \N$ and $t \geq 0$.

From the rates obtained above, we see that if we define
\begin{equation}\label{Wbeq}
W^b_j(t) = X^b_j(t) - \int_0^t (\mu X_{j-1}(u) + B_j(u) X_j(u)) \: du
\end{equation}
and
\begin{equation}\label{Wdeq}
W^d_j(t) = X_j^d(t) - \int_0^t D_j(u) X_j(u) \: du,
\end{equation}
then the processes $(W_j^b(t), t \geq 0)$ and $(W_j^d(t), t \geq 0)$ are martingales for all $j \in \Z^+$.  Therefore, if we define $W_j(t) = W_j^b(t) - W_j^d(t)$ for all $t \geq 0$, then the process $(W_j(t), t \geq 0)$ is a martingale for all $j \in \Z^+$.  Let $\Delta W_j(t) = W_j(t) - W_j(t-)$.  Because the process $W_j$ is locally of bounded variation, the quadratic variation is given by $$[W_j](t) = \sum_{u \in [0, t]} \Delta W_j(u)^2 = X_j^b(t) + X_j^d(t)$$ (see (8.19) of \cite{kle}).  Because $W_j^b + W_j^d$, being the sum of two martingales, is a martingale, we get (see Definition 8.22 of \cite{kle})
\begin{equation}\label{angleWj}
\langle W_j \rangle(t) = \int_0^t (\mu X_{j-1}(u) + B_j(u)X_j(u) + D_j(u)X_j(u)) \: du.
\end{equation}

We will work primarily with a different martingale.  For all $t \geq 0$ and $j \in \Z^+$, let $$G_j^*(t) = B_j(t) - D_j(t) = NF_j(t) - 1 - \mu.$$  As long as every individual's fitness is strictly positive, we have
\begin{equation}\label{GGstar}
G_j^*(t) = N \cdot \frac{1 + s(j - M(t))}{N} - 1 - \mu = s(j - M(t)) - \mu = G_j(t),
\end{equation}
where $G_j(t)$ was defined in (\ref{Gdef}).  We interpret $G_j^*(t)$ as the growth rate of the type $j$ population a time $t$.  In Proposition \ref{Zmart} below, we define a martingale that will be very useful for studying how the number of type $j$ individuals evolves over time.  This is similar to the martingale studied in section 4 of \cite{dm11}.

\begin{Prop}\label{Zmart}
For all $t \geq 0$ and $j \in \Z^+$, let
\begin{equation}\label{Zjdef}
Z_j(t) = e^{-\int_0^t G_j^*(v) \: dv} X_j(t) - \int_0^t \mu X_{j-1}(u) e^{-\int_0^u G_j^*(v) \: dv} \: du - X_j(0).
\end{equation}
Then $(Z_j(t), t \geq 0)$ is a mean zero martingale with
$$\Var(Z_j(t)) = E \bigg[ \int_0^t e^{-2 \int_0^u G_j^*(v) \: dv} (\mu X_{j-1}(u) + B_j(u) X_j(u) + D_j(u) X_j(u)) \: du \bigg].$$  
\end{Prop}

\begin{proof}
For $t \geq 0$ and $j \in \Z^+$, define
\begin{equation}\label{newMG0}
I_j(t) = e^{-\int_0^t G_j^*(v) \: dv}.
\end{equation}
The processes $X_j$ and $I_j$ are both semimartingales, so the Integration by Parts Formula (see Corollary 8.7 of \cite{kle}) gives
\begin{equation}\label{newMG1}
I_j(t) X_j(t) = I_j(0) X_j(0) + \int_0^t X_j(u-) \: dI_j(u) + \int_0^t I_j(u-) \: dX_j(u) + [X_j, I_j]_t.
\end{equation}
Because the processes $X_j$ and $I_j$ are locally of bounded variation, and the process $I_j$ has continuous paths, we have (see (8.19) of \cite{kle})
\begin{equation}\label{newMG4}
[X_j, I_j]_t = 0 \hspace{.2in} \mbox{for all $t$ a.s.}
\end{equation}
Also,
\begin{equation}\label{newMG2}
\int_0^t X_j(u-) \: dI_j(u) = - \int_0^t X_j(u) G_j^*(u) I_j(u) \: du.
\end{equation}
Because $$X_j(t) = X_j(0) + X_j^b(t) - X_j^d(t) = X_0(t) + W_j(t) + \int_0^t (\mu X_{j-1}(u) + G_j^*(u) X_j(u)) \: du$$ and $I_j(t)$ is a continuous function of $t$, we get
\begin{equation}\label{newMG3}
\int_0^t I_j(u-) \: dX_j(u) = \int_0^t I_j(u)(\mu X_{j-1}(u) + G_j^*(u) X_j(u)) \: du + \int_0^t I_j(u) \: dW_j(u).
\end{equation}
Combining (\ref{newMG1}), (\ref{newMG4}), (\ref{newMG2}), and (\ref{newMG3}) and using that $I_j(0) = 1$, we get $$I_j(t) X_j(t) = X_j(0) + \int_0^t I_j(u) \mu X_{j-1}(u) \: du + \int_0^t I_j(u) \: dW_j(u).$$  Therefore, in view of (\ref{Zjdef}) and (\ref{newMG0}), we have $$Z_j(t) = \int_0^t I_j(u) \: dW_j(u).$$  Note that $D_j(t) \leq 1 + \mu$ for all $t$.  Also, because $0 \leq F_j(t) \leq 1$ for all $t$, we have $B_j(t) \leq N$ for all $t$, and so the process $(G_j^*(t), t \geq 0)$ is bounded.  Therefore, using (\ref{angleWj}), for each fixed $t > 0$, we have
$$E \bigg[ \int_0^t I_j^2(u) \: d \langle W_j \rangle(u) \bigg] = E \bigg[ \int_0^t e^{-2 \int_0^u G_j^*(v) \: dv} (\mu X_{j-1}(u) + B_j(u)X_j(u) + D_j(u)X_j(u)) \: du \bigg] < \infty.$$  Therefore (see Theorem 8.32 of \cite{kle}), the process $(Z_j(t), t \geq 0)$ is a square integrable martingale and $$\langle Z_j \rangle(t) = \int_0^t I_j^2(u) \: d \langle W_j \rangle(u) = \int_0^t e^{-2 \int_0^u G_j^*(v) \: dv} (\mu X_{j-1}(u) + B_j(u)X_j(u) + D_j(u)X_j(u)) \: du.$$  Because $Z_j(0) = 0$, the process $(Z_j(t), t \geq 0)$ is a mean zero martingale.  Finally, because $\Var(Z_j(t)) = E[Z_j^2(t)] = E[\langle Z_j \rangle(t)]$ (see Corollary 8.25 of \cite{kle}), the result follows.
\end{proof}

\subsection{Generalizations}

It will often be useful to consider the martingale of Proposition \ref{Zmart} started or stopped at a stopping time.  Let $({\cal F}_t)_{t \geq 0}$ be the natural filtration of the process $((X_0(t), X_1(t), \dots), t \geq 0)$.  Let $\tau$ be a stopping time with respect to $({\cal F}_t)_{t \geq 0}$.  Let $X_j^{\tau}(t) = X_j(t \wedge \tau)$ and $Z_j^{\tau}(t) = Z_j(t \wedge \tau)$ for all $t \geq 0$.  Then the process $((X^{\tau}_0(t), X^{\tau}_1(t), X_2^{\tau}(t), \dots), t \geq 0)$ represents the population modified so that it does not change after time $\tau$.  Because stopped martingales are martingales, the process $(Z_j^{\tau}(t), t \geq 0)$ is a martingale with $\langle Z_j^{\tau} \rangle(t) = \langle Z_j \rangle(t \wedge \tau)$, and we have the following corollary.

\begin{Cor}\label{ZmartCor}
Let $\tau$ be a stopping time, and let $Z_j^{\tau}(t) = Z_j(t \wedge \tau)$ for all $t \geq 0$ and $j \in \Z^+$.  Then $(Z_j^{\tau}(t), t \geq 0)$ is a mean zero martingale with $$\Var(Z_j^{\tau}(t)) = E \bigg[ \int_0^{t \wedge \tau} e^{-2 \int_0^u G_j^*(v) \: dv} (\mu X_{j-1}(u) + B_j(u) X_j(u) + D_j(u) X_j(u)) \: du \bigg].$$  
\end{Cor}

Also, the process $((X_0(t), X_1(t), X_2(t), \dots), t \geq 0)$ is a Markov chain on the countable state space $S = \{(x_0, x_1, \dots): x_j \in \Z^+ \mbox{ for all $j$ and }\sum_{j=0}^{\infty} x_j = N\}$ and therefore satisfies the Strong Markov Property.  Combining Corollary \ref{ZmartCor} with the Strong Markov Property leads to the following result.

\begin{Cor}\label{ZmartCor2}
Let $\kappa$ and $\tau$ be stopping times with $\kappa \leq \tau$.  For all $j \in \Z^+$, let $Z_j^{\kappa, \tau}(t) = 0$ if $t < \kappa$, and if $t \geq \kappa$, let
$$Z_j^{\kappa, \tau}(t) = e^{-\int_{\kappa}^{t \wedge \tau} G_j^*(v) \: dv} X_j(t \wedge \tau) - \int_{\kappa}^{t \wedge \tau} \mu X_{j-1}(u) e^{-\int_{\kappa}^u G_j^*(v) \: dv} \: du - X_j(\kappa).$$  Then $(Z_j^{\kappa, \tau}(\kappa + t), t \geq 0)$ is a mean zero martingale with $$\Var(Z_j^{\kappa, \tau}(\kappa + t)|{\cal F}_{\kappa}) = E \bigg[ \int_{\kappa}^{(\kappa + t) \wedge \tau} e^{-2 \int_{\kappa}^u G_j^*(v) \: dv} (\mu X_{j-1}(u) + B_j(u) X_j(u) + D_j(u) X_j(u)) \: du \bigg| {\cal F}_{\kappa} \bigg].$$
\end{Cor}

Also, we will sometimes need to consider the type $j$ individuals that are descended from an individual that gets its $j$th mutation during some time interval.  The following result is established in the same way as Proposition \ref{Zmart} and Corollary \ref{ZmartCor2} except that the mutation rate is set to zero outside of the time interval $(\kappa, \gamma]$.  

\begin{Cor}\label{ZmartCor4}
Let $\kappa$ and $\gamma$ be stopping times with $\kappa \leq \gamma$.  Let $j \in \Z^+$.  For $t \geq 0$, let $X_j^{[\kappa, \gamma]}(t)$ be the number of type $j$ individuals in the population at time $t$ that are descended from individuals that acquired a $j$th mutation during the time interval $(\kappa, \gamma]$.  Let $Z_j^{[\kappa, \gamma]}(t) = 0$ if $t < \kappa$.  If $t \geq \kappa$, let $$Z_j^{[\kappa, \gamma]}(t) = e^{-\int_{\kappa}^t G_j^*(v) \: dv} X_j^{[\kappa, \gamma]}(t) - \int_{\kappa}^{t \wedge \gamma} \mu X_{j-1}(u) e^{-\int_{\kappa}^u G_j^*(v) \: dv} \: du.$$  Then $(Z_j^{[\kappa, \gamma]}(\kappa + t), t \geq 0)$ is a mean zero martingale.  Denoting by $B_j^{[\kappa, \gamma]}(t)$ and $D_j^{[\kappa, \gamma]}(t)$ the expressions on the right-hand sides of (\ref{Bjeq}) and (\ref{Djeq}) with $X_j^{[\kappa, \gamma]}(t)$ in place of $X_j(t)$, we have
\begin{align}\label{varkappaeq}
&\Var(Z_j^{[\kappa, \gamma]}(\kappa + t)|{\cal F}_{\kappa}) \nonumber \\
&\hspace{.05in}= E \bigg[ \int_{\kappa}^{\kappa + t} e^{-2 \int_{\kappa}^u G_j^*(v) \: dv} (\mu X_{j-1}(u) \1_{u \in (\kappa, \gamma]} + B_j^{[\kappa, \gamma]}(u)X_j^{[\kappa, \gamma]}(u) + D_j^{[\kappa, \gamma]}(u)X_j^{[\kappa, \gamma]}(u)) \: du \bigg| {\cal F}_{\kappa} \bigg]. 
\end{align}
Furthermore, if $\tau$ is a stopping time with $\kappa \leq \tau$, then $(Z_j^{[\kappa, \gamma]}((\kappa + t) \wedge \tau), t \geq 0)$ is a mean zero martingale, and $\Var(Z_j^{[\kappa, \gamma]}((\kappa + t) \wedge \tau)|{\cal F}_{\kappa})$ is obtained by replacing $\kappa + t$ with $(\kappa + t) \wedge \tau$ in (\ref{varkappaeq}).
\end{Cor}

\begin{Rmk}\label{smprem}{\em
By the Strong Markov Property, the result of Corollary \ref{ZmartCor4} holds even if $j$ is random, as long as $j$ is ${\cal F}_{\kappa}$-measurable.}
\end{Rmk}

\subsection{A related supermartingale}

We will also need to consider a supermartingale that involves not just the individuals of type $j$ but the individuals of all types less than or equal to $j$.  For $j \in \Z^+$ and $t \geq 0$, let $$S_j(t) = X_0(t) + X_1(t) + \dots + X_j(t).$$  There are two ways that the value of the process $S_j$ could change at time $t$:
\begin{enumerate}
\item The number of individuals with $j$ or fewer mutations could increase by one because of a birth.  This happens when one of the $N - S_j(t-)$ individuals with more than $j$ mutations dies and is replaced by an individual with $j$ or fewer mutations.  Because each individual dies at rate 1, and when a death occurs at time $t$, the probability that a type $\ell$ individual is born is $X_{\ell}(t-) F_{\ell}(t-)$, the rate at which this occurs is
\begin{equation}\label{birthrate}
(N - S_j(t-)) \sum_{\ell = 0}^j X_{\ell}(t-)F_{\ell}(t-).
\end{equation}

\item The number of individuals with $j$ or fewer mutations could decrease by one because of a mutation or death.  The rate at which one of the type $j$ individuals acquires a $(j+1)$st mutation is $\mu X_j(t-)$.  There are $S_j(t-)$ individuals with $j$ or fewer mutations that could die, and when a death occurs, the probability that the new individual born has more than $j$ mutations is $1 - \sum_{\ell = 0}^j X_{\ell}(t-)F_{\ell}(t-)$.  Therefore, the total rate of events that reduce the number of type $j$ individuals is
\begin{equation}\label{deathrate}
S_j(t-) \bigg(1 - \sum_{\ell = 0}^j X_{\ell}(t-)F_{\ell}(t-) \bigg) + \mu X_j(t-).
\end{equation}
\end{enumerate}
Let $$V_j(t) = N \sum_{\ell = 0}^j X_{\ell}(t) F_{\ell}(t) - S_j(t) - \mu X_j(t),$$ and note that the difference between the expressions in (\ref{birthrate}) and (\ref{deathrate}) is $V_j(t-)$.  Thus, reasoning as in the argument following (\ref{Wbeq}) and (\ref{Wdeq}), the process $(S_j(t) - \int_0^t V_j(u) \: du, \: t \geq 0)$ is a martingale.  This leads to the following proposition.

\begin{Prop}\label{supprop}
For all $j \in \Z^+$ and $t \geq 0$, let $${\tilde G}_j(t) = \max_{\ell \in \{0, 1, \dots, j\}} (N F_{\ell}(t) - 1 - \mu \1_{\{\ell = j\}}),$$ and let $Y_j(t) = e^{-\int_0^t {\tilde G}_j(v) \: dv} S_j(t)$.  Then $(Y_j(t), t \geq 0)$ is a supermartingale for all $j \in \Z^+$.
\end{Prop}

\begin{proof}
Lemma 3.2 in Chapter 4 of \cite{ek86} states that if $(X(t), t \geq 0)$ is a process which takes its values in a complete separable metric space $E$ and is adapted to $({\cal F}_t)_{t \geq 0}$, and if $f: E \rightarrow \R$ and $g: E \rightarrow \R$ are bounded measurable functions such that $\inf_{x \in E} f(x) > 0$ and $(f(X(t)) - \int_0^t g(X(u)) \: du, \: t \geq 0)$ is a martingale with respect to $({\cal F}_t)_{t \geq 0}$, then the process whose value at time $t$ is $$f(X(t)) \exp \bigg( - \int_0^t \frac{g(X(u))}{f(X(u))} \: du \bigg)$$ is a martingale with respect to $({\cal F}_t)_{t \geq 0}$.  We can apply this result with $X(t) = (X_0(t), X_1(t), \dots)$, $f(X(t)) = S_j(t) + \eta$ where $\eta > 0$, and $g(X(t)) = V_j(t) \1_{\{S_j(t) > 0\}}$ to get that if
$$Y_j^{\eta}(t) = (S_j(t) + \eta) \exp \bigg( - \int_0^t \frac{V_j(v)}{S_j(v) + \eta} \1_{\{S_j(v) > 0\}} \: dv \bigg),$$
then $(Y_j^{\eta}(t), t \geq 0)$ is a martingale.  Note that $F_{\ell}(t) \leq F_j(t)$ for all $\ell \leq j$.  Therefore,
\begin{equation}\label{YSGin}
\frac{V_j(t)}{S_j(t) + \eta} = \sum_{\ell=0}^j \frac{X_{\ell}(t)}{S_j(t) + \eta} (N F_{\ell}(t) - 1 - \mu \1_{\{\ell = j\}}) \leq \sum_{\ell=0}^j \frac{X_{\ell}(t)}{S_j(t) + \eta} {\tilde G}_j(t) \leq {\tilde G}_j(t).
\end{equation}  
If $0 \leq u < t$, then for all $\eta > 0$,
\begin{align*}
E \big[ e^{-\int_0^t {\tilde G}_j(v) \1_{\{S_j(v) > 0\}} dv} (S_j(t) + \eta) \big| {\cal F}_u \big]
&= E\big[ Y_j^{\eta}(t) e^{-\int_0^t ({\tilde G}_j(v) - (V_j(v)/(S_j(v) + \eta)) \1_{\{S_j(v) > 0\}} dv} \big| {\cal F}_u \big] \nonumber \\
&\leq E[Y^{\eta}_j(t)| {\cal F}_u] e^{-\int_0^u ({\tilde G}_j(v) - (V_j(v)/(S_j(v) + \eta)))\1_{\{S_j(v) > 0\}} dv} \nonumber \\
&= Y_j^{\eta}(u) e^{-\int_0^u ({\tilde G}_j(v) - (V_j(v)/(S_j(v) + \eta)))\1_{\{S_j(v) > 0\}} dv} \nonumber \\
&= e^{-\int_0^u {\tilde G}_j(v) \1_{\{S_j(v) > 0\}} dv} (S_j(u) + \eta).
\end{align*}
Letting $\eta \rightarrow 0$, we get 
\begin{equation}\label{Ymart}
E \big[ e^{-\int_0^t {\tilde G}_j(v) \1_{\{S_j(v) > 0\}} dv} S_j(t) \big| {\cal F}_u \big] \leq e^{-\int_0^u {\tilde G}_j(v) \1_{\{S_j(v) > 0\}} dv} S_j(u)
\end{equation}
Because $S_j(t) = 0$ whenever $S_j(v) = 0$ for some $v < t$, the indicators on both sides of (\ref{Ymart}) can be removed.  It follows that $E[Y_j(t)|{\cal F}_u] \leq Y_j(u)$.  That is, $(Y_j(t), t \geq 0)$ is a supermartingale.
\end{proof}

\begin{Rmk}\label{Gjtilde}
{\em As long as $G_j(t) = G_j^*(t)$, since we are assuming $N$ is large enough that $s \geq \mu$, for all $\ell < j$ we have $$NF_{\ell}(t) - 1 = G^*_{\ell}(t) + \mu = G_{\ell}(t) + \mu = s(\ell - M(t)) \leq s(j - M(t)) - \mu = G_j(t),$$ and therefore ${\tilde G}_j(t) = G_j(t)$.}
\end{Rmk}

\section{Proof of Proposition \ref{earlyprop}}\label{earlysec}

In this section, we study the behavior of the process before the time $t^*$ defined in (\ref{tstardef}).  We prove Proposition \ref{earlyprop}.  Recall the definitions of $k_N$, $k_N^-$, $k_N^+$, and $k^*$ from (\ref{kNdef}), (\ref{kNminus}), (\ref{kNplus}), and (\ref{kstar}).  Part 1 of Proposition \ref{earlyprop} says that for $j \leq k_N^-$, the number of type $j$ individuals at time $t \in [0, t^*]$ is well approximated by $x_j(t)$, which is defined in (\ref{Xjbardef}).  Part 2 handles the delicate case in which there is an integer $j$ in the interval $(k_N^-, k_N^+)$.  Parts 3 and 4 say that for $j \geq k_N^+$, no type $j$ individuals appear before time $t^*$, and there are fewer than $s/\mu$ individuals of type $k^*$ through time $t^*$.

\subsection{Bounding the mean number of mutations}

Before time $t^*$, the mean number of mutations in the population is close to zero.  Accordingly, let $\eta = \mu k_N^5/s$, and define the stopping time $$\tau = \inf\{t: M(t) \geq \eta\}.$$  
Recall the definition of the martingale $(Z_j(t), t \geq 0)$ from (\ref{Zjdef}).
We will consider the processes $(X_j^{\tau}(t), t \geq 0)$ and $(Z_j^{\tau}(t), t \geq 0)$, where $X_j^{\tau}(t) = X_j(t \wedge \tau)$ and $Z_j^{\tau}(t) = Z_j(t \wedge \tau)$ for all $t \geq 0$.  From assumption A3 and (\ref{muspower}), we see that for all $a > 0$,
\begin{equation}\label{muskNa}
\frac{\mu}{s} k_N^a = \frac{\mu}{s^{a+1}} (s k_N)^a \rightarrow 0 \hspace{.2in}\mbox{as }N \rightarrow \infty,
\end{equation} 
so in particular $\eta \rightarrow 0$ and $s \eta \rightarrow 0$ as $N \rightarrow \infty$.  Therefore, we may and will assume throughout this section that $N$ is large enough that $s \eta < 1$.  This implies that the fitness of every individual is strictly positive before time $\tau$, which means $G^*_j(t) = G_j(t) = s(j - M(t)) - \mu$ for all $j \in \Z^+$ and $t < \tau$.  Our first goal is to show that with high probability, we have $\tau > t^*$, and so stopping the process at time $\tau$ does not change the behavior of the process before time $t^*$.  To do this, we need the upper bound on $E[X_j^{\tau}(t)]$ provided by the following lemma.  This lemma will also be useful for first moment estimates later in the proof.

\begin{Lemma}\label{EXjupper}
For all $t \geq 0$ and $j \in \Z^+$, we have
\begin{equation}\label{EXjeq}
E[X_j^{\tau}(t)] \leq \frac{N \mu^j (e^{st} - 1)^j}{s^j j!}.
\end{equation}
\end{Lemma}

\begin{proof}
Let $t \geq 0$.  Let $m \in \N$, and for $i \in \{0, 1, \dots, m\}$, let $t_i = (i/m) t$.  Let $X_j^{[t_i, t_{i+1}]}(t)$ be the number of type $j$ individuals at time $t$ that are descended from individuals that acquired their $j$th mutation during the time interval $(t_i, t_{i+1}]$.  The process $(Z_j^{[t_i, t_{i+1}]}(t_i + t), t \geq 0)$ introduced in Corollary \ref{ZmartCor4} is a mean zero martingale.  The process stopped at time $\tau$ is also a mean zero martingale, so 
$$E\bigg[e^{-\int_{t_i}^{t \wedge \tau} G_j(v) \: dv} X_j^{[t_i, t_{i+1}]}(t \wedge \tau)\bigg] = E \bigg[ \int_{t_i}^{t_{i+1} \wedge \tau} \mu X_{j-1}^{\tau}(u) e^{-\int_{t_i}^u G_j(v) \: dv} \: du \bigg].$$  Now $sj - s \eta - \mu \leq G_j(v) \leq sj$ for all $v \in [0, \tau)$, which implies that
\begin{align*}
E[X_j^{[t_i, t_{i+1}]}(t \wedge \tau)] &\leq e^{sj(t - t_i)} E\bigg[e^{-\int_{t_i}^{t \wedge \tau} G_j(v) \: dv} X_j^{[t_i, t_{i+1}]}(t \wedge \tau)\bigg] \\
&\leq e^{sj(t - t_i)} E \bigg[ \int_{t_i}^{t_{i+1} \wedge \tau} \mu X_{j-1}^{\tau}(u) e^{-(sj - s \eta - \mu)(u - t_i)} \: du \bigg] \\
&\leq e^{(s \eta + \mu)(t_{i+1} - t_i)} \int_{t_i}^{t_{i+1}} \mu e^{sj(t - u)} E[X_{j-1}^{\tau}(u)] \: du. 
\end{align*}
Summing over $i \in \{0, 1, \dots, m-1\}$ gives $$E[X_j^{\tau}(t)] \leq e^{(s \eta + \mu)t/m} \int_0^t \mu e^{sj(t-u)} E[X_{j-1}^{\tau}(u)] \: du$$ and then letting $m \rightarrow \infty$ gives 
\begin{equation}\label{expforind}
E[X_j^{\tau}(t)] \leq \int_0^t \mu e^{sj(t-u)} E[X_{j-1}^{\tau}(u)] \: du.
\end{equation}

We now use (\ref{expforind}) to prove (\ref{EXjeq}) by induction.  Because $X_0^{\tau}(t) \leq N$ for all $t \geq 0$, we have $E[X_0^{\tau}(t)] \leq N$ for all $t \geq 0$, which establishes the result when $j = 0$.  Suppose $j \geq 1$ and (\ref{EXjeq}) holds with $j-1$ in place of $j$.  Then by (\ref{expforind}),
$$E[X_j^{\tau}(t)] \leq \int_0^t \mu e^{sj(t-u)} \frac{N \mu^{j-1}}{s^{j-1} (j-1)!} (e^{su} - 1)^{j-1} \: du \leq \frac{N \mu^j e^{sjt}}{s^{j-1} (j-1)!} \int_0^t e^{-sju} (e^{su} - 1)^{j-1} \: du.$$
Because
\begin{equation}\label{expintegral}
\int_0^t e^{-sju} (e^{su} - 1)^{j-1} \: du = \frac{(1 - e^{-st})^j}{sj},
\end{equation}
the result follows by induction.  
\end{proof}

\begin{Lemma}\label{tautstarlem}
We have ${\displaystyle \lim_{N \rightarrow \infty} P(\tau \leq t^*) = 0}$.
\end{Lemma}

\begin{proof}
If $\tau \leq t^*$, then $M(t^* \wedge \tau) = M(\tau) \geq \eta$, so by Markov's Inequality, $$P(\tau \leq t^*) \leq P(M(t^* \wedge \tau) \geq \eta) \leq \frac{E[M(t^* \wedge \tau)]}{\eta}.$$  By Lemma \ref{EXjupper},
\begin{align*}
E[M(t^* \wedge \tau)] &= \frac{1}{N} \sum_{j=1}^{\infty} j E[X_j^{\tau}(t^*)] \\
&\leq \frac{1}{N} \sum_{j=1}^{\infty} j \cdot \frac{N \mu^j}{s^j j!} (e^{st^*} - 1)^j \\
&= \frac{\mu}{s} (e^{st^*} - 1) \sum_{j=1}^{\infty} \frac{1}{(j-1)!} \bigg( \frac{\mu}{s} (e^{st^*} - 1) \bigg)^{j-1} \\
&= ye^y,
\end{align*}
where $y = (\mu/s) (e^{st^*} - 1).$  Recalling (\ref{tstardef}), we have $y \leq (\mu/s) e^{st^*} \leq (\mu/s) k_N^4$.  Therefore, in view of (\ref{muskNa}), we have $y \rightarrow 0$ as $N \rightarrow \infty$, and thus $y e^y \leq 2y$ for sufficiently large $N$.  Using (\ref{A1prime}), for sufficiently large $N$, 
$$P(\tau \leq t^*) \leq \frac{2y}{\eta} \leq \frac{2 (\mu/s) k_N^4}{(\mu/s)k_N^5} \rightarrow 0 \hspace{.2in}\mbox{as }N \rightarrow \infty,$$ as claimed.
\end{proof}

\subsection{Controlling the fluctuations in $X_j$}

Our goal in this subsection is to obtain sharp bounds on the fluctuations of the number of type $j$ individuals before time $t^*$.  Because the randomness can be expressed in terms of the martingales $Z_j$, the key result is the next lemma, which will provide control on the value of $|Z_j(t)|$.  Before stating this lemma, we establish a simple bound on the birth and death rates that will be useful throughout the paper.  Note that for all $t$ such that all individuals at time $t$ have a strictly positive fitness, and in particular for all $t < \tau$, we have 
\begin{align*}
B_j(t) + D_j(t) &= (N - 2 X_j(t))F_j(t) + 1 + \mu \\
&= \frac{(N - 2X_j(t))(1 + s(j - M(t)))}{N} + 1 + \mu \\
&\leq 2 + sj + \mu.
\end{align*}
Because $s k_N^+ \rightarrow 0$ as $N \rightarrow \infty$ by (\ref{kdiff}) and assumption A3 and $\mu \rightarrow 0$ as $N \rightarrow \infty$, we have for $j \leq k_N^+$,
\begin{equation}\label{BD3}
B_j(t) + D_j(t) \leq 3 \hspace{.1in}\mbox{ for sufficiently large }N.
\end{equation}
For future reference, note that (\ref{BD3}) also holds for all $j \leq J = 3 k_N T + k^* + 1$.

\begin{Lemma}\label{Z6lem}
Let $\eps > 0$.  For sufficiently large $N$, we have $$P \bigg( \sup_{t \in [0, t^*]} |Z^{\tau}_j(t)| \leq 16 \sqrt{\frac{N \mu^j t^*k_N}{\eps s^j j!}} \mbox{ for all }j \leq k_N^+ \bigg) > 1 - \frac{\eps}{15}.$$
\end{Lemma}

\begin{proof}
By Corollary \ref{ZmartCor}, the process $(Z_j^{\tau}(t), t \geq 0)$ is a mean zero martingale.  Since $G_j^*(t) = G_j(t)$ for $t < \tau$, we have
$$\Var(Z_j^{\tau}(t^*)) = E \bigg[ \int_0^{t^*\wedge \tau} e^{-2 \int_0^u G_j(v) \: dv} (\mu X_{j-1}(u) + B_j(u) X_j(u) + D_j(u) X_j(u)) \: du \bigg].$$  Combining this result with (\ref{BD3}) and Lemma \ref{EXjupper}, we get
\begin{align}\label{mainvarZ}
\Var(Z_j^{\tau}(t^*)) &\leq \int_0^{t^*} e^{-2 (sj - s \eta - \mu)u} \bigg( \frac{N \mu^j (e^{su} - 1)^{j-1}}{s^{j-1} (j-1)!} + \frac{3 N \mu^j (e^{su} - 1)^j}{s^j j!} \bigg) \: du \nonumber \\
&\leq e^{2(s \eta + \mu) t^*} \cdot \frac{N \mu^j}{s^j j!} \int_0^{t^*} e^{-2sju} \big( (e^{su} - 1)^{j-1} sj + 3 (e^{su} - 1)^j \big) \: du \nonumber \\
&\leq e^{2(s \eta + \mu) t^*} \cdot \frac{N \mu^j}{s^j j!} \int_0^{t^*} (e^{-s(j+1) u} sj + 3 e^{-sju}) \: du.
\end{align}
For $j \leq k_N^+$, the result (\ref{muskNa}) implies that 
\begin{equation}\label{jtstarto0}
2(s \eta + \mu) t^* \leq 2(\mu k_N^5 + \mu) \bigg( \frac{4}{s} \log k_N \bigg) \rightarrow 0 \hspace{.2in}\mbox{as }N \rightarrow \infty
\end{equation}
and therefore $e^{2(s \eta + \mu) t^*} \rightarrow 1$ as $N \rightarrow \infty$.  Also, as a consequence of assumption A3, we have $e^{-s(j+1) u} sj + 3 e^{-sju} \leq sk_N^+ + 3 \rightarrow 3$ as $N \rightarrow \infty$ for all $u \in [0, t^*]$.  Thus, for sufficiently large $N$,
\begin{equation}\label{mainvarZfinal}
\Var(Z_j^{\tau}(t^*)) \leq \frac{4 N \mu^j t^*}{s^j j!}
\end{equation}
for all $j \leq k_N^+$.  By the $L^2$ Maximum Inequality for martingales,
\begin{equation}\label{Zsup10}
P \bigg( \sup_{t \in [0, t^*]} |Z_j^{\tau}(t)| > 16 \sqrt{\frac{N \mu^j t^* k_N}{\eps s^j j!}} \bigg) \leq 4 \Var(Z_j^{\tau}(t^*)) \cdot \frac{\eps s^j j!}{256 N \mu^j t^* k_N} \leq \frac{\eps}{16 k_N}
\end{equation}
for all $j \leq k_N^+$ if $N$ is sufficiently large.  Since $k_N \rightarrow \infty$ as $N \rightarrow \infty$ by (\ref{A1prime}) and $k_N^+ - k_N \rightarrow 0$ as $N \rightarrow \infty$ by (\ref{kdiff}), we have $(k_N^+ + 1)/k_N \rightarrow 1$ as $N \rightarrow \infty$.  The result thus follows from (\ref{Zsup10}) by taking the union over $j \in \{0, 1, \dots, \lfloor k_N^+ \rfloor\}$.
\end{proof}

The next lemma shows that when the processes $Z_j$ are bounded as indicated in Lemma \ref{Z6lem}, the processes $X_j$ will stay fairly close to the deterministic functions $x_j$ defined in (\ref{Xjbardef}).  Because the difference between $X_j$ and $x_j$ depends in part on the difference between $X_{j-1}$ and $x_{j-1}$, the proof proceeds by induction.  Rather precise bounds are needed to prevent the errors from accumulating too rapidly during the induction process, so some technical work is required to obtain sufficiently sharp estimates.

\begin{Lemma}\label{XXbarlem}
On the event that $t^* < \tau$ and
\begin{equation}\label{supZassump}
\sup_{t \in [0, t^*]} |Z_j(t)| \leq 16 \sqrt{ \frac{N \mu^j t^* k_N}{\eps s^j j!}} \mbox{ for all }j \leq k_N^+,
\end{equation}
we have, for all $t \in [0, t^*]$ and $\ell \leq k_N^+$,
\begin{equation}\label{XXbareq}
|X_{\ell}(t) - x_{\ell}(t)| \leq x_{\ell}(t) \bigg( (\ell + 1)(s \eta + \mu) t + 16 \sum_{j=0}^{\ell} \sqrt{\frac{s^j t^* k_N}{\eps N \mu^j j!}} \cdot \frac{\ell!}{(\ell - j)!} (1 - e^{-st})^{-j} \bigg).
\end{equation}
In particular, for all $\ell \leq k_N^+$, we have
\begin{equation}\label{XXbareq2}
\sup_{t \in [0, t^*]} |X_{\ell}(t) - x_{\ell}(t)| \leq x_{\ell}(t^*) \bigg( (\ell + 1)(s \eta + \mu) t^* + 16 \sum_{j=0}^{\ell} \sqrt{\frac{s^j t^* k_N}{\eps N \mu^j j!}} \cdot \frac{\ell!}{(\ell - j)!} (1 - e^{-st^*})^{-j} \bigg).
\end{equation}
\end{Lemma}

\begin{proof}
Throughout the proof, we will assume that $t^* < \tau$ and that (\ref{supZassump}) holds.  This implies that $G_j^*(u) = G_j(u)$ for $u \leq t^*$.  For $j \leq k_N^+$ and $t \in [0, t^*]$, define $$H_j(t) = (s \eta + \mu) t x_j(t) + 16 e^{sjt} \sqrt{\frac{N \mu^j t^* k_N}{\eps s^j j!}}.$$  We will first show by induction that for $\ell \in \{0, 1, \dots, \lfloor k_N^+ \rfloor\}$ and $t \in [0, t^*]$, we have
\begin{equation}\label{Hind}
|X_{\ell}(t) - x_{\ell}(t)| \leq H_{\ell}(t) + \sum_{j=0}^{\ell - 1} \frac{\mu^{\ell - j}}{(\ell - j - 1)! s^{\ell - j - 1}} \int_0^t H_j(u) e^{s \ell (t - u)} (1 - e^{-s(t - u)})^{\ell - j - 1} \: du.
\end{equation}

Consider first $\ell = 0$.  Suppose $t \in [0, t^*]$.  From (\ref{Zjdef}), we get $X_0(t) = e^{\int_0^t G_0(v) \: dv} (N + Z_0(t)).$  Because $-s \eta - \mu \leq G_0(v) \leq 0$ for all $v \in [0, t]$, it follows that $$|X_0(t) - N| \leq N (1 - e^{\int_0^t G_0(v) \: dv}) + |Z_0(t)| \leq N (s \eta + \mu) t + 16 \sqrt{\frac{N t^* k_N}{\eps}}.$$  Therefore, since $x_0(t) = N$ for all $t \in [0, t^*]$, we have $$|X_0(t) - x_0(t)| \leq N(s \eta + \mu)t + 16 \sqrt{\frac{N t^* k_N}{\eps}} = H_0(t),$$
so (\ref{Hind}) holds for $\ell = 0$.

Next, suppose (\ref{Hind}) holds for $\ell - 1$, where $\ell \geq 1$.  Let $t \in [0, t^*]$.  Equation (\ref{Zjdef}) gives
\begin{align}\label{Xl4}
X_{\ell}(t) &= e^{\int_0^t G_{\ell}(v) \: dv} \bigg( \int_0^t \mu X_{\ell - 1}(u) e^{-\int_0^u G_{\ell}(v) \: dv} \: du + Z_{\ell}(t) \bigg) \nonumber \\
&= \int_0^t \mu X_{\ell - 1}(u) e^{\int_u^t G_{\ell}(v) \: dv} \: du + e^{\int_0^t G_{\ell}(v) \: dv} Z_{\ell}(t) \nonumber \\
&= \int_0^t \mu (X_{\ell - 1}(u) - x_{\ell - 1}(u)) e^{\int_u^t G_{\ell}(v) \: dv} \: du + \int_0^t \mu x_{\ell - 1}(u) (e^{\int_u^t G_{\ell}(v) \: dv} - e^{s \ell (t-u)}) \: du \nonumber \\
&\hspace{.5in}+ \int_0^t \mu x_{\ell - 1}(u) e^{s \ell (t-u)} \: du + e^{\int_0^t G_{\ell}(v) \: dv} Z_{\ell}(t).
\end{align}
Equation (\ref{expintegral}) gives
\begin{equation}\label{Xl43}
\int_0^t \mu x_{\ell - 1}(u)e^{s \ell (t - u)} \: du = \frac{N \mu^{\ell} e^{s \ell t}}{s^{\ell - 1} (\ell - 1)!} \int_0^t (e^{su} - 1)^{\ell - 1} e^{- s \ell u} \: du = x_{\ell}(t).
\end{equation}
Also, $|e^{\int_u^t G_{\ell}(v) \: dv} - e^{s \ell (t - u)}| = e^{s \ell (t - u)}|e^{-\int_u^t (s M(v) + \mu) \: dv} - 1| \leq e^{s \ell (t - u)}(s \eta + \mu)(t-u).$  Therefore, (\ref{Xl43}) implies that
\begin{equation}\label{Xl42}
\int_0^t \mu x_{\ell - 1}(u) |e^{\int_u^t G_{\ell}(v) \: dv} - e^{s \ell (t-u)}| \: du \leq (s \eta + \mu)t \int_0^t \mu x_{\ell - 1}(u) e^{s \ell (t-u)} \: du = (s \eta + \mu) t x_{\ell}(t). 
\end{equation}
Furthermore, because $e^{\int_0^t G_{\ell}(v) \: dv} \leq e^{s \ell t}$ and we are assuming that (\ref{supZassump}) holds,
\begin{equation}\label{Xl44}
e^{\int_0^t G_{\ell}(v) \: dv} |Z_{\ell}(t)| \leq 16 e^{s \ell t} \sqrt{ \frac{N \mu^{\ell} t^* k_N}{\eps s^{\ell} \ell!}}.
\end{equation}
Combining (\ref{Xl4}), (\ref{Xl43}), (\ref{Xl42}), and (\ref{Xl44}) leads to
$$|X_{\ell}(t) - x_{\ell}(t)| \leq H_{\ell}(t) + \int_0^t \mu |X_{\ell - 1}(u) - x_{\ell - 1}(u)| e^{s \ell (t - u)} \: du.$$
Using the induction hypothesis to bound the integral, we get
\begin{align}\label{Xlpre}
|X_{\ell}(t) - x_{\ell}(t)| &\leq H_{\ell}(t) + \int_0^t \mu e^{s \ell (t - u)} \bigg( H_{\ell-1}(u) + \sum_{j=0}^{\ell - 2} \frac{\mu^{\ell - 1 - j}}{(\ell - j - 2)! s^{\ell - j - 2}} \nonumber \\
&\hspace{.8in} \times \int_0^u H_j(v) e^{s (\ell - 1) (u - v)} (1 - e^{-s(u - v)})^{\ell - j - 2} \: dv \bigg) \: du.
\end{align}
The first term $\int_0^t \mu e^{s \ell (t - u)} H_{\ell - 1}(u) \: du$ in the integral on the right-hand side of (\ref{Xlpre}) matches the $j = \ell - 1$ term on the right-hand side of (\ref{Hind}).  For $j \in \{0, 1, \dots, \ell - 2\}$, the term corresponding to $j$ in the sum on the right-hand side of (\ref{Xlpre}) can be expressed as
\begin{align*}
&\int_0^t \mu e^{s \ell (t - u)} \cdot \frac{\mu^{\ell - 1 - j}}{(\ell - j - 2)!s^{\ell - j - 2}} \bigg( \int_0^u H_j(v) e^{s(\ell - 1)(u - v)} (1 - e^{-s(u-v)})^{\ell - j - 2} \: dv \bigg) \: du \\
&\hspace{.5in}= \frac{\mu^{\ell - j}}{(\ell - j - 2)! s^{\ell - j - 2}} \int_0^t H_j(v) e^{s \ell (t - v)} \bigg( \int_v^t e^{-s(u - v)}(1 - e^{-s(u - v)})^{\ell - j - 2} \: du \bigg) \: dv,
\end{align*}
which matches the term corresponding to $j$ on the right-hand side of (\ref{Hind}) because the substitution $x = u - v$ combined with (\ref{expintegral}) gives
$$\int_v^t e^{-s(u - v)}(1 - e^{-s(u - v)})^{\ell - j - 2} \: du = \int_0^{t - v} e^{-sx}(1 - e^{-sx})^{\ell - j - 2} \: dx = \frac{(1 - e^{-s(t - v)})^{\ell - j - 1}}{s(\ell - j - 1)}.$$  Thus, by induction, (\ref{Hind}) holds for all $\ell \in \{0, 1, \dots, \lfloor k_N^+ \rfloor\}$ and $t \in [0, t^*]$.

Next we will obtain (\ref{XXbareq}) from (\ref{Hind}).  For $j \in \{0, 1, \dots, \ell - 1\}$, the term corresponding to $j$ in the sum in (\ref{Hind}) can be written as
\begin{equation}\label{2tmjbd}
\frac{\mu^{\ell - j}}{(\ell - j - 1)! s^{\ell - j - 1}} \int_0^t \bigg( (s \eta + \mu) u x_j(u) + 16 e^{sju} \sqrt{ \frac{N \mu^j t^* k_N}{\eps s^j j!}} \bigg) e^{s \ell (t - u)}(1 - e^{-s(t - u)})^{\ell - j - 1} \: du.
\end{equation}
The first of the two terms in this expression is bounded above by
$$ (s \eta + \mu) t \cdot \frac{N \mu^{\ell}}{s^{\ell - 1} j! (\ell - j - 1)!} \int_0^t (e^{su} - 1)^j e^{s \ell (t - u)} (1 - e^{-s(t - u)})^{\ell - j - 1} \: du.$$  By making the substitution $x = e^{su}$ and $y = e^{st}$ and then applying the result (3.199) of \cite{table}, we see that $$\int_0^t (e^{su} - 1)^j e^{s \ell (t - u)} (1 - e^{-s(t - u)})^{\ell - j - 1} \: du = (e^{st} - 1)^{\ell} \cdot  \frac{j! (\ell - j - 1)!}{s \ell!},$$ so upper bound on the first term in (\ref{2tmjbd}) becomes
\begin{equation}\label{ubtm1}
(s \eta + \mu) t \cdot \frac{N \mu^{\ell} (e^{st} - 1)^{\ell}}{s^{\ell} \ell!} = (s \eta + \mu) t \cdot x_{\ell}(t).
\end{equation}
The second term in (\ref{2tmjbd}) equals
$$\frac{16 \mu^{\ell - j}}{(\ell - j - 1)! s^{\ell - j - 1}} \sqrt{\frac{N \mu^j t^* k_N}{\eps s^j j!}} \cdot e^{s \ell t} \int_0^t e^{-s(\ell - j)u} (1 - e^{-s(t - u)})^{\ell - j - 1} \: du.$$
Making the substitution $x = t - u$, we get
$$\int_0^t e^{-s(\ell - j)u} (1 - e^{-s(t - u)})^{\ell - j - 1} \: du = e^{-s(\ell - j)t} \int_0^t e^{sx} (e^{sx} - 1)^{\ell - j - 1} \: dx = \frac{(1 - e^{-st})^{\ell - j}}{s(\ell - j)}.$$
Also, $e^{s \ell t} (1 - e^{-st})^{\ell - j} = (e^{st} - 1)^{\ell}(1 - e^{-st})^{-j}$ so the second term in (\ref{2tmjbd}) equals
\begin{equation}\label{ubtm2}
\frac{16 \mu^{\ell - j}}{(\ell - j)! s^{\ell - j}} \sqrt{\frac{N \mu^j t^* k_N}{\eps s^j j!}} \cdot (e^{st} - 1)^{\ell} (1 - e^{-st})^{-j}
= 16 \sqrt{\frac{s^j t^* k_N}{\eps N \mu^j j!}} \cdot \frac{\ell!}{(\ell - j)!} (1 - e^{-st})^{-j} x_{\ell}(t),
\end{equation}
which matches the term corresponding to $j$ in (\ref{XXbareq}).  Furthermore, we have
\begin{equation}\label{ubtm3}
H_{\ell}(t) = (s \eta + \mu)t x_{\ell}(t) + 16 \sqrt{\frac{s^{\ell} t^* k_N}{\eps N \mu^{\ell} \ell!}} \cdot \ell! (1 - e^{-st})^{-\ell} x_{\ell}(t),
\end{equation}
and the second term matches the $j = \ell$ term in (\ref{XXbareq}).  Combining the bound in (\ref{ubtm1}) with the results in (\ref{ubtm2}) and (\ref{ubtm3}) gives the bound in (\ref{XXbareq}).

Finally, note that if $t \in [0, t^*]$ and $0 \leq j \leq \ell \leq k_N^+$, then
\begin{align*}
x_{\ell}(t)(1 - e^{-st})^{-j} &= (e^{st} - 1)^{\ell} (e^{st^*} - 1)^{-\ell}(1 - e^{-st})^{-j} x_{\ell}(t^*) \\
&= (e^{st} - 1)^{\ell-j}(e^{st^*} - 1)^{-\ell}e^{stj}x_{\ell}(t^*) \\
&\leq (e^{st^*} - 1)^{-j} e^{st^*j} x_{\ell}(t^*) \\
&= (1 - e^{-st^*})^{-j} x_{\ell}(t^*),
\end{align*}
so (\ref{XXbareq2}) follows from (\ref{XXbareq}).
\end{proof}

\subsection{Proof of part 1 of Proposition \ref{earlyprop}}

Here we show how the results in the previous section can be used to obtain the desired control on the difference between $X_j$ and $x_j$ up to time $t^*$ for $j \leq k_N^-$.  The result (\ref{p552}) below is essentially a restatement of part 1 of Proposition \ref{earlyprop}.

\begin{Prop}\label{early1prop}
Let $t_0 = (1/s) \log k_N$.  For sufficiently large $N$, we have
\begin{equation}\label{p551}
P \big( |X_j(t) - x_j(t)| \leq \delta x_j(t) \mbox{ for all }j \leq k_N^- \mbox{ and }t \in [t_0, t^*] \big) > 1 - \frac{\eps}{12}
\end{equation}
and
\begin{equation}\label{p552}
P \bigg( \sup_{t \in [0, t^*]} |X_j(t) - x_j(t)| \leq \delta x_j(t^*) \mbox{ for all }j \leq k_N^- \bigg) > 1 - \frac{\eps}{12}.
\end{equation}
\end{Prop}

\begin{proof}
It follows from Lemmas \ref{tautstarlem} and \ref{Z6lem} that the probability that $t^* < \tau$ and (\ref{supZassump}) holds is at least $1 - \eps/12$ for sufficiently large $N$.  Thus, the proposition will follow from Lemma \ref{XXbarlem} provided that for sufficiently large $N$, we have
\begin{equation}\label{p553}
(\ell + 1)(s \eta + \mu)t + 16 \sum_{j=0}^{\ell} \sqrt{\frac{s^j t^* k_N}{\eps N \mu^j j!}} \cdot \frac{\ell!}{(\ell - j)!} (1 - e^{-st})^{-j} \leq \delta
\end{equation}
for all $t \in [t_0, t^*]$ and $\ell \leq k_N^-$.  It will suffice to show that the two terms on the left-hand side of (\ref{p553}) each tend to zero as $N \rightarrow \infty$ uniformly in $\ell \leq k_N^-$ and $t \in [t_0, t^*]$.  The first term tends to zero by the reasoning in (\ref{jtstarto0}), so it remains to consider the second term.

For $j \leq \ell \leq k_N^-$ and $t^* \geq (1/s) \log k_N$, we have 
$(1 - e^{-st})^{-j} \leq (1 - k_N^{-1})^{-k_N} \rightarrow e$ as $N \rightarrow \infty$.  Therefore, for sufficiently large $N$, we have $(1 - e^{-st})^{-j} \leq 3$.  It now follows from the Binomial Theorem that
$$\sum_{j=0}^{\ell} \sqrt{\frac{s^j t^* k_N}{\eps N \mu^j j!}} \cdot \frac{\ell!}{(\ell - j)!} (1 - e^{-st})^{-j} \leq 3 \sqrt{\frac{t^* k_N}{\eps N}} \sqrt{\ell!} \sum_{j=0}^{\ell} \binom{\ell}{j} \sqrt{\frac{s^j}{\mu^j}} = 3 \sqrt{\frac{t^* k_N }{\eps N}} \sqrt{\ell!} \bigg(1 + \sqrt{ \frac{s}{\mu} } \bigg)^{\ell}$$ for sufficiently large $N$.  To show that this expression tends to zero as $N \rightarrow \infty$ for all $\ell \leq k_N^-$, it suffices to show that
\begin{equation}\label{stslog}
\lim_{N \rightarrow \infty} \log \bigg( \sqrt{ \frac{t^* k_N}{\eps N}} \sqrt{ k_N^- !} \bigg(1 + \sqrt{ \frac{s}{\mu}} \bigg)^{k_N^-} \bigg) = - \infty.
\end{equation}
We use $o(k_N)$ to denote a term which, when divided by $k_N$, tends to zero as $N \rightarrow \infty$ and $O(1)$ to denote a term that stays bounded as $N \rightarrow \infty$.  Because $n! \sim \sqrt{2 \pi} n^{n + 1/2} e^{-n}$ by Stirling's Formula, we have
\begin{equation}\label{logfact}
\log k_N^-! = \bigg( k_N^- + \frac{1}{2} \bigg) \log k_N^- - k_N^- + O(1) = k_N^- \log k_N^- - k_N^-  + o(k_N).
\end{equation}
Also, because $s/\mu \rightarrow \infty$ as $N \rightarrow \infty$ by (\ref{muspower}), $$\log \bigg( \bigg(1 + \sqrt{\frac{s}{\mu}} \bigg)^{k_N^-} \bigg) = k_N^- \log \bigg( 1 + \sqrt{\frac{s}{\mu}} \bigg) = \frac{k_N^-}{2} \log \bigg( \frac{s}{\mu} \bigg) + o(k_N),$$
and because assumption A1 implies that
\begin{equation}\label{klog1s}
\lim_{N \rightarrow \infty} \frac{k_N}{\log (1/s)} = \infty,
\end{equation}
we have $$\log t^* = \log \log k_N + \log (1/s) + O(1) = o(k_N).$$  
Finally, note that $k_N^- \log k_N^- = k_N \log k_N + o(k_N)$.
Therefore, the logarithm on the left-hand side of (\ref{stslog}) is
\begin{align}\label{mainlogcalc}
&\frac{1}{2} \bigg( \log t^* +  \log k_N - \log \eps - \log N \bigg) + \log \sqrt{ k_N^-} + \log \bigg( \bigg(1 + \sqrt{\frac{s}{\mu}} \bigg)^{k_N^-} \bigg) \nonumber \\
&\hspace{.5in}= \frac{1}{2} \bigg( -\log N + k_N^- \log k_N^- - k_N^- + k_N^- \log \bigg( \frac{s}{\mu} \bigg) \bigg) + o(k_N) \nonumber \\
&\hspace{.5in}= \frac{1}{2} \bigg( -\log N + k_N^- \log k_N^- - k_N^- + \log N - \frac{\log N}{\log(s/\mu)} \log \bigg( \frac{\log N}{\log(s/\mu)} \bigg) \bigg) + o(k_N) \nonumber \\
&\hspace{.5in}= - \frac{1}{2} k_N + o(k_N),
\end{align}
which tends to $-\infty$ as $N \rightarrow \infty$.  The result follows.
\end{proof}

\subsection{Proof of part 2 of Proposition \ref{earlyprop}}

In this subsection, we consider the case in which there is an integer $j \in (k_N^-, k_N^+)$.  As noted before the statement of Proposition \ref{earlyprop}, for sufficiently large $N$ there can be at most one such integer, so we will assume that $N$ is large enough to ensure this.  Also, such a $j$ may not exist for every $N$, so in this subsection asymptotic statements as $N \rightarrow \infty$ should be understood to mean that we consider a subsequence of integers $(N_i)_{i=1}^{\infty}$ tending to infinity such that there is an integer in $(k_{N_i}^-, k_{N_i}^+)$ for all $i$.  

Recall that we can write $j$ as in (\ref{bjdef}), with $-1 < b_j < 2$, and $d_j = \max\{0, b_j\}$.  Recall also that when such a $j$ exists, we have $t^* = (4/s) \log k_N$.  In this case, we can not use the same argument as in the proof of Part 1 of Proposition \ref{earlyprop} because the expression in (\ref{mainlogcalc}) does not tend to $-\infty$ as $N \rightarrow \infty$ if $k_N^-$ is replaced by $k_N^+$.  Instead, we will break the type $j$ individuals into three subpopulations.  Define the times
$$r_1 = \max \bigg\{0, \frac{(b_j + 1) \log k_N - 2}{s} \bigg\}, \hspace{.3in} r_2 = \frac{(d_j + 1) \log k_N}{s}.$$  Note that $0 \leq r_1 < r_2 < t^*$.  For each type $j$ individual in the population, we can consider the time when this individual or its ancestor acquired its $j$th mutation.  For $t \in [0, t^*]$, using the notation of Corollary \ref{ZmartCor4}, we can write
\begin{equation}\label{j1j2j3}
X_j(t) = X_j^{[0, r_1]}(t) + X_j^{[r_1, r_2]}(t) + X_j^{[r_2, t^*]}(t).
\end{equation}
Here we are dividing the type $j$ population into three groups, depending on whether the $j$th mutation occurred before time $r_1$, between times $r_1$ and $r_2$, or after time $r_2$.  We will consider these three subpopulations separately in the next three lemmas.

\begin{Lemma}\label{subpop1}
We have $$\lim_{N \rightarrow \infty} P(X_j^{[0, r_1]}(t^*) = 0) = 1.$$
\end{Lemma}

\begin{proof}
Clearly $X_j^{[0, r_1]}(t) = 0$ for all $t \in [0, t^*]$ when $r_1 = 0$, so we will assume that $r_1 > 0$.  Each type $j-1$ individual is acquiring mutations at rate $\mu$.  Therefore, by Lemma \ref{EXjupper}, the expected number of times, before time $r_1 \wedge \tau$, that a type $j-1$ individual acquires a $j$th mutation is at most
\begin{equation}\label{nummut}
\int_0^{r_1} \mu E[X_{j-1}^{\tau}(t)] \: dt \leq \frac{N \mu^j}{s^{j-1}(j-1)!} \int_0^{r_1} (e^{st} - 1)^{j-1} \: dt.
\end{equation}
We have $$\int_0^{r_1} (e^{st} - 1)^{j-1} \: dt \leq \int_0^{r_1} e^{sjt} \: dt \leq \frac{e^{sjr_1}}{sj} = \frac{k_N^{j(b_j + 1)} e^{-2j}}{sj}.$$  Therefore, for sufficiently large $N$, the expression in (\ref{nummut}) is bounded above by
$$\frac{N \mu^j k_N^{j(b_j + 1)} e^{-2j}}{s^j j!}.$$  By Markov's Inequality, this expression also gives an upper bound for the probability that at least one type $j-1$ individual acquires a $j$th mutation by time $r_1 \wedge \tau$.  
Using (\ref{bjdef}), the reasoning in (\ref{logfact}), and the fact that
\begin{equation}\label{jlogj}
j \log j = j \log k_N + o(k_N) = k_N \log k_N + o(k_N),
\end{equation}
we get
\begin{align*}
\log \bigg( \frac{N \mu^j k_N^{j(b_j + 1)} e^{-2j}}{s^j j!} \bigg) &= \log N - j \log \bigg( \frac{s}{\mu} \bigg) + (b_j + 1)j \log k_N - 2j - \log j! \\
&= - b_j k_N \log k_N + (b_j + 1)j \log k_N - 2j - (j \log j - j) + o(k_N) \\
&= -j + o(k_N),
\end{align*}
which tends to $-\infty$ as $N \rightarrow \infty$.  Thus, the probability that some individual acquires a $j$th mutation by time $r_1 \wedge \tau$ tends to zero as $N \rightarrow \infty$.  Combining this observation with Lemma \ref{tautstarlem} gives the result.
\end{proof}

\begin{Lemma}\label{subpop2}
For sufficiently large $N$,
$$P \bigg( X_j^{[r_1, r_2]}(t^*) > \frac{195}{\eps} k_N^{-d_j} x_j(t^*) \bigg) \leq \frac{\eps}{12}.$$
\end{Lemma}

\begin{proof}
First, suppose $b_j > 0$.  By applying the argument that leads to (\ref{expforind}) followed by the result of Lemma \ref{EXjupper} and then (\ref{expintegral}), we get
\begin{align}\label{j2star1}
E[X_j^{[r_1, r_2]}(t^* \wedge \tau)] &\leq \int_{r_1}^{r_2} \mu e^{sj(t^* - u)} \: E[X_{j-1}^{\tau}(u)] \: du \nonumber \\
&\leq \frac{N \mu^j e^{sjt^*}}{s^{j-1}(j-1)!} \int_{r_1}^{r_2} e^{-sju} (e^{su} - 1)^{j-1} \: du \nonumber \\
&= \frac{N \mu^j e^{sjt^*}}{s^j j!} \bigg( (1 - e^{-s r_2})^j - (1 - e^{-sr_1})^j \bigg).
\end{align}
Since $\frac{d}{dx} (1 - e^{-sx})^j = j (1 - e^{-sx})^{j-1} s e^{-sx} \leq sje^{-sx}$ and $r_2 - r_1 = 2/s$ if $N$ is sufficiently large, we have 
\begin{equation}\label{j2star2}
(1 - e^{-sr_2})^j - (1 - e^{-sr_1})^j \leq (r_2 - r_1) s j e^{-s r_1} \leq 2 e^2 j k_N^{-(b_j + 1)}.
\end{equation}
Since $j/k_N \rightarrow 1$ as $N \rightarrow \infty$ and
\begin{equation}\label{tstark4}
\frac{(e^{st^*} - 1)^j}{e^{sjt^*}} = \bigg(1 - \frac{1}{k_N^4} \bigg)^j \rightarrow 1 \hspace{.2in}\mbox{as } N \rightarrow \infty,
\end{equation}
it follows from (\ref{j2star1}) and (\ref{j2star2}) that for sufficiently large $N$, we have $$E[X_j^{[r_1, r_2]}(t^* \wedge \tau)] \leq \frac{15 N \mu^j (e^{st^*} - 1)^j}{s^j j!} \cdot k_N^{-b_j} =  15 k_N^{-b_j} x_j(t^*).$$  When $b_j \leq 0$, we can use instead Lemma \ref{EXjupper} to get $E[X_j^{[r_1, r_2]}(t^* \wedge \tau)] \leq E[X_j^{\tau}(t^*)] \leq x_j(t^*)$.  Combining these results gives $$E[X_j^{[r_1, r_2]}(t^* \wedge \tau)] \leq 15 k_N^{-d_j} x_j(t^*)$$ for sufficiently large $N$.  By Markov's Inequality,
$$P \bigg( X_j^{[r_1, r_2]}(t^* \wedge \tau) > \frac{195}{\eps} k_N^{-d_j} x_j(t^*) \bigg) \leq \frac{\eps E[X_j^{[r_1, r_2]}(t^* \wedge \tau)]}{195 k_N^{-d_j} x_j(t^*)} \leq \frac{\eps}{13}.$$
The result now follows from Lemma \ref{tautstarlem}.
\end{proof}

\begin{Lemma}\label{subpop3}
There exist positive constants $c$ and $c'$, not depending on $\eps$, such that for sufficiently large $N$,
\begin{equation}\label{stage3final}
P \big( c k_N^{-d_j} x_j(t^*) \leq X_j^{[r_2, t^*]}(t^*) \leq c' k_N^{-d_j} x_j(t^*) \big) \geq 1 - \frac{\eps}{5}.
\end{equation}
\end{Lemma}

\begin{proof}
For $t \in [r_2, t^*]$, write
$$Z_j^{[r_2, t^*]}(t) = e^{-\int_{r_2}^t G_j^*(v) \: dv} X_j^{[r_2, t^*]}(t) - \int_{r_2}^t \mu X_{j-1}(u) e^{-\int_{r_2}^u G_j^*(v) \: dv} \: du$$ as in Corollary \ref{ZmartCor4}.  Then
\begin{equation}\label{mainXbracket}
X_j^{[r_2, t^*]}(t^* \wedge \tau) = \int_{r_2}^{t^* \wedge \tau} \mu X_{j-1}(u) e^{\int_u^{t^* \wedge \tau} G_j^*(v) \: dv} \: du + e^{\int_{r_2}^{t^* \wedge \tau} G_j(v) \: dv} Z_j^{[r_2, t^*]}(t^* \wedge \tau).
\end{equation}
Note that $r_2 \geq (1/s) \log k_N$.  Assume for now that $\tau > t^*$ and that the event in (\ref{p551}) holds so that, in particular,
\begin{equation}\label{1-dXj-1}
(1 - \delta)x_{j-1}(t) \leq X_{j-1}(t) \leq (1 + \delta)x_{j-1}(t) \hspace{.2in}\mbox{ for all }t \in [r_2, t^*].
\end{equation}
Using (\ref{expintegral}),
\begin{align}\label{l58new}
\int_{r_2}^{t^*} \mu x_{j-1}(u) e^{sj(t^* - u)} \: du &= \frac{N \mu^j e^{sjt^*}}{s^{j-1} (j-1)!} \int_{r_2}^{t^*} e^{-sju} (e^{su} - 1)^{j-1} \: du \nonumber \\
&= \frac{N \mu^j e^{sjt^*}}{s^j j!} \big( (1 - e^{-st^*})^j - (1 - e^{-sr_2})^j \big) \nonumber \\
&= \frac{N \mu^j e^{sjt^*}}{s^j j!} \bigg( \bigg(1 - \frac{1}{k_N^4} \bigg)^j - \bigg(1 - \frac{1}{k_N^{d_j + 1}} \bigg)^j \bigg).
\end{align}
We need to consider the asymptotic behavior of $h_N = (1 - k_N^{-4})^j - (1 - k_N^{-(d_j + 1)})^j$ as $N \rightarrow \infty$.  First suppose $b_j \leq 0$.  Note that $\frac{d}{dx} (1-x)^j = -j(1-x)^{j-1}$.  Therefore, using $\sim$ to denote that the ratio of the two sides tends to one as $N \rightarrow \infty$, we have
\begin{equation}\label{l581}
h_N \leq j(k_N^{-(d_j + 1)} - k_N^{-4}) \leq j k_N^{-(d_j + 1)} \sim k_N^{-d_j}
\end{equation}
and
\begin{equation}\label{l582}
h_N \geq j(1 - k_N^{-1})^{j-1}(k_N^{-(d_j + 1)} - k_N^{-4}) \sim e^{-1} k_N^{-d_j}.
\end{equation}
Combining (\ref{tstark4}), (\ref{1-dXj-1}), (\ref{l58new}), (\ref{l581}), and (\ref{l582}), we get that there are positive constants $c_1$ and $c_2$ such that for sufficiently large $N$,
\begin{equation}\label{Xbracket1}
c_1 k_N^{-d_j} x_j(t^*)  \leq \int_{r_2}^{t^*} \mu X_{j-1}(u) e^{sj(t^* - u)} \: du \leq c_2 k_N^{-d_j} x_j(t^*) .
\end{equation}
In view of (\ref{deltadef}), the constants $c_1$ and $c_2$ can be chosen so that the equation holds for all allowable values of $\delta$.  Also, using (\ref{1-dXj-1}) and then reasoning as in (\ref{Xl42}), we get
\begin{equation}\label{Xbracket2}
0 \leq \int_{r_2}^{t^*} \mu X_{j-1}(u) (e^{\int_u^{t^* \wedge \tau} G^*_j(v) \: dv} - e^{sj(t^* - u)}) \: du \leq (1 + \delta)(s \eta + \mu) t^* x_j(t^*).
\end{equation}
Since $(1 + \delta)(s \eta + \mu)t^* k_N^{d_j} \rightarrow 0$ as $N \rightarrow \infty$ by the reasoning in (\ref{jtstarto0}), it follows from (\ref{Xbracket1}) and (\ref{Xbracket2}) that there are positive constants $c_3$ and $c_4$ such that for sufficiently large $N$,
\begin{equation}\label{Xbracket3}
c_3 k_N^{-d_j} x_j(t^*) \leq \int_{r_2}^{t^*} \mu X_{j-1}(u)  e^{\int_u^{t^* \wedge \tau} G_j^*(v) \: dv} \: du \leq c_4 k_N^{-d_j} x_j(t^*).
\end{equation}

We still need to control the second term on the right-hand side of (\ref{mainXbracket}), which requires bounding $Z_j^{[r_2, t^*]}(t^* \wedge \tau)$.  By Corollary \ref{ZmartCor4}, 
\begin{align*}
&\Var(Z_j^{[r_2, t^*]}(t^* \wedge \tau)|{\cal F}_{r_2}) \\
&\hspace{.2in} = E \bigg[ \int_{r_2}^{t^* \wedge \tau} e^{-2 \int_{r_2}^u G_j(v) \: dv}(\mu X_{j-1}(u) + B_j^{[r_2, t^*]}(u) X_j^{[r_2, t^*]}(u) + D_j^{[r_2, t^*]}(u) X_j^{[r_2, t^*]}(u)) \: du \bigg| {\cal F}_{r_2} \bigg].
\end{align*}
We now take expectations of both sides of this equation.  Using that $X_j^{[r_2, t^*]}(u) \leq X_j(u)$ for $u \leq \tau$, that $B_j^{[r_2, t^*]}(u) + D_j^{[r_2, t^*]}(u) \leq 3$ by the reasoning that leads to (\ref{BD3}), and that Lemma \ref{EXjupper} holds, we get for sufficiently large $N$,
\begin{align*}
E\big[\Var(Z_j^{[r_2, t^*]}(t^* \wedge \tau)|{\cal F}_{r_2})\big]
&\leq E \bigg[ \int_{r_2}^{t^*} e^{-2 \int_{r_2}^u G_j(v) \: dv} (\mu X_{j-1}^{\tau}(u) + 3 X_j^{\tau}(u)) \: du \bigg] \nonumber \\
&\leq \int_{r_2}^{t^*} e^{-2(sj - s \eta - \mu)(u - r_2)} \bigg( \frac{N \mu^j (e^{su} - 1)^{j-1}}{s^{j-1} (j-1)!} + \frac{3N \mu^j (e^{su} - 1)^j}{s^j j!} \bigg) \: du \nonumber \\
&\leq e^{2(s \eta + \mu) t^*} e^{2sj r_2} \cdot \frac{N \mu^j}{s^j j!} \int_{r_2}^{t^*} e^{-2sju} \big( (e^{su} - 1)^{j-1} sj + 3(e^{su} - 1)^j \big) \: du \nonumber \\
&\leq e^{2(s \eta + \mu) t^*} e^{2sj r_2} \cdot \frac{N \mu^j}{s^j j!} \int_{r_2}^{t^*} (e^{-s(j+1)u} sj + 3 e^{-sju}) \: du.
\end{align*}
Reasoning as in the derivation of (\ref{mainvarZfinal}) from (\ref{mainvarZ}), we have $e^{-s(j+1)u} sj + 3e^{-sju} \leq e^{-sjr_2}(e^{-su}sj + 3)$ for $u \geq r_2$, so for sufficiently large $N$,
\begin{equation}\label{displayvar}
E\big[\Var(Z_j^{[r_2, t^*]}(t^* \wedge \tau)|{\cal F}_{r_2})\big] \leq \frac{4 N \mu^j t^*}{s^j j!} e^{s j r_2}.
\end{equation}
Note that if $Y$ is a random variable and ${\cal G}$ is a $\sigma$-field such that $E[Y|{\cal G}] = 0$, then by the conditional Chebyshev's Inequality,
$$P(|Y| > a) = E[P(|Y| > a|{\cal G})] \leq E \bigg[ \frac{\Var(Y|{\cal G})}{a^2} \bigg] = \frac{E[\Var(Y)|{\cal G}]}{a^2}.$$
Therefore, (\ref{displayvar}) implies 
\begin{equation}\label{Z3bound}
P\bigg(|Z_j^{[r_2, t^*]}(t^* \wedge \tau)| > \sqrt{\frac{48 N \mu^j t^* e^{sjr_2}}{\eps s^j j!}} \bigg) \leq \frac{\eps}{12}.
\end{equation}
In view of (\ref{tstark4}), when the event in (\ref{Z3bound}) holds, for sufficiently large $N$ we have
$$e^{\int_{r_2}^{t^* \wedge \tau} G_j(v) \: dv} |Z_j^{[r_2, t^*]}(t^* \wedge \tau)| \leq e^{sj(t^* - r_2)} \sqrt{\frac{48 N \mu^j t^* e^{s j r_2}}{\eps s^j j!}} \leq \sqrt{\frac{49 s^j j! t^* e^{-s j r_2}}{\eps N \mu^j}} \: x_j(t^*).$$
Combining this result with (\ref{mainXbracket}) and (\ref{Xbracket3}), we get that when equation (\ref{1-dXj-1}) and the event in (\ref{Z3bound}) hold and when $\tau > t^*$, we have
\begin{equation}\label{c3c4}
(c_3 - y_N) k_N^{-d_j} x_j(t^*) \leq X_j^{[r_2, t^*]}(t^*) \leq (c_4 + y_N) k_N^{-d_j} x_j(t^*)
\end{equation} 
for sufficiently large $N$,
where $$y_N = \sqrt{\frac{49 s^j j! t^* e^{-s j r_2}}{\eps N \mu^j}} \: k_N^{d_j}.$$  In view of Lemma \ref{tautstarlem}, Proposition \ref{early1prop}, and equation (\ref{Z3bound}), the result will follow if we can show that $y_N \rightarrow 0$ as $N \rightarrow \infty$.  To show this, we make a calculation similar to the calculation in the proof of Part 1 of Proposition \ref{earlyprop}.  Noting that $e^{-sjr_2} = e^{-(d_j + 1) j \log k_N}$ and using (\ref{logfact}), (\ref{klog1s}), and (\ref{jlogj}), we get
\begin{align*}
&\log \bigg( \sqrt{\frac{49 s^j j! t^* e^{-sjr_2}}{\eps N \mu^j}} \: k_N^{d_j} \bigg) \\
&\hspace{.2in} = \frac{1}{2} \bigg(\log 49 + j \log\bigg(\frac{s}{\mu} \bigg) + \log j! + \log t^* - (d_j + 1) j \log k_N - \log \eps - \log N \bigg) + d_j \log k_N \\
&\hspace{.2in} = \frac{1}{2} \bigg( j \log \bigg( \frac{s}{\mu} \bigg) + j \log j - j - (d_j + 1)j \log k_N - \log N \bigg) + o(k_N) \\
&\hspace{.2in} = \frac{1}{2} \bigg( \log N + b_j k_N \log k_N + k_N \log k_N - j - (d_j + 1) k_N \log k_N - \log N \bigg) + o(k_N) \\
&\hspace{.2in} = - \frac{j}{2} + \frac{(b_j - d_j) k_N \log k_N}{2}o(k_N),
\end{align*}
which tends to $- \infty$ as $N \rightarrow \infty$.  Thus, $y_N \rightarrow 0$ as $N \rightarrow \infty$, which completes the proof.
\end{proof}

Combining Lemmas \ref{subpop1}, \ref{subpop2}, and \ref{subpop3} and using (\ref{j1j2j3}), we arrive immediately at the following result, which is essentially part 2 of Proposition \ref{earlyprop}.

\begin{Prop}\label{early2prop}
There exist positive constants $C_1$ and $C_2$ such that for sufficiently large $N$, we have, for all $j \in (k_N^-, k_N^+)$, $$P \big( C_1 k_N^{-d_j} x_j(t^*) \leq X_j(t^*) \leq C_2 k_N^{-d_j} x_j(t^*) \big) > 1 - \frac{\eps}{3}.$$
\end{Prop}

\subsection{Proof of parts 3 and 4 of Proposition \ref{earlyprop}}

In this subsection, we complete the proof of Proposition \ref{earlyprop}.  We will need the following lemma.  Recall from (\ref{kstar}) that $k^* = \max\{j \in \N: j < k_N^+\}$.

\begin{Lemma}\label{logto0}
We have $$\lim_{N \rightarrow \infty} \frac{N \mu^{k^* + 1} e^{sk^*t^*}}{s^{k^*+1}k^*!} = 0.$$
\end{Lemma}

\begin{proof}
We have, using the reasoning in (\ref{logfact}),
\begin{align}\label{logpt3}
\log \bigg( \frac{N \mu^{k^* + 1} e^{sk^* t^*}}{s^{k^* + 1} k^*!} \bigg) &= \log N - (k^* + 1) \log \bigg( \frac{s}{\mu} \bigg) + s k^* t^* - \log k^*!  \nonumber \\
&= - (k^* + 1 - k_N) \log \bigg( \frac{s}{\mu} \bigg) + s k^* t^* - k^* \log k^* + k^* + o(k_N).
\end{align}
We consider two cases.  First, suppose $k^* + 1 - k_N \geq 1/2$.  It follows from assumption A2 that $(k_N \log k_N)/\log(s/\mu) \rightarrow 0$ as $N \rightarrow \infty$.  Therefore, the first term dominates the expression in (\ref{logpt3}), so the expression tends to $-\infty$ as $N \rightarrow \infty$.  On the other hand, suppose $k^* + 1 - k_N < 1/2$.  Then $k^* < k_N - 1/2$, which for sufficiently large $N$ implies that $k^* < k_N^-$ by (\ref{kdiff}).  It follows that there are no integers in the interval $(k_N^-, k_N^+)$, which means $t^* = (2/s) \log k_N$.  Because $k^* + 1 \geq k_N^+$, we have $k^* + 1 - k_N \geq k_N^+ - k_N$, so in this case, starting from (\ref{logpt3}),
\begin{align*}
\log \bigg( \frac{N \mu^{k^* + 1} e^{sk^* t^*}}{s^{k^* + 1} k^*!} \bigg) &\leq -(k_N^+ - k_N) \log \bigg( \frac{s}{\mu} \bigg) + 2 k^* \log k_N - k^* \log k^* + k^* + o(k_N) \\
&= - 2 k_N \log k_N + 2 k^* \log k_N - k^* \log k^* + k^* + o(k_N),
\end{align*}
which tends to $-\infty$ as $N \rightarrow \infty$ because $k_N \sim k^*$ as $N \rightarrow \infty$.  The result follows.
\end{proof}

The results below establish parts 3 and 4 of Proposition \ref{earlyprop}.  Proposition \ref{earlyprop} follows immediately from Propositions \ref{early1prop}, \ref{early2prop}, \ref{early3prop}, and \ref{early4prop}.

\begin{Prop}\label{early3prop}
For sufficiently large $N$,
$$P(X_{k^*}(t) < s/\mu \mbox{ for all }t \in [0, t^*]) \geq 1 - \frac{\eps}{24}.$$
\end{Prop}

\begin{proof}
For $t < \tau$, we have $B_{k^*}(t) - D_{k^*}(t) = G_{k^*}^*(t) = G_{k^*}(t) \geq s(k^* - \eta) - \mu > 0$ for sufficiently large $N$.  Since the rate of events that increase the number of type $k^*$ individuals by one is thus always greater than the rate of events that decrease the number of type $k^*$ individuals by one, the process $(X_{k^*}^{\tau}(t), t \geq 0)$ is a submartingale.  By Doob's Maximal Inequality and Lemma \ref{EXjupper}, for sufficiently large $N$,
$$P \bigg( \sup_{t \in [0, t^*]} X_{k^*}^{\tau}(t) \geq \frac{s}{\mu} \bigg) \leq \frac{E[X_{k^*}^{\tau}(t^*)]}{s/\mu} \leq \frac{N \mu^{k^*+1} (e^{st^*} - 1)^{k^*}}{s^{k^*+1}k^*!}.$$
This expression tends to zero as $N \rightarrow \infty$ by Lemma \ref{logto0} which, in view of Lemma \ref{tautstarlem}, implies the result.
\end{proof}

\begin{Prop}\label{early4prop}
For sufficiently large $N$,
$$P\big(X_j(t^*) = 0 \mbox{ for all }j \geq k_N^+ \mbox{ and }t \in [0, t^*] \big) \geq 1 - \frac{\eps}{24}.$$
\end{Prop}

\begin{proof}
Each individual of type $k^*$ acquires mutations at rate $\mu$.  Therefore, by Lemma \ref{EXjupper}, the expected number of times, before time $t^* \wedge \tau$, that a type $k^*$ individual acquires a $(k^* + 1)$st mutation is at most $$\int_0^{t^*} \mu E[X_{k^*}^{\tau}(t)] \: dt \leq \frac{N \mu^{k^* + 1}}{s^{k^*} k^*!} \int_0^{t^*} (e^{st} - 1)^{k^*} \: dt \leq \frac{N \mu^{k^* + 1} e^{sk^* t^*}}{s^{k^* + 1} k^*! k^*}.$$  This expression tends to zero as $N \rightarrow \infty$ by Lemma \ref{logto0} and the fact that $k^* \rightarrow \infty$ as $N \rightarrow \infty$.  The result now follows from Markov's Inequality and Lemma \ref{tautstarlem}.  
\end{proof}

\section{Proof of part 1 of Proposition \ref{zetaprop}}\label{zetasec1}

Recall that Proposition \ref{meanprop} states that $M(t)$ is close to zero for $t \leq a_N$ and close to $j$ during the time interval $[\gamma_j, \gamma_{j+1})$.  The time $\zeta_2$ can be interpreted as the first time at which the approximation to $M(t)$ given in Proposition \ref{meanprop} fails to hold.  Part 1 of Proposition \ref{zetaprop} stipulates that, during the time interval $[t^*, a_N T]$, the time $\zeta_2$ can not happen until either $\zeta_1$ or $\zeta_3$ has occurred.  That is, as long as the behavior of the type $j$ individuals follows the description in Propositions \ref{prop1}, \ref{prop2} and \ref{tauprop}, the mean number of mutations in the population must satisfy the approximation in Proposition \ref{meanprop}.

Note that part 1 of Proposition \ref{zetaprop} is a deterministic statement.  To prove it, we will assume that $\zeta_0 = \infty$, meaning that until time $t^*$ the population behaves according to Proposition \ref{earlyprop}.  We will show that if $t \in (t^*, a_NT]$ and $\zeta_1 \wedge \zeta_3 > t$, then the approximation in Proposition \ref{meanprop} is valid up through time $t$.  We begin with two lemmas.  The first one gives a useful bound that follows from (\ref{prop21}) and (\ref{prop22}), and the second one shows that if $t < \tau_{j+1}$, then type $j$ individuals contribute little to the mean number of mutations at time $t$.

\begin{Lemma}\label{smulem}
Let $C_6 = C_3 + 1 + 4 \delta$, where $C_3$ comes from (\ref{prop21}).
Suppose $j \geq k^*+1$.  Suppose $\tau_{j+1} < \zeta_{1,j}$ and $\tau_{j+1} \leq a_N T$.  Suppose also that either $\tau_{j+1} < \zeta_3$ or $j \leq J$.  Then for sufficiently large $N$,
\begin{equation}\label{mainintG}
\frac{s}{C_6 \mu}\leq \exp \bigg( \int_{\tau_j}^{\tau_{j+1}} G_j(v) \: dv \bigg) \leq \frac{2s}{\mu}.
\end{equation}
Also, suppose $k^*+1 \leq j \leq J$, $\tau_j \leq t < \zeta_{1,j}$, and $t \leq a_N T$.  Then for sufficiently large $N$, if
\begin{equation}\label{mainintG2}
\exp \bigg( \int_{\tau_j}^t G_j(v) \: dv \bigg) \geq \frac{2s}{\mu},
\end{equation}
we have $\tau_{j+1} \leq t$.
\end{Lemma}

\begin{proof}
Suppose $j \geq k^*+1$, $\tau_{j+1} < \zeta_{1,j}$, and $\tau_{j+1} \leq a_N T$.  If $\tau_{j+1} < \zeta_3$, then (\ref{tauspacing}) gives $\tau_{j+1} \geq \tau_j + a_N/3k_N \geq \tau_j^*$.  Now suppose instead $j \leq J$. Then for $t < \tau_j^* \wedge \zeta_{1,j}$, part 1 of Proposition \ref{prop2} gives $X_{j,1}(t) \leq s/2 \mu$.  Since $X_{j,2}(t) = 0$ for $t \leq \xi_j$, part 2 of Proposition \ref{prop2} gives for $t < \tau_j^* \wedge \zeta_{1,j}$, $$X_{j,2}(t) \leq (1 + 4 \delta) e^{\int_{\tau_j}^t G_j(v) \: dv} \leq (1 + 4 \delta) e^{sJ(\tau_j^* - \tau_j)} = (1 + 4 \delta) \bigg( \frac{s}{\mu} \bigg)^{J/4Tk_N},$$ which by (\ref{Jdef}) is less than $s/2 \mu$ for sufficiently large $N$.  Thus, $X_j(t) < s/\mu$ if $N$ is sufficiently large and $t < \tau_j^* \wedge \zeta_{1,j}$, which means $\tau_{j+1} \geq \tau_j^*$ in this case also.

Recall from (\ref{taujdef}) that $X_j(\tau_{j+1}) = \lceil s/\mu \rceil$.  When $\tau_{j+1} \geq \tau_j^*$, equations (\ref{prop21}) and (\ref{prop22}) with $t = \tau_{j+1}$ give
$$(1 - 4 \delta) e^{\int_{\tau_j}^{\tau_{j+1}} G_j(v) \: dv} \leq \lceil s/\mu \rceil \leq (1 + 4 \delta + C_3) e^{\int_{\tau_j}^{\tau_{j+1}} G_j(v) \: dv}.$$  The result (\ref{mainintG}) now follows immediately from rearranging this equation and observing that $\lceil s/\mu \rceil/ ((s/\mu)(1 - 4 \delta)) \leq 2$ for sufficiently large $N$.

To prove the last statement of the lemma, suppose $k^*+1 \leq j \leq J$, $\tau_j \leq t < \zeta_{1,j}$, $t \leq a_N T$, and (\ref{mainintG2}) holds.  For $t < \tau_j^*$, if $N$ is sufficiently large, then $$\int_{\tau_j}^t G_j(v) \: dv \leq sJ (\tau_j^* - \tau_j) < \log\bigg( \frac{s}{\mu} \bigg),$$ contradicting (\ref{mainintG2}).  Therefore, we must have $t \geq \tau_j^*$.  If $\tau_j^* \leq t < \tau_{j+1}$, then equation (\ref{prop22}) gives $X_j(t) \geq 2(1 - 4 \delta)(s/\mu) \geq s/\mu$, contradicting the definition of $\tau_{j+1}$.  Thus, if $N$ is sufficiently large, then $\tau_{j+1} \leq t$, as claimed.
\end{proof}

\begin{Lemma}\label{newjbound}
If $t \in (t^*, a_N T]$ and $\zeta_1 \wedge \zeta_3 > t$, then $$\frac{1}{N} \sum_{j=k^* + 1}^{\infty} j X_j(t) \1_{\{\tau_{j+1} > t\}} \leq \frac{J s}{N \mu}.$$
\end{Lemma}

\begin{proof}
Suppose $j \geq k^* + 1$.  If the statement of part 1 of Proposition \ref{prop2} holds, then no early type $j$ individual can get a type $j+1$ mutation until after time $\tau_{j+1} \wedge a_N T$.  No other type $j$ individual appears until after time $\xi_j \geq \tau_j$.
Thus, we have $X_{j+1} = 0$ for $t \leq \tau_{j+1} \wedge \tau_j \wedge a_N T$.  Also, $\tau_{j+1} \geq \tau_j \wedge a_N T$, as noted in Remark \ref{tauorder}.  Thus, since we are assuming that $\zeta_1 \wedge \zeta_3 > t$, we have $X_{j+1}(t) = 0$ on the event that $t \leq \tau_j$.  Therefore, when $t \in (t^*, a_N T]$ and $\zeta_1 \wedge \zeta_3 > t$, there can be at most one value of $j$ for which $\tau_{j+1} > t$ but $X_j(t) > 0$.

Because $\zeta_3 > t$, the calculation in (\ref{JRemeq}) implies that $\tau_J > t$ and thus $X_j(t) = 0$ for all $j > J$.  Because $X_j(t) \leq s/\mu$ when $t < \tau_{j+1}$, the result follows.
\end{proof}

The approximation in Proposition \ref{meanprop} has four parts.  The first part pertains to the case $t \leq a_N$, the second part pertains to the case $t \in [a_N, \gamma_{k^* + 1})$, and the third part pertains to the case in which $t \in [\gamma_j, \gamma_{j+1})$ for some $j \geq k^* + 1$.  The fourth part will be a consequence of the first three.  Proposition \ref{tint1} below handles the case of $t \leq a_N$.  

\begin{Prop}\label{tint1}
For sufficiently large $N$, on the event that $\zeta_0 = \infty$ and $\zeta_1 \wedge \zeta_3 > t$, we have for all $t \in (t^*, a_N]$, $$M(t) < 3e^{-s(a_N - t)}.$$
\end{Prop}

\begin{proof}
Fix $t \in (t^*, a_N]$.  We will assume throughout the proof that $\zeta_0 = \infty$ and $\zeta_1 \wedge \zeta_3 > t$.  Suppose first that $0 \leq j \leq k_N^-$.  Write $\alpha = (1 + \delta)^2/(1 - \delta)^2$.  By equations (\ref{prop11}) and (\ref{early1}) and the fact that $G_j(t) - G_0(t) = sj$ for all $t \geq 0$, we have
\begin{equation}\label{Xjbound}
\frac{X_j(t)}{N} \leq \frac{X_j(t)}{X_0(t)} \leq \frac{\alpha x_j(t^*) e^{\int_{t^*}^t G_j(v) \: dv}}{x_0(t^*) e^{\int_{t^*}^t G_0(v) \: dv}} = \frac{\alpha \mu^j (e^{st^*} - 1)^j}{s^j j!} e^{sj(t - t^*)} \leq \frac{\alpha \mu^j e^{sjt}}{s^j j!}.
\end{equation}
Therefore, 
$$\frac{1}{N} \sum_{j=1}^{\lfloor k_N \rfloor} j X_j(t) \leq \alpha \sum_{j=1}^{\lfloor k_N \rfloor} \frac{j}{j!} \bigg( \frac{\mu e^{st}}{s} \bigg)^j = \frac{\alpha \mu e^{st}}{s} \sum_{j=0}^{\lfloor k_N \rfloor - 1} \frac{1}{j!} \bigg( \frac{\mu e^{st}}{s} \bigg)^j \leq \frac{\alpha \mu e^{st}}{s} e^{(\mu/s)e^{st}}.$$  Now
\begin{equation}\label{musest}
\frac{\mu}{s} e^{st} = \frac{\mu}{s} e^{sa_N} e^{-s(a_N - t)} = e^{-s(a_N - t)}.
\end{equation}
Because $e^{-s(a_N - t)} \leq 1$ and thus $e^{(\mu/s) e^{st}} \leq e$, it follows that
\begin{equation}\label{MaN1}
\frac{1}{N} \sum_{j=1}^{\lfloor k_N \rfloor} j X_j(t) \leq \alpha e e^{-s(a_N - t)}.
\end{equation}

Next, suppose $j \in (k_N^-, k_N^+)$.  Then, using (\ref{earlypt2}) instead of (\ref{early1}), the same reasoning used in (\ref{Xjbound}) gives that for some positive constant $C_7$, $$\frac{X_j(t)}{N} \leq \frac{C_7 \mu^j e^{jst}}{s^j j!}.$$
For sufficiently large $N$, there will be at most one integer in the interval $(k_N^-, k_N^+)$.  In this case, using (\ref{musest}) and then using that $(\mu/s)e^{st} \leq 1$ for the last inequality, we get
\begin{equation}\label{MaN2}
\frac{1}{N} \sum_{j \in (k_N^-, k_N^+) \cap \Zsm} j X_j(t) \leq \frac{C_7 \mu e^{st}}{s} \sum_{j \in (k_N^-, k_N^+) \cap \Zsm} \frac{1}{j!} \bigg( \frac{\mu e^{st}}{s} \bigg)^{j-1} \leq \frac{C_7}{\lceil k_N^- \rceil !} e^{-s(a_N - t)}.
\end{equation}

Consider now the case in which $j \geq k^* + 1$ and $\tau_{j+1} \leq t$.  Then by (\ref{prop23}),
\begin{equation}\label{Xjup}
X_j(t) \leq (1 + \delta)(s/\mu) e^{\int_{\tau_{j+1}}^t G_j(v) \: dv}.
\end{equation}
The assumption that $t < \zeta_1$ entails that $t^* < \tau_{k^*+1} \leq \tau_{j+1}$
in view of part 3 of Proposition \ref{earlyprop} and Remark \ref{tauorder}, so using (\ref{prop11}), we get $$s/\mu \leq X_{k^*}(\tau_{k^*+1}) \leq (1 + \delta) X_{k^*}(t^*) e^{\int_{t^*}^{\tau_{k^* + 1}} G_{k^*}(v) \: dv}.$$  Therefore, another application of (\ref{prop11}) leads to
\begin{equation}\label{Xklow}
X_{k^*}(t) \geq (1 - \delta) X_{k^*}(t^*) e^{\int_{t^*}^t G_{k^*}(v) \: dv} \geq \frac{(1 - \delta) s}{(1 + \delta) \mu} e^{\int_{\tau_{k^*+1}}^t G_{k^*}(v) \: dv}.
\end{equation}
Thus, since $G_j(v) - G_{k^*}(v) = s(j - k^*)$ for all $v \geq 0$, combining (\ref{Xjup}) and (\ref{Xklow}) leads to
$$\frac{X_j(t)}{X_{k^*}(t)} \leq \frac{(1 + \delta)^2}{1 - \delta} e^{s (j - k^*)(t - \tau_{j+1})} e^{-\int_{\tau_{k^* + 1}}^{\tau_{j+1}} G_{k^*}(v) \: dv}.$$  Now
\begin{align*}
\int_{\tau_{k^* + 1}}^{\tau_{j+1}} G_{k^*}(v) \: dv &= \sum_{m=1}^{j-k^*} \bigg( \int_{\tau_{k^* + m}}^{\tau_{k^* + m + 1}} G_{k^* + m}(v) \: dv - sm(\tau_{k^*+m+1} - \tau_{k^*+m})\bigg) \\
&\geq \bigg( \sum_{m=1}^{j-k^*} \int_{\tau_{k^* + m}}^{\tau_{k^* + m + 1}} G_{k^* + m}(v) \: dv \bigg) - s (j - k^*) (\tau_{j+1} - \tau_{k^*+1}).
\end{align*}
By Lemma \ref{smulem}, for sufficiently large $N$ we have 
$$\exp \bigg( \sum_{m=1}^{j-k^*} \int_{\tau_{k^* + m}}^{\tau_{k^* + m + 1}} G_{k^* + m}(v) \: dv \bigg) \geq \bigg( \frac{s}{C_6 \mu} \bigg)^{j - k^*}.$$  Combining these observations gives that for sufficiently large $N$,
\begin{align*}
\frac{X_j(t)}{X_{k^*}(t)} &\leq \alpha (1 - \delta) e^{s(j - k^*)(t - \tau_{j+1})} \bigg( \frac{C_6 \mu}{s} \bigg)^{j - k^*} e^{s(j - k^*)(\tau_{j+1} - \tau_{k^*+1})} \\
&= \alpha (1 - \delta) \bigg( \frac{C_6 \mu}{s} \bigg)^{j-k^*} e^{s (j - k^*)(t - \tau_{k^*+1})}.
\end{align*}
Since $\tau_{k^*+1} > t^*$ and $e^{-s(j - k^*)t^*} \leq k_N^{-2(j - k^*)}$, it follows that for sufficiently large $N$,
$$\frac{X_j(t)}{X_{k^*}(t)} \leq \alpha \bigg( \frac{C_6 \mu}{s} \bigg)^{j-k^*} e^{s(j - k^*) t} k_N^{-2(j-k^*)} = \alpha \bigg( \frac{C_6 \mu e^{st}}{sk_N^2} \bigg)^{j-k^*}.$$  Thus, making the substitution $\ell = j - k^*$, for sufficiently large $N$,
\begin{equation}\label{Xjinfsum}
\frac{1}{N} \sum_{j=k^* + 1}^{\infty} j X_j(t) \1_{\{\tau_{j+1} \leq t\}} \leq \sum_{j=k^* + 1}^{\infty} \frac{j X_j(t)}{X_{k^*}(t)} \1_{\{\tau_{j+1} \leq t\}} \leq \alpha \sum_{\ell = 1}^{\infty} (k^* + \ell) \bigg( \frac{C_6 \mu e^{st}}{s k_N^2} \bigg)^{\ell}.
\end{equation}
In view of (\ref{musest}), we see that $C_6 \mu e^{st}/sk_N^2 \rightarrow 0$ as $N \rightarrow \infty$, and therefore the infinite sum on the right-hand side of (\ref{Xjinfsum}) is dominated by the leading term when $N$ is large.  Therefore, for sufficiently large $N$, using (\ref{musest}) again,
\begin{equation}\label{MaN3}
\frac{1}{N} \sum_{j=k^* + 1}^{\infty} j X_j(t) \1_{\{\tau_{j+1} \leq t\}} \leq \frac{2 \alpha C_6 k^*}{k_N^2} \cdot \frac{\mu e^{st}}{s} = \frac{2 \alpha C_6 k^*}{k_N^2} e^{-s(a_N - t)}.
\end{equation}

It remains only to consider the case in which $j \geq k^* + 1$ and $\tau_{j+1} > t$, for which the necessary bound is given in Lemma \ref{newjbound}.  Combining (\ref{MaN1}), (\ref{MaN2}), (\ref{MaN3}), and Lemma \ref{newjbound}, we get that for sufficiently large $N$, $$M(t) \leq \bigg( \alpha e + \frac{C_7}{\lceil k_N^- \rceil !} + \frac{2 \alpha C_6 k^*}{k_N^2} + \frac{J s e^{s(a_N - t)}}{N \mu} \bigg) e^{-s(a_N - t)}.$$  As $N \rightarrow \infty$, clearly $C_7/\lceil k_N^- \rceil ! \rightarrow 0$ and $2 \alpha C_6 k^*/k_N^2 \rightarrow 0$.  As for the fourth term, we have $e^{s(a_N - t)} \leq e^{sa_N} = s/\mu$, which means
\begin{equation}\label{JsN}
\frac{J s e^{s(a_N - t)}}{N \mu} \leq \frac{J s^2}{N \mu^2} \rightarrow 0
\end{equation}
as $N \rightarrow \infty$ by (\ref{muspower}) and (\ref{muNpower}).  The result $M(t) \leq 3 e^{-s(a_N - t)}$ follows because $\alpha e < 3$ by (\ref{deltadef}).
\end{proof}

We next consider the case in which $t \in (a_N, \gamma_{k^* + 1})$.  During this period of time, the mean number of mutations in the population increases rapidly from near zero at time $a_N$ to near $k^*$ at time $\gamma_{k^* + 1}$.  The upper bound on the mean number of mutations given by Proposition \ref{tint2} below will be sufficient for our purposes.  Before stating this proposition, we prove a lemma which will also be useful in studying the population at later times.

\begin{Lemma}\label{Xjllem}
Suppose $j$ and $\ell$ are positive integers with $j \geq k^*$.  Let $\alpha_j = (1 + \delta)^2/(1 - \delta)$ if $j = k^*$ and $\alpha_j = (1 + \delta)/(1 + \delta)$ if $j \geq k^* + 1$.  Suppose $t \in [\tau_{j + \ell + 1}, \gamma_{j + K}] \cap [0, a_N T]$.  Suppose also that $\zeta_0 = \infty$ and $\zeta_1 \wedge \zeta_3 > t$.  Then for sufficiently large $N$,
$$\frac{X_{j + \ell}(t)}{X_j(t)} \leq \alpha_j \bigg( \frac{C_6 \mu}{s} \bigg)^{\ell} \bigg( \frac{\mu}{s} \bigg)^{\ell(\ell - 1)/6k_N} e^{s \ell (t - \tau_{j+1})}.$$
\end{Lemma}

\begin{proof}
Assume for now that $j \geq k^* + 1$.  Then because $\zeta_1 > t$, the bounds in (\ref{prop23}), combined with the facts that $\gamma_{j+K} < \gamma_{j+\ell+K}$ and $\tau_{j+1} < \tau_{j+\ell+1}$ by Remark \ref{tauorder}, give
\begin{align}\label{Xratio}
\frac{X_{j + \ell}(t)}{X_j(t)} &\leq \frac{\alpha_j e^{\int_{\tau_{j+\ell+1}}^t G_{j+ \ell}(v) \: dv}} {e^{\int_{\tau_{j+1}}^t G_j(v) \: dv}} \nonumber \\
&= \alpha_j e^{-\int_{\tau_{j+1}}^{\tau_{j+\ell+1}} G_{j+\ell}(v) \: dv} e^{\int_{\tau_{j+1}}^t (G_{j+\ell}(v) - G_j(v)) \: dv} \nonumber \\
&= \alpha_j e^{-\int_{\tau_{j+1}}^{\tau_{j+\ell+1}} G_{j+\ell}(v) \: dv} e^{s \ell (t - \tau_{j+1})}.
\end{align}
If instead $j = k^*$, then we use (\ref{prop11}), as in (\ref{Xklow}), rather than (\ref{prop23}) to get the lower bound on $X_j(t)$, and we again obtain (\ref{Xratio}).  In both cases,
\begin{align*}
\int_{\tau_{j+1}}^{\tau_{j+\ell+1}} G_{j+\ell}(v) \: dv &= \sum_{m=1}^{\ell} \int_{\tau_{j+m}}^{\tau_{j+m+1}} G_{j+\ell}(v) \: dv \\
&= \sum_{m=1}^{\ell} \bigg( \int_{\tau_{j+m}}^{\tau_{j+m+1}} G_{j+m}(v) \: dv + s(\ell - m)(\tau_{j+m+1} - \tau_{j+m}) \bigg).
\end{align*}
We now apply Lemma \ref{smulem} and (\ref{tauspacing}) to get that for sufficiently large $N$,
\begin{align*}
\exp \bigg( \int_{\tau_{j+1}}^{\tau_{j+\ell+1}} G_{j+\ell}(v) \: dv \bigg) &\geq \bigg( \frac{s}{C_6 \mu} \bigg)^{\ell} \exp \bigg( \sum_{m=1}^{\ell} s (\ell - m) \cdot \frac{a_N}{3 k_N} \bigg) \\
&= \bigg( \frac{s}{C_6 \mu} \bigg)^{\ell} \exp \bigg( \frac{s a_N \ell (\ell - 1)}{6 k_N} \bigg).
\end{align*}  
Because $e^{sa_N} = s/\mu$, combining this inequality with (\ref{Xratio}) gives the result.
\end{proof}

\begin{Prop}\label{tint2}
There is a positive constant $C_4$ such that if $N$ is sufficiently large, then for all $t \in (a_N, \gamma_{k^* + 1})$, on the event that $\zeta_0 = \infty$ and $\zeta_1 \wedge \zeta_3 > t$ we have $$M(t) < k_N + C_4.$$ 
\end{Prop}

\begin{proof}
Suppose $t \in (a_N, \gamma_{k^*+1})$.  Suppose also that $\zeta_0 = \infty$ and $\zeta_1 \wedge \zeta_3 > t$.  Note that 
\begin{equation}\label{Mtsplit}
M(t) = \frac{1}{N} \sum_{j=0}^{\infty} j X_j(t) \leq \frac{1}{N} \sum_{j=0}^{\infty} k^* X_j(t) + \frac{1}{N} \sum_{\ell = 1}^{\infty} \ell X_{k^* + \ell}(t) = k^* + \frac{1}{N} \sum_{\ell = 1}^{\infty} \ell X_{k^* + \ell}(t).
\end{equation}
By Lemma \ref{Xjllem}, for sufficiently large $N$,
\begin{align*}
\frac{1}{N} \sum_{\ell = 1}^{\infty} \ell X_{k^* + \ell}(t) \1_{\{\tau_{k^*+\ell+1} \leq t\}} &\leq \sum_{\ell = 1}^{\infty} \frac{\ell X_{k^* + \ell}(t)}{X_{k^*}(t)} \1_{\{\tau_{k^*+\ell+1} \leq t\}} \\
&\leq \frac{(1 + \delta)^2}{1 - \delta}\sum_{\ell = 1}^{\infty} \ell \bigg( \frac{C_6 \mu}{s} \bigg)^{\ell} \bigg( \frac{\mu}{s} \bigg)^{\ell(\ell - 1)/6k_N} e^{s \ell (t - \tau_{k^*+1})}.
\end{align*}
Because $t - \tau_{k^* + 1} \leq \gamma_{k^* + 1} - \tau_{k^* + 1} = a_N$ and $e^{s \ell a_N} = (s/\mu)^\ell$, we have for sufficiently large $N$,
\begin{equation}\label{mainubsum}
\frac{1}{N} \sum_{\ell = 1}^{\infty} \ell X_{k^* + \ell}(t) \1_{\{\tau_{k^*+\ell+1} \leq t\}} \leq \frac{(1 + \delta)^2}{1 - \delta}\sum_{\ell = 1}^{\infty} \ell C_6^{\ell} \bigg( \frac{\mu}{s} \bigg)^{\ell(\ell - 1)/6k_N}.
\end{equation}
If $r_{\ell}$ denotes the $\ell$th term in the sum on the right-hand side of (\ref{mainubsum}), then $r_1 = C_6$ and for $\ell \geq 1$,
\begin{equation}\label{rratio}
\frac{r_{\ell+1}}{r_{\ell}} = \frac{C_6 (\ell + 1)}{\ell} \bigg( \frac{\mu}{s} \bigg)^{\ell/3k_N} \leq 2C_6 \bigg( \frac{\mu}{s} \bigg)^{1/3k_N},
\end{equation}
which tends to zero as $N \rightarrow \infty$ because $$\log \bigg( \frac{\mu}{s} \bigg)^{1/3k_N} = - \frac{1}{3 k_N} \log \bigg( \frac{s}{\mu} \bigg) = - \frac{[\log(s/\mu)]^2}{3 \log N},$$ which tends to $-\infty$ as $N \rightarrow \infty$ by (\ref{A2prime}).  Therefore, the first term dominates the sum on the right-hand side of (\ref{mainubsum}) for sufficiently large $N$, so for sufficiently large $N$ we have
\begin{equation}\label{Mtsplit1}
\frac{1}{N} \sum_{\ell = 1}^{\infty} \ell X_{k^* + \ell}(t) \1_{\{\tau_{k^*+\ell+1} \leq t\}} \leq \frac{(1 + \delta)^2}{1 - \delta} \cdot 2C_6.
\end{equation}
Finally, Lemma \ref{newjbound} and equations (\ref{muspower}) and (\ref{muNpower}) give
\begin{equation}\label{Mtsplit2}
\frac{1}{N} \sum_{j=k^* + 1}^{\infty} j X_j(t) \1_{\{\tau_{j+1} > t\}} \leq \frac{J s}{N \mu} \rightarrow 0 \hspace{.2in}\mbox{as }N \rightarrow \infty.
\end{equation}
Because $k^* - k_N \leq k_N^+ - k_N \rightarrow 0$ as $N \rightarrow \infty$ by (\ref{kdiff}), the result follows from (\ref{Mtsplit}), (\ref{Mtsplit1}), and (\ref{Mtsplit2}).
\end{proof}

It remains to consider the case in which $t \in [\gamma_j, \gamma_{j+1})$ for some $j \geq k^* + 1$.  In this case, we will need to consider carefully the contributions to $M(t)$ not just from individuals with an unusually large number of mutations, as in the proofs of Propositions \ref{tint1} and \ref{tint2}, but also from individuals with an unusually small number of mutations.  Therefore, we will use the following two lemmas, which parallel Lemma \ref{Xjllem}.

\begin{Lemma}\label{Xjllem2}
Suppose $j$ and $\ell$ are positive integers such that $j - \ell \geq k^* + 1$.  Suppose that $t \in [\gamma_j, \gamma_{j+K}] \cap [0, a_N T]$.  Let $\alpha_{\ell}(t) = (1 + \delta)/(1 - \delta)$ if $t \leq \gamma_{j-\ell+K}$, and let $\alpha_{\ell}(t) = k_N^2/(1 - \delta)$ if $t > \gamma_{j-\ell+K}$.  Suppose also that $\zeta_0 = \infty$ and $\zeta_1 \wedge \zeta_3 > t$.  Then for sufficiently large $N$,
$$\frac{X_{j - \ell}(t)}{X_j(t)} \leq \alpha_{\ell}(t) \bigg( \frac{2s}{\mu} \bigg)^{\ell} \bigg( \frac{\mu}{s} \bigg)^{\ell (\ell - 1)/6k_N}e^{-s \ell (t - \tau_j)}.$$
\end{Lemma}

\begin{proof}
Because $\zeta_1 \wedge \zeta_3 > t$ and $t \in [\gamma_j, \gamma_{j+K}]$, we can use (\ref{prop23}) to obtain a lower bound on $X_j(t)$.  Also, we can obtain an upper bound on $X_{j-\ell}$ from (\ref{prop23}) when $t \leq \gamma_{j - \ell + K}$, and from (\ref{prop24}) when $t > \gamma_{j - \ell + K}$.  This leads to
\begin{equation}\label{betaeq}
\frac{X_{j - \ell}(t)}{X_j(t)} \leq \frac{\alpha_{\ell}(t) e^{\int_{\tau_{j-\ell+1}}^t G_{j-\ell}(v) \: dv}}{e^{\int_{\tau_{j+1}}^t G_j(v) \: dv}}.
\end{equation}
Therefore,
\begin{equation}\label{betaeq2}
\frac{X_{j - \ell}(t)}{X_j(t)} \leq \frac{\alpha_{\ell}(t) e^{\int_{\tau_{j - \ell + 1}}^{\tau_j} G_{j-\ell}(v) \: dv} e^{\int_{\tau_j}^t G_{j - \ell}(v) \: dv}}{e^{-\int_{\tau_j}^{\tau_{j+1}} G_j(v) \: dv} e^{\int_{\tau_j}^t G_j(v) \: dv}} = \frac{\alpha_{\ell}(t) e^{\int_{\tau_{j - \ell + 1}}^{\tau_j} G_{j-\ell}(v) \: dv} e^{-s \ell (t - \tau_j)}}{e^{-\int_{\tau_j}^{\tau_{j+1}} G_j(v) \: dv}}.
\end{equation}
Because $\zeta_1 \wedge \zeta_3 > t$, it follows from Lemma \ref{smulem} that for sufficiently large $N$,
\begin{equation}\label{tm13}
e^{\int_{\tau_j}^{\tau_{j+1}} G_j(v) \: dv} \leq \frac{2s}{\mu}.
\end{equation}
Also, using (\ref{tauspacing}),
\begin{align*}
\int_{\tau_{j-\ell+1}}^{\tau_j} G_{j - \ell}(v) \: dv &= \sum_{m = 1}^{\ell - 1} \bigg( \int_{\tau_{j-m}}^{\tau_{j-m+1}} G_{j-m}(v) \: dv - s (\ell - m)(\tau_{j-m+1} - \tau_{j-m}) \bigg) \\
&\leq \bigg( \sum_{m=1}^{\ell - 1} \int_{\tau_{j-m}}^{\tau_{j-m+1}} G_{j-m}(v) \: dv \bigg) - \frac{s \ell (\ell - 1) a_N}{6k_N},
\end{align*}
so using Lemma \ref{smulem} again, for sufficiently large $N$,
\begin{equation}\label{tm14}
e^{\int_{\tau_{j-\ell+1}}^{\tau_j} G_{j-\ell}(v) \: dv} \leq \bigg( \frac{2s}{\mu} \bigg)^{\ell - 1} \bigg( \frac{\mu}{s} \bigg)^{\ell(\ell - 1)/6k_N}.
\end{equation}
The result now follows from (\ref{betaeq2}), (\ref{tm13}), and (\ref{tm14}).
\end{proof}

\begin{Lemma}\label{Xjlem3}
Suppose $i$ and $j$ are positive integers such that $0 \leq i \leq k^*$ and $j \geq k^* + 1$.  Let $\kappa(t) = 1 + \delta$ if $t \leq \gamma_{k^*+K}$, and let $\kappa(t) = k_N^2$ if $t > \gamma_{k^*+K}$.  There is a positive constant $C_8$ such that if $N$ is sufficiently large, then for all $t \in [\gamma_j, \gamma_{j+K}] \cap [0, a_N T]$,
on the event that $\zeta_0 = \infty$ and $\zeta_1 \wedge \zeta_3 > t$ we have
$$\frac{X_i(t)}{X_j(t)} \leq C_8 \kappa(t) 2^{j - k^*} k_N^{-(k^* - i)} \bigg( \frac{\mu}{s} \bigg)^{(j - k^*)(j - k^* - 1)/6k_N} e^{-s(j-i)(t - \gamma_j)}.$$
\end{Lemma}

\begin{proof}
Because $\zeta_1 > t$, equations (\ref{prop11}) and (\ref{prop12}) give
\begin{equation}\label{iupper}
X_i(t) \leq \kappa(t) X_i(t^*) e^{\int_{t^*}^t G_i(v) \: dv} = \kappa(t) \cdot \frac{X_i(t^*) e^{-s(k^* - i)(t - t^*)}}{X_{k^*}(t^*)} \cdot X_{k^*}(t^*) e^{\int_{t^*}^t G_{k^*}(v) \: dv}.
\end{equation}
Because we are working on the event that $\zeta_0 = \infty$, we can use the bounds on $X_i(t^*)$ and $X_{k^*}(t^*)$ from (\ref{early1}) and (\ref{earlypt2}).  Recall that for sufficiently large $N$, there is at most one integer $j$ such that $k_N^- < j < k_N^+$, which then must be $k^*$.  Let
\begin{displaymath}
\lambda_i = \left\{
\begin{array}{ll} 1 & \mbox{ if  }i = k^*  \\
(1 + \delta) k_N^{d_j}/C_1 & \mbox{ if }i < k^* \mbox{ and }k_N^- < k^* < k_N^+ \\
(1 + \delta)/(1 - \delta) & \mbox{ if }i < k^* \mbox{ and }k^* \leq k_N^-
\end{array} \right.
\end{displaymath}
Because $d_j \geq 0$, it follows from Proposition \ref{earlyprop} that for sufficiently large $N$, $$\frac{X_i(t^*)}{X_{k^*}(t^*)} \leq \frac{\lambda_i x_i(t^*)}{x_{k^*}(t^*)} = \lambda_i \bigg( \frac{s}{\mu} \bigg)^{k^* - i} \frac{k^*!}{i!} (e^{st^*} - 1)^{i - k^*}.$$ 
Now $k^*!/i! \leq (k^*)^{k^* - i}$ and $$\bigg( \frac{e^{st^*} - 1}{e^{st^*}} \bigg)^{i - k^*} \leq \bigg(1 - \frac{1}{k_N^2} \bigg)^{-k^*} \rightarrow 1 \hspace{.2in}\mbox{as }N \rightarrow \infty,$$ so for sufficiently large $N$,
\begin{align}\label{ikstar}
\frac{X_i(t^*) e^{-s(k^* - i)(t - t^*)}}{X_{k^*}(t^*)} &\leq 2 \lambda_i \bigg( \frac{s}{\mu} \bigg)^{k^* - i} (k^*)^{k^* - i} e^{st^*(i - k^*)} e^{-s(k^* - i)(t - t^*)} \nonumber \\
&= 2 \lambda_i \bigg( \frac{s}{\mu} \bigg)^{k^* - i} (k^*)^{k^* - i} e^{-s(k^* - i) t}.
\end{align}
Also, equation (\ref{prop11}) implies that for sufficiently large $N$,
$$1 + \frac{s}{\mu} \geq X_{k^*}(\tau_{k^* + 1}) \geq (1 - \delta) X_{k^*}(t^*) e^{\int_{t^*}^{\tau_{k^*} + 1} G_{k^*}(v) \: dv},$$  and therefore,
\begin{equation}\label{kstarlower}
X_{k^*}(t^*) e^{\int_{t^*}^t G_{k^*}(v) \: dv} \leq \frac{(1 + s/\mu)}{1 - \delta} e^{\int_{\tau_{k^* + 1}}^t G_{k^*}(v) \: dv}.
\end{equation}
Combining (\ref{iupper}), (\ref{ikstar}), and (\ref{kstarlower}), we get that for sufficiently large $N$,
\begin{equation}\label{Xiupper2}
X_i(t) \leq 2\lambda_i \kappa(t) \frac{1 + s/\mu}{1-\delta} \bigg( \frac{s}{\mu} \bigg)^{k^* - i} (k^*)^{k^* - i} e^{-s(k^* - i)t} e^{\int_{\tau_{k^* + 1}}^t G_{k^*}(v) \: dv}.
\end{equation}
By (\ref{prop23}), for sufficiently large $N$,
$$X_j(t) \geq \frac{(1 - \delta)s}{\mu} e^{\int_{\tau_{j+1}}^t G_j(v) \: dv}.$$  Combining this result with (\ref{Xiupper2}) gives that for sufficiently large $N$,
\begin{equation}\label{XiXj}
\frac{X_i(t)}{X_j(t)} \leq 2\lambda_i \kappa(t) \frac{1 + \mu/s}{(1-\delta)^2} \bigg( \frac{s}{\mu} \bigg)^{k^* - i} (k^*)^{k^* - i} e^{-s(k^* - i)t} \cdot \frac{e^{\int_{\tau_{k^* + 1}}^t G_{k^*}(v) \: dv}}{e^{\int_{\tau_{j+1}}^t G_j(v) \: dv}}.
\end{equation}
Note that the ratio of exponentials on the right-hand side of (\ref{XiXj}) is the same as the ratio of exponentials on the right-hand side of (\ref{betaeq}) with $j - k^*$ in place of $\ell$.  Consequently, the argument used to prove Lemma \ref{Xjllem2} gives
$$\frac{e^{\int_{\tau_{k^* + 1}}^t G_{k^*}(v) \: dv}}{e^{\int_{\tau_{j+1}}^t G_j(v) \: dv}} \leq \bigg( \frac{2s}{\mu} \bigg)^{j - k^*} \bigg( \frac{\mu}{s} \bigg)^{(j - k^*)(j - k^* - 1)/6k_N} e^{-s(j - k^*)(t - \tau_j)}.$$
Putting this result together with (\ref{XiXj}) gives that for sufficiently large $N$,
\begin{equation}\label{XiXj2}
\frac{X_i(t)}{X_j(t)} \leq 2 \lambda_i \kappa(t) 2^{j - k^*} \frac{1 + \mu/s}{(1 - \delta)^2} \bigg( \frac{s}{\mu} \bigg)^{j - i} \bigg( \frac{\mu}{s} \bigg)^{(j - k^*)(j - k^* - 1)/6k_N} (k^*)^{k^* - i} e^{-s(k^* - i) t} e^{-s(j - k^*)(t - \tau_j)}.
\end{equation}
Because $\zeta_0 = \infty$, we have $\tau_j \geq t^* = (\theta/s) \log k_N$, where $\theta = 4$ if $k_N^- < k^* < k_N^+$ and $\theta = 2$ otherwise. Therefore, recalling (\ref{gammajdef}),
\begin{align}\label{piece1}
e^{-s(k^* - i) t} e^{-s(j - k^*)(t - \tau_j)} &=  e^{-s(j-i)(t - \tau_j)} e^{-s(k^* - i)\tau_j} \nonumber \\
&\leq e^{-s(j - i)(t - \tau_j)} k_N^{-\theta(k^* - i)} \nonumber \\
&= \bigg( \frac{\mu}{s} \bigg)^{j - i} e^{-s(j - i)(t - \gamma_j)} k_N^{-\theta(k^* - i)}.
\end{align}
Also, because $k^* - k_N \leq k_N^+ - k_N \rightarrow 0$ as $N \rightarrow \infty$ by (\ref{kdiff}), we have
\begin{equation*}
\bigg(\frac{k^*}{k_N} \bigg)^{k^* - i} \leq \bigg( \frac{k_N^+}{k_N} \bigg)^{k^* - i} \rightarrow 1 \hspace{.2in}\mbox{as }N \rightarrow \infty.
\end{equation*}
Recall that $\lambda_i = 1$ when $i = k^*$.  Also, for $i < k^*$, we have
$\lambda_i = (1 + \delta)/(1 - \delta)$ when $\theta = 2$ and $\lambda_i = (1 + \delta) k_N^2/C_1$ when $\theta = 4$.  It follows that for sufficiently large $N$,
\begin{equation}\label{piece2}
\lambda_i (k^*)^{k^* - i} k_N^{-\theta(k^* - i)} \leq 2 \lambda_i k_N^{-(\theta - 1)(k^* - i)} \leq \frac{2(1 + \delta)}{\min\{1 - \delta, C_1\}} k_N^{-(k^* - i)}.
\end{equation}
Combining (\ref{XiXj2}), (\ref{piece1}), and (\ref{piece2}) gives the result.
\end{proof}

\begin{Prop}\label{tint3}
There exists a positive constant $C_5$ such that for sufficiently large $N$, if $t \in [\gamma_j, \gamma_{j+1}) \cap [0, a_N T]$ for some $j \geq k^* + 1$, then on the event that $\zeta_0 = \infty$ and $\zeta_1 \wedge \zeta_3 > t$, we have
$$|M(t) - j| < C_5 (e^{-s(t - \gamma_j)} + e^{-s(\gamma_{j+1} - t)}).$$
\end{Prop}

\begin{proof}
Throughout the proof, we work on the event that $\zeta_0 = \infty$ and $\zeta_1 \wedge \zeta_3 > t$.  We also assume that $t \in [\gamma_j, \gamma_{j+1})$.  Note that
\begin{align}\label{M3terms}
|M(t) - j| &\leq \frac{1}{N} \sum_{\ell = 1}^{\infty} \ell X_{j + \ell}(t) + \frac{1}{N} \sum_{\ell = 1}^j \ell X_{j - \ell}(t) \nonumber \\
&= \frac{1}{N} \sum_{\ell = 1}^{\infty} \ell X_{j + \ell}(t) + \frac{1}{N} \sum_{\ell = 1}^{j - k^* - 1} \ell X_{j - \ell}(t) + \frac{1}{N} \sum_{i=0}^{k^*} (j - i) X_i(t).
\end{align}
The argument for bounding the first term is similar to that in the proof of Proposition \ref{tint2}.  By Lemma \ref{Xjllem}, for sufficiently large $N$,
\begin{align*}
\frac{1}{N} \sum_{\ell = 1}^{\infty} \ell X_{j + \ell}(t) \1_{\{\tau_{j+\ell+1} \leq t\}} &\leq \sum_{\ell = 1}^{\infty} \frac{\ell X_{j + \ell}(t)}{X_j(t)} \1_{\{\tau_{j+\ell+1} \leq t\}} \\
&= \frac{1 + \delta}{1 - \delta} \sum_{\ell = 1}^{\infty} \ell \bigg( \frac{C_6 \mu}{s} \bigg)^{\ell} \bigg( \frac{\mu}{s} \bigg)^{\ell(\ell - 1)/6k_N} e^{s \ell (t - \tau_{j+1})}.
\end{align*}
Now $t - \tau_{j+1} = t - \gamma_{j+1} + \gamma_{j+1} - \tau_{j+1} = t - \gamma_{j+1} + a_N$.  Since $e^{s \ell a_N} = (s/\mu)^\ell$, it follows that $e^{s \ell (t - \tau_{j+1})} = (s/\mu)^{\ell} e^{-s \ell (\gamma_{j+1} - t)}$ and therefore
\begin{equation}\label{sumupper}
\frac{1}{N} \sum_{\ell = 1}^{\infty} \ell X_{j + \ell}(t) \1_{\{\tau_{j+\ell+1} \leq t\}} \leq \frac{1 + \delta}{1 - \delta} \sum_{\ell = 1}^{\infty} \ell C_6^{\ell} \bigg( \frac{\mu}{s} \bigg)^{\ell(\ell - 1)/6k_N} e^{-s \ell (\gamma_{j+1} - t)}.
\end{equation}
Let $r_{\ell}$ be the $\ell$th term in the sum on the right-hand side of (\ref{sumupper}).  Then $r_1 = C_6 e^{-s(\gamma_{j+1} - t)}$ and for $\ell \geq 1$, $$\frac{r_{\ell + 1}}{r_{\ell}} = \frac{C_6 (\ell + 1)}{\ell} e^{-s(\gamma_{j+1} - t)} \bigg( \frac{\mu}{s} \bigg)^{\ell/3k_N} \leq 2C_6 \bigg( \frac{\mu}{s} \bigg)^{1/3k_N},$$
which goes to zero as $N \rightarrow \infty$ by the argument following (\ref{rratio}).  Therefore, the first term dominates the sum on the right-hand side of (\ref{sumupper}), so for sufficiently large $N$ we have
\begin{equation}\label{Mterm21}
\frac{1}{N} \sum_{\ell = 1}^{\infty} \ell X_{j + \ell}(t) \1_{\{\tau_{j+\ell+1} \leq t\}} \leq 2C_6 e^{-s(\gamma_{j+1} - t)}.
\end{equation}
Also, by Lemma \ref{newjbound},
\begin{equation}\label{JsNu}
\frac{1}{N} \sum_{\ell = 1}^{\infty} \ell X_{j + \ell}(t) \1_{\{\tau_{j+\ell+1} > t\}} \leq \frac{1}{N} \sum_{j=k^* + 1}^{\infty} j X_j(t) \1_{\{\tau_{j+1} > t\}} \leq \frac{Js}{N \mu} \leq \frac{J s e^{s(\gamma_{j+1} - \gamma_j)}}{N \mu} e^{-s(\gamma_{j+1} - t)}.
\end{equation}
Because $t > \zeta_3$, equation (\ref{tauspacing}) gives $\gamma_{j+1} - \gamma_j = \tau_{j+1} - \tau_j \leq 2a_N/k_N$.  Therefore, $$\frac{J s e^{s(\gamma_{j+1} - \gamma_j)}}{N \mu} \leq \frac{J}{N} \bigg( \frac{s}{\mu} \bigg)^{1 + 2/k_N} \rightarrow 0$$ as $N \rightarrow \infty$ by (\ref{JsN}).  Combining this result with (\ref{Mterm21}), we get that for sufficiently large $N$,
\begin{equation}\label{Msecondterm}
\frac{1}{N} \sum_{\ell = 1}^{\infty} \ell X_{j+\ell}(t) \leq (2C_6 + 1) e^{-s(\gamma_{j+1} - t)}.
\end{equation}

Consider now the second term in (\ref{M3terms}).  Suppose $\ell \leq j - k^* - 1$, so that $j - \ell \geq k^* + 1$.  As in Lemma \ref{Xjllem2}, write $\alpha_{\ell}(t) = (1 + \delta)/(1 - \delta)$ if $t \leq \gamma_{j-\ell+K}$ and $\alpha_{\ell}(t) = k_N^2/(1 - \delta)$ if $t > \gamma_{j-\ell+K}$.  Then Lemma \ref{Xjllem2} implies that for sufficiently large $N$,
$$\frac{X_{j-\ell}(t)}{X_j(t)} \leq \alpha_{\ell}(t) \bigg( \frac{2s}{\mu} \bigg)^{\ell} \bigg( \frac{\mu}{s} \bigg)^{\ell(\ell - 1)/6k_N} e^{-s \ell (t - \tau_j)}.$$
Because $\gamma_j - \tau_j = a_N$ and $e^{sa_N} = s/\mu$, we have 
\begin{equation}\label{tm12}
e^{-s \ell (t - \tau_j)} = e^{-s \ell (t - \gamma_j)} e^{-s \ell a_N} = \bigg( \frac{\mu}{s} \bigg)^{\ell} e^{-s \ell (t - \gamma_j)}.
\end{equation}
Therefore, for sufficiently large $N$,
$$\frac{X_{j - \ell}(t)}{X_j(t)} \leq \alpha_{\ell}(t) 2^{\ell} \bigg( \frac{\mu}{s} \bigg)^{\ell (\ell - 1)/6 k_N} e^{-s \ell (t - \gamma_j)},$$
and so
\begin{equation}\label{sumlower}
\frac{1}{N} \sum_{\ell = 1}^{j - k^* - 1} \ell X_{j - \ell}(t) \leq \sum_{\ell = 1}^{\infty} \alpha_{\ell}(t) \ell 2^{\ell}  \bigg( \frac{\mu}{s} \bigg)^{\ell (\ell - 1)/6 k_N} e^{-s \ell (t - \gamma_j)}.
\end{equation}
Let $v_{\ell}$ denote the $\ell$th term on the right-hand side of (\ref{sumlower}).  Note that $t < \gamma_{j+1} \leq \gamma_{j-1+K}$ as long as $N$ is large enough that $K \geq 2$.  Therefore, $v_1 = 2((1+\delta)/(1-\delta)) e^{-s (t - \gamma_j)}$ and for $\ell \geq 1$, $$\frac{v_{\ell + 1}}{v_{\ell}} \leq \frac{2 k_N^2 (\ell + 1)}{\ell} \bigg( \frac{\mu}{s} \bigg)^{\ell/3 k_N} e^{-s(t - \gamma_j)}  \leq 4 k_N^2 \bigg( \frac{\mu}{s} \bigg)^{1/3k_N}.$$  To see that this expression tends to zero as $N \rightarrow \infty$, note that
\begin{equation}\label{kNmusto0}
\log \bigg( k_N^2 \bigg( \frac{\mu}{s} \bigg)^{1/3k_N} \bigg) = 2 \log k_N - \frac{1}{3k_N} \log \bigg( \frac{s}{\mu} \bigg) = 2 \log k_N - \frac{[\log(s/\mu)]^2}{3 \log N},
\end{equation}
which tends to $-\infty$ as $N \rightarrow \infty$ by assumption A2.  Therefore, the first term dominates the sum on the right-hand side of (\ref{sumlower}) when $N$ is large.  For sufficiently large $N$, we therefore have
\begin{equation}\label{M31ub}
\frac{1}{N} \sum_{\ell = 1}^{j - k^* + 1} \ell X_{j - \ell}(t) \leq 3 e^{-s(t - \gamma_j)}.
\end{equation} 

Finally, we consider the third term in (\ref{M3terms}).  Suppose $0 \leq i \leq k^*$.  Define $\kappa(t)$ as in the statement of Lemma \ref{Xjlem3}.  By Lemma \ref{Xjlem3}, for sufficiently large $N$,
\begin{align*}
\frac{1}{N} \sum_{i=0}^{k^*} (j - i) X_i(t) &\leq \sum_{i=0}^{k^*} \frac{(j - i) X_i(t)}{X_j(t)} \\
&\leq C_8 \kappa(t) 2^{j - k^*} \bigg( \frac{\mu}{s} \bigg)^{(j - k^*)(j - k^* - 1)/6k_N} \sum_{i=0}^{k^*} (j - i) k_N^{-(k^*-i)} e^{-s(j-i)(t - \gamma_j)}.
\end{align*}
Because $j - k^* \geq 1$, we have $e^{-s(j-i)(t - \gamma_j)} \leq e^{-s(t - \gamma_j)}$ for $i \in \{0, 1, \dots, k^*\}$.  Also, if we let $v_i = (j - i)k_N^{-(k^* - i)}$, then $v_{i-1}/v_i \leq 2/k_N \rightarrow 0$ as $N \rightarrow \infty$ for $i \in \{1, 2, \dots, k^*\}$.  Therefore, for sufficiently large $N$, the sum $\sum_{i=0}^{k^*} v_i$ is dominated by the $i = k^*$ term, and we get $$\sum_{i=0}^{k^*} (j - i) k_N^{-(k^* - i)} \leq 2 (j - k^*).$$  It follows that
\begin{equation}\label{iupper2}
\frac{1}{N} \sum_{i=0}^{k^*} (j - i) X_i(t) \leq 2 C_8 \kappa(t) 2^{j - k^*} (j - k^*) \bigg( \frac{\mu}{s} \bigg)^{(j - k^*)(j - k^* - 1)/6k_N} e^{-s(t - \gamma_j)}
\end{equation}
for sufficiently large $N$.  If $j = k^*+1$ and $N$ is sufficiently large, then $\kappa(t) = 1 + \delta$, and so
\begin{equation}\label{C51}
2 C_8 \kappa(t) 2^{j - k^*} (j - k^*) \bigg( \frac{\mu}{s} \bigg)^{(j - k^*)(j - k^* - 1)/6k_N} = 4 (1 + \delta) C_8.
\end{equation}
If $j - k^* \geq 2$, then $\kappa(t) \leq k_N^2$.  For $\ell \geq 2$, let $$w_{\ell} = 2 C_8 k_N^2 2^{\ell} \ell \bigg( \frac{\mu}{s} \bigg)^{\ell (\ell - 1)/6k_N}.$$  Then, for $\ell \geq 2$, we have $w_{\ell+1}/w_{\ell} \leq 3 (\mu/s)^{2/3k_N}$, which tends to zero as $N \rightarrow \infty$ by the argument following (\ref{rratio}).  Therefore, for sufficiently large $N$, the $\ell = 2$ term is largest, so if $j \geq k^* + 2$, then
\begin{equation}\label{C52}
2 C_8 \kappa(t) 2^{j - k^*} (j - k^*) \bigg( \frac{\mu}{s} \bigg)^{(j - k^*)(j - k^* - 1)/6k_N} \leq 16 C_8 k_N^2 \bigg( \frac{\mu}{s} \bigg)^{1/3k_N},
\end{equation}
which tends to zero as $N \rightarrow \infty$ by the argument around (\ref{kNmusto0}).  Combining (\ref{iupper2}) with the bounds in (\ref{C51}) and (\ref{C52}) gives that for sufficiently large $N$, 
\begin{equation}\label{Mlastterm}
\frac{1}{N} \sum_{i=0}^{k^*} (j - i) X_i(t) \leq 5 C_8 e^{-s(t - \gamma_j)}.
\end{equation}
The result now follows from (\ref{M3terms}), (\ref{Msecondterm}), (\ref{M31ub}), and (\ref{Mlastterm}).
\end{proof} 

\begin{Rmk}\label{meanrem}
{\em If $t \in [t^*, \gamma_{j+1}] \cap [0, a_N T]$, then on the event that $\zeta_0 = \infty$ and $\zeta_1 \wedge \zeta_3 > t$, it follows from (\ref{Mterm21}) and (\ref{JsNu}) that $$\frac{1}{N} \sum_{i=j+1}^{\infty} X_i(t) \leq C_5 e^{-s(\gamma_{j+1} - t)} + \frac{s}{N \mu},$$ where we get $s/N \mu$ in place of $Js/N \mu$ for the second term from the argument in the proof of Lemma \ref{newjbound} that there can be at most one value of $i$ for which $\tau_{i+1} > t$ but $X_i(t) > 0$.  Likewise, if $t \in [\gamma_j, \gamma_{j+K}] \cap [0, a_N T]$, then on the event that $\zeta_0 = \infty$ and $\zeta_1 \wedge \zeta_3 > t$,
equations (\ref{M31ub}) and (\ref{Mlastterm}) imply that 
$$\frac{1}{N} \sum_{i=0}^{j-1} X_i(t) \leq C_5 e^{-s(t - \gamma_j)}.$$  In particular, for $t \in [\gamma_j, \gamma_{j+1})$, unless $t$ is close to $\gamma_j$ or $\gamma_{j+1}$, nearly all individuals in the population at time $t$ will be of type $j$.}
\end{Rmk}

Proposition \ref{tint4} below establishes the fourth part of Proposition \ref{meanprop}.  Part 1 of Proposition \ref{zetaprop} follows immediately from Propositions \ref{tint1}, \ref{tint2}, \ref{tint3}, and \ref{tint4}.

\begin{Prop}\label{tint4}
For sufficiently large $N$, if $t \in [\tau_j, \tau_{j+1})$ for some $j \geq k^* + 1$, then on the event that $\zeta_0 = \infty$ and $\zeta_1 \wedge \zeta_3 > t$, we have $M(t) < j - 1$.
\end{Prop}

\begin{proof}
Suppose $\zeta_0 = \infty$ and $\zeta_1 \wedge \zeta_3 > t$.  Suppose $t \in [\tau_j, \tau_{j+1})$, where $j \geq k^* + 1$.  We consider three cases.  First, suppose $t \leq a_N$.  Then $M(t) < 3 \leq j - 1$ by Proposition \ref{tint1} for sufficiently large $N$.

Second, suppose $t \in (a_N, \gamma_{k^* + 1})$.  Then $M(t) < k_N + C_4$ for sufficiently large $N$ by Proposition \ref{tint2}.  Because $t > \zeta_3$, the result of part 1 of Proposition \ref{tauprop} implies that $\tau_{k^* + 1} \leq 2 a_N/k_N$.   Therefore, (\ref{tauspacing}) implies that for sufficiently large $N$, $$\tau_{k^* + 1 + k_N/3} \leq \frac{2 a_N}{k_N} + \frac{2 a_N}{k_N} \cdot \frac{k_N}{3} = a_N \bigg( \frac{2}{k_N} + \frac{2}{3} \bigg) < a_N.$$  Therefore, $\tau_{j+1} > a_N > \tau_{k^* + 1 + k_N/3}$, which means $j \geq k^* + k_N/3$.  For sufficiently large $N$, we are guaranteed $k_N + C_4 < k^* + k_N/3 - 1$, and thus $M(t) < j - 1$.

Finally, suppose $t \in [\gamma_{\ell}, \gamma_{\ell+1})$ for some $\ell \geq k^* + 1$.  Then $M(t) < \ell + 2 C_5$ for sufficiently large $N$ by Proposition \ref{tint3}.  Also, since $t \geq \gamma_{\ell} = \tau_{\ell} + a_N$, equation (\ref{tauspacing}) gives $$\tau_{\ell + k_N/2} \leq \tau_{\ell} + \frac{2 a_N}{k_N} \cdot \frac{k_N}{2} \leq t < \tau_{j+1},$$ which means $j \geq \ell + k_N/2 - 1$.  Since $\ell + 2C_5 \leq \ell + k_N/2 - 2$ for sufficiently large $N$, we again obtain $M(t) < j - 1$.
\end{proof}

\section{Proof of part 2 of Proposition \ref{zetaprop}}\label{zetasec2}

Recall that Proposition \ref{tauprop} consists of three parts.  The first part simply bounds $\tau_{k^*+1}$.
The second part is concerned with $R(t)$, which can be interpreted as the number of new types that have emerged between times $a_N (t-1)$ and $a_N t$.  The third part pertains to the spacings between the times $\tau_j$.

The time $\zeta_3$ is the first time at which one of the statements of Proposition \ref{tauprop} fails to hold.  Part 2 of Proposition \ref{zetaprop} stipulates that $\zeta_3$ can not happen until either $\zeta_1$ or $\zeta_2$ has occurred.  That is, as long as the behavior of the type $j$ individuals follows the description in Propositions \ref{earlyprop} and \ref{prop1}, and the mean number of mutations in the population behaves as described in Proposition \ref{meanprop}, the results of Proposition \ref{tauprop} must continue to hold.  Part 2 of Proposition \ref{zetaprop}, like part 1, is a deterministic statement.  To prove it, we will assume that $\zeta_0 = \infty$.  We will fix a time $t \in [t^*, a_N T]$ and show that if $\zeta_1 > t$ and $\zeta_2 \geq t$, then $\zeta_3 > t$, which means that the conclusions of Proposition \ref{tauprop} are valid through time $t$.  

\subsection{An upper bound on $\tau_{k^* + 1}$}

In this subsection, we establish the following result, which gives part 1 of Proposition \ref{tauprop}.

\begin{Prop}\label{taukstar}
For sufficiently large $N$, on the event that $\zeta_0 = \infty$, $\zeta_1 > 2 a_N/k_N$, and $\zeta_2 \geq 2 a_N/k_N$, we have $\tau_{k^* + 1} \leq 2 a_N/k_N$.
\end{Prop}

\begin{proof}
Suppose $\zeta_0 = \infty$, $\zeta_1 > 2 a_N/k_N$, and $\zeta_2 \geq 2a_N/k_N$.  We need to show $X_{k^*}(2a_N/k_N) \geq s/\mu$.  By (\ref{prop11}),
\begin{equation}\label{taukstar1}
X_{k^*}(2a_N/k_N) \geq (1 - \delta)X_{k^*}(t^*) e^{\int_{t^*}^{2a_N/k_N} G_{k^*}(v) \: dv}.
\end{equation}
Because $\zeta_2 \geq 2a_N/k_N$, we have $\int_{t^*}^{2a_N/k_N} M(v) \: dv \leq 3/s$ by part 1 of Proposition \ref{meanprop}.  Therefore, since $2 \mu a_N/k_N \rightarrow 0$ as $N \rightarrow \infty$ by (\ref{muspower}), for sufficiently large $N$ we have
\begin{align}\label{taukstar2}
\int_{t^*}^{2a_N/k_N} G_{k^*}(v) \: dv &= s k^* (2a_N/k_N - t^*) - \mu(2a_N/k_N - t^*) - \int_{t^*}^{2a_N/k_N} s M(v) \: dv. \nonumber \\
 &\geq sk^*(2a_N/k_N - t^*) - 4.
\end{align}
Also, by Proposition \ref{earlyprop}, if we set $d = 0$ when $k^* \leq k_N^-$ and $d = d_{k^*}$ when $k^* > k_N^-$, we get
\begin{equation}\label{taukstar3}
X_{k^*}(t^*) \geq \min\{(1 - \delta), C_1\} k_N^{-d} x_{k^*}(t^*).
\end{equation}
Combining (\ref{taukstar1}), (\ref{taukstar2}), and (\ref{taukstar3}), we see that there is a constant $c > 0$ such that
$$X_{k^*}(2a_N/k_N) \geq \frac{c N \mu^{k^*}}{s^{k^*} k^*!} \bigg( \frac{e^{st^*} - 1}{e^{st^*}} \bigg)^{k^*} k_N^{-d} e^{2sk^* a_N/k_N}.$$  Because $(1 - e^{-st^*})^{k^*} \rightarrow 1$ as $N \rightarrow \infty$, to show that $X_{k^*}(2a_N/k_N) \geq s/\mu$ for sufficiently large $N$, it suffices to show that
\begin{equation}\label{stslogratio}
\lim_{N \rightarrow \infty} \frac{N \mu^{k^* + 1}}{s^{k^* + 1} k^*!} k_N^{-d} e^{2 sk^*  a_N/k_N} = \infty.
\end{equation}
Arguing as in (\ref{logpt3}), we get
\begin{align}\label{logfinal}
\log \bigg( \frac{N \mu^{k^* + 1}}{s^{k^* + 1} k^*!} k_N^{-d} e^{2s k^* a_N/k_N} \bigg) &= (k_N - k^* - 1) \log \bigg( \frac{s}{\mu} \bigg) - k^* \log k^* + k^* + \frac{2 s k^* a_N}{k_N} + o(k_N) \nonumber \\
&= \bigg( k_N - k^* - 1 + \frac{2k^*}{k_N} \bigg) \log \bigg( \frac{s}{\mu}\bigg) - k^* \log k^* + k^* + o(k_N).
\end{align}
Because $k^*/k_N \rightarrow 1$ as $N \rightarrow \infty$, and $k_N - k^* \geq k_N - k_N^+ \rightarrow 0$ as $N \rightarrow \infty$ by (\ref{kdiff}), the first term on the right-hand side of (\ref{logfinal}) is at least $(1/2) \log(s/\mu)$ for sufficiently large $N$.  Because $(k_N \log k_N)/\log(s/\mu) \rightarrow 0$ as $N \rightarrow \infty$ by assumption A2, it follows that the first term dominates the right-hand side of (\ref{logfinal}), and thus the expression in (\ref{logfinal}) tends to infinity as $N \rightarrow \infty$.  Hence, (\ref{stslogratio}) holds, which completes the proof.
\end{proof}

\subsection{Approximating $R(a_N t)/k_N$ by $q(t)$}

In this subsection, we establish the second part of Proposition \ref{tauprop}, which states that $R(a_N t)/k_N$ can be well approximated by $q(t)$, where $q$ is the function defined in (\ref{qdef}).  The first lemma collects some properties of the function $q$.

\begin{Lemma}\label{Qlem}
There is a unique bounded function $q: [0, \infty) \rightarrow [0, \infty)$ satisfying (\ref{qdef}).  
The function $q$ is right continuous on $[0, \infty)$ and continuous on $[0, 1) \cup (1, \infty)$.  Also, $1 \leq q(t) \leq e$ for all $t \geq 0$ and
\begin{equation}\label{qlim2}
\lim_{t \rightarrow \infty} q(t) = 2.
\end{equation}
\end{Lemma}

\begin{proof}
Note that (\ref{qdef}) is equivalent to the renewal equation
\begin{equation}\label{renewal}
q(t) = g(t) + \int_0^t q(t - u) f(u) \: du,
\end{equation}
where $f(u) = g(u) = \1_{\{0 \leq u < 1\}}$.  That this equation has a unique solution which is nonnegative and bounded on every finite interval is a consequence of Theorem 2 in \cite{feller}.  Another consequence of Theorem 2 in \cite{feller} is that the function $t \mapsto q(t) - g(t)$ is continuous, which implies that $q$ is right continuous on $[0, \infty)$ and continuous on $[0, 1) \cup (1, \infty)$.

To obtain the bounds on $q$, let $u = \inf\{t \geq 1: q(u) \geq e \mbox{ or }q(u) \leq 1\}$.  Suppose $u < \infty$.  Then either $q(u) = e$ or $q(u) = 1$.  However, $q(u) = \int_{u-1}^u q(t) \: dt \in (1, e)$, a contradiction.  Thus $u = \infty$, which means $1 \leq q(t) \leq e$ for all $t \geq 0$.  Equation (\ref{qlim2}) is a consequence of Theorem 4 in \cite{feller}.  See also Remark \ref{renewalrmk}.
\end{proof}

The next lemma controls the value of $R(t)$ for $t < a_N$.

\begin{Lemma}\label{Rrange1}
Let $0 < \eta < 1$.  If $N$ is sufficiently large, then for all $t \in [0, a_N)$, on the event that $\zeta_0 = \infty$, $\zeta_1 > t$, and $\zeta_2 \geq t$, we have
\begin{equation}\label{Rrange1eq}
(1 - \eta) k_N e^{(1 - \eta) t/a_N} < R(t) < (1 + \eta) k_N e^{(1 + \eta)t/a_N}.
\end{equation}
\end{Lemma}

\begin{proof}
On the event $\zeta_0 = \infty$, Proposition \ref{earlyprop} implies that $\tau_j > t^*$ for all $j \geq k^* + 1$ and therefore $R(t) = k^*$ for $t \in [0, t^*]$.  Because $k^*/k_N \rightarrow 1$ as $N \rightarrow \infty$ and $t^*/a_N \rightarrow 0$ as $N \rightarrow \infty$ by (\ref{A2prime}), it follows that for sufficiently large $N$, equation (\ref{Rrange1eq}) holds for all $t \in [0, t^*]$.

Consider next the case in which $t^* < t < a_N$.  Suppose also that $\zeta_0 = \infty$, $\zeta_1 > t$, and $\zeta_2 \geq t$.  Let $\theta > 0$.  If $k^* + 1 \leq \ell \leq J$ and $\tau_{\ell+1} \leq t$, then Lemma \ref{smulem} implies that for sufficiently large $N$,
\begin{equation}\label{ExpG1}
\frac{s}{C_6 \mu} \leq e^{\int_{\tau_\ell}^{\tau_{\ell+1}} G_\ell(v) \: dv} \leq \frac{2s}{\mu}.
\end{equation}  
Note that
\begin{equation}\label{ExpG2}
\int_{\tau_\ell}^{\tau_{\ell+1}} G_\ell(v) \: dv = s\ell(\tau_{\ell+1} - \tau_{\ell}) - s \int_{\tau_{\ell}}^{\tau_{\ell+1}} M(v) \: dv - \mu(\tau_{\ell+1} - \tau_{\ell}).
\end{equation}
Note that $\tau_{\ell} \geq t^*$ by parts 3 and 4 of Proposition \ref{earlyprop}.  Therefore,
because $\zeta_2 \geq t$, part 1 of Proposition \ref{meanprop} implies that
\begin{equation}\label{ExpG3}
0 \leq s \int_{\tau_{\ell}}^{\tau_{\ell+1}} M(v) \: dv \leq s \int_{t^*}^{a_N} M(v) \: dv \leq 3 s \int_{t^*}^{a_N} e^{-s(a_N - v)} \: dv < 3.
\end{equation}
Since $2e^3 < 41$ and $\mu a_N \rightarrow 0$ as $N \rightarrow \infty$, it follows from (\ref{ExpG1}), (\ref{ExpG2}), and (\ref{ExpG3}) that for sufficiently large $N$, $$\frac{s}{C_6 \mu} \leq e^{s\ell (\tau_{\ell+1} - \tau_{\ell})} \leq \frac{41s}{\mu}.$$  Therefore, for sufficiently large $N$,
\begin{equation}\label{taueta}
\frac{(1 - \theta)a_N}{\ell} = \frac{1 - \theta}{s\ell} \log \bigg( \frac{s}{\mu} \bigg) \leq \tau_{\ell+1} - \tau_{\ell} \leq \frac{1 + \theta}{s\ell} \log \bigg( \frac{s}{\mu} \bigg) = \frac{(1 + \theta) a_N}{\ell}.
\end{equation}
Furthermore, by repeating the above argument with $t$ in place of $\tau_{j+1}$, we see that for sufficiently large $N$, if $\tau_{\ell} \leq t$ and $t - \tau_{\ell} \geq (1 + \theta) a_N/\ell$, then
\begin{equation}\label{new73}
\int_{\tau_{\ell}}^t G_{\ell}(v) \: dv \geq (1 + \theta) a_N s - 3 - \mu a_N \geq \log \bigg( \frac{2s}{\mu} \bigg),
\end{equation}
in which case the last statement of Lemma \ref{smulem} implies that $\tau_{\ell+1} \leq t$.

Therefore, if $k^* + 1 \leq j \leq J$ and $\tau_j \leq t$, then (\ref{taueta}) implies that for sufficiently large $N$,
$$t \geq \tau_j \geq \sum_{\ell=k^* + 1}^{j-1} (\tau_{\ell+1} - \tau_{\ell}) \geq (1 - \theta) a_N \sum_{\ell=k^*+1}^{j-1} \frac{1}{\ell} \geq a_N (1 - \theta) \log \bigg( \frac{j}{k^*+1} \bigg),$$ and rearranging this equation gives $j \leq (k^* + 1) e^{t/[a_N(1 - \theta)]}.$  In view of (\ref{Rdef}), it follows that for sufficiently large $N$,
\begin{equation}\label{Rtupper}
R(t) \leq (k^* + 1)\exp\bigg(\frac{t}{a_N(1 - \theta)}\bigg).
\end{equation}
Likewise, equation (\ref{taueta}) and Proposition \ref{taukstar} imply that if $k^* + 1 \leq j \leq J$ and $\tau_j \leq t$, then for sufficiently large $N$,
$$\tau_j = \tau_{k^* + 1} + \sum_{\ell=k^*+1}^{j-1} (\tau_{\ell+1} - \tau_{\ell}) \leq \frac{2a_N}{k_N} + (1 + \theta)a_N \sum_{\ell=k^*+1}^{j-1} \frac{1}{\ell} \leq a_N \bigg( \frac{2}{k_N} + (1 + \theta) \log \bigg( \frac{j-1}{k^*} \bigg) \bigg),$$ and the observation following (\ref{new73}) thus implies that if 
$$t \geq a_N \bigg( \frac{2}{k_N} + (1 + \theta) \log \bigg( \frac{j-1}{k^*} \bigg) \bigg),$$ 
or equivalently if $j \leq 1 + k^* \exp([t/a_N - 2/k_N]/(1+\theta))$, then $\tau_j \leq t$ if $N$ is sufficiently large.  It follows that
\begin{equation}\label{Rtlower}
R(t) \geq 1 + k^* \exp \bigg( \frac{t}{a_N (1 + \theta)} - \frac{2}{k_N(1 + \theta)} \bigg)
\end{equation}
for sufficiently large $N$.  Because $k_N \rightarrow \infty$ and $k^*/k_N \rightarrow 0$ as $N \rightarrow \infty$, we can see from (\ref{Rtupper}) and (\ref{Rtlower}) that for sufficiently large $N$, equation (\ref{Rrange1eq}) holds for all $t \in (t^*, a_N)$ as long as $\theta$ is chosen to be sufficiently small relative to $\eta$.
\end{proof}

We next consider the value of $R(t)$ for $t \in [a_N, a_N T]$.  We will find it useful to introduce the following notation.  For $t \in [0, \zeta_1 \wedge a_NT)$, let 
\begin{equation}\label{Mbardef}
{\bar M}(t) = \left\{
\begin{array}{ll} 0 & \mbox{ if }t < a_N \\
k^* & \mbox{ if }t \in [a_N, \gamma_{k^* + 1}) \\
j & \mbox{ if }t \in [\gamma_j, \gamma_{j+1}) \mbox{ for }j \geq k^* + 1.
\end{array} \right.
\end{equation}
Note that $M(t)$ is well-defined because, by Remark \ref{tauorder}, we have $\tau_j < \tau_{j+1}$, and therefore $\gamma_j < \gamma_{j+1}$, whenever $\tau_j < \zeta_1$.  As long as the conclusions of Proposition \ref{meanprop} hold, ${\bar M}(t)$ is a good approximation to the mean number of mutations in the population at time $t$.  

\begin{Lemma}\label{MMbarint}
If $\zeta_2 \geq a_N$, then $$\int_0^{a_N} |M(t) - {\bar M}(t)| \: dt \leq \frac{3}{s}.$$  For sufficiently large $N$, if $\zeta_2 \geq \gamma_{k^*+1}$ and $\zeta_1 > 2 a_N/k_N$, then
$$\int_{a_N}^{\gamma_{k^* + 1}} |M(t) - {\bar M}(t)| \: dt \leq \frac{2 k^*}{k_N} a_N.$$  Finally, for all $j \geq k^* + 1$, if $\zeta_2 \geq \gamma_{j+1}$ then
\begin{equation}\label{Mbareq}
\int_{\gamma_j}^{\gamma_{j+1}} |M(t) - {\bar M}(t)| \: dt \leq \frac{2 C_5}{s}.
\end{equation}
\end{Lemma}

\begin{proof}
The first and third statements follow immediately from integrating the result of Proposition \ref{meanprop}.  For the second statement, note that for sufficiently large $N$, we have $k_N + C_4 \leq 2k^*$.  Then for $t \in [a_N, \gamma_{k^*+1})$, it follows that when $\zeta_2 \geq \gamma_{k^*+1}$, we have $0 \leq M(t) \leq 2 k^*$ and thus $|M(t) - {\bar M}(t)| \leq k^*$.  The result follows because when $\zeta_1 > 2 a_N/k_N$ and $\zeta_2 \geq \gamma_{k^*+1}$, we have $\gamma_{k^* + 1} - a_N = \tau_{k^* + 1} \leq 2a_N/k_N$ by Proposition \ref{taukstar}.
\end{proof}

\begin{Lemma}\label{Rjlem}
Suppose $j \geq k^* + 1$.  Also, suppose $t \in [\tau_j, \tau_{j+1})$ and either $\zeta_1 > t$ or $\zeta_3 > t$.  Then $R(t) = j - {\bar M}(t)$.
\end{Lemma}

\begin{proof}
First suppose that $t \geq \gamma_{k^* + 1}$, so that $t - a_N \geq \tau_{k^* + 1}$.  Then ${\bar M}(t) = \ell$ implies that $t \in [\gamma_{\ell}, \gamma_{\ell + 1})$, and thus $t - a_N \in [\tau_{\ell}, \tau_{\ell + 1})$.  Thus, in view of Remark \ref{tauorder} when $\zeta_1 > t$ or (\ref{tauspacing}) when $\zeta_3 > t$, the times $\tau_{\ell + 1}, \tau_{\ell + 2}, \dots, \tau_j$ occur in the interval $(t - a_N, t]$.  Because $R(t)$ is the number of integers $i \geq k^*+1$ such that $t - a_N < \tau_i \leq t$, we have $R(t) = j - \ell$, as claimed.  The other possibility is that $t < \gamma_{k^* + 1}$.  Because $t - a_N < \tau_{k^* + 1}$, the times $\tau_{k^* + 1}, \dots, \tau_j$ occur in the interval $(t - a_N, t]$.  Therefore $R(t) = j - k^*$ if $t \geq a_N$ and $R(t) = j$ if $t < a_N$, which again matches the conclusion of the lemma in view of (\ref{Mbardef}).
\end{proof}

The lemma below is the key to obtaining the integral equation for the limit function $q$. 

\begin{Lemma}\label{Rrange2}
Let $0 < \eta < 1$.  If $N$ is sufficiently large, then 
for all $t \in [a_N, a_N T]$, on the event that $\zeta_0 = \infty$, $\zeta_1 > t$, and $\zeta_2 \geq t$, we have
$$\frac{(1 - \eta)}{a_N} \int_{t - a_N}^t R(u) \: du < R(t) < \frac{(1 + \eta)}{a_N} \int_{t - a_N}^t R(u) \: du$$
provided that 
\begin{equation}\label{Rbounds}
\frac{k_N}{2} \leq \inf_{u \in [0, t]} R(u) \leq \sup_{u \in [0, t]} R(u) \leq 3k_N.
\end{equation}
\end{Lemma}

\begin{proof}
Fix $t \in [a_N, a_NT]$, and suppose $\zeta_0 = \infty$, $\zeta_1 > t$, and $\zeta_2 \geq t$.
Let $L_1 = \min\{j: \tau_j > t - a_N\}$ and $L_2 = \max\{j: \tau_j \leq t\}$.  In view of Remark \ref{tauorder}, we can write $$(t - a_N, t] = (t - a_N, \tau_{L_1}] \cup \bigg(\bigcup_{j = L_1}^{L_2 - 1} (\tau_j, \tau_{j + 1}] \bigg) \cup (\tau_{L_2}, t].$$  
For $u \in [0, t]$, let
\begin{equation*}
S(u) = \left\{
\begin{array}{ll} 0 & \mbox{ if }u < \tau_{k^* + 1} \\
G_j(u)/s & \mbox{ if }t \in [\tau_j, \tau_{j+1}) \mbox{ for }j \geq k^* + 1. \\
\end{array} \right.
\end{equation*}
If $L_1 \leq j < L_2$, then since $j \leq J$ by Remark \ref{JRmk}, Lemma \ref{smulem} implies that for sufficiently large $N$, $$\log \bigg( \frac{s}{C_6 \mu} \bigg) \leq \int_{\tau_j}^{\tau_{j+1}} G_j(u) \: du \leq \log \bigg( \frac{2s}{\mu} \bigg).$$  Dividing by $s$, we get that for sufficiently large $N$,
\begin{equation}\label{Seq1}
\bigg(1 - \frac{\eta}{3} \bigg) a_N \leq \int_{\tau_j}^{\tau_{j+1}} S(u) \: du \leq \bigg(1 + \frac{\eta}{3}\bigg) a_N.
\end{equation}
Because $\zeta_2 \geq t$, we have $S(u) \geq 0$ for all $u \in [0, t)$ by the result of part 4 of Proposition \ref{meanprop}.  Combining this observation with the last statement of Lemma \ref{smulem}, we get
\begin{equation}\label{Seq2}
0 \leq \int_{\tau_{L_2}}^t S(u) \: du = \frac{1}{s} \int_{\tau_{L_2}}^t G_j(u) \: du \leq \frac{1}{s} \log \bigg( \frac{2s}{\mu} \bigg) \leq \bigg(1 + \frac{\eta}{3} \bigg) a_N
\end{equation}
for sufficiently large $N$.  Likewise, if $L_1 = k^*+1$, then $S(u) = 0$ for $u < \tau_{L_1}$, and if $L_1 > k^* + 1$, then $$\int_{t - a_N}^{\tau_{L_1}} S(u) \: du \leq \int_{\tau_{L_1 - 1}}^{\tau_{L_1}} \frac{G_{L_1 - 1}(u)}{s} \: du.$$  Therefore, Lemma \ref{smulem} implies that for sufficiently large $N$,
\begin{equation}\label{Seq3}
0 \leq \int_{t - a_N}^{\tau_{L_1}} S(u) \: du \leq \frac{1}{s} \log \bigg( \frac{2s}{\mu} \bigg) \leq \bigg(1 + \frac{\eta}{3} \bigg) a_N.
\end{equation}
By Remark \ref{tauorder}, the times $\tau_{L_1}, \tau_{L_1+1}, \dots, \tau_{L_2}$ are in $(t - a_N, t]$, so $R(t) = L_2 - L_1 + 1$.  Therefore, we can sum (\ref{Seq1}) over $j$ from $L_1$ to $L_2 - 1$ and combine this result with (\ref{Seq2}), and (\ref{Seq3}) to get that for sufficiently large $N$,
$$(R(t) - 1)\bigg(1 - \frac{\eta}{3} \bigg) a_N \leq \int_{t - a_N}^t S(u) \: du \leq (R(t) + 1)\bigg(1 + \frac{\eta}{3} \bigg) a_N.$$
Rearranging this equation, we get, for sufficiently large $N$,
\begin{equation}\label{prelimR}
- 1 + \frac{1}{(1 + \eta/3) a_N} \int_{t - a_N}^t S(u) \: du \leq R(t) \leq 1 + \frac{1}{(1 - \eta/3) a_N} \int_{t - a_N}^t S(u) \: du.
\end{equation}

We now relate $S(u)$ to $R(u)$.  By Lemma \ref{Rjlem}, if $u \in [\tau_j, \tau_{j+1}) \cap [0, t]$ with $j \geq k^* + 1$, then $R(u) = j - {\bar M}(u)$.  Therefore, for $u \in [\tau_j, \tau_{j+1}) \cap [0, t]$,
$$S(u) = \frac{G_j(u)}{s} = (j - {\bar M}(u)) + ({\bar M}(u) - M(u)) - \frac{\mu}{s} = R(u) + ({\bar M}(u) - M(u)) - \frac{\mu}{s}.$$  If $0 < t < \tau_{k^*+1}$, then $S(u) = 0$ and $R(u) = k^*$.  Therefore,
\begin{equation}\label{RS1}
\int_{t - a_N}^t |S(u) - R(u)| \: du \leq k^* \tau_{k^* + 1} + \int_{(t - a_N) \vee \tau_{k^* + 1}}^t |{\bar M}(u) - M(u)| \: du + \frac{\mu}{s} a_N.
\end{equation}
By Proposition \ref{taukstar}, for sufficiently large $N$,
\begin{equation}\label{RS2}
k^* \tau_{k^* + 1} \leq \frac{2 k^*}{k_N} a_N.
\end{equation}
The number of values of $\gamma_{\ell}$ between $t - a_N$ and $t$ is the same as the number of values of $\tau_{\ell}$ between $t - 2a_N$ and $t - a_N$, which is either $R(t - a_N)$ or $R(t - a_N) - k^*$ depending on the value of $t$.  This means that at most $R(t - a_N) + 1$ intervals of the form $[\gamma_{\ell}, \gamma_{\ell+1})$ can intersect the interval $[t - a_N, t]$.  Therefore, by Lemma \ref{MMbarint}, for sufficiently large $N$
\begin{equation}\label{RS3}
\int_{(t - a_N) \vee \tau_{k^* + 1}}^t |{\bar M}(u) - M(u)| \: du \leq \frac{3}{s} + \frac{2 k^*}{k_N} a_N + \frac{2 C_5}{s}(R(t - a_N) + 1).
\end{equation}
Therefore, combining (\ref{RS1}), (\ref{RS2}), and (\ref{RS3}), we get that for sufficiently large $N$,
$$\int_{t - a_N}^t |S(u) - R(u)| \: du \leq \frac{3}{s} + \frac{4 k^*}{k_N} a_N + \frac{2 C_5}{s} (R(t - a_N) + 1) + \frac{\mu}{s} a_N.$$  Therefore, if (\ref{Rbounds}) holds, then for sufficiently large $N$,
$$\frac{1}{a_N} \int_{t - a_N}^t |S(u) - R(u)| \: du \leq \frac{3 + 2C_5 (3k_N + 1)}{s a_N} + \bigg( \frac{4k^*}{k_N} + \frac{\mu}{s} \bigg).$$  Because $sa_N \rightarrow \infty$ by (\ref{muspower}), it follows that for sufficiently large $N$, when (\ref{Rbounds}) holds we have
\begin{equation}\label{finalRS}
\frac{1}{a_N} \int_{t - a_N}^t |S(u) - R(u)| \: du \leq \frac{\eta}{6} k_N \leq \frac{\eta}{3 a_N} \int_{t - a_N}^t R(u) \: du.
\end{equation}
From (\ref{prelimR}) and (\ref{finalRS}), we conclude that for sufficiently large $N$, when (\ref{Rbounds}) holds we have
$$-1 + \frac{1 - \eta/3}{(1 + \eta/3) a_N} \int_{t - a_N}^t R(u) \: du \leq R(t) \leq 1 + \frac{1 + \eta/3}{(1 - \eta/3) a_N} \int_{t - a_N}^t R(u) \: du.$$
The result follows since $1 - \eta < (1 - \eta/3)/(1 + \eta/3) < (1 + \eta/3)/(1-\eta/3) < 1 + \eta$ if $0 < \eta < 1$.
\end{proof}

The following deterministic result will help us to obtain the second part of Proposition \ref{tauprop} from Lemmas \ref{Rrange1} and \ref{Rrange2}.

\begin{Lemma}\label{rqlem}
Let $0 < \eta < 1$.  Suppose $r: [0, T] \rightarrow \R$ is a right continuous function such that $(1 - \eta) e^{(1 - \eta)t} < r(t) < (1 + \eta) e^{(1 + \eta)t}$ for $0 \leq t < 1$ and $(1 - \eta)\int_{t-1}^t r(u) \: du < r(t) < (1 + \eta)\int_{t-1}^t r(u) \: du$ for $1 \leq t \leq T$.  Let $q$ be the function defined in (\ref{qdef}).  Then $$\sup_{t \in [0, T]} |r(t) - q(t)| \leq 4 \eta e^{(1 + \eta)T}.$$
\end{Lemma}

\begin{proof}
Let $r_1: [0, T] \rightarrow [0, \infty)$ and $r_2: [0, T] \rightarrow [0, \infty)$ be the unique bounded functions satisfying
\begin{displaymath}
r_1(t) = \left\{
\begin{array}{ll} (1 - \eta) e^{(1 - \eta)t} & \mbox{if }0 \leq t < 1 \\
(1 - \eta) \int_{t-1}^t r_1(u) \: du & \mbox{if }1 \leq t \leq T, \\
\end{array} \right. \hspace{.02in}
r_2(t) = \left\{
\begin{array}{ll} (1 + \eta) e^{(1 + \eta)t} & \mbox{if }0 \leq t < 1 \\
(1 + \eta) \int_{t-1}^t r_2(u) \: du & \mbox{if }1 \leq t \leq T. \\
\end{array} \right.
\end{displaymath}
The existence and uniqueness of these functions, and their continuity away from $1$, follows from Theorem 2 in \cite{feller} as in the proof of Lemma \ref{Qlem} because the functions $r_1$ and $r_2$ satisfy (\ref{renewal}) if we replace the functions $f$ and $g$ by $f_1$ and $g_1$ or $f_2$ and $g_2$ respectively, where $f_1(u) = g_1(u) = (1 - \eta) \1_{\{0 \leq u < 1\}}$ and $f_2(u) = g_2(u) = (1 + \eta) \1_{\{0 \leq u < 1\}}$.

We claim that $r_1(t) < r(t) < r_2(t)$ and $r_1(t) < q(t) < r_2(t)$ for all $t \in [0, T]$.  To see this, let $u = \inf\{t: r(t) \geq r_2(t)\}$.  Seeking a contradiction, suppose $u \leq T$.  Clearly $u \geq 1$, and so $r_2(u) - r(u) \geq (1 + \eta) \int_{u-1}^u (r_2(t) - r(t)) \: dt > 0$, which contradicts the right continuity of $r$ and $r_2$.  Therefore, $r(t) \leq r_2(t)$ for all $t \in [0, T]$.  A parallel argument gives $r(t) \geq r_1(t)$ for all $t \in [0, T]$.  The result for $q$ is a special case of the result for $r$, which completes the proof of the claim.

Let $d(t) = r_2(t) - r_1(t)$ for all $t \in [0, T]$.  The claim above implies that
\begin{equation}\label{rqbound}
\sup_{t \in [0, T]} |r(t) - q(t)| \leq \sup_{t \in [0, T]} d(t).
\end{equation}
We have $d(t) = (1 + \eta)e^{(1 + \eta)t} - (1 - \eta)e^{(1 - \eta)t}$ for $t \in [0, 1]$.  Note that if $t \in [0, 1]$, then
\begin{equation}\label{rqb2}
d(t) \leq d(1) \leq e^{1 + \eta} - e^{1 - \eta} + 2 \eta e^{1 + \eta} \leq 4 \eta e^{1 + \eta}.
\end{equation}
If $1 \leq t \leq T$, then
$$d(t) = (1 + \eta) \int_{t-1}^t r_2(u) \: du - (1 - \eta) \int_{t-1}^t r_1(u) \: du = (1 + \eta) \int_{t-1}^t d(u) \: du + 2 \eta \int_{t-1}^t r_1(u) \: du.$$  Therefore, using that $0 \leq r_1(t) \leq q(t) \leq e$ for all $t$ by Lemma \ref{Qlem}, we see that if $1 < t \leq T$, then $$d'(t) = (1 + \eta)(d(t) - d(t-1)) + 2 \eta (r_1(t) - r_1(t-1)) \leq (1 + \eta) d(t) + 2e \eta.$$  Solutions to the differential equation $f'(t) = (1 + \eta) f(t) + 2e \eta$ can be expressed in the form $f(t) = C e^{(1 + \eta)t} - 2 e \eta/(1 + \eta)$, where $C$ is a constant.  If $f(1) = d(1)$, then $C = (d(1) + 2 e \eta/(1 + \eta)) e^{-(1 + \eta)}$.  Therefore, if $1 \leq t \leq T$, then
\begin{equation}\label{dbound}
d(t) \leq C e^{(1 + \eta) t} - \frac{2 e \eta}{1 + \eta} \leq 4 \eta e^{(1 + \eta) t}.
\end{equation}
The result follows from (\ref{rqbound}), (\ref{rqb2}), and (\ref{dbound}).
\end{proof}

\begin{Prop}\label{rqprop}
For sufficiently large $N$, on the event that $\zeta_0 = \infty$, we have $$\bigg| \frac{R(a_N t)}{k_N} - q(t) \bigg| < \delta$$
for all $t \in [0, T]$ such that $\zeta_1 > a_N t$ and $\zeta_2 \geq a_N t$.
\end{Prop}

\begin{proof}
Suppose that $\zeta_0 = \infty$, $\zeta_1 > a_N t$, and $\zeta_2 \geq a_N t$.  Choose $\eta > 0$ small enough that $4 \eta e^{(1 + \eta)T} < \delta$.  For $u \in [0, T]$, let $r(u) = R(a_N u)/k_N$.  Lemma \ref{Rrange1} implies that if $u < 1$ and $u \leq t$, then $(1 - \eta) e^{(1 - \eta)u} \leq r(u) \leq (1 + \eta) e^{(1 + \eta)u}$.  Define $\kappa = \inf\{u: r(u) \geq 3 \mbox{ or }r(u) \leq 1/2\}$.  By Lemma \ref{Rrange2}, if $1 \leq u < \kappa$ and $u \leq t$, then $$(1 - \eta) \int_{u-1}^u r(v) \: dv \leq r(u) \leq (1 + \eta) \int_{u-1}^u r(v) \: dv.$$  Note that $R$ is right continuous, and therefore so is $r$, so we can apply Lemma \ref{rqlem} to the function $r$ to get
\begin{equation}\label{rqeq}
\sup_{u \in [0, t] \cap [0, \kappa)} |r(u) - q(u)| < \delta.
\end{equation}
The result will follow from (\ref{rqeq}) if we can establish that $\kappa > t$.  In view of Remark \ref{tauorder}, we have $|R(u) - R(u-)| \in \{-1, 0, 1\}$ for all $u \in [0, a_N t]$.  In particular, if $\kappa \leq t$, then $|r(\kappa) - r(\kappa-)| \leq 1/k_N$, which contradicts (\ref{rqeq}) for sufficiently large $N$ because $1 \leq q(u) \leq e$ for all $u \geq 0$ by Lemma \ref{Qlem}.  Therefore, $\kappa > t$, and the proof is complete.
\end{proof}

\subsection{The spacings between $\tau_j$ and $\tau_{j+1}$}

The third part of Proposition \ref{tauprop} primarily pertains to the spacings between $\tau_j$ and $\tau_{j+1}$.  The proposition below establishes the necessary relationship between the times $\tau_j$ and the function $q$, and leads easily to the main result (\ref{tauspacing}).

\begin{Prop}\label{qintprop}
If $N$ is sufficiently large, then for all $j \in \{k^* + 1, \dots, J-1\}$ such that $\zeta_0 = \infty$, $\zeta_1 > \tau_{j+1}$, $\zeta_2 \geq \tau_{j+1}$, and $\tau_{j+1} \leq a_N T$, we have
\begin{equation}\label{qtauupper}
\int_{\tau_j/a_N}^{\tau_{j+1}/a_N} q(u) \: du \leq \frac{1 + 2 \delta}{k_N}
\end{equation}
and
\begin{equation}\label{qtaulower}
\int_{\tau_j/a_N}^{\tau_{j+1}/a_N} (q(u) + \1_{\{u \in [1, \gamma_{k^* + 1}/a_N)\}}) \: du \geq \frac{1 - 2 \delta}{k_N}.
\end{equation}
Also, if $N$ is sufficiently large, then for all $j \in \{k^* + 1, \dots, J-1\}$ and all $t \in [0, a_N T]$, on the event that $\zeta_0 = \infty$, $\zeta_1 > t$, and $\zeta_2 \geq t$, if
\begin{equation}\label{qtbound}
\int_{\tau_j/a_N}^{t/a_N} q(u) \: du \geq \frac{1 + 2 \delta}{k_N},
\end{equation}
then $\tau_{j+1} \leq t$.
\end{Prop}

\begin{proof}
We prove the result by induction on $j$.  Suppose $\zeta_0 = \infty$, $\zeta_1 > \tau_{j+1}$, $\zeta_2 \geq \tau_{j+1}$, and $\tau_{j+1} \leq a_N T$.  Suppose also that (\ref{qtauupper}) and (\ref{qtaulower}) hold with $\ell$ in place of $j$ for $\ell \in \{k^* + 1, \dots, j-1\}$.  Let $\eta > 0$.  From (\ref{Seq1}) we see that if $N$ is sufficiently large, then
\begin{equation}\label{elemG}
(1 - \eta) a_N \leq \int_{\tau_j}^{\tau_{j+1}} \frac{G_j(v)}{s} \: dv \leq (1 + \eta) a_N.
\end{equation}
By Lemma \ref{Rjlem}, for $v \in [\tau_j, \tau_{j+1})$,
\begin{equation}\label{GReq}
\frac{G_j(v)}{s} = j - M(v) - \frac{\mu}{s} = R(v) + ({\bar M}(v) - M(v)) - \frac{\mu}{s}.
\end{equation}
Let $L_j$ be the number of integers $\ell \geq k^* + 1$ such that $\gamma_{\ell} \in [\tau_j, \tau_{j+1})$.  Then the interval $[\tau_j, \tau_{j+1})$ intersects at most $L_j + 1$ intervals of the form $[\gamma_{\ell - 1}, \gamma_{\ell})$, so by Lemma \ref{MMbarint}, if $N$ is sufficiently large, then
\begin{equation}\label{preMM}
\int_{\tau_j}^{\tau_{j+1}} |{\bar M}(v) - M(v)| \1_{\{v \notin [a_N, \gamma_{k^* + 1})\}} \: dv \leq \frac{3}{s} + \frac{2(L_j + 1)C_5}{s}.
\end{equation}
By the induction hypothesis, (\ref{qtaulower}) holds if $j$ is replaced by $\ell \in\{k^* + 1, \dots, j-1\}$.
By Lemma \ref{Qlem}, we have $q(u) \leq e$ for all $u \geq 0$.  Also, $\gamma_{k^*+1}/a_N = 1 + \tau_{k^* + 1}/a_N \leq 1 + 2/k_N$ for sufficiently large $N$ by Lemma \ref{taukstar}.  Since $q$ is right continuous and $q(1) = e - 1$, it follows that for sufficiently large $N$, we have
\begin{equation}\label{supq}
\sup_{u \geq 0} \: (q(u) + \1_{\{u \in [1, \gamma_{k^* + 1}/a_N)\}}) < e + \delta.
\end{equation}
Thus, using (\ref{deltadef}) and (\ref{qtaulower}), for sufficiently large $N$, $$\gamma_{\ell+1} - \gamma_{\ell} = \tau_{\ell+1} - \tau_{\ell} \geq \frac{a_N(1 - 2 \delta)}{k_N (e + \delta)} \geq \frac{a_N}{3k_N}.$$  It follows that $L_j \leq 1 + (3k_N/a_N)(\tau_{j+1} - \tau_j)$.  Combining this observation with (\ref{preMM}) gives
$$\int_{\tau_j}^{\tau_{j+1}} |{\bar M}(v) - M(v)| \1_{\{v \notin [a_N, \gamma_{k^* + 1})\}} \: dv \leq \frac{3 + 4C_5}{s} + \bigg( \frac{6k_N C_5}{a_N s} \bigg) (\tau_{j+1} - \tau_j).$$
Write $C = 3 + 4 C_5$.  Because $k_N/(a_N s) \rightarrow 0$ as $N \rightarrow \infty$ by (\ref{A2prime}) and $\mu/s \rightarrow 0$ as $N \rightarrow \infty$, it follows that for sufficiently large $N$,
\begin{equation}\label{abserror}
\int_{\tau_j}^{\tau_{j+1}} \bigg( |{\bar M}(v) - M(v)| \1_{\{v \notin [a_N, \gamma_{k^* + 1})\}} + \frac{\mu}{s} \bigg) \: dv \leq \frac{C}{s} + \eta (\tau_{j+1} - \tau_j).
\end{equation}
Combining (\ref{abserror}) with (\ref{elemG}) and (\ref{GReq}), we get that for sufficiently large $N$,
\begin{align*}
\bigg| \int_{\tau_j}^{\tau_{j+1}} (R(v) + ({\bar M}(v) - M(v))\1_{\{v \in [a_N, \gamma_{k^* + 1})\}}) \: dv - a_N \bigg| \leq \eta a_N + \frac{C}{s} + \eta(\tau_{j+1} - \tau_j).
\end{align*}
To simplify notation, write $h(v) = ({\bar M}(v) - M(v))\1_{\{v \in [a_N, \gamma_{k^* + 1})\}}$.  Make the substitution $u = v/a_N$ and divide both sides by $a_N k_N$ to get
$$\bigg| \int_{\tau_j/a_N}^{\tau_{j+1}/a_N} \bigg( \frac{R(a_N u)}{k_N} + \frac{h(a_N u)}{k_N} \bigg) \: du - \frac{1}{k_N} \bigg|  \leq \frac{\eta}{k_N} + \frac{C}{s a_N k_N} + \frac{\eta(\tau_{j+1} - \tau_j)}{a_N k_N}$$
for sufficiently large $N$.  By Proposition \ref{rqprop}, we have $|R(a_N u)/k_N - q(u)| < \delta$ for $u < \tau_{j+1}/a_N$, so for sufficiently large $N$,
\begin{equation}\label{mainabsbound}
\bigg| \int_{\tau_j/a_N}^{\tau_{j+1}/a_N} \bigg( q(u) + \frac{h(a_N u)}{k_N} \bigg) \: du - \frac{1}{k_N} \bigg|  \leq \frac{\eta}{k_N} + \frac{C}{s a_N k_N} + \frac{\eta(\tau_{j+1} - \tau_j)}{a_N k_N} + \frac{\delta(\tau_{j+1} - \tau_j)}{a_N}.
\end{equation}

We now pursue the upper and lower bounds separately.  In view of part 2 of Proposition \ref{meanprop}, because $\zeta_2 \geq \tau_{j+1}$, we have $h(v) \geq k^* - k_N - C_4$ for all $v \in [a_N, \gamma_{k^* + 1}) \cap [\tau_j, \tau_{j+1})$.  Therefore, because $k^*/k_N \rightarrow 1$ as $N \rightarrow \infty$ and $\gamma_{k^*+1} - a_N = \tau_{k^* + 1} \leq 2a_N/k_N$ by Proposition \ref{taukstar}, for sufficiently large $N$ we have
$$\int_{\tau_j/a_N}^{\tau_{j+1}/a_N} \frac{h(a_N u)}{k_N} \: du \geq \bigg( \frac{k^* - k_N - C_4}{k_N} \bigg) \bigg( \frac{\gamma_{k^* + 1} - a_N}{a_N} \bigg) \geq - \frac{\eta}{k_N}.$$  Combining this result with (\ref{mainabsbound}) yields
\begin{equation}\label{preqdelta}
\int_{\tau_j/a_N}^{\tau_{j+1}/a_N} q(u) \: du \leq \frac{1 + 2 \eta}{k_N} + \frac{C}{s a_N k_N} + \frac{\eta(\tau_{j+1} - \tau_j)}{a_N k_N} + \frac{\delta(\tau_{j+1} - \tau_j)}{a_N}.
\end{equation}
Since $s a_N \rightarrow \infty$, we have $C/(s a_N) < \eta$ for sufficiently large $N$.  Therefore, bringing the last two terms on the right-hand side of (\ref{preqdelta}) to the left-hand side, we get
$$\int_{\tau_j/a_N}^{\tau_{j+1}/a_N} \bigg( q(u) - \frac{\eta}{k_N} - \delta \bigg) \: du \leq \frac{1 + 3 \eta}{k_N}$$
for sufficiently large $N$.  Also, since $q(u) \geq 1$ for all $u \geq 0$ by Lemma \ref{Qlem}, we have $q(u)(1 - \alpha) \leq q(u) - \alpha$ for all $u \geq 0$ and $\alpha > 0$.  Therefore, for sufficiently large $N$,
$$\int_{\tau_j/a_N}^{\tau_{j+1}/a_N} q(u) \: du \leq \bigg(1 - \frac{\eta}{k_N} - \delta \bigg)^{-1} \bigg(\frac{1 + 3 \eta}{k_N} \bigg).$$  The upper bound (\ref{qtauupper}) follows as long as $\eta$ is chosen to be small enough relative to $\delta$.

To obtain (\ref{qtaulower}), note that $h(v) \leq k^*$ for all $v \in [a_N, \gamma_{k^* + 1}) \cap [\tau_j, \tau_{j+1})$.  Therefore, for sufficiently large $N$,
\begin{align*}
\int_{\tau_j/a_N}^{\tau_{j+1}/a_N} \frac{h(a_N u)}{k_N} \: du &\leq \int_{\tau_j/a_N}^{\tau_{j+1}/a_N} \1_{\{u \in [1, \gamma_{k^* + 1}/a_N)\}} \: du + \bigg(\frac{k^* - k_N}{k_N}\bigg) \bigg(\frac{\gamma_{k^* + 1} - a_N}{a_N} \bigg). \\
&\leq \int_{\tau_j/a_N}^{\tau_{j+1}/a_N} \1_{\{u \in [1, \gamma_{k^* + 1}/a_N)\}} \: du + \frac{\eta}{k_N}.
\end{align*}
Combining this result with (\ref{mainabsbound}) and using that $sa_N \rightarrow \infty$ as $N \rightarrow \infty$, we get for sufficiently large $N$,
$$\int_{\tau_j/a_N}^{\tau_{j+1}/a_N} (q(u) + \1_{\{u \in [1, \gamma_{k^* + 1}/a_N)\}}) \: du \geq \frac{1 - 3 \eta}{k_N} - \frac{\eta(\tau_{j+1} - \tau_j)}{a_N k_N} - \frac{\delta(\tau_{j+1} - \tau_j)}{a_N}$$ and therefore
$$\int_{\tau_j/a_N}^{\tau_{j+1}/a_N} \bigg(q(u) + \1_{\{u \in [1, \gamma_{k^* + 1}/a_N)\}} + \frac{\eta}{k_N} + \delta \bigg) \: du \geq \frac{1 - 3 \eta}{k_N}.$$
If $x \geq 1$ and $\alpha > 0$, then $x(1 + \alpha) \geq x + \alpha$.  Therefore, for sufficiently large $N$,
$$\int_{\tau_j/a_N}^{\tau_{j+1}/a_N} (q(u) + \1_{\{u \in [1, \gamma_{k^* + 1}/a_N)\}}) \: du \geq \bigg(1 + \frac{\eta}{k_N} + \delta \bigg)^{-1} \bigg(\frac{1 - 3 \eta}{k_N}\bigg).$$
The lower bound (\ref{qtaulower}) follows as long as $\eta$ is chosen to be small enough relative to $\delta$.

It remains to prove the last statement of the proposition.  Suppose now that $\zeta_0 = \infty$, $\zeta_1 > t$, $\zeta_2 \geq t$, $t \leq a_N T$, and (\ref{qtbound}) holds.  We need to show that $\tau_{j+1} \leq t$.  By Lemma \ref{smulem}, if $N$ is large enough, it suffices to show that $$e^{\int_{\tau_j}^t G_j(v) \: dv} \geq \frac{2s}{\mu}.$$  Therefore, it suffices to show that for sufficiently large $N$,
\begin{equation}\label{Gsts}
\int_{\tau_j}^t \frac{G_j(v)}{s} \: dv \geq (1 + \eta) a_N.
\end{equation}
Using (\ref{GReq}), the bound in part 2 of Proposition \ref{meanprop}, and the reasoning leading to (\ref{abserror}) with $t$ in place of $\tau_{j+1}$, we get for sufficiently large $N$,
\begin{equation}\label{GSlower}
\int_{\tau_j}^t \frac{G_j(v)}{s} \: dv \geq \int_{\tau_j}^t R(v) \: dv - \frac{C}{s} - \eta(t - \tau_j) - (k^* - k_N - C_4)(\gamma_{k+1}^* - a_N).
\end{equation}
By Proposition \ref{rqprop}, for sufficiently large $N$,
\begin{equation}\label{Rintlower}
\int_{\tau_j}^t R(v) \: dv = a_N \int_{\tau_j/a_N}^{t/a_N} R(a_N u) \: du \geq a_N k_N \int_{\tau_j/a_N}^{t/a_N} (q(u) - \delta) \: du.
\end{equation}
Using Proposition \ref{taukstar}, we have $(k^* - k_N - C_4)(\gamma_{k^* + 1} - a_N) \leq (k^* - k_N - C_4)(2 a_N/k_N) \leq \eta a_N$ for sufficiently large $N$.  Combining this bound with (\ref{GSlower}) and (\ref{Rintlower}), and then using (\ref{qtbound}), we get that for sufficiently large $N$,
\begin{align*}
\int_{\tau_j}^t \frac{G_j(v)}{s} \: dv &\geq a_N k_N \int_{\tau_j/a_N}^{t/a_N} \bigg( q(u) - \delta - \frac{\eta}{k_N} \bigg) \: du - \frac{C}{s} - \eta a_N \\
&\geq a_N k_N \bigg(1 - \delta - \frac{\eta}{k_N} \bigg) \int_{\tau_j/a_N}^{t/a_N} q(u) \: du - \frac{C}{s} - \eta a_N \\
&\geq a_N \bigg( 1 - \delta - \frac{\eta}{k_N} \bigg) (1 + 2 \delta) - \frac{C}{s} - \eta a_N,
\end{align*}
which implies (\ref{Gsts}) as long as $\eta$ is chosen to be small enough relative to $\delta$, in view of the fact that $s a_N \rightarrow \infty$ as $N \rightarrow \infty$.
\end{proof}

\begin{proof}[Proof of part 2 of Proposition \ref{zetaprop}]
Recall that $q(u) \geq 1$ for all $u \geq 0$ by Lemma \ref{Qlem}.  Therefore, if (\ref{qtauupper}) holds, then $$\frac{\tau_{j+1} - \tau_j}{a_N} \leq \frac{1 + 2 \delta}{k_N} \leq \frac{2}{k_N}.$$  Also, in view of (\ref{supq}), for sufficiently large $N$, if (\ref{qtaulower}) holds, then $$\frac{\tau_{j+1} - \tau_j}{a_N} \geq \frac{1 - 2 \delta}{(e + \delta) k_N} \geq \frac{1}{3k_N}.$$  Thus, if (\ref{qtauupper}) and (\ref{qtaulower}) hold, then so does (\ref{tauspacing}).  Also, if $\tau_j + 2 a_N/k_N \leq a_N T$, then (\ref{qtbound}) holds with $t = \tau_j + 2a_N/k_N$.  Therefore, part 2 of Proposition \ref{zetaprop} follows from Propositions \ref{taukstar}, \ref{rqprop}, and \ref{qintprop}.
\end{proof}

\section{Proof of part 3 of Proposition \ref{zetaprop}}\label{zetasec3}

To prove part 3 of Proposition \ref{zetaprop}, we need to show that with high probability, the results of Propositions \ref{prop1} and \ref{prop2} hold as long as the results of Propositions \ref{meanprop} and \ref{tauprop} hold.  Propositions \ref{prop1} and \ref{prop2} describe the behavior of the number of type $j$ individuals.  The proof proceeds by induction on $j$, in the sense that to show that the number of type $j$ individuals behaves as predicted, we will need to know that the number of type $j-1$ individuals does so.  Define the stopping time 
$$\rho_j = \zeta_0 \wedge \zeta_2 \wedge \zeta_3 \wedge \zeta_{1, j-1} \wedge a_N T.$$  We then need to show that 
\begin{equation}\label{stszeta}
\sum_{j=0}^J P(\{\zeta_0 = \infty\} \cap \{\zeta_{1,j} \leq \rho_j\}) < \frac{\eps}{2}.
\end{equation}
Essentially, this means that the number of type $j$ individuals behaves as expected with high probability until after time $\rho_j$.

Note that if $t < \rho_j$, then the reasoning in Remark \ref{JRmk} implies that no individual of type $J+1$ or higher can appear until after time $t$.  Because assumption A3 implies that $s k_N \rightarrow \infty$, we have $sJ \leq 1$ for sufficiently large $N$.  It follows that $1 + s(j - M(t)) \geq 0$ for all $j \geq 0$, and therefore $G_j^*(t) = G_j(t)$ for all $j \geq 0$ as noted in (\ref{GGstar}).  Throughout this section, we will assume that $N$ is large enough that $sJ \leq 1$, which will make it possible to ignore the distinction between $G_j^*(t)$ and $G_j(t)$.

\subsection{Individuals of type $j \leq k^*$}

In this subsection, we consider the behavior of individuals of type $j$ for $j \in \{0, 1, \dots, k^*\}$ and show that with high probability this behavior matches what is described in Proposition \ref{prop1}.  Central to the analysis will be the martingales $Z_j^{\kappa, \tau}$ from Corollary \ref{ZmartCor2}, with $\kappa = t^*$ and $\tau = (\rho_j \wedge \gamma_{k^* + K}) \vee t^*$.  To lighten notation, we denote this process by $Z_j'$.  We let $\rho_j^* = (\rho_j \wedge \gamma_{k^* + K}) \vee t^*$ and then, for $t \geq t^*$, we let 
\begin{equation}\label{Zjprimedef}
Z_j'(t) = e^{-\int_{t^*}^{t \wedge \rho_j^*} G_j(v) \: dv} X_j(t \wedge \rho_j^*) - \int_{t^*}^{t \wedge \rho_j^*} \mu X_{j-1}(u) e^{-\int_{t^*}^u G_j(v) \: dv} \: du - X_j(t^*).
\end{equation}
Note that when $j = 0$, we are using the convention $X_{-1}(u) = 0$.  For $t \in [t^*, \rho_j^*]$, 
\begin{align}\label{Xj3terms}
X_j(t) &= e^{\int_{t^*}^t G_j(v) \: dv} X_j(t^*) + \int_{t^*}^t \mu X_{j-1}(u) e^{\int_u^t G_j(v) \: dv} \: du + e^{\int_{t^*}^t G_j(v) \: dv} Z_j'(t) \nonumber \\
&= T_{j,1}(t) + T_{j,2}(t) + T_{j,3}(t),
\end{align}
where $T_{j,1}(t)$, $T_{j,2}(t)$, and $T_{j,3}(t)$ denote the three terms in the previous line.  To establish the result of part 1 of Proposition \ref{prop1}, we need to show that $|T_{j,2}(t) + T_{j,3}(t)|/T_{j,1}(t) < \delta$ with high probability for $t \in [t^*, \rho_j^*]$.  We first bound $T_{j,2}(t)/T_{j,1}(t)$.

\begin{Lemma}\label{jj1ratio}
For sufficiently large $N$, if $1 \leq j \leq k^*$, then on $\{\zeta_0 = \infty\}$, $$\frac{\mu X_{j-1}(t^*)}{s X_j(t^*)} < \frac{\delta}{3}.$$ 
\end{Lemma}

\begin{proof}
Suppose $\zeta_0 = \infty$.  By (\ref{early1}), if $j \leq k_N^-$, then
\begin{equation}\label{Tjratio1}
\frac{\mu X_{j-1}(t^*)}{s X_j(t^*)} \leq \frac{\mu (1 + \delta)}{s (1 - \delta)} \cdot \frac{x_{j-1}(t^*)}{x_j(t^*)} = \frac{1 + \delta}{1 - \delta} \cdot \frac{j}{e^{st^*} - 1} \leq \frac{1 + \delta}{1 - \delta} \cdot \frac{k^*}{k_N^2 - 1}.
\end{equation}
Suppose instead $j \in (k_N^-, k_N^+)$.  Because $k_N^+ - k_N^- \rightarrow 0$ as $N \rightarrow \infty$ by (\ref{kdiff}), for sufficiently large $N$ we know that $j-1 \leq k_N^-$.  For such $N$, because $d_j \leq 2$, equation (\ref{earlypt2}) yields 
\begin{equation}\label{Tjratio2}
\frac{\mu X_{j-1}(t^*)}{s X_j(t^*)} \leq \frac{\mu (1 + \delta) }{C_1 s} \cdot \frac{x_{j-1}(t^*)}{k_N^{-d_j} x_j(t^*)} = \frac{1 + \delta}{C_1 (1 - \delta)} \cdot \frac{j}{k_N^{-d_j}(e^{st^*} - 1)} \leq \frac{1 + \delta}{C_1 (1 - \delta)} \cdot \frac{k_N^2 k^*}{k_N^4 - 1}.
\end{equation}
Because the right-hand sides of (\ref{Tjratio1}) and (\ref{Tjratio2}) tend to zero as $N \rightarrow \infty$, the result follows.
\end{proof}

\begin{Lemma}\label{T2j}
For sufficiently large $N$, if $0 \leq j \leq k^*$ and $t \in (t^*, \rho_j^*]$, then $T_{j,2}(t)/T_{j,1}(t) \leq \delta/2$.
\end{Lemma}

\begin{proof}
Since $T_{0,2}(t) = 0$, we may assume $1 \leq j \leq k^*$.  Because $\zeta_{1,j-1} \leq \rho_j^*$, the conclusion of part 1 of Proposition \ref{prop1} holds for $j-1$ up to time $\rho_j^*$.  Therefore, if $u \in (t^*, \rho_j^*)$, then $$X_{j-1}(u) \leq (1 + \delta) X_{j-1}(t^*) e^{\int_{t^*}^u G_{j-1}(v) \: dv}.$$  It follows that if $t \in (t^*, \rho_j^*]$, then
\begin{align*}T_{j,2}(t) &\leq \mu (1 + \delta) X_{j-1}(t^*) \int_{t^*}^t e^{\int_{t*}^u G_{j-1}(v) \: dv} e^{\int_u^t G_j(v) \: dv} \: du \\
&= \mu (1 + \delta) X_{j-1}(t^*) e^{\int_{t^*}^t G_j(v) \: dv} \int_{t^*}^t e^{-s(u - t^*)} \: du \\
&\leq \frac{\mu(1 + \delta)}{s} X_{j-1}(t^*) e^{\int_{t^*}^t G_j(v) \: dv}.
\end{align*}
Thus, if $t \in (t^*, \rho_j^*]$, then
$$\frac{T_{j,2}(t)}{T_{j,1}(t)} \leq \frac{\mu (1 + \delta) X_{j-1}(t^*)}{s X_j(t^*)}.$$
The result now follows from Lemma \ref{jj1ratio}.
\end{proof}

To bound $T_{j,3}(t)/T_{j,1}(t)$, we will need to control the fluctuations of the process $(Z_j'(t), t \geq t^*)$.  The following preliminary bound will be useful.

\begin{Lemma}\label{expGlem}
For sufficiently large $N$, if $0 \leq j \leq k^*$ and $u \in (t^*, \rho_j^*]$, then
$$\exp \bigg( - \int_{t^*}^u G_j(v) \: dv \bigg) \leq w(u) e^{-sj(u - t^*)},$$ where
\begin{equation}\label{wdef}
w(u) = \left\{
\begin{array}{ll} 21 & \mbox{ if }u \in (t^*, a_N]  \\
(s/\mu)^{2 k_N/3} & \mbox{ if }u > a_N.
\end{array} \right.
\end{equation}
\end{Lemma}

\begin{proof}
Note that
\begin{equation}\label{expGeq1}
e^{-\int_{t^*}^u G_j(v) \: dv} = e^{-sj(u - t^*)} e^{\int_{t^*}^u s M(v) \: dv + \mu (u - t^*)}.
\end{equation}
In view of parts 1 and 3 of Proposition \ref{tauprop}, we have $\mu(u - t^*) \leq \mu \gamma_{k^* + K} \leq \mu(a_N + 2K a_N/k_N) \rightarrow 0$ as $N \rightarrow \infty$.
If $u \leq a_N$, then $\int_{t^*}^u M(v) \: dv \leq 3/s$ by Lemma \ref{MMbarint} and therefore, for sufficiently large $N$,
\begin{equation}\label{expGeq2}
e^{\int_{t^*}^u s M(v) \: dv + \mu(u - t^*)} \leq e^{3 + \mu(u - t^*)} \leq 21 = w(u).
\end{equation}
Suppose instead $a_N < u \leq \gamma_{k^* + K}$.  By the results of Propositions \ref{meanprop} and \ref{tauprop},
\begin{align*}
\int_{t^*}^u M(v) \: dv &\leq \int_{t^*}^{a_N} M(v) \: dv + \int_{a_N}^{\gamma_{k^*+1}} M(v) \: dv + \sum_{\ell = 1}^{K-1} \int_{\gamma_{k^* + \ell}}^{\gamma_{k^*+\ell+1}} M(v) \: dv \\
&\leq \frac{3}{s} + (k_N + C_4)(\gamma_{k^* + 1} - a_N) + \sum_{\ell = 1}^{K-1} (k^* + \ell + 2C_5)(\gamma_{k^*+\ell+1} - \gamma_{k^*+\ell}) \\
&\leq \frac{2a_N}{k_N} \bigg( \frac{3k_N}{2s a_N} + (k_N + C_4) + (K - 1)(k^* + K + 2C_5) \bigg).
\end{align*}
Because $K = \lfloor k_N/4 \rfloor$, we have $K(k_N + K) \leq (5/16) k_N^2$.  Since the other terms are of a smaller order of magnitude for large $N$, it follows that there is a positive constant $c < 1/3$ such that $3k_N/(2sa_N) + (k_N + C_4) + (K-1)(k^* + K + 2C_5) < c k_N^2$ for sufficiently large $N$.  Hence, for sufficiently large $N$,
$$\int_{t^*}^u M(v) \: dv \leq 2c a_N k_N$$ and therefore
\begin{equation}\label{expGeq3}
e^{\int_{t^*}^u s M(v) \: dv + \mu(u - t^*)} \leq e^{2 s a_N k_N/3} = \bigg( \frac{s}{\mu} \bigg)^{2k_N/3} = w(u).
\end{equation}
The result follows from (\ref{expGeq1}), (\ref{expGeq2}), and (\ref{expGeq3}).
\end{proof}

\begin{Lemma}\label{Zprimelem}
For sufficiently large $N$, if $0 \leq j \leq k^*$, then $$P \bigg( \sup_{t \in (t^*, \: \rho_j^*]} |Z_j'(t)| > \frac{\delta}{2} X_j(t^*)\bigg) < \frac{\eps}{64 k_N}.$$
\end{Lemma}

\begin{proof}
The process $(Z_j'(t), t \geq t^*)$ is a mean zero martingale.  By Corollary \ref{ZmartCor2} and (\ref{BD3}), for $t \geq t^*$,
\begin{align}\label{varZj'}
&\Var(Z_j'(t)|{\cal F}_{t^*}) \leq E \bigg[ \int_{t^*}^{t \wedge \rho_j^*} e^{-2 \int_{t^*}^u G_j(v) \: dv} (\mu X_{j-1}(u) + 3 X_j(u)) \: du \bigg| {\cal F}_{t^*} \bigg] \nonumber \\
&\hspace{.2in}= E \bigg[ \int_{t^*}^t e^{-\int_{t^*}^u G_j(v) \: dv} \big(\mu e^{-\int_{t^*}^u G_j(v) \: dv} X_{j-1}(u) + 3 e^{-\int_{t^*}^u G_j(v) \: dv} X_j(u) \big) \1_{\{u \leq \rho_j^*\}} \: du \bigg| {\cal F}_{t^*} \bigg].
\end{align}
For $u < \rho_j^*$, the conclusion of part 1 of Proposition \ref{prop1} holds for $j-1$ through time $u$, and so
\begin{align}\label{GXj-1}
e^{-\int_{t^*}^u G_j(v) \: dv} X_{j-1}(u) &\leq (1 + \delta) X_{j-1}(t^*) e^{\int_{t^*}^u G_{j-1}(v) \: dv} e^{-\int_{t^*}^u G_j(v) \: dv} \nonumber \\
&= (1 + \delta) X_{j-1}(t^*) e^{-s(u - t^*)}.
\end{align}
Plugging this result and the result of Lemma \ref{expGlem} into (\ref{varZj'}), and then bringing the conditional expectation inside the integral, we get for $t > t^*$,
\begin{align*}
&\Var(Z_j'(t)|{\cal F}_{t^*}) \\
&\hspace{.1in} \leq E \bigg[ \int_{t^*}^t e^{-sj(u - t^*)} w(u) \big(\mu (1 + \delta) X_{j-1}(t^*) e^{-s(u - t^*)} + 3 e^{-\int_{t^*}^u G_j(v) \: dv} X_j(u) \big) \1_{\{u \leq \rho_j^* \}} \: du \bigg| {\cal F}_{t^*} \bigg] \\
&\hspace{.1in} \leq \int_{t^*}^t e^{-sj(u - t^*)} w(u) \big(\mu (1 + \delta) X_{j-1}(t^*) e^{-s(u - t^*)} + 3 E \big[ e^{-\int_{t^*}^u G_j(v) \: dv} X_j(u) \1_{\{u \leq \rho_j^* \}} \big| {\cal F}_{t^*} \big] \big) \: du.
\end{align*}
Because $(Z'_j(u), u \geq t^*)$ is a martingale with $Z'_j(t^*) = 0$, we have $E[Z_j'(u)|{\cal F}_{t^*}] = 0$ for $u \geq t^*$.  Using this fact along with (\ref{Zjprimedef}) followed by (\ref{GXj-1}), we get for $u > t^*$,
\begin{align*}
E \big[ e^{-\int_{t^*}^u G_j(v) \: dv} X_j(u) \1_{\{u \leq \rho_j^* \}} \big| {\cal F}_{t^*} \big] &\leq 
E \big[ e^{-\int_{t^*}^{u \wedge \rho_j^*} G_j(v) \: dv} X_j(u \wedge \rho_j) \big| {\cal F}_{t^*} \big] \\
&= E \bigg[ \int_{t^*}^{u \wedge \rho_j^*} \mu X_{j-1}(r) e^{-\int_{t^*}^r G_j(v) \: dv} \: dr \bigg| {\cal F}_{t^*} \bigg] + X_j(t^*) \\
&\leq E \bigg[ \mu (1 + \delta) X_{j-1}(t^*) \int_{t^*}^{u \wedge \rho_j^*} e^{-s(r - t^*)} \: dr \bigg| {\cal F}_{t^*} \bigg] + X_j(t^*) \\
&\leq \frac{\mu(1 + \delta) X_{j-1}(t^*)}{s} + X_j(t^*).
\end{align*}
Thus, for $t > t^*$,
$$\Var(Z_j'(t)|{\cal F}_{t^*}) \leq \int_{t^*}^t e^{-sj(u - t^*)} w(u) \bigg( \mu (1 + \delta) X_{j-1}(t^*) \bigg(1 + \frac{3}{s} \bigg) + 3 X_j(t^*) \bigg) \: du.$$  By Lemma \ref{jj1ratio}, for sufficiently large $N$ we have $\mu(1 + \delta)X_{j-1}(t^*)(1 + 3/s) \leq X_j(t^*)$ on $\{\zeta_0 = \infty\}$.  Therefore, for $t > t^*$, if $N$ is sufficiently large, then on $\{\zeta_0 = \infty\} \in {\cal F}_{t^*}$,
\begin{equation}\label{simplevarZj}
\Var(Z_j'(t)|{\cal F}_{t^*}) \leq 4 X_j(t^*) \int_{t^*}^t e^{-sj(u - t^*)} w(u) \: du.
\end{equation}
When $j = 0$, we take $N$ large enough that $(s/\mu)^{2k_N/3} \geq 21$, and then, using the bound that $e^{-sj(u - t^*)} \leq 1$, equations (\ref{simplevarZj}) and (\ref{wdef}) imply that on $\{\zeta_0 = \infty\}$,
\begin{equation}\label{varj0}
\Var(Z_j'(t)|{\cal F}_{t^*}) \leq 4 X_j(t^*) \bigg( \frac{s}{\mu} \bigg)^{2k_N/3} t.
\end{equation}
When $1 \leq j \leq k^*$, we break the integral in (\ref{simplevarZj}) into two pieces and use (\ref{wdef}) to get that on $\{\zeta_0 = \infty\}$,
\begin{align}\label{varjpos}
\Var(Z_j'(t)|{\cal F}_{t^*}) &\leq 4 X_j(t^*) \bigg( 21 \int_{t^*}^{a_N} e^{-sj(u - t^*)} \: du + \bigg( \frac{s}{\mu} \bigg)^{2k_N/3} \int_{a_N}^{\infty} e^{-sj(u - t^*)} \: du \bigg) \nonumber \\
&\leq 4 X_j(t^*) \bigg( \frac{21}{sj} + \bigg( \frac{s}{\mu} \bigg)^{2k_N/3} \frac{e^{-sj (a_N - t^*)}}{sj} \bigg) \nonumber \\
&= \frac{4 X_j(t^*)}{sj} \bigg( 21 + e^{sjt^*} \bigg( \frac{s}{\mu} \bigg)^{-j + 2k_N/3} \bigg).
\end{align}
Parts 1 and 3 of Proposition \ref{tauprop} imply that if $\gamma_{k^* + K} \leq \zeta_3$, then
\begin{equation}\label{gammakK}
\gamma_{k^* + K} = \tau_{k^* + K} + a_N \leq \frac{2Ka_N}{k_N} + a_N \leq \frac{3a_N}{2}.
\end{equation}
In particular, we must have $\rho_j^* \leq 3a_N/2$.  Combining this observation with the $L^2$ Maximum Inequality, we get that on $\{\zeta_0 = \infty\}$,
\begin{align}\label{mainvarsts}
P \bigg( \sup_{t \in (t^*, \: \rho_j^*]} |Z_j'(t)| > \frac{\delta}{2} X_j(t^*) \bigg| {\cal F}_{t^*} \bigg) &\leq P \bigg( \sup_{t \in [t^*, 3a_N/2]} |Z_j'(t)| > \frac{\delta}{2} X_j(t^*) \bigg| {\cal F}_{t^*} \bigg) \nonumber \\
&\leq \frac{4 \Var(Z_j'(3a_N/2)|{\cal F}_{t^*})}{(\delta X_j(t^*)/2)^2} \nonumber \\
&= \frac{\eps}{64 k_N} \cdot \frac{1024 k_N \Var(Z_j'(3a_N/2)|{\cal F}_{t^*})}{\eps \delta^2 X_j(t^*)^2}.
\end{align}
If we can show that, on $\{\zeta_0 = \infty\}$, the second factor on the right-hand side of (\ref{mainvarsts}) is less than one for sufficiently large $N$, the result will follow by taking expectations of both sides in (\ref{mainvarsts}).  We will assume that $\zeta_0 = \infty$ and show that this factor tends to zero as $N \rightarrow \infty$, uniformly in $j$.

We consider separately the cases $j = 0$ and $1 \leq j \leq k^*$.  Suppose first that $j = 0$.  We have $X_0(t^*) \geq (1 - \delta) N$ by Proposition \ref{earlyprop}, so using (\ref{varj0}),
\begin{equation}\label{j0var2}
\frac{k_N \Var(Z_0'(3a_N/2)|{\cal F}_{t^*})}{X_0(t^*)^2} \leq \frac{6 k_N a_N}{X_0(t^*)} \bigg( \frac{s}{\mu} \bigg)^{2k_N/3} \leq \frac{6 k_N a_N}{(1 - \delta) N} \bigg( \frac{s}{\mu} \bigg)^{2k_N/3}.
\end{equation}
Note that
\begin{align*}
\log \bigg(\frac{k_N a_N}{N} \bigg( \frac{s}{\mu} \bigg)^{2k_N/3} \bigg) &= \log k_N + \log \bigg( \frac{1}{s} \bigg) + \log \log \bigg( \frac{s}{\mu} \bigg) - \log N + \frac{2 k_N}{3} \log \bigg( \frac{s}{\mu} \bigg) \\
&= \log k_N + \log \bigg( \frac{1}{s} \bigg) + \log \log \bigg( \frac{s}{\mu} \bigg) - \frac{1}{3} \log N,
\end{align*}
which tends to $-\infty$ as $N \rightarrow \infty$ because $(\log k_N)/(\log N) \rightarrow 0$ as $N \rightarrow \infty$ and because, by assumption A1, we have $\log(1/s)/\log N \rightarrow 0$ and $(\log \log(s/\mu))/\log N \rightarrow 0$ as $N \rightarrow \infty$.  It follows that the expression in (\ref{j0var2}) tends to zero as $N \rightarrow \infty$.

Next, suppose $1 \leq j \leq k^*$.  Then, using (\ref{varjpos}),
\begin{equation}\label{2tmsvarZj}
\frac{k_N \Var(Z_j'(3a_N/2)|{\cal F}_{t^*})}{X_j(t^*)^2} \leq \frac{4 k_N}{s j X_j(t^*)} \bigg( 21 + e^{sjt^*} \bigg( \frac{s}{\mu} \bigg)^{-j + 2k_N/3} \bigg). 
\end{equation}
We will show that the two terms on the right-hand side of (\ref{2tmsvarZj}) each go to zero as $N \rightarrow \infty$.  For the first term, we use Proposition \ref{earlyprop}, equation (\ref{logfact}), and the fact that $\log(1/s)/k_N \rightarrow 0$ by assumption A1 to get
\begin{align}\label{logjtm1}
\log \bigg( \frac{k_N}{sj X_j(t^*)} \bigg) &\leq \log \bigg( \frac{k_N^3 s^j j!}{\min\{C_1, 1 - \delta\} sj N \mu^j (e^{st^*} - 1)^j} \bigg) \nonumber \\
&= \log \bigg( \frac{k_N^3}{\min\{C_1, 1 - \delta\} sj} \bigg) - \log N + j \log \bigg( \frac{s}{\mu} \bigg) + \log j! - j \log(e^{st^*} - 1) \nonumber \\
&= - \log N + j \log \bigg( \frac{s}{\mu} \bigg) + j \log j - j - s j t^* + o(k_N).
\end{align}
If $j \leq k_N^-$, then $-\log N + j \log(s/\mu) \leq 0$, and $sjt^* \geq 2j \log k_N$.  Therefore, the expression in (\ref{logjtm1}) tends to $-\infty$ as $N \rightarrow \infty$.
If instead $j \in (k_N^-, k_N^+)$, then $t^* = (4/s)\log k_N$, and we can we write $j$ as in (\ref{bjdef}) to get
$$\log \bigg( \frac{k_N}{sj X_j(t^*)} \bigg) \leq -\log N + \log N + b_j k_N \log k_N + j \log j - j - 4 j \log k_N + o(k_N),$$
which tends to $-\infty$ as $N \rightarrow \infty$ because $b_j < 2$ and (\ref{jlogj}) holds.  Thus, the first term on the right-hand side of (\ref{2tmsvarZj}) tends to zero as $N \rightarrow \infty$.  To bound the second term, we use (\ref{logjtm1}) to get 
\begin{align*}
\log \bigg( \frac{k_N}{sj X_j(t^*)} \cdot e^{sjt^*} \bigg( \frac{s}{\mu} \bigg)^{-j + 2k_N/3} \bigg) &= -\log N + \frac{2k_N}{3} \log \bigg( \frac{s}{\mu} \bigg) + j \log j - j + o(k_N). \\
&= - \frac{1}{3} \log N + j \log j - j + o(k_N),
\end{align*}
which tends to $-\infty$ as $N \rightarrow \infty$ because $(k_N \log k_N)/\log N \rightarrow 0$ as $N \rightarrow \infty$.  It follows that the right-hand side of (\ref{2tmsvarZj}) tends to zero as $N \rightarrow \infty$. 
\end{proof}

\begin{Prop}\label{kstar1}
For sufficiently large $N$, if $0 \leq j \leq k^*$, then
$$P \big( (1 - \delta) X_j(t^*) e^{\int_{t^*}^t G_j(v) \: dv} \leq X_j(t) \leq (1 + \delta) X_j(t^*) e^{\int_{t^*}^t G_j(v) \: dv} \mbox{ for all }t \in (t^*, \rho_j^*] \big) > 1 - \frac{\eps}{64k_N}.$$
\end{Prop}

\begin{proof}
By (\ref{Xj3terms}), we have
\begin{equation}\label{mainXjbound}
(1 - \delta) X_j(t^*) e^{\int_{t^*}^t G_j(v) \: dv} \leq X_j(t) \leq (1 + \delta) X_j(t^*) e^{\int_{t^*}^t G_j(v) \: dv}
\end{equation}
as long as $T_{j,2}(t)/T_{j,1}(t) \leq \delta/2$ and $|T_{j,3}(t)|/T_{j,1}(t) = |Z_j'(t)|/X_j(t^*) \leq \delta/2$.  Therefore, the result follows from Lemmas \ref{T2j} and \ref{Zprimelem}.
\end{proof}

\begin{Prop}\label{kstar23}
For sufficiently large $N$, if $0 \leq j \leq k^*$, then
\begin{equation}\label{btwKL}
P \big( X_j(t) > k_N^2 X_j(t^*) e^{\int_{t^*}^t G_j(v) \: dv} \mbox{ for some }t \in (\gamma_{k^* + K}, \rho_j] \big) <\frac{\eps}{48k_N}
\end{equation}
and
\begin{equation}\label{afterL}
P \big( X_j(t) > 0 \mbox{ for some }t \in [\gamma_{k^* + L}, \rho_j] \big) < \frac{\eps}{48k_N}.
\end{equation}
\end{Prop}

\begin{proof}
Fix $j \in \{0, 1, \dots, k^*\}$  Assume for now that (\ref{mainXjbound}) holds for all $t \in (t^*, \rho_j^*]$.  Then, on the event $\{\gamma_{k^* + K} \leq \rho_j\}$, for all $\ell \in \{1, \dots, j\}$ we have
\begin{equation}\label{Xllrat}
\frac{X_{\ell-1}(\gamma_{k^* + K})}{X_{\ell}(\gamma_{k^* + K})} \leq \frac{(1 + \delta) X_{{\ell}-1}(t^*) e^{\int_{t^*}^{\gamma_{k^* + K}} G_{{\ell}-1}(v) \: dv}}{(1 - \delta)X_{\ell}(t^*) e^{\int_{t^*}^{\gamma_{k^* + K}} G_{\ell}(v) \: dv}} = \frac{(1 + \delta) X_{\ell-1}(t^*) e^{-s(\gamma_{k^* + K} - t^*)}}{(1 - \delta) X_{\ell}(t^*)}.
\end{equation}
Note that on $\{\gamma_{k^*+K} \leq \rho_j\}$, the result (\ref{tauspacing}) implies that for sufficiently large $N$,
\begin{equation}\label{gamkt}
\gamma_{k^*+K} - t^* \geq a_N + \tau_{k^*+K} - \tau_{k^*+1} \geq a_N + \frac{a_N(K-1)}{3k_N} \geq \frac{14a_N}{13}.
\end{equation}
Combining (\ref{Xllrat}) and (\ref{gamkt}) with Lemma \ref{jj1ratio}, we get that on $\{\gamma_{k^*+K} \leq \rho_j\}$, for sufficiently large $N$,
\begin{equation}\label{lrat}
\frac{X_{\ell-1}(\gamma_{k^* + K})}{X_{\ell}(\gamma_{k^* + K})} \leq \frac{(1 + \delta) \delta s}{3 (1 - \delta) \mu} e^{-14s a_N/13} = \frac{(1 + \delta) \delta}{3(1 - \delta)} \bigg( \frac{s}{\mu} \bigg)^{-1/13},
\end{equation}
which tends to zero as $N \rightarrow \infty$.  This means that, among individuals with $j$ or fewer mutations, the fraction with $j$ mutations at time $\gamma_{k^* + K}$ must tend to one as $N \rightarrow \infty$.  Recalling that $S_j(t) = X_0(t) + X_1(t) + \dots + X_j(t)$, for sufficiently large $N$ we have
\begin{equation}\label{SjgammaK}
S_j(\gamma_{k^* + K}) \leq \frac{3}{2} X_j(\gamma_{k^* + K})
\end{equation}
on the event $\{\gamma_{k^* + K} \leq \rho_j\}$.  

By Proposition \ref{supprop} and Remark \ref{Gjtilde}, the process $$\big(e^{-\int_{\gamma_{k^* + K}}^{(\gamma_{k^* + K} + t) \wedge \rho_j} G_j(v) \: dv} S_j((\gamma_{k^* + K} + t) \wedge \rho_j), \: t \geq 0 \big)$$ is a nonnegative supermartingale.  Therefore,
\begin{equation}\label{OSTSpre}
P \bigg( \sup_{t \in (\gamma_{k^* + K}, \: \rho_j]} \: e^{-\int_{\gamma_{k^* + K}}^t G_j(v) \: dv} S_j(t) > \frac{k_N^2}{2} S_j(\gamma_{k^* + K}) \bigg| {\cal F}_{\gamma_{k^* + K}} \bigg) \leq \frac{2}{k_N^2}.
\end{equation}
Combining this result with (\ref{SjgammaK}), we get
\begin{equation}\label{OSTS}
P \bigg( S_j(t) > \frac{3k_N^2}{4} X_j(\gamma_{k^* + K}) e^{\int_{\gamma_{k^*+K}}^t G_j(v) \: dv} \mbox{ for some }t \in (\gamma_{k^* + K}, \rho_j] \bigg| {\cal F}_{\gamma_{k^*+K}} \bigg) \leq \frac{2}{k_N^2}.
\end{equation}
Taking expectations of both sides of (\ref{OSTS}), and then using Proposition \ref{kstar1} along with the facts that $X_j(t) \leq S_j(t)$ for all $t \geq 0$ and $\eps/64k_N + 2/k_N^2 < \eps/48k_N$ for sufficiently large $N$, we obtain (\ref{btwKL}).

To get (\ref{afterL}), observe that when the complement of the event in (\ref{OSTSpre}) holds and $\rho_j \geq \gamma_{k^* + L}$, we have
\begin{equation}\label{finSj1}
S_j(\gamma_{k^* + L}) \leq \frac{k_N^2}{2} S_j(\gamma_{k^* + K}) e^{\int_{\gamma_{k^* + K}}^{\gamma_{k^*+L}} G_j(v) \: dv}.
\end{equation}
For $v \in [\gamma_{k^* + K}, \rho_j)$, the result of Proposition \ref{meanprop} implies that for sufficiently large $N$,
\begin{equation}\label{finSj2}
G_j(v) = s(j - M(v)) - \mu \leq s(k^* - (k^* + K - 2C_5)) \leq - \frac{sk_N}{5}.
\end{equation}
Also, the result of Proposition \ref{tauprop} implies that if $\rho_j \geq \gamma_{k^* + L}$ then for sufficienly large $N$,
\begin{equation}\label{finSj3}
\gamma_{k^* + L} - \gamma_{k^* + K} \geq (L - K) \cdot \frac{a_N}{3k_N} \geq \frac{16 a_N}{3}.
\end{equation}
Also, $S_j(\gamma_{k^* + K}) \leq N$, so combining (\ref{finSj1}), (\ref{finSj2}), and (\ref{finSj3}), we get that when the complement of the event in (\ref{OSTSpre}) holds and $\rho_j \geq \gamma_{k^* + L}$, for sufficiently large $N$,
\begin{equation}\label{finSj4}
S_j(\gamma_{k^* + L}) \leq \frac{N k_N^2}{2} e^{-(16/15) s k_N a_N} = \frac{N k_N^2}{2} \bigg( \frac{s}{\mu} \bigg)^{-16 k_N /15}.
\end{equation}
The logarithm of the right-hand side of (\ref{finSj4}) is $$\log N - \frac{16 k_N }{15} \log \bigg( \frac{s}{\mu} \bigg) + 2 \log k_N - \log 2 = - \frac{1}{15} \log N + 2 \log k_N - \log 2$$ which tends to $-\infty$ as $N \rightarrow \infty$.  Thus, the right-hand side of (\ref{finSj4}) tends to zero as $N \rightarrow \infty$ and thus is guaranteed to be less than one if $N$ is sufficiently large.  Since $S_j(\gamma_{k^* + L})$ is an integer, it must be zero.  Furthermore, if $S_j(\gamma_{k^* + L}) = 0$, then $S_j(t) = 0$ for all $t \geq \gamma_{k^* + L}$, which implies that $X_j(t) = 0$ for all $t \geq \gamma_{k^* + L}$.  We can now conclude (\ref{afterL}).
\end{proof}

\begin{Rmk}\label{zeta1jstar}
{\em It follows immediately from Propositions \ref{kstar1} and \ref{kstar23} that if $0 \leq j \leq k^*$, then for sufficiently large $N$, $$\sum_{j=0}^{k^*} P(\{\zeta_0 = \infty\} \cap \{\zeta_{1,j} \leq \rho_j\}) \leq (k^* + 1) \bigg( \frac{\eps}{64 k_N} + \frac{\eps}{48 k_N} + \frac{\eps}{48 k_N} \bigg) < \frac{\eps}{16}.$$}
\end{Rmk}

\subsection{Other type $j$ individuals before time $\tau_{j+1}$}

For the rest of section \ref{zetasec3}, we assume that $j \in \{k^* + 1, \dots, J\}$.  In this subsection, we focus on type $j$ individuals that are not early, meaning they are descended from type $j$ mutations that occurred after the time $\xi_j$ defined in (\ref{xijdef}).  We will show that the claim in part 2 of Proposition \ref{prop2} holds with high probability.  We will begin with three preliminary lemmas.

Define the random set
\begin{equation}\label{Thetadef}
\Theta = \bigg\{j: a_N - \frac{2a_N}{k_N} \leq \tau_j \leq a_N + \frac{2a_N}{k_N} \bigg\}.
\end{equation}
Recall from (\ref{qjdef}) that as long as $q_j > 1$, we have $q_j = j - k_N$ if $j \in \Theta$ and $q_j = j - M(\tau_j)$ if $j \notin \Theta$.  When $j \in \Theta$, it will be difficult to bound $X_j(t)$ as tightly as when $j \notin \Theta$, so we will structure the proof so that we can allow a larger probability of $\zeta_{1,j} \leq \rho_j$ when $j \in \Theta$.  Because the times $\tau_i$ are spaced at least $a_N/3k_N$ apart until time $\zeta_3$ by Proposition \ref{tauprop}, there can be at most 12 values of $j$ for which $\tau_j < \rho_j$ and $j \in \Theta$. 

\begin{Lemma}\label{Gqlem}
There is a positive constant $C_9$ for sufficienty large $N$, the following hold:
\begin{enumerate}
\item If $j \notin \Theta$ and $t \in [\tau_j, \tau_{j+1} \wedge \rho_j)$, then $s(q_j - C_9) \leq G_j(t) \leq s (q_j + C_9)$.

\item If $t \in [\tau_j, \tau_{j+1} \wedge \rho_j)$, then $(1 - 2 \delta) s k_N \leq G_j(t) \leq G_j(t) + \mu \leq (e + 2 \delta) s k_N$.

\item If $\tau_j < \rho_j$, then $(1 - 2 \delta) k_N \leq q_j \leq (e + 2 \delta) k_N.$
\end{enumerate}
\end{Lemma}

\begin{proof}
First suppose $t \in [\tau_j, \tau_{j+1} \wedge \rho_j)$ and $j \notin \Theta$.  In view of part 4 of Proposition \ref{meanprop}, we have $j - M(\tau_j) > 1$ and therefore $q_j = j - M(\tau_j)$.  Therefore, $G_j(t) - sq_j = s(M(\tau_j) - M(t)) - \mu$, which means
\begin{equation}\label{GQabs}
|G_j(t) - sq_j| \leq s \bigg( |M(\tau_j) - {\bar M}(\tau_j)| + |{\bar M}(\tau_j) - {\bar M}(t)| + |{\bar M}(t) - M(t)| + \frac{\mu}{s} \bigg).
\end{equation}
It follows from Proposition \ref{meanprop} that
\begin{equation}\label{MMbargood}
|{\bar M(u)} - M(u)| \leq \max\{3, 2C_5\} \hspace{.1in} \mbox{if } u \notin [a_N, \gamma_{k^*+1}).
\end{equation}
The results of Proposition \ref{tauprop} imply that since $j \notin \Theta$, we have $[\tau_j, \tau_{j+1} \wedge \rho_j) \cap [a_N, \gamma_{k^*+1}) = \emptyset$.
Also, because $t - \tau_j \leq 2a_N/k_N$ by the upper bound in (\ref{tauspacing}), the lower bound in (\ref{tauspacing}) implies that at most six of the times $\tau_i$ can occur between times $\tau_j - a_N$ and $t - a_N$.  It thus follows from (\ref{Mbardef}) that $|{\bar M}(\tau_j) - {\bar M}(t)| \leq 6$.  Since $\mu/s \leq 1$ for sufficiently large $N$ by (\ref{muspower}), combining these observations with (\ref{GQabs}) gives
$$|G_j(t) - sq_j| \leq s(7 + 2 \max\{3, 2C_5\}),$$ which implies part 1 of the lemma.

To prove part 2, we assume $t \in [\tau_j, \tau_{j+1} \wedge \rho_j)$ but no longer assume that $j \notin \Theta$.  It follows from Lemma \ref{Rjlem} that $R(t) = j - {\bar M}(t)$.  Therefore,
\begin{equation}\label{GReq2}
G_j(t) = s R(t) + s({\bar M}(t) - M(t)) - \mu.
\end{equation}
By part 2 of Proposition \ref{tauprop} and Lemma \ref{Qlem}, we have
\begin{equation}\label{Rbounds2}
k_N(1 - \delta) \leq R(t) \leq k_N(e + \delta).
\end{equation}
Also, $\mu/s \rightarrow 0$ as $N \rightarrow \infty$ by (\ref{muspower}).  Therefore, if $t \notin [a_N, \gamma_{k^* + 1})$, then part 2 of the lemma follows from (\ref{MMbargood}), (\ref{GReq2}), and (\ref{Rbounds2}).  Now suppose instead that $t \in [a_N, \gamma_{k^* + 1})$.  Proposition \ref{meanprop} implies that $k^* - k_N - C_4 \leq {\bar M}(t) - M(t) \leq k^*$.  Therefore, using (\ref{GReq2}) and part 2 of Proposition \ref{tauprop} again, we have
\begin{equation}\label{Gjbds}
s k_N (q(t/a_N) - \delta) + s(k^* - k_N - C_4) - \mu \leq G_j(t) \leq s k_N (q(t/a_N) + \delta) + s k^* - \mu.
\end{equation}
Since $\gamma_{k^*+1}/a_N \rightarrow 1$ as $N \rightarrow \infty$ by part 1 of Proposition \ref{tauprop}, and since $q$ is a right continuous function with $q(1) = e-1$ by Lemma \ref{Qlem}, we have $e - 1 - \delta/2 \leq q(t/a_N) \leq e - 1 + \delta/2$ for sufficiently large $N$.  Part 2 of the lemma follows because $k^*/k_N \rightarrow 1$ as $N \rightarrow \infty$.  

Finally, we prove part 3.  When $j \notin \Theta$, we have $sq_j = G_j(\tau_j) + \mu$, so part 3 follows immediately from part 2.  Suppose instead $\tau_j < \rho_j$ and $j \in \Theta$, which means $a_N - 2a_N/k_N \leq \tau_j \leq a_N + 2a_N/k_N$.  Since $q(1) = e$, it follows from part 2 of Proposition \ref{tauprop} that $k_N(e - 2 \delta) \leq R(a_N - 2a_N/k_N) \leq k_N(e + \delta)$ if $N$ is sufficiently large.  Therefore, in view of (\ref{tauspacing}), we have $k_N(e - 2 \delta) \leq j \leq k_N(e + \delta) + 12$ and thus $k_N(e - 1 - 2 \delta) \leq q_j \leq k_N(e - 1 + \delta) + 12$ if $N$ is sufficiently large.  Therefore, part 3 of the lemma holds in this case as well.
\end{proof}

Define
\begin{equation}\label{xijminus}
\xi_j^- = \tau_j + \frac{1}{sq_j} \log \bigg( \frac{1}{sq_j} \bigg) - \frac{b}{sq_j}.
\end{equation}

\begin{Lemma}\label{newxilem}
For sufficiently large $N$, if $\tau_j < \rho_j$, then $\tau_j < \xi_j^- < \xi_j < \tau_j^*$ and $\tau_j^* - \xi_j \geq a_N/8Tk_N$.
\end{Lemma}

\begin{proof}
Because $sk_N \rightarrow 0$ by assumption A3, for sufficiently large $N$ we have $\log(1/(3sk_N)) > b$.  Whenever $\tau_j < \rho_j$, Lemma \ref{Gqlem} implies that $q_j < (e + 2 \delta) k_N$.  Therefore, for sufficiently large $N$, we have
$\tau_j < \xi_j^- < \xi_j$.
Also, because $q_j \geq (1 - 2 \delta) k_N$ for sufficiently large $N$ if $\tau_j < \rho_j$, and because (\ref{muspower}) implies that $\log(s/\mu)/\log(1/sk_N) \geq \log(s/\mu)/\log(1/s) \rightarrow \infty$ as $N \rightarrow \infty$, for sufficiently large $N$ we have
\begin{equation}\label{tauxi2}
\xi_j = \tau_j + \frac{1}{sq_j} \log \bigg( \frac{1}{sq_j} \bigg) + \frac{b}{sq_j} \leq \tau_j + \frac{a_N}{8Tk_N}.
\end{equation}
Therefore, $\tau_j^* - \xi_j \geq a_N/(8Tk_N)$ if $\tau_j < \rho_j$.
\end{proof}

\begin{Lemma}\label{imm}
For sufficiently large $N$, if $\tau_j \leq t \leq \gamma_{j-1+K}$ and $t < \rho_j$, then $$\frac{(1 - 3 \delta)s}{\mu} e^{\int_{\tau_j}^t G_{j-1}(v) \: dv} \leq X_{j-1}(t) \leq \frac{(1 + 3 \delta) s}{\mu} e^{\int_{\tau_j}^t G_{j-1}(v) \: dv}.$$
\end{Lemma}

\begin{proof}
If $j \geq k^* + 2$, the result is immediate from (\ref{prop23}).  If instead $j = k^* + 1$, then by (\ref{prop11}) when $t = \tau_{k^*+1}$ and the fact that $s/\mu \leq X_{k^*}(\tau_{k^*+1}) \leq 1 + s/\mu$, we get
\begin{equation}\label{starimm}
\frac{s}{\mu(1 + \delta)} \leq X_{k^*}(t^*) e^{\int_{t^*}^{\tau_{k^*+1}} G_{k^*}(v) \: dv} \leq \frac{1 + s/\mu}{1 - \delta}.
\end{equation}
Because $1 - 3 \delta \leq (1-\delta)/(1+\delta) \leq (1+\delta)(1+s/\mu)/[(1 - \delta)(s/\mu)] \leq 1 + 3 \delta$ for sufficiently large $N$, another application of (\ref{prop11}) gives the result.
\end{proof}

Recall that $X_{j,2}(t)$ denotes the number of type $j$ individuals at time $t$ descended from an individual that acquired a type $j$ mutation after time $\xi_j$.  Then, using the notation of Corollary \ref{ZmartCor4}, for $t \in [\xi_j, \tau_{j+1} \wedge \rho_j]$, we have
$$Z_j^{[\xi_j, \tau_{j+1}]}(t) = e^{-\int_{\xi_j}^t G_j(v) \: dv} X_{j,2}(t) - \int_{\xi_j}^{t \wedge \tau_{j+1}} \mu X_{j-1}(u) e^{-\int_{\xi_j}^u G_j(v) \: dv} \: du.$$  Let $${\bar \rho}_j = \tau_{j+1} \wedge \rho_j,$$ and for $t \geq \xi_j$, let $$Z_j'(t) = Z_j^{[\xi_j, \tau_{j+1}]}(t \wedge {\bar \rho}_j),$$ with the convention that $Z_j'(t) = 0$ if ${\bar \rho}_j \leq \xi_j$.  
Then for $t \geq \xi_j$, we have
\begin{equation}\label{Xj22tms}
X_{j,2}(t \wedge {\bar \rho}_j) = \int_{\xi_j}^{t \wedge {\bar \rho}_j} \mu X_{j-1}(u) e^{\int_u^{t \wedge {\bar \rho}_j} G_j(v) dv} \: du + e^{\int_{\xi_j}^{t \wedge {\bar \rho}_j} G_j(v) \: dv} Z_j'(t).
\end{equation}
We will separately consider the two terms on the right-hand side of (\ref{Xj22tms}). 
Lemma \ref{immtermbound} below gives the required bounds on the first term.

\begin{Lemma}\label{immtermbound}
For sufficiently large $N$, we have
$$\int_{\xi_j}^t \mu X_{j-1}(u) e^{\int_u^t G_j(v) dv} \: du \leq (1 + 3 \delta) e^{\int_{\tau_j}^t G_j(v) \: dv}$$
for all $t \in [\xi_j, {\bar \rho}_j]$ and $$\int_{\xi_j}^t \mu X_{j-1}(u) e^{\int_u^t G_j(v) dv} \: du \geq \bigg(1 - \frac{7 \delta}{2} \bigg) e^{\int_{\tau_j}^t G_j(v) \: dv}$$ for all $t \in [\tau_j^*, {\bar \rho}_j]$.
\end{Lemma}

\begin{proof}
Suppose $t \in [\xi_j, {\bar \rho}_j]$.  By Lemma \ref{imm}, for sufficiently large $N$,
\begin{align}\label{xijXjupper}
\int_{\xi_j}^t \mu X_{j-1}(u) e^{\int_u^t G_j(v) dv} \: du &\leq (1 + 3 \delta) s \int_{\xi_j}^t e^{\int_{\tau_j}^u G_{j-1}(v) \: dv} e^{\int_u^t G_j(v) dv} \: du \nonumber \\
&= (1 + 3 \delta) s e^{\int_{\tau_j}^t G_j(v) \: dv} \int_{\xi_j}^t e^{-s(u - \tau_j)} \: du \nonumber \\
&= (1 + 3 \delta) e^{\int_{\tau_j}^t G_j(v) \: dv} \big( e^{-s(\xi_j - \tau_j)} - e^{-s(t - \tau_j)} \big).
\end{align}
Because $\xi_j \geq \tau_j$, we have $e^{-s(\xi_j - \tau_j)} - e^{-s(t - \tau_j)} \leq 1$, which gives the upper bound in the lemma.

Now suppose $t \in [\tau_j^*, {\bar \rho}_j]$.  The same argument that yields (\ref{xijXjupper}) implies that for sufficiently large $N$,
\begin{equation}\label{prelower}
\int_{\xi_j}^t \mu X_{j-1}(u) e^{\int_u^t G_j(v) dv} \: du \geq (1 - 3 \delta) e^{\int_{\tau_j}^t G_j(v) \: dv} \big( e^{-s(\xi_j - \tau_j)} - e^{-s(t - \tau_j)} \big).
\end{equation}
Now if $\tau_j < \rho_j$, then
\begin{equation}\label{sxitau}
s(\xi_j  - \tau_j) = \frac{1}{q_j} \log \bigg( \frac{1}{sq_j} \bigg) + \frac{b}{q_j}.
\end{equation}
For sufficiently large $N$, part 3 of Lemma \ref{Gqlem} gives $q_j \geq (1 - 2 \delta) k_N$ when $\tau_j < \rho_j$, which by assumption A1 implies that $s(\xi_j  - \tau_j) \rightarrow 0$ and therefore $e^{-s(\xi_j - \tau_j)} \rightarrow 1$ uniformly in $j$ as $N \rightarrow \infty$.  Furthermore, if $t \geq \tau_j^*$, then $e^{-s(t - \tau_j)} \leq e^{-sa_N/4Tk_N} \rightarrow 0$ as $N \rightarrow \infty$ by (\ref{A2prime}).  Consequently, the lower bound in the lemma follows from (\ref{prelower}).
\end{proof}

It remains to show that the second term on the right-hand side of (\ref{Xj22tms}) is small.  We know from Corollary \ref{ZmartCor4} that the process $(Z'(\xi_j + t), t \geq 0)$ is a mean zero martingale, so the problem is to control the fluctuations of this process.  The next result gives the key second moment estimate.

\begin{Lemma}\label{varZprime}
For sufficiently large $N$, we have, for all $t \geq 0$, $$\Var(Z_j'(\xi_j + t)|{\cal F}_{\xi_j}) \leq 5 e^{\int_{\tau_j}^{\xi_j} G_j(v) \: dv} \cdot \frac{1}{s k_N^2}.$$  
\end{Lemma}

\begin{proof}
By Corollary \ref{ZmartCor4}, we have
\begin{align*}
&\Var(Z_j'(\xi_j + t)|{\cal F}_{\xi_j}) \\
&= E \bigg[ \int_{\xi_j}^{(\xi_j + t) \wedge {\bar \rho}_j} e^{-2 \int_{\xi_j}^u G_j(v) \: dv}( \mu X_{j-1}(u) + B_j^{[\xi_j, \tau_{j+1}]}(u) X_{j,2}(u) + D_j^{[\xi_j, \tau_{j+1}]}(u) X_{j,2}(u)) \: du \bigg| {\cal F}_{\xi_j} \bigg].
\end{align*}
We now can use the reasoning leading to (\ref{BD3}) to get
\begin{equation}\label{var1}
\Var(Z_j'(\xi_j + t)|{\cal F}_{\xi_j}) \leq E \bigg[ \int_{\xi_j}^{\xi_j + t} e^{-2 \int_{\xi_j}^u G_j(v) \: dv}(\mu X_{j-1}(u) + 3 X_{j,2}(u)) \1_{\{u < {\bar \rho}_j\}} \: du \bigg| {\cal F}_{\xi_j} \bigg].
\end{equation}
Using Lemma \ref{imm}, we get that if $u < {\bar \rho}_j$, then
\begin{align}\label{var2}
e^{-\int_{\xi_j}^u G_j(v) \: dv} \mu X_{j-1}(u) &\leq (1 + 3 \delta) s e^{-\int_{\xi_j}^u G_j(v) \: dv} e^{\int_{\tau_j}^u G_{j-1}(v) \: dv} \nonumber \\
&= (1 + 3 \delta) s e^{\int_{\tau_j}^{\xi_j} G_j(v) \: dv} e^{-s(u - \tau_j)}.
\end{align}
Also, from (\ref{Xj22tms}) and (\ref{var2}), if $u < {\bar \rho}_j$, then
\begin{align}\label{var3}
e^{-\int_{\xi_j}^u G_j(v) \: dv} X_{j,2}(u) &= \int_{\xi_j}^u \mu X_{j-1}(w) e^{-\int_{\xi_j}^w G_j(v) \: dv} \: dw + Z'_j(u) \nonumber \\
&\leq (1 + 3 \delta) s e^{\int_{\tau_j}^{\xi_j} G_j(v) \: dv} \int_{\xi_j}^u e^{-s(w - \tau_j)} \: dw + Z_j'(u) \nonumber \\
&= (1 + 3 \delta) e^{\int_{\tau_j}^{\xi_j} G_j(v) \: dv} (e^{-s(\xi_j - \tau_j)} - e^{-s(u - \tau_j)}) + Z_j'(u).
\end{align}
Combining (\ref{var1}), (\ref{var2}), and (\ref{var3}), and using that $3(1 + 3 \delta) < 4$ by (\ref{deltadef}), we get
\begin{align}\label{var4}
&\Var(Z_j'(\xi_j + t)|{\cal F}_{\xi_j}) \leq 4 E \bigg[ \int_{\xi_j}^{\xi_j + t} e^{-\int_{\xi_j}^u G_j(v) \: dv} \big(e^{\int_{\tau_j}^{\xi_j} G_j(v) \: dv}(s e^{-s(u - \tau_j)} \nonumber \\
&\hspace{2in}+ e^{-s(\xi_j - \tau_j)} - e^{-s(u - \tau_j)}) + Z_j'(u)\big) \1_{\{u < {\bar \rho}_j\}} \: du \bigg| {\cal F}_{\xi_j} \bigg].
\end{align}
By part 2 of Lemma \ref{Gqlem}, we have $G_j(v) \geq (1 - 2 \delta)sk_N$ for all $v \in [\tau_j, {\bar \rho}_j)$, which means that for $u \geq \xi_j$, we have
\begin{equation}\label{var5}
e^{-\int_{\xi_j}^u G_j(v) \: dv} \1_{\{u < {\bar \rho}_j\}} \leq e^{-(1 - 2 \delta) sk_N(u - \xi_j)}.
\end{equation}
Also, although $Z'_j(u)$ can be negative, it can be seen from (\ref{var1}) that the integrand in (\ref{var4}) must be nonnegative so, in particular,
\begin{equation}\label{Zpos}
e^{\int_{\tau_j}^{\xi_j} G_j(v) \: dv} (s e^{-s(u - \tau_j)} + e^{-s(\xi_j - \tau_j)} - e^{-s(u - \tau_j)}) + Z_j'(u) \geq 0
\end{equation}
for $u \in [\xi_j, {\bar \rho}_j)$.  Because $s < 1$ for sufficiently large $N$, we see that $s e^{-s(u - \tau_j)} - e^{-s(u - \tau_j)}$ is an increasing function of $u$.  Also, $Z_j'(u) = Z_j'({\bar \rho}_j)$ for all $u \geq \rho_j$.  Therefore, (\ref{Zpos}) holds for all $u \geq \xi_j$.  Thus, combining (\ref{var4}) and (\ref{var5}) gives
\begin{align}\label{var6}
&\Var(Z_j'(\xi_j + t)|{\cal F}_{\xi_j}) \leq 4 E \bigg[ \int_{\xi_j}^{\xi_j + t} e^{-(1 - 2 \delta) sk_N(u - \xi_j)} \big(e^{\int_{\tau_j}^{\xi_j} G_j(v) \: dv}(s e^{-s(u - \tau_j)} \nonumber \\
&\hspace{2.5in}+ e^{-s(\xi_j - \tau_j)} - e^{-s(u - \tau_j)}) + Z_j'(u)\big) \: du \bigg| {\cal F}_{\xi_j} \bigg].
\end{align}
Every expression in the integrand in (\ref{var6}) is ${\cal F}_{\xi_j}$-measurable except $Z'_j(u)$.  Since $(Z'(\xi_j + t), t \geq 0)$ is a mean zero martingale by Corollary \ref{ZmartCor4}, we can apply Fubini's Theorem and then evaluate the conditional expectation in (\ref{var6}) to get
\begin{align*}
&\Var(Z_j'(\xi_j + t)|{\cal F}_{\xi_j}) \\
&\hspace{.5in}\leq 4 e^{\int_{\tau_j}^{\xi_j} G_j(v) \: dv} \int_{\xi_j}^{\xi_j + t} e^{-(1 - 2 \delta) sk_N(u - \xi_j)} (s e^{-s(u - \tau_j)} + e^{-s(\xi_j - \tau_j)} - e^{-s(u - \tau_j)}) \: du.
\end{align*}
Now for all $u \geq \xi_j$,
\begin{align*}
s e^{-s(u - \tau_j)} + e^{-s(\xi_j - \tau_j)} - e^{-s(u - \tau_j)} &\leq s + e^{-s(\xi_j - \tau_j)}(1 - e^{-s(u - \xi_j)}) \\
&\leq s + e^{-s(\xi_j - \tau_j)} \cdot s(u - \xi_j) \\
&\leq s(1 + u - \xi_j),
\end{align*}
so for sufficiently large $N$,
\begin{align*}
\Var(Z_j'(\xi_j + t)|{\cal F}_{\xi_j}) &\leq 4 e^{\int_{\tau_j}^{\xi_j} G_j(v) \: dv} \int_{\xi_j}^{\xi_j + t} e^{-(1 - 2 \delta)sk_N(u - \xi_j)} s (1 + u - \xi_j) \: du \\
&\leq 4 e^{\int_{\tau_j}^{\xi_j} G_j(v) \: dv} \int_0^{\infty} e^{-(1 - 2 \delta)sk_N y} s(1 + y) \: dy \\
&= 4 e^{\int_{\tau_j}^{\xi_j} G_j(v) \: dv} \bigg( \frac{s}{(1 - 2 \delta) sk_N} + \frac{s}{((1 - 2 \delta) s k_N)^2} \bigg) \\
&\leq 5 e^{\int_{\tau_j}^{\xi_j} G_j(v) \: dv} \cdot \frac{1}{s k_N^2},
\end{align*}
as claimed.
\end{proof}

\begin{Lemma}\label{Gqxi}
For sufficiently large $N$, if $\xi_j < \rho_j$, then $$\int_{\tau_j}^{\xi_j} G_j(v) \: dv \geq s q_j (\xi_j - \tau_j) - \delta.$$
\end{Lemma}

\begin{proof}
Suppose $\xi_j < \rho_j$.  Consider first the case in which $j \notin \Theta$.  Then for sufficiently large $N$, we have $G_j(v) \geq s (q_j - C_9)$ for $v \in [\tau_j, \xi_j]$ by part 1 of Lemma \ref{Gqlem}.  Therefore, $$\int_{\tau_j}^{\xi_j} G_j(v) \: dv \geq s(q_j - C_9)(\xi_j - \tau_j) = sq_j(\xi_j - \tau_j) - C_9 s (\xi_j - \tau_j),$$ and the result follows because $s(\xi_j - \tau_j) \rightarrow 0$ as $N \rightarrow \infty$ by the argument following (\ref{sxitau}).

Next, suppose $j \in \Theta$, which means $q_j = j - k_N$.  Using (\ref{tauxi2}), we get $$a_N - \frac{2 a_N}{k_N} \leq \tau_j \leq \xi_j \leq a_N + \bigg(2 + \frac{1}{8T} \bigg)\frac{a_N}{k_N}.$$  By Proposition \ref{meanprop}, if $t < a_N \wedge \rho_j$ then $M(t) \leq 3$.  If $a_N \leq t < \gamma_{k^* + 1} \wedge \rho_j$, then $M(t) < k_N + C_4$.  In view of (\ref{tauspacing}), if $\gamma_{k^* + 1} \leq t \leq (a_N + (2 + 1/8T)a_N/k_N) \wedge \rho_j$, then $t \leq \gamma_{k^* + 8}$ and therefore $M(t) \leq k^* + 7 + 2C_5$.  Combining the results for these three cases, there is a positive constant $C$ such that if $t \in [\tau_j, \xi_j]$, then $M(t) \leq k_N + C$.  It follows that $$\int_{\tau_j}^{\xi_j} G_j(v) \: dv \geq \int_{\tau_j}^{\xi_j} (s(j - k_N - C) - \mu) \: dv = (s(q_j - C) - \mu)(\xi_j - \tau_j),$$
and the result follows because $(Cs + \mu)(\xi_j - \tau_j) \rightarrow 0$ as $N \rightarrow \infty$.
\end{proof}

\begin{Lemma}\label{varZprimelem}
For sufficiently large $N$,
$$P \bigg( e^{\int_{\xi_j}^t G_j(v) \: dv} |Z_j'(t)| \leq \frac{\delta}{2} e^{\int_{\tau_j}^t G_j(v) \: dv} \mbox{ for all }t \in [\xi_j, {\bar \rho}_j] \bigg) \geq 1 - \frac{\eps}{25J}.$$
\end{Lemma}

\begin{proof}
By the $L^2$ Maximum Inequality and Lemma \ref{varZprime},
\begin{align}\label{Doobmax}
P \bigg( \sup_{t \geq 0} |Z_j'(\xi_j + t)| > \frac{\delta}{2} e^{\int_{\tau_j}^{\xi_j} G_j(v) \: dv} \bigg| {\cal F}_{\xi_j} \bigg) &\leq \frac{16}{\delta^2} e^{-2 \int_{\tau_j}^{\xi_j} G_j(v) \: dv} \cdot \sup_{t \geq 0} \Var(Z_j'(\xi_j + t)|{\cal F}_{\xi_j} \big) \nonumber \\
&\leq \frac{80}{\delta^2 s k_N^2} e^{-\int_{\tau_j}^{\xi_j} G_j(v) \: dv}.
\end{align}
By Lemma \ref{Gqxi} and part 3 of Lemma \ref{Gqlem}, if $\xi_j < \rho_j$ then
$$e^{-\int_{\tau_j}^{\xi_j} G_j(v) \: dv} \leq e^{\delta} e^{-sq_j(\xi_j - \tau_j)} = e^{\delta} e^{-b}sq_j \leq 3 e^{-b} s k_N.$$
Plugging this result into (\ref{Doobmax}), then taking expectations and using (\ref{bdef}) and the fact that $J \leq 4T k_N$ for sufficiently large $N$, we get that for sufficiently large $N$,
$$P \bigg( \sup_{t \geq 0} |Z_j'(\xi_j + t)| > \frac{\delta}{2} e^{\int_{\tau_j}^{\xi_j} G_j(v) \: dv} \bigg) \leq \frac{240e^{-b}}{\delta^2 k_N} \leq \frac{960 e^{-b} T}{\delta^2 J} = \frac{\eps}{25 J}.$$  The lemma follows.
\end{proof}

Combining (\ref{Xj22tms}) with Lemmas \ref{immtermbound} and \ref{varZprimelem} and then summing over $j$ immediately yields the following corollary, which shows that the result of part 2 of Proposition \ref{prop2} holds with high probability.

\begin{Cor}\label{pt2zeta}
For sufficiently large $N$,
\begin{align*}
&\sum_{j=k^*+1}^J P \bigg( \big\{ X_{j,2}(t) < (1 - 4 \delta) e^{\int_{\tau_j}^t G_j(v) \: dv} \mbox{ for some }t \in [\tau_j^*, {\bar \rho}_j] \big\} \\
&\hspace{1.2in}\cup \big\{X_{j,2}(t) > (1 + 4 \delta) e^{\int_{\tau_j}^t G_j(v) \: dv} \mbox{ for some }t \in [\xi_j, {\bar \rho}_j] \big\} \bigg) \leq \frac{\eps}{25}.
\end{align*}
\end{Cor}

\subsection{Early type $j$ individuals before time $\tau_{j+1}$}

In this subsection, we continue to assume $j \in \{k^* + 1, \dots, J\}$.  We consider early type $j$ individuals, which are descended from type $j$ mutations that occur at or before the time $\xi_j$.  We will show that the claims of part 1 of Proposition \ref{prop2} hold with high probability.  Note that (\ref{prop21}) involves a constant $C_3$, which we will define to be
\begin{equation}\label{C3def}
C_3 = \frac{204 b T}{\eps}.
\end{equation}
 We will assume throughout this section that $N$ is large enough that the conclusions of Lemma \ref{newxilem} hold.

From part 1 of Proposition \ref{prop2}, we know that if $j \geq k^* + 2$, then no early type $j-1$ individual acquires a $j$th mutation until time $\tau_j \wedge \rho_j \wedge  a_N T$.  In particular, no type $j$ individual can appear until time $\xi_{j-1} \wedge \rho_j$.  This result is also true when $j = k^* + 1$ if we define $\xi_{k^*} = t^*$ because, according to Proposition \ref{earlyprop}, on $\{\zeta_0 = \infty\}$, no individuals of type $k^* + 1$ appear until after time $t^*$.  Therefore, using the notation from Corollary \ref{ZmartCor4} in which $X_j^{[u,v]}(t)$ denotes the number of type $j$ individuals at time $t$ descended from individuals that acquired a $j$th mutation during the time interval $(u, v]$, as long as $\xi_{j-1} < \rho_j$, we have
\begin{equation}\label{split}
X_{j,1}(t) = X_j^{[\xi_{j-1}, \tau_j]}(t) + X_j^{[\tau_j, \xi_j^-]}(t) + X_j^{[\xi_j^-, \xi_j]}(t).
\end{equation}
We will consider these three processes separately.

\begin{Lemma}\label{bplem}
Let $(Z(t), t \geq 0)$ be a continuous-time birth and death process in which each individual independently dies at rate $\nu > 0$ and gives birth to a new individual at rate $\lambda > \nu$.  Assume that $Z(0) = 1$.  Then
\begin{equation}\label{survivetot}
P(Z(t) > 0) = \frac{\lambda - \nu}{\lambda - \nu e^{-(\lambda - \nu) t}}.
\end{equation}
Also, if $n \in \N$, then
\begin{equation}\label{reachn}
P \big( \sup_{t \geq 0} Z(t) \geq n \big) = \frac{1 - \nu/\lambda}{1 - (\nu/\lambda)^n}.
\end{equation}
\end{Lemma}

\begin{proof}
It is well-known (see section 5 of Chapter III in \cite{athney}) that the generating function for this process is $$F(s,t) = \sum_{k=0}^{\infty} P(Z(t) = k) s^k = \frac{\nu(s - 1) - e^{-(\lambda - \nu)t}(\lambda s - \nu)}{\lambda(s - 1) - e^{-(\lambda - \nu)t} (\lambda s - \nu)}.$$  Because $P(Z(t) > 0) = 1 - F(0, t)$, the result (\ref{survivetot}) follows after some algebra.

Also, at any given time, the probability that the next event is a birth is $\lambda/(\lambda + \nu)$, while the probability that the next event is a death is $\nu/(\lambda + \nu)$.  Therefore, (\ref{reachn}) follows from well-known results for asymmetric random walks (see, for example, section 3 of chapter 3 in \cite{karlin}).
\end{proof}

\begin{Lemma}\label{bpsurvival}
Suppose $\kappa$ is an $({\cal F}_t)_{t \geq 0}$ stopping time such that $\xi_{j-1} \leq \kappa \leq \xi_j$ and, with positive probability, a type $j$ mutation occurs at time $\kappa$.  For sufficiently large $N$, the following hold:
\begin{enumerate}
\item  Given that a type $j$ mutation occurs at time $\kappa$, the probability that the number of type $j$ descendants of this mutation exceeds $(s/\mu)^{1-\delta}$ before time $\rho_j$ is at most $3sk_N$.

\item Given that a type $j$ mutation occurs at time $\kappa$, the probability that $\kappa + a_N/8Tk_N < \rho_j$ and at least one type $j$ individual descended from this mutation is alive at time $\kappa + a_N/8Tk_N$ is at most $3sk_N$.
\end{enumerate}
\end{Lemma}

\begin{proof}
Suppose a type $j$ mutation occurs at time $\kappa$.  By the reasoning leading to (\ref{Bjeq}), each type $j$ descendant of the individual that gets this mutation gives birth at rate less than or equal to $1 + s(j - M(t))$.  Since $s(j - M(t)) = G_j(t) + \mu$, it follows from Lemma \ref{Gqlem} that until time $\rho_j$, the birth rate is at most $\lambda = 1 + (e + 2 \delta) s k_N$.  As long as the number of type $j$ individuals descended from this mutation is less than $(s/\mu)^{1 - \delta}$, the reasoning leading to (\ref{Djeq}) implies that the rate at which each such individual either acquires a mutation or dies and gets replaced by an individual that is not a type $j$ individual descended from this mutation is at least $\mu + 1 - (s/\mu)^{1 - \delta}(1 + s(j - M(t)))/N$.  Using Lemma \ref{Gqlem} and (\ref{muNpower}), we see that for sufficiently large $N$, this quantity is at least $\nu = \mu + 1 - \delta s k_N$ until time $\rho_j$.  Therefore, until time $\rho_j$ occurs or the number of type $j$ individuals descended from this mutation reaches $(s/\mu)^{1 - \delta}$, the number of such individuals is dominated by a continuous-time branching process in which each individual gives birth at rate $\lambda$ and dies at rate $\nu$.

By Lemma \ref{bplem}, the probability that the number of type $j$ individuals descended from this mutation exceeds $(s/\mu)^{1 - \delta}$ before time $\rho_j$ is at most
\begin{equation}\label{bpexp1}
\frac{1 - \nu/\lambda}{1 - (\nu/\lambda)^{(s/\mu)^{1-\delta}}}.
\end{equation}
Likewise, the probability that $\kappa + a_N/8Tk_N < \rho_j$ and at least one type $j$ individual descended from this mutation is alive at time $\kappa + a_N/8Tk_N$ is less than or equal to
\begin{equation}\label{bpexp2}
\frac{\lambda - \nu}{\lambda - \nu e^{-(\lambda - \nu) (a_N/8T k_N)}}.
\end{equation}
We must show that the expressions in (\ref{bpexp1}) and (\ref{bpexp2}) are bounded above by $3sk_N$ for sufficiently large $N$.
We have $$1 - \nu/\lambda \leq \lambda - \nu \leq (e + 3 \delta) sk_N.$$  Because $e + 3 \delta < 3$ by (\ref{deltadef}), it remains only to show that the denominators of the expressions in (\ref{bpexp1}) and (\ref{bpexp2}) tend to one as $N \rightarrow \infty$.
If $N$ is large enough that $\nu < 1$, then we have $(\nu/\lambda)^{(s/\mu)^{1-\delta}} \leq (1 + (e + 2 \delta) sk_N)^{-(s/\mu)^{1-\delta}}$, which tends to zero as $N \rightarrow \infty$ because $(s k_N)(s/\mu)^{1-\delta} \rightarrow \infty$ as $N \rightarrow \infty$ by (\ref{muspower}).  Likewise, $\nu e^{-(\lambda - \nu)a_N/8Tk_N} \rightarrow 0$ as $N \rightarrow \infty$ because $(\lambda - \nu)a_N/8Tk_N \geq (e + 3 \delta) s a_N/8T \rightarrow \infty$ as $N \rightarrow \infty$.  The result follows.
\end{proof}

Lemmas \ref{Xj3lem}, \ref{nomutlem}, and \ref{expearlyj} below give us the bounds that we will need to establish that the result of part 1 of Proposition \ref{prop2} holds with high probability.  We will use the notation $o(k_N^{-1})$ for a collection of probabilities $p_{j,N}$ such that $$\lim_{N \rightarrow \infty} k_N \sup_{j \in \{k^* + 1, \dots, J\}} p_{j,N} = 0.$$  Lemma \ref{Xj3lem} shows that it is highly unlikely that any type $j$ mutations appearing before time $\tau_j$ will have descendants alive in the population after time $\tau_j^*$.  As a result, it will be possible essentially to ignore such mutations.

\begin{Lemma}\label{Xj3lem}
We have 
\begin{equation}\label{vearly1}
P \bigg( X_j^{[\xi_{j-1}, \tau_j]}(t) > \bigg(\frac{s}{\mu}\bigg)^{1-\delta} \mbox{ for some }t \in [\xi_{j-1}, \tau_j^* \wedge \rho_j] \bigg) = o(k_N^{-1})
\end{equation}
and
\begin{equation}\label{vearly2}
P \big( X_j^{[\xi_{j-1}, \tau_j]}(t) > 0 \mbox{ for some }t \in [\tau_j^*, \rho_j] \big) = o(k_N^{-1}).
\end{equation}
\end{Lemma}

\begin{proof}
Write ${\tilde \rho}_j = \tau_j \wedge \rho_j$.
Suppose first that $j \geq k^* + 2$.  Because ${\tilde \rho}_j \leq \zeta_{1, j-1}$, the result of part 2 of Proposition \ref{prop2} holds for type $j-1$ individuals up to time ${\tilde \rho}_j$, which means
$$\int_{\xi_{j-1}}^{{\tilde \rho}_j} \mu X_{j-1,2}(t) \: dt \leq \mu (1 + 4 \delta) \int_{\xi_{j-1}}^{{\tilde \rho}_j} e^{\int_{\tau_{j-1}}^{t} G_{j-1}(v) \: dv} \: dt.$$
Also, since ${\tilde \rho}_j \leq \zeta_{1,j-1}$, Lemma \ref{smulem} implies that
$$e^{\int_{\tau_{j-1}}^{{\tilde \rho}_j} G_{j-1}(v) \: dv} \leq \frac{2s}{\mu}$$
for sufficiently large $N$, which leads to
\begin{equation}\label{muXj-1bound}
\int_{\xi_{j-1}}^{{\tilde \rho}_j} \mu X_{j-1,2}(t) \: dt \leq 2s(1 + 4 \delta) \int_{\xi_{j-1}}^{{\tilde \rho}_j} e^{-\int_{t}^{{\tilde \rho}_j} G_{j-1}(v) \: dv} \: dt.
\end{equation}
Now suppose instead that $j = k^* + 1$, and recall that $\xi_{k^*} = t^*$ by definition.  Then because ${\tilde \rho}_j \leq \zeta_{1,j-1}$, the result of part 1 of Proposition \ref{prop1} gives $$\int_{\xi_{j-1}}^{{\tilde \rho}_j} \mu X_{j-1,2}(t) \: dt \leq \mu (1 + \delta) \int_{\xi_{j-1}}^{{\tilde \rho}_j} X_{j-1}(t^*) e^{\int_{t^*}^t G_{j-1}(v) \: dv} \: dt.$$  Reasoning as in the proof of Lemma \ref{smulem} but using (\ref{prop11}), we get that for sufficiently large $N$,
$$X_{j-1}(t^*) e^{\int_{t^*}^{{\tilde \rho}_j} G_{j-1}(v) \: dv} \leq \frac{2s}{\mu},$$ so (\ref{muXj-1bound}) holds in this case as well.  Therefore, combining (\ref{muXj-1bound}) with part 2 of Lemma \ref{Gqlem} and writing $C_{10} = 2(1 + 4 \delta)/(1 - 2 \delta)$, we get
$$\int_{\xi_{j-1}}^{{\tilde \rho}_j} \mu X_{j-1,2}(t) \: dt \leq 2s(1 + 4 \delta) \int_{\xi_{j-1}}^{{\tilde \rho}_j} e^{-(1 - 2 \delta) sk_N({\tilde \rho}_j - t)} \: dt \leq \frac{C_{10}}{k_N}.$$

Because ${\tilde \rho}_j \leq \zeta_{1, j-1}$, the last statement of part 1 of Proposition \ref{prop2} implies that no early type $j-1$ individual acquires a $j$th mutation before time ${\tilde \rho}_j$.  Because each type $j-1$ individual acquires mutations at rate $\mu$, the number of times that type $j-1$ individuals that are not early acquire a $j$th mutation between the times $\xi_{j-1}$ and $$\inf\bigg\{u: \int_{\xi_{j-1}}^u \mu X_{j-1,2}(t) \: dt \geq \frac{C_{10}}{k_N} \bigg\}$$ is Poisson with mean $C_{10}/k_N$.  In particular, the probability that at least one such mutation occurs during this time period is at most $C_{10}/k_N$, and the probability that two or more such mutations occur during this time period is at most $C_{10}^2/k_N^2$.  If such a mutation occurs before time ${\tilde \rho}_j$, then by Lemma \ref{bpsurvival}, the probability that the number of type $j$ descendants of this mutation exceeds $(s/\mu)^{1-\delta}$ before time $\rho_j$ is at most $3sk_N$.  Likewise, the probability that some type $j$ descendant of this individual is still alive at time $\tau_j^* \wedge \rho_j$ is at most $3sk_N$.  Thus, the probabilities of the events in (\ref{vearly1}) and (\ref{vearly2}) are both bounded above by
\begin{align*}
\frac{C_{10}^2}{k_N^2} + \frac{C_{10}}{k_N} \cdot 3sk_N.
\end{align*}
This expression is $o(k_N^{-1})$ because $sk_N \rightarrow 0$ as $N \rightarrow \infty$ by assumption A3.
\end{proof}

Lemma \ref{nomutlem} bounds the probability that, when $j \notin \Theta$, we have an early type $j$ mutation with descendants alive after time $\tau_j^*$.  This bound is given in (\ref{nomut2}) below.  A sharper bound is given in (\ref{nomut1}) for the probability that such a mutation occurs before time $\xi_j^-$.

\begin{Lemma}\label{nomutlem}
For sufficiently large $N$, we have 
\begin{equation}\label{nomut1}
P \big( \big\{ X_j^{[\tau_j, \xi_j^-]}(t) > 0 \mbox{ for some }t \in [ \tau_j^*, \rho_j ] \big\} \cap \{j \notin \Theta\} \big) < \frac{\eps}{16J}.
\end{equation}
and
\begin{equation}\label{nomut2}
P \big( \big\{ X_j^{[\tau_j, \xi_j]}(t) > 0 \mbox{ for some }t \in [ \tau_j^*, \rho_j ] \big\} \cap \{j \notin \Theta\} \big) \leq \frac{13 e^b}{k_N}.
\end{equation}
\end{Lemma}

\begin{proof}
By Lemma \ref{imm} and part 1 of Lemma \ref{Gqlem}, on the event $\{j \notin \Theta\}$, we have
\begin{align}\label{nomuteq1}
\int_{\tau_j}^{\xi_j^- \wedge \rho_j} \mu X_{j-1}(t) \: dt &\leq (1 + 3 \delta) s \int_{\tau_j}^{\xi_j^- \wedge \rho_j} e^{\int_{\tau_j}^t G_{j-1}(v) \: dv} \: dt \nonumber \\
&\leq (1 + 3 \delta) s \int_{\tau_j}^{\xi_j^- \wedge \rho_j} e^{s(q_j + C_9)(t - \tau_j)} \: dt \nonumber \\
&\leq (1 + 3 \delta) s \cdot \frac{e^{s(q_j + C_9)(\xi_j^- - \tau_j)}}{s(q_j + C_9)}.
\end{align}
By part 3 of Lemma \ref{Gqlem}, we have $(1 + 3 \delta)/(q_j + C_9) \leq 2/k_N$ for sufficiently large $N$.  Also, recalling (\ref{xijminus}) and observing that $\log(1/sq_j)/q_j \rightarrow 0$ as $N \rightarrow \infty$ on $\{\tau_j < \rho_j\}$ by assumption A1 and part 3 of Lemma \ref{Gqlem}, we get that for sufficiently large $N$, on $\{\tau_j < \rho_j\}$,
\begin{equation}\label{nomuteq2}
e^{s(q_j + C_9)(\xi_j^- - \tau_j)} = \frac{e^{-b}}{sq_j} \exp \bigg( \frac{C_9}{q_j} \log \bigg( \frac{1}{sq_j} \bigg) - \frac{C_9 b}{q_j} \bigg) \leq \frac{2 e^{-b}}{sq_j}.
\end{equation}
Therefore, on the event $\{j \notin \Theta\}$, we have
\begin{equation}\label{nummut1}
\int_{\tau_j}^{\xi_j^- \wedge \rho_j} \mu X_{j-1}(t) \: dt \leq \frac{4 e^{-b}}{s k_N q_j}.
\end{equation}
Likewise, if we replace $\xi_j^-$ by $\xi_j$ in (\ref{nomuteq1}), (\ref{nomuteq2}), and (\ref{nummut1}), we get that on the event $\{j \notin \Theta\}$,
\begin{equation}\label{nummut2}
\int_{\tau_j}^{\xi_j \wedge \rho_j} \mu X_{j-1}(t) \: dt \leq \frac{4e^b}{sk_N q_j}.
\end{equation}

Let $\Gamma_1$ be the number of type $j$ mutations between times $\tau_j$ and $\xi_j^- \wedge \rho_j$, and let $\Gamma_2$ be the number of type $j$ mutations between times $\tau_j$ and $\xi_j \wedge \rho_j$.  Because each type $j-1$ individual acquires mutations at rate $\mu$, equations (\ref{nummut1}) and (\ref{nummut2}) imply that $E[\Gamma_1 \1_{\{j \notin \Theta\}}|{\cal F}_{\tau_j}] \leq 4e^{-b}/(sk_N q_j)$ and $E[\Gamma_2 \1_{\{j \notin \Theta\}}|{\cal F}_{\tau_j}] \leq 4e^b/(sk_N q_j)$.   Let $A_i$ be the event that $\tau_j^* \leq \rho_j$ and the individual that gets the $i$th type $j$ mutation between times $\tau_j$ and $\xi_j$ has type $j$ descendants alive at time $\tau_j^*$.  By Lemma \ref{newxilem}, this individual must have type $j$ descendants alive for at least a time $a_N/8T k_N$ after the time of the mutation.  Therefore, by Lemma \ref{bpsurvival}, we have $P(A_i|\Gamma \geq i) \leq 3sk_N$.  Using part 3 of Lemma \ref{Gqlem}, equation (\ref{deltadef}), and the fact that $J/k_N \leq 4T$ for sufficiently large $N$, we get
$$P \bigg( \{j \notin \Theta\} \cup \bigcup_{i=1}^{\Gamma_1} A_i \bigg| {\cal F}_{\tau_j} \bigg) \leq 3sk_N E[\Gamma_1 \1_{\{j \notin \Theta\}}|{\cal F}_{\tau_j}] \leq \frac{12 e^{-b}}{q_j} \leq \frac{13 e^{-b}}{k_N} \leq \frac{\eps}{16J} \cdot \frac{832 T e^{-b}}{\eps}.$$  Equation (\ref{nomut1}) follows because $e^{-b} < \eps/832T$ by (\ref{bdef}).  Likewise,
$$P \bigg( \{j \notin \Theta\} \cup \bigcup_{i=1}^{\Gamma_2} A_i \bigg| {\cal F}_{\tau_j} \bigg) \leq 3sk_N E[\Gamma_2 \1_{\{j \notin \Theta\}}|{\cal F}_{\tau_j}] \leq \frac{12 e^b}{q_j} \leq \frac{13 e^b}{k_N},$$ which implies (\ref{nomut2}).
\end{proof}

\begin{Lemma}\label{expearlyj}
For sufficiently large $N$, on the event $\{\rho_j > \tau_j\}$, we have
\begin{equation}\label{taujxij}
P \big( X^{[\tau_j, \xi_j]}_j(t) > C_3 e^{\int_{\tau_j}^t G_j(v) \: dv} \mbox{ for some }t \in [\tau_j, \tau_{j+1} \wedge \rho_j ] \big| {\cal F}_{\tau_j} \big) \leq \frac{\eps}{97}
\end{equation}
and
\begin{equation}\label{xiximinus}
P \big( X^{[\xi_j^-, \xi_j]}_j(t) > C_3 e^{\int_{\tau_j}^t G_j(v) \: dv} \mbox{ for some }t \in [\tau_j, \tau_{j+1} \wedge \rho_j ] \big| {\cal F}_{\tau_j} \big) \leq \frac{\eps}{17J}.
\end{equation}
Also, 
\begin{equation}\label{s4mubound}
P \bigg( X^{[\tau_j, \xi_j]}_j(t) > \bigg(\frac{s}{\mu}\bigg)^{1-\delta} \mbox{ for some }t \in [\tau_j, \tau_j^* \wedge \rho_j] \bigg) = o(k_N^{-1}).
\end{equation}
Furthermore, (\ref{taujxij}) holds even if $j$ is random, as long as $\tau_j$ is a stopping time.
\end{Lemma}

\begin{proof}
Let ${\bar \rho}_j = \tau_{j+1} \wedge \rho_j$.  Using the notation of Corollary \ref{ZmartCor4}, if $t \geq \tau_j$, then
\begin{equation}\label{Doob}
e^{-\int_{\tau_j}^{t \wedge {\bar \rho}_j} G_j(v) \: dv} X^{[\tau_j, \xi_j]}_j(t \wedge {\bar \rho}_j) = \int_{\tau_j}^{t \wedge \xi_j \wedge {\bar \rho}_j} \mu X_{j-1}(u) e^{-\int_{\tau_j}^u G_j(v) \: dv} \: du + Z_j^{[\tau_j, \xi_j]}(t \wedge {\bar \rho}_j).
\end{equation}
By Corollary \ref{ZmartCor4}, the process $(Z_j^{[\tau_j, \xi_j]}(\tau_j + t), \: t \geq 0)$ is a martingale.  Therefore, if we define $$Y(t) = e^{-\int_{\tau_j}^{t \wedge {\bar \rho}_j} G_j(v) \: dv} X^{[\tau_j, \xi_j]}_j(t \wedge {\bar \rho}_j)$$ for all $t \geq \tau_j$, then the process $(Y(t), t \geq \tau_j)$, having been expressed in (\ref{Doob}) as the sum of an increasing process and a martingale, is a submartingale.  By Doob's Maximal Inequality,
\begin{equation}\label{YC3bound}
P(Y(t) \geq C_3 \mbox{ for some }t \in [\tau_j, {\bar \rho}_j] | {\cal F}_{\tau_j}) \leq \frac{1}{C_3} E[Y({\bar \rho}_j)|{\cal F}_{\tau_j}].\end{equation}
By (\ref{Doob}) and Lemma \ref{imm},
\begin{align}\label{Yrhobar}
E[Y({\bar \rho}_j)|{\cal F}_{\tau_j}] &= E \bigg[ \int_{\tau_j}^{{\bar \rho}_j \wedge \xi_j} \mu X_{j-1}(u) e^{-\int_{\tau_j}^u G_j(v) \: dv} \: du \bigg| {\cal F}_{\tau_j} \bigg] \nonumber \\
&\leq (1 + 3 \delta) s E \bigg[ \int_{\tau_j}^{{\bar \rho}_j \wedge \xi_j} e^{-s(u - \tau_j)} \: du \bigg| {\cal F}_{\tau_j} \bigg] \nonumber \\
&\leq 1 + 3 \delta.
\end{align}
Now (\ref{taujxij}) follows immediately from (\ref{YC3bound}) and (\ref{Yrhobar}), as long as $C_3 \geq 97(1 + 3 \delta)/\eps$, which is true by (\ref{C3def}).  Note that Remark \ref{smprem} implies that (\ref{taujxij}) holds when $j$ is random, provided that $\tau_j$ is a stopping time.

To obtain (\ref{s4mubound}), note that if $t \leq \tau_j^* \wedge \rho_j$, then by part 2 of Lemma \ref{Gqlem}, we have $$e^{-\int_{\tau_j}^t G_j(v) \: dv} \geq e^{-(a_N/4Tk_N)((e + 2 \delta) sk_N)} = \bigg( \frac{s}{\mu} \bigg)^{-(e + 2 \delta)/4T}.$$  Therefore, 
\begin{align*}
&P \bigg( X^{[\tau_j, \xi_j]}_j(t) > \bigg(\frac{s}{\mu}\bigg)^{1-\delta} \mbox{ for some }t \in [\tau_j, \tau_j^* \wedge \rho_j] \bigg) \\
&\hspace{0.7in}= P\bigg(Y(t) > e^{-\int_{\tau_j}^t G_j(v) \: dv} \cdot \bigg(\frac{s}{\mu}\bigg)^{1-\delta} \mbox{ for some }t \in [\tau_j, \tau_j^* \wedge \rho_j] \bigg) \\
&\hspace{0.7in}\leq P\bigg( Y(t) > \bigg( \frac{s}{\mu} \bigg)^{1 - \delta - (e + 2 \delta)/4T} \mbox{ for some }t \in [\tau_j, \tau_j^* \wedge \rho_j] \bigg).
\end{align*}
Write $\theta = 1 - \delta - (e + 2 \delta)/4T$, which is positive by (\ref{deltadef}).
Arguing as in the derivations of (\ref{YC3bound}) and (\ref{Yrhobar}) but using $(s/\mu)^{\theta}$ in place of $C_3$ and $\tau_j^* \wedge \rho_j$ in place of ${\bar \rho}_j$, we get
$$P \bigg( X^{[\tau_j, \xi_j]}_j(t) > \bigg(\frac{s}{\mu}\bigg)^{1-\delta} \mbox{ for some }t \in t \in [\tau_j, \tau_j^* \wedge \rho_j] \bigg) \leq (1 + 3 \delta) \bigg(\frac{s}{\mu} \bigg)^{-\theta}.$$  The result (\ref{s4mubound}) follows because $(s/\mu)^{-\theta} k_N \rightarrow 0$ as $N \rightarrow \infty$, as can be seen by taking the logarithm and using (\ref{A2prime}).

The argument for (\ref{xiximinus}) is similar to the argument for (\ref{taujxij}).  Again using Corollary \ref{ZmartCor4}, we have
\begin{equation}\label{Doob2}
e^{-\int_{\xi_j^-}^{t \wedge {\bar \rho}_j} G_j(v) \: dv} X^{[\xi_j^-, \xi_j]}_j(t \wedge {\bar \rho}_j) = \int_{\xi_j^-}^{t \wedge \xi_j \wedge {\bar \rho}_j} \mu X_{j-1}(u) e^{-\int_{\xi_j^-}^u G_j(v) \: dv} \: du + Z_j^{[\xi_j^-, \xi_j]}(t \wedge {\bar \rho}_j),
\end{equation}
where $(Z_j^{[\xi_j^-, \xi_j]}(\xi_j^- + t), \: t \geq 0)$ is a martingale.  For $t \geq \xi_j^-$, let $$W(t) = e^{-\int_{\xi_j^-}^{t \wedge {\bar \rho}_j} G_j(v) \: dv} X^{[\xi_j^-, \xi_j]}_j(t \wedge {\bar \rho}_j).$$  By (\ref{Doob2}), the process $(W(\xi_j^- + t), t \geq 0)$ is a submartingale.  By Doob's Maximal Inequality,
\begin{equation}\label{W3bound}
P(e^{-\int_{\tau_j}^{\xi_j^-} G_j(v) \: dv} W(t) > C_3 \mbox{ for some }t \in [\xi_j^-, {\bar \rho}_j] | {\cal F}_{\xi_j^-}) \leq \frac{e^{-\int_{\tau_j}^{\xi_j^-} G_j(v) \: dv}}{C_3} E[W({\bar \rho}_j)|{\cal F}_{\xi_j^-}].
\end{equation}
By (\ref{Doob2}) and Lemma \ref{imm},
\begin{align*}
E[W({\bar \rho}_j)|{\cal F}_{\xi_j^-}] &= E \bigg[ \int_{\xi_j^-}^{{\bar \rho}_j \wedge \xi_j} \mu X_{j-1}(u) e^{-\int_{\xi_j^-}^u G_j(v) \: dv} \: du \bigg| {\cal F}_{\xi_j^-} \bigg] \nonumber \\
&\leq (1 + 3 \delta) s E \bigg[ \int_{\xi_j^-}^{{\bar \rho}_j \wedge \xi_j} e^{\int_{\tau_j}^u G_{j-1}(v) \: dv} e^{-\int_{\xi_j^-}^u G_j(v) \: dv} \: du \bigg| {\cal F}_{\xi_j^-} \bigg] \\
&= (1 + 3 \delta) s e^{\int_{\tau_j}^{\xi_j^-} G_j(v) \: dv} E \bigg[ \int_{\xi_j^-}^{{\bar \rho}_j \wedge \xi_j} e^{-s(u - \tau_j)} \: du \bigg| {\cal F}_{\xi_j^-} \bigg] \nonumber \\
&\leq (1 + 3 \delta) s (\xi_j - \xi_j^-) e^{\int_{\tau_j}^{\xi_j^-} G_j(v) \: dv}. \nonumber \\
&= \frac{2(1 + 3 \delta) b}{q_j} \: e^{\int_{\tau_j}^{\xi_j^-} G_j(v) \: dv}.
\end{align*}
Since $q_j \geq (1 - 2 \delta)/k_N$ on $\{\tau_j < \rho_j\}$ for sufficiently large $N$ by part 3 of Lemma \ref{Gqlem}, it follows that for sufficiently large $N$, we have, on $\{\tau_j < \rho_j\}$, $$E[W({\bar \rho}_j)|{\cal F}_{\xi_j^-}] \leq \frac{3b}{k_N} \cdot e^{\int_{\tau_j}^{\xi_j^-} G_j(v) \: dv}.$$  Therefore, recalling (\ref{C3def}) and noting that $J \leq 4T k_N$ for sufficiently large $N$, we get
$$P(e^{-\int_{\tau_j}^{\xi_j^-} G_j(v) \: dv} W(t) > C_3 \mbox{ for some }t \in [\tau_j, {\bar \rho}_j] | {\cal F}_{\xi_j^-}) \leq \frac{3b}{C_3 k_N} \leq \frac{12bT}{C_3 J} = \frac{\eps}{17J}.$$
Taking conditional expectations of both sides with respect to ${\cal F}_{\tau_j}$ yields (\ref{xiximinus}).
\end{proof}

We now combine Lemmas \ref{Xj3lem}, \ref{nomutlem}, and \ref{expearlyj} to establish that the result of part 1 of Proposition \ref{prop2} holds with high probability.  In view of the fact that $2(s/\mu)^{1-\delta} \leq s/2\mu$ for sufficiently large $N$, Proposition \ref{pt1zeta} establishes the first two statements of this result.  Proposition \ref{8155} establishes the last statement.

\begin{Prop}\label{pt1zeta}
For sufficiently large $N$,
\begin{align*}
& \sum_{j=k^*+1}^{J} P \bigg( \big\{ X_{j,1}(t) > C_3 e^{\int_{\tau_j}^t G_j(v) \: dv} \mbox{ for some }t \in [\tau_j^*, \tau_{j+1} \wedge \rho_j] \big\} \\
&\hspace{1.2in}\cup \bigg\{ X_{j,1}(t) > 2 \bigg( \frac{s}{\mu} \bigg)^{1-\delta} \mbox{ for some }t \leq \tau_j^* \wedge \rho_j \bigg\} \bigg) \leq \frac{\eps}{4}.
\end{align*}
\end{Prop}

\begin{proof}
Recall from the discussion before (\ref{split}) that if a type $j$ individual appears before time $\xi_{j-1}$, then $\rho_j$ occurs at that time, so we only need to consider type $j$ mutations after time $\xi_{j-1}$. 
Combining (\ref{vearly1}) and (\ref{s4mubound}), we see that for sufficiently large $N$, 
\begin{equation}\label{pt1zeta1}
P(X_{j,1}(t) > 2(s/\mu)^{1-\delta} \mbox{ for some }t \leq \tau_j^* \wedge \rho_j) = o(k_N^{-1}).
\end{equation}
By (\ref{split}), (\ref{vearly2}), (\ref{nomut1}), and (\ref{xiximinus}),
\begin{equation}\label{pt1zeta3}
\sum_{j = k^* + 1}^J P\big( \big\{X_{j,1}(t) > C_3 e^{\int_{\tau_j}^t G_j(v) \: dv} \mbox{ for some }t \in [\tau_j^*, \tau_{j+1} \wedge \rho_j] \big\} \cap \{j \notin \Theta\} \big) \leq \frac{\eps}{8}
\end{equation}
for sufficiently large $N$.  Because we observed that there can be at most 12 values of $j$ for which $\tau_j < \rho_j$ and $j \in \Theta$, it follows from (\ref{vearly2}) and (\ref{taujxij}) that for sufficiently large $N$,
\begin{equation}\label{pt1zeta4}
\sum_{j = k^* + 1}^J P\big( \big\{X_{j,1}(t) > C_3 e^{\int_{\tau_j}^t G_j(v) \: dv} \mbox{ for some }t \in [\tau_j^*, \tau_{j+1} \wedge \rho_j] \big\} \cap \{j \in \Theta\} \big) \leq \frac{\eps}{8}.
\end{equation}
Note that the values of $j$ that are in $\Theta$ are random, so we are using the statement in Lemma \ref{expearlyj} that (\ref{taujxij}) holds when $j$ is random, as long as $\tau_j$ is a stopping time.
\end{proof}

\begin{Prop}\label{8155}
Let ${\bar \rho}_j = \tau_{j+1} \wedge \rho_j$.  Let $A_j$ be the event that an early type $j$ individual acquires a $(j+1)$st mutation at or before time ${\bar \rho}_j$.  Let
\begin{align*}
E_1 &= \big\{X_{j,1}(t) \leq 2(s/\mu)^{1-\delta} \mbox{ for all }t \leq \tau_j^* \wedge \rho_j \mbox{ and }j \in \{k^*+1, \dots, J\} \big\} \\
E_2 &= \big\{X_{j,1}(t) \leq C_3 e^{\int_{\tau_j}^t G_j(v) \: dv} \mbox{ for all }t \in [\tau_j^*, {\bar \rho}_j] \mbox{ and }j \in \{k^*+1, \dots, J\} \big\} \\
E_3 &= \big\{(1 - 4 \delta) e^{\int_{\tau_j}^t G_j(v) \: dv} \leq X_{j,2}(t) \leq (1 + 4 \delta) e^{\int_{\tau_j}^t G_j(v) \: dv} \\
&\hspace{2.2in} \mbox{ for all }t \in [\tau_j^*, {\bar \rho}_j] \mbox{ and }j \in \{k^*+1, \dots, J\} \big\}.
\end{align*}
Then, for sufficiently large $N$,
$$P \bigg( \bigg( \bigcup_{j=k^*+1}^J A_j \bigg) \cap E_1 \cap E_2 \cap E_3 \bigg) \leq \frac{\eps}{48}.$$
\end{Prop}

\begin{proof}
We first bound the probability that an early type $j$ individual gets a $(j+1)$st mutation between times $\xi_{j-1}$ and $\tau_j^*$.  When $E_1$ occurs, we have, using (\ref{tauspacing}),
\begin{align}\label{sumint1}
\sum_{j = k^* + 1}^J  \int_{\xi_{j-1}}^{\tau_j^* \wedge \rho_j} \mu X_{j,1}(t) \: dt &\leq 2 \mu \bigg( \frac{s}{\mu} \bigg)^{1-\delta} \sum_{j=k^*+1}^J (\tau_j^* \wedge \rho_j - \xi_{j-1}) \nonumber \\
&\leq 2 \mu (J - k^*) \bigg( \frac{s}{\mu} \bigg)^{1-\delta} \cdot \frac{a_N}{k_N} \bigg(2 + \frac{1}{4T} \bigg) \nonumber \\
&= \bigg(4 + \frac{1}{2T} \bigg) \frac{(J - k^*)}{k_N} \bigg( \frac{\mu}{s} \bigg)^{\delta} \log \bigg( \frac{s}{\mu} \bigg) \rightarrow 0
\end{align}
as $N \rightarrow \infty$.  Because each type $j$ individual acquires mutations at rate $\mu$, the expression on the right-hand side of (\ref{sumint1}) bounds the probability that $E_1$ occurs and, for some $j \in \{k^*+1, \dots, J\}$, an early type $j$ individual gets another mutation between times $\xi_{j-1}$ and $\tau_j^*$.

Consider next the possibility that such a mutation occurs between times $\tau_j^*$ and $\tau_{j+1} \wedge \rho_j$.
In view of (\ref{nomut2}) and the fact that there are at most 12 values of $j$ for which $\tau_j < \rho_j$ and $j \in \Theta$, the probability that there are fewer than $k_N^{1/2}$ values of $j$ for which $X_{j,1}(t) > 0$ for some $t \in [\tau_j^*, \rho_j]$ tends to one as $N \rightarrow \infty$.  Suppose there are indeed fewer than $k_N^{1/2}$ such values of $j$, and suppose $E_2$ and $E_3$ occur.  Then, for sufficiently large $N$,
$$e^{\int_{\tau_j}^{{\bar \rho}_j} G_j(v) \: dv} \leq \frac{X_{j,2}({\bar \rho}_j)}{1 - 4 \delta} \leq \frac{1 + s/\mu}{1 - 4 \delta} \leq \frac{2s}{\mu}.$$  Therefore, using part 2 of Lemma \ref{Gqlem},
\begin{align}\label{sumint2}
\sum_{j=k^* + 1}^J \int_{\tau_j^*}^{{\bar \rho}_j} \mu X_{j,1}(t) \: dt &\leq C_3 k_N^{1/2} \mu  \int_{\tau_{j^*}}^{{\bar \rho}_j} e^{\int_{\tau_j}^t G_j(v) \: dv} \: dt \nonumber \\
&= C_3 k_N^{1/2} \mu e^{\int_{\tau_j}^{{\bar \rho}_j} G_j(v) \: dv} \int_{\tau_j^*}^{{\bar \rho}_j} e^{-\int_t^{{\bar \rho}_j} G_j(v) \: dv} \: dt \nonumber \\
&\leq C_3 k_N^{1/2} \mu \cdot \frac{2s}{\mu} \int_{\tau_j^*}^{{\bar \rho}_j} e^{-(1 - 2 \delta) sk_N({\bar \rho}_j - t)} \: dt \nonumber \\
&\leq \frac{2 C_3}{(1 - 2 \delta) k_N^{1/2}} \rightarrow 0
\end{align}
as $N \rightarrow \infty$.  The expression on the right-hand side of (\ref{sumint2}) bounds the probability that for some $j \in \{k^*+1, \dots, J\}$, an early type $j$ individual gets another mutation between times $\tau_j^*$ and ${\bar \rho}_j$.
Equations (\ref{sumint1}) and (\ref{sumint2}) thus imply that the probability that $E_1$, $E_2$, and $E_3$ occur but $A_j$ also occurs for some $j \in \{k^* + 1, \dots, J\}$ tends to zero as $N \rightarrow \infty$.  The result follows.
\end{proof}

\subsection{Type $j$ individuals between times $\tau_{j+1}$ and $\gamma_{j+K}$}

In this subsection, we show that the number of type $j$ individuals behaves quite predictably between times $\tau_{j+1}$ and $\gamma_{j+K}$.  In particular, we show that the result of part 3 of Proposition \ref{prop2} holds with high probability.  The key to the argument will be showing that the fluctuations in $X_j(t)$ are small.  We assume throughout the subsection that $j \in \{k^*+1, \dots, J\}$.  Let $$\rho_j' = \rho_j \wedge \gamma_{j+K}.$$  We apply Corollary \ref{ZmartCor2} with $\tau_{j+1}$ in place of $\kappa$ and $\rho_j'$ in place of $\tau$ to get that for $t \geq \tau_{j+1}$,
\begin{equation}\label{mainXj}
X_j(t \wedge \rho_j') = e^{\int_{\tau_{j+1}}^{t \wedge \rho_j'} G_j(v) \: dv} \bigg( X_j(\tau_{j+1}) + \int_{\tau_{j+1}}^{t \wedge \rho_j'} \mu X_{j-1}(u) e^{-\int_{\tau_{j+1}}^u G_j(v) \: dv} \: du + Z_j^{\tau_{j+1}, \rho_j'}(t) \bigg).
\end{equation}
To lighten notation, we will set $$Z_j''(t) = Z_j^{\tau_{j+1}, \rho_j'}(t),$$ and then the process $(Z_j''(\tau_{j+1} + t), t \geq 0)$
is a mean zero martingale.    By definition, we have $s/\mu \leq X_j(\tau_{j+1}) \leq 1 + s/\mu$, so the first term in (\ref{mainXj}) is very close to the expression in (\ref{prop23}).  Therefore, to show that (\ref{prop23}) holds with high probability, we need to show that the second and third terms in (\ref{mainXj}) are small relative to the first term with high probability.  We begin with a result similar to Lemma \ref{imm} that holds between times $\gamma_{j-1+K}$ and $\gamma_{j+K}$.

\begin{Lemma}\label{lateimm}
For sufficiently large $N$, if $\gamma_{j-1+K} \leq t < \rho_j'$, then
$$X_{j-1}(t) \leq \frac{(1 + 2 \delta) k_N^2 s}{\mu} e^{\int_{\tau_j}^t G_{j-1}(v) \: dv}.$$
\end{Lemma}

\begin{proof}
If $j \geq k^* + 2$, the result is immediate from (\ref{prop24}).  Suppose instead $j = k^* + 1$.  Then (\ref{starimm}) holds.  Because $(1 + s/\mu)/[(1 - \delta)(s/\mu)] \leq 1 + 2 \delta$ for sufficiently large $N$, an application of (\ref{prop12}) then gives the result.
\end{proof}

The next lemma controls the second term in (\ref{mainXj}).  We will consider the event
\begin{equation}\label{Fjdef}
F_j = \big\{X_{j,2}(\tau_{j+1}) \geq (1 - 4 \delta) e^{\int_{\tau_j}^{\tau_{j+1}} G_j(v) \: dv}\big\}.
\end{equation}
By Corollary \ref{pt2zeta}, with probability at least $1 - \eps/25$, for all $j \in \{k^* + 1, \dots, J\}$ either $F_j$ occurs or $\tau_{j+1} > \rho_j$.

\begin{Lemma}\label{immterm}
For sufficiently large $N$, if $t \in [\tau_{j+1}, \rho_j']$ and $F_j$ occurs, then
$$\int_{\tau_{j+1}}^t \mu X_{j-1}(u) e^{-\int_{\tau_{j+1}}^u G_j(v) \: dv} \: du \leq \frac{\delta s}{3 \mu}.$$
\end{Lemma}

\begin{proof}
First suppose $\tau_{j+1} \leq u < \rho_j' \wedge \gamma_{j-1+K}$.  Then by Lemma \ref{imm},
\begin{align*}
\mu X_{j-1}(u) e^{-\int_{\tau_{j+1}}^u G_j(v) \: dv} &\leq (1 + 3 \delta) s e^{\int_{\tau_j}^u G_{j-1}(v) \: dv} e^{-\int_{\tau_{j+1}}^u G_j(v) \: dv} \\
&= (1 + 3 \delta) s e^{\int_{\tau_j}^{\tau_{j+1}} G_j(v) \: dv} e^{-s(u - \tau_j)}.
\end{align*}
On the event $F_j$, we have
\begin{equation}\label{expGjint}
e^{\int_{\tau_j}^{\tau_{j+1}} G_j(v) \: dv} \leq \frac{1 + s/\mu}{1 - 4 \delta},
\end{equation}
so for sufficiently large $N$, on $F_j$,
\begin{equation}\label{mainimmterm}
\mu X_{j-1}(u) e^{-\int_{\tau_{j+1}}^u G_j(v) \: dv} \leq \frac{2s^2}{\mu} e^{-s(u - \tau_j)}.
\end{equation}
Next, suppose $\gamma_{j-1+K} \leq u < \rho_j'$.  Then by Lemma \ref{lateimm} and (\ref{expGjint}), for sufficiently large $N$, on $F_j$ we have
\begin{align}\label{lateimmterm}
\mu X_{j-1}(u) e^{-\int_{\tau_{j+1}}^u G_j(v) \: dv} &\leq (1 + 2 \delta) k_N^2 s e^{\int_{\tau_j}^u G_{j-1}(v) \: dv} e^{-\int_{\tau_{j+1}}^u G_j(v) \: dv} \nonumber \\
&= (1 + 2 \delta) k_N^2 s e^{\int_{\tau_j}^{\tau_{j+1}} G_j(v) \: dv} e^{-s(u - \tau_j)} \nonumber \\
&\leq \frac{2 k_N^2 s^2}{\mu} e^{-s(u - \tau_j)}.
\end{align}
By (\ref{mainimmterm}) and (\ref{lateimmterm}), if $t \in [\tau_{j+1}, \rho_j')$ then on $F_j$,
\begin{align}\label{mainimmbound}
\int_{\tau_{j+1}}^t \mu X_{j-1}(u) e^{-\int_{\tau_{j+1}}^u G_j(v) \: dv} \: du &\leq \frac{2s^2}{\mu} \int_{\tau_{j+1}}^{t \wedge \gamma_{j-1+K}} e^{-s(u - \tau_j)} \: du + \frac{2k_N^2 s^2}{\mu} \int_{\gamma_{j-1+K}}^t e^{-s(u - \tau_j)} \: du. \nonumber \\
&\leq \frac{2s}{\mu} \big( e^{-s(\tau_{j+1} - \tau_j)} + k_N^2 e^{-s(\gamma_{j-1+K} - \tau_j)} \big).
\end{align}
If $t \in [\tau_{j+1}, \rho_j')$, then $\tau_{j+1} < \rho_j$, which means $\tau_{j+1} - \tau_j \geq a_N/3k_N$ by (\ref{tauspacing}) and $\gamma_{j-1+K} - \tau_j \geq \gamma_j - \tau_j = a_N$.  Therefore, by (\ref{A2prime}), we have
\begin{equation}\label{expb1}
s (\tau_{j+1} - \tau_j) \geq \frac{a_N s}{3 k_N} = \frac{\log (s/\mu)^2}{3 \log N} \rightarrow \infty \hspace{.2in}\mbox{as } N \rightarrow \infty
\end{equation}
and
\begin{equation}\label{expb2}
k_N^2 e^{-s(\gamma_{j-1+K} - \tau_j)} \leq k_N^2 e^{-sa_N} = \frac{k_N^2 \mu}{s} \rightarrow 0 \hspace{.2in}\mbox{as } N \rightarrow \infty.
\end{equation}
The lemma follows from (\ref{mainimmbound}), (\ref{expb1}), and (\ref{expb2}).
\end{proof}

It remains to bound the third term on the right-hand side of (\ref{mainXj}).  To bound this term, we will need to control the fluctuations of the martingale $(Z_j''(\tau_{j+1} + t), t \geq 0)$.  Lemma \ref{bigvarZj} below gives the required second moment bound.  Before stating this lemma, we provide some estimates on $G_j(v)$ in the following two lemmas.

\begin{Lemma}\label{Gjest1}
For sufficiently large $N$,
if $j \geq k^* + 1 + K$ and $u \in [\tau_{j+1}, \gamma_{j-K} \wedge \rho_j')$, or if $k^* + 1 \leq j \leq k^* + K$ and $u \in [\tau_{j+1}, a_N \wedge \rho_j')$, then
\begin{equation}\label{Gj0}
e^{-\int_{\tau_{j+1}}^u G_j(v) \: dv} \leq e^{-sk_N(u - \tau_{j+1})/5}.
\end{equation}
\end{Lemma}

\begin{proof}
We will use the results of Proposition \ref{meanprop}, which by definition hold up to time $\rho_j'$.  Also, recall that $K = \lfloor k_N/4 \rfloor$.  First, suppose $j \geq k^* + 1 + K$ and $t < \rho_j'$.  If $t \leq a_N$, then by part 1 of Proposition \ref{meanprop}, for sufficiently large $N$,
\begin{equation}\label{Gj1}
G_j(t) = s(j - M(t)) - \mu \geq s(j - 3) - \mu \geq \frac{sk_N}{5}.
\end{equation}
If $t \in (a_N, \gamma_{k^* + 1})$, then by part 2 of Proposition \ref{meanprop}, for sufficiently large $N$,
\begin{equation}\label{Gj2}
G_j(t) = s(j - M(t)) - \mu \geq s(j - k_N - C_4) - \mu \geq \frac{s k_N}{5}.
\end{equation}
If $t \in [\gamma_{k^* + 1}, \gamma_{j-K})$, then by part 3 of Proposition \ref{meanprop}, for sufficiently large $N$,
\begin{equation}\label{Gj3}
G_j(t) = s(j - M(t)) - \mu \geq s(j - (j-K-1) - 2C_5) - \mu \geq \frac{s k_N}{5}.
\end{equation}
Combining (\ref{Gj1}), (\ref{Gj2}), and (\ref{Gj3}), we get for $u \in [\tau_{j+1}, \gamma_{j-K} \wedge \rho_j')$
$$\int_{\tau_{j+1}}^u G_j(v) \: dv \geq \frac{s k_N}{5} (u - \tau_{j+1}),$$ which leads to (\ref{Gj0}).
Next, suppose $k^* + 1 \leq j \leq k^* + K$.  If $t < a_N \wedge \rho_j'$, then (\ref{Gj1}) holds as before, which again yields (\ref{Gj0}).
\end{proof}

\begin{Lemma}\label{Gjest2}
For sufficiently large $N$, if $j \geq k^* + 1 + K$ and $u \in [\gamma_{j-K}, \rho_j')$, or if $k^* + 1 \leq j \leq k^* + K$ and $u \in [a_N, \rho_j')$, then
$$e^{-\int_{\tau_{j+1}}^u G_j(v) \: dv} \leq \bigg( \frac{s}{\mu} \bigg)^{-k_N/241}.$$
\end{Lemma}

\begin{proof}
Recall the definition of ${\bar M}(t)$ from (\ref{Mbardef}).  Write $\gamma^* = \gamma_{k^* + 1}$ if $k^* + 1 \leq j \leq k^* + K$ and $\gamma^* = \gamma_{j-K}$ if $j \geq k^* + 1 + K$.  Also, write $i = k^* + 1$ if $ k^* + 1 \leq j \leq k^* + K$ and $i = j - K$ if $j \geq k^* + 1 + K$.  Suppose $u \in [\gamma^*, \rho_j')$, and note from the definition of $\rho_j'$ that this means $u < \gamma_{j+K}$.  Now
\begin{align}\label{Gj6}
\int_{\gamma^*}^u G_j(v) \: dv &= \int_{\gamma^*}^u \big( s(j - {\bar M}(v)) + s ({\bar M}(v) - M(v)) - \mu \big) \: dv \nonumber \\
&= \sum_{\ell = i}^{j + K - 1} \int_{\gamma_{\ell} \wedge u}^{\gamma_{\ell + 1} \wedge u} \big( s(j - \ell) + s ({\bar M}(v) - M(v)) - \mu \big) \: dv.
\end{align}
Because positive terms can be bounded below by zero, we have, using (\ref{tauspacing}),
\begin{align}\label{Gj7}
\sum_{\ell = i}^{j+K-1} \int_{\gamma_{\ell} \wedge u}^{\gamma_{\ell+1} \wedge u} s(j - \ell) \: dv &\geq \sum_{\ell = j+1}^{j+K-1} s (j - \ell)(\gamma_{\ell+1} \wedge u - \gamma_{\ell} \wedge u) \nonumber \\
&\geq \frac{2 a_N}{k_N} \sum_{\ell = j+1}^{j+K-1} s(j - \ell) \nonumber \\
&= - \frac{K (K-1) s a_N}{k_N}.
\end{align}
Using (\ref{tauspacing}) and (\ref{Mbareq}) and the fact that there are at most $2K$ terms in the sum, we get
\begin{equation}\label{Gj8}
\sum_{\ell = i}^{j +K - 1} \int_{\gamma_{\ell} \wedge u}^{\gamma_{\ell+1} \wedge u} \big( s({\bar M}(v) - M(v)) - \mu \big) \: dv \geq -2K \bigg( 2 C_5 + \frac{2 a_N \mu}{k_N} \bigg).
\end{equation}
Now since $s/\mu \rightarrow \infty$ as $N \rightarrow \infty$ by (\ref{muspower}) and $s a_N/k_N \rightarrow \infty$ as $N \rightarrow \infty$ by (\ref{A2prime}), we have $4C_5 + 4 a_N \mu/k_N \leq sa_N/k_N$ for sufficiently large $N$.
Combining this observation with (\ref{Gj6}), (\ref{Gj7}), and (\ref{Gj8}) yields
\begin{equation}\label{Gj9}
\int_{\gamma^*}^u G_j(v) \: dv \geq - \frac{K^2 s a_N}{k_N} \geq - \frac{k_N s a_N}{16} = -\frac{k_N}{16} \log \bigg( \frac{s}{\mu} \bigg).
\end{equation}

It remains to consider the integral between times $\tau_{j+1}$ and $\gamma^*$.  Suppose first that $j \geq k^* + 1 + K$.  
In view of part 3 of Proposition \ref{tauprop}, for sufficiently large $N$, as long as $\gamma^* < \rho_j'$, we have
$$\gamma^* - \tau_{j+1} = \gamma_{j-K} - \tau_{j-K} + \tau_{j-K} - \tau_{j+1} = a_N - (\tau_{j+1} - \tau_{j-K}) \geq a_N - \frac{2(K-1)a_N}{k_N} \geq \frac{a_N}{3}.$$  Thus, assuming that $\gamma^* < \rho_j'$, Lemma \ref{Gjest1} implies that for sufficiently large $N$,
\begin{equation}\label{GjC1}
\int_{\tau_{j+1}}^{\gamma^*} G_j(v) \: dv \geq \frac{s k_N}{5} (\gamma^* - \tau_{j+1}) \geq \frac{s k_N a_N}{15} = \frac{k_N}{15} \log \bigg( \frac{s}{\mu} \bigg).
\end{equation}
Suppose next that $k^* + 1 \leq j \leq k^* + K$.  Then, as long as $a_N < \rho_j'$, parts 1 and 3 of Proposition \ref{tauprop} imply that for sufficiently large $N$,
$$a_N - \tau_{j+1} \geq a_N - \tau_{k^*+1} + \tau_{k^*+1} - \tau_{j+1} \geq a_N - \frac{2a_N}{k_N} - \frac{2Ka_N}{k_N} \geq \frac{a_N}{3},$$
and the same reasoning that yields (\ref{GjC1}) gives
\begin{equation}\label{Gj4}
\int_{\tau_{j+1}}^{a_N} G_j(v) \: dv \geq \frac{k_N}{15} \log \bigg( \frac{s}{\mu} \bigg).
\end{equation}
Now suppose $u \in [a_N, \gamma_{k^* + 1} \wedge \rho_j')$.  Then for $v \in (a_N, u)$, since $k^* + 1 \geq k_N^+ \geq k_N$ and $M(v) < k_N + C_4$ by the result of part 2 of Proposition \ref{meanprop}, we have
$$G_j(v) = s(j - M(v)) - \mu \geq s(k^* + 1 - k_N - C_4) - \mu \geq -s C_4 - \mu.$$  Therefore, since $(\gamma_{k^*+1} \wedge \rho_j') - a_N \leq 2a_N/k_N$ by part 1 of Proposition \ref{tauprop}, for sufficiently large $N$ we have
\begin{equation}\label{Gj5}
\int_{a_N}^u G_j(v) \: dv \geq -(s C_4 + \mu)(u - a_N) \geq -\frac{2 (s C_4 + \mu) a_N}{k_N} = -\frac{2(C_4 + \mu/s)}{k_N} \log \bigg( \frac{s}{\mu} \bigg) \geq -\log \bigg( \frac{s}{\mu} \bigg).
\end{equation}
By combining (\ref{Gj9}) and (\ref{GjC1}) when $j \geq k^* + 1 + K$ and $u \in [\gamma_{j-K}, \rho_j')$, and by combining (\ref{Gj9}), (\ref{Gj4}), and (\ref{Gj5}) when $k^* + 1 \leq j \leq k^* + K$ and $u \in [a_N, \rho_j')$, we obtain for sufficiently large $N$ in both cases,
$$\int_{\tau_{j+1}}^u G_j(v) \: dv \geq \bigg( \frac{k_N}{15} - \frac{k_N}{16} \bigg) \log \bigg( \frac{s}{\mu} \bigg) - \log \bigg( \frac{s}{\mu} \bigg) \geq \frac{k_N}{241} \log \bigg( \frac{s}{\mu} \bigg).$$
The result of the lemma follows.
\end{proof}

\begin{Lemma}\label{bigvarZj}
For sufficiently large $N$, we have, for all $t \geq 0$,
$$\Var(Z_j''(\tau_{j+1} + t)|{\cal F}_{\tau_{j+1}}) \leq \frac{21}{\mu k_N}$$
on the event $F_j$.
\end{Lemma}

\begin{proof}
By Corollary \ref{ZmartCor2} and (\ref{BD3}), for all $t \geq 0$,
\begin{equation}\label{zvar1}
\Var(Z_j''(\tau_{j+1} + t)|{\cal F}_{\tau_{j+1}}) \leq E \bigg[ \int_{\tau_{j+1}}^{\tau_{j+1} + t} e^{-2 \int_{\tau_{j+1}}^u G_j(v) \: dv}(\mu X_{j-1}(u) + 3 X_j(u)) \1_{\{u < \rho_j'\}} \: du \bigg| {\cal F}_{\tau_{j+1}} \bigg].
\end{equation}
Using (\ref{mainimmterm}) when $u < \gamma_{j-1+K}$ and using (\ref{lateimmterm}) combined with (\ref{expb2}) when $u \geq \gamma_{j-1+K}$, we obtain that if $u \in [\tau_{j+1}, \rho_j')$, then for sufficiently large $N$,
\begin{equation}\label{zvar2}
e^{-\int_{\tau_{j+1}}^u G_j(v) \: dv} \mu X_{j-1}(u) \leq \frac{2s^2}{\mu}.
\end{equation}
Also, by (\ref{mainXj}) and Lemma \ref{immterm}, if $u \in [\tau_{j+1}, \rho_j')$, then on the event $F_j$, for sufficiently large $N$,
\begin{align}\label{zvar3}
e^{-\int_{\tau_{j+1}}^u G_j(v) \: dv} X_j(u) &= X_j(\tau_{j+1}) + \int_{\tau_{j+1}}^u \mu X_{j-1}(w) e^{-\int_{\tau_{j+1}}^w G_j(v) \: dv} \: dw + Z_j''(u) \nonumber \\
&\leq 1 + \frac{s}{\mu} + \frac{\delta s}{3 \mu} + Z_j''(u).
\end{align}
Combining (\ref{zvar1}), (\ref{zvar2}), and (\ref{zvar3}), and noting that $2s^2/\mu + 3(1 + s/\mu + \delta s/3 \mu) \leq 4s/\mu$ for sufficiently large $N$, we get
\begin{equation}\label{zvar4}
\Var(Z_j''(\tau_{j+1} + t)|{\cal F}_{\tau_{j+1}}) \leq E \bigg[ \int_{\tau_{j+1}}^{\tau_{j+1} + t} e^{-\int_{\tau_{j+1}}^u G_j(v) \: dv} \bigg( \frac{4s}{\mu} + 3 Z_j''(u) \bigg) \1_{\{u < \rho_j'\}} \: du \bigg| {\cal F}_{\tau_{j+1}} \bigg]
\end{equation}
on $F_j$ for sufficiently large $N$.

To bound the right-hand side of (\ref{zvar4}), we split the integral into two pieces.  Let $\gamma' = \gamma_{j-K}$ if $j \geq k^* + 1 + K$, and let $\gamma' = a_N$ if $k^* + 1 \leq j \leq k^* + K$.  Consider first the contribution to the integral from $u < \gamma'$.  Because the integrand in (\ref{zvar1}) is nonnegative and $Z_j''(u) = Z_j''(\rho_j')$ for all $u \geq \rho_j'$, we have $4s/\mu + 3 Z_j''(u) \geq 0$ for all $u \geq \tau_{j+1}$.  Then by Lemma \ref{Gjest1},
for all $u \geq 0$, $$e^{-\int_{\tau_{j+1}}^u G_j(v) \: dv} \bigg( \frac{4s}{\mu} + 3 Z_j''(u) \bigg) \1_{\{u < \gamma' \wedge \rho_j'\}} \leq e^{-s k_N (u - \tau_{j+1})/5} \bigg( \frac{4s}{\mu} + 3 Z_j''(u) \bigg).$$  Combining this observation with Fubini's Theorem and the fact that $(Z_j''(\tau_{j+1} + t), t \geq 0)$ is a mean zero martingale, we get for all $t \geq 0$,
\begin{align}\label{gampt1}
& E \bigg[ \int_{\tau_{j+1}}^{\tau_{j+1} + t} e^{-\int_{\tau_{j+1}}^u G_j(v) \: dv} \bigg( \frac{4s}{\mu} + 3 Z_j''(u) \bigg) \1_{\{u < \gamma' \wedge \rho_j'\}} \: du \bigg| {\cal F}_{\tau_{j+1}} \bigg] \nonumber \\
&\hspace{1.5in}\leq E \bigg[ \int_{\tau_{j+1}}^{\infty} e^{-sk_N(u - \tau_{j+1})/5} \bigg( \frac{4s}{\mu} + 3 Z_j''(u) \bigg) \: du \bigg| {\cal F}_{\tau_{j+1}} \bigg] \nonumber \\
&\hspace{1.5in} = \frac{4s}{\mu} \int_{\tau_{j+1}}^{\infty} e^{-sk_N(u - \tau_{j+1})/5} \: du \nonumber \\
&\hspace{1.5in} = \frac{20}{\mu k_N}.
\end{align}
Likewise, by Lemma \ref{Gjest2},
\begin{equation}\label{inti1}
e^{-\int_{\tau_{j+1}}^u G_j(v) \: dv} \bigg( \frac{4s}{\mu} + 3 Z_j''(u) \bigg) \1_{\{\gamma' \leq u <\rho_j'\}} \leq \bigg( \frac{s}{\mu} \bigg)^{-k_N/241} \bigg( \frac{4s}{\mu} + 3Z''(u) \bigg).
\end{equation}
Also, using (\ref{tauspacing}), which is valid up to time $\rho_j'$,
\begin{equation}\label{inti2}
\rho_j' - \tau_{j+1} = a_N + \rho_j' - \gamma_{j+1} \leq a_N + \frac{2K a_N}{k_N} \leq \frac{3 a_N}{2}.
\end{equation}
Combining (\ref{inti1}) and (\ref{inti2}) with with Fubini's Theorem and the fact that $(Z_j''(\tau_{j+1} + t), t \geq 0)$ is a mean zero martingale, we get for all $t \geq 0$,
\begin{align}\label{gampt2}
& E \bigg[ \int_{\tau_{j+1}}^{\tau_{j+1} + t} e^{-\int_{\tau_{j+1}}^u G_j(v) \: dv} \bigg( \frac{4s}{\mu} + 3 Z_j''(u) \bigg) \1_{\{\gamma' \leq u <\rho_j'\}} \: du \bigg| {\cal F}_{\tau_{j+1}} \bigg] \nonumber \\
&\hspace{1.5in}\leq E \bigg[ \int_{\tau_{j+1}}^{\tau_{j+1} + 3a_N/2} \bigg( \frac{s}{\mu} \bigg)^{-k_N/241} \bigg( \frac{4s}{\mu} + 3Z''(u) \bigg) \: du \bigg| {\cal F}_{\tau_{j+1}} \bigg] \nonumber \\
&\hspace{1.5in}\leq \frac{6 s a_N}{\mu} \bigg( \frac{s}{\mu} \bigg)^{-k_N/241} \nonumber \\
&\hspace{1.5in} = \frac{6}{\mu} \bigg( \frac{s}{\mu} \bigg)^{-k_N/241} \log \bigg( \frac{s}{\mu} \bigg).
\end{align}
Because $k_N (s/\mu)^{-k_N/241} \log (s/\mu) \rightarrow 0$ as $N \rightarrow \infty$, as can be easily seen by taking logarithms, the lemma follows from (\ref{zvar4}), (\ref{gampt1}), and (\ref{gampt2}).
\end{proof}

\begin{Lemma}\label{finZpp}
For sufficiently large $N$,
$$P \bigg( \bigg\{ |Z_j''(t)| > \frac{\delta s}{3 \mu} \mbox{ for some }t \in [\tau_{j+1}, \rho_j'] \bigg\} \cap F_j \bigg) \leq \frac{756 \mu}{\delta^2 s^2 k_N}.$$
\end{Lemma}

\begin{proof}
By the $L^2$ Maximum Inequality for martingales and Lemma \ref{bigvarZj}, on the event $F_j$,
$$P \bigg( \sup_{t \geq 0} |Z_j''(\tau_{j+1} + t)| > \frac{\delta s}{3 \mu} \bigg| {\cal F}_{\tau_{j+1}} \bigg) \leq \frac{36 \mu^2}{\delta^2 s^2} \cdot \sup_{t \geq 0} \Var(Z_j''(\tau_{j+1} + t)|{\cal F}_{\tau_{j+1}}) \leq \frac{756 \mu}{\delta^2 s^2 k_N}.$$
Taking expectations of both sides yields the result.
\end{proof}

\begin{Cor}\label{pt3zeta}
For sufficiently large $N$,
$$\sum_{j = k^* + 1}^J P \bigg( X_j(t) \notin \bigg[ \frac{(1 - \delta) s}{\mu} e^{\int_{\tau_{j+1}}^t G_j(v) \: dv}, \: \frac{(1 + \delta) s}{\mu} e^{\int_{\tau_{j+1}}^t G_j(v) \: dv} \bigg] \mbox{ for some }t \in [\tau_{j+1}, \rho_j'] \bigg) \leq \frac{\eps}{24}.$$
\end{Cor}

\begin{proof}
By (\ref{mainXj}), Lemmas \ref{immterm} and \ref{finZpp}, and the fact that $s/\mu \leq X_j(\tau_{j+1}) \leq (s/\mu)(1+\delta/3)$ for sufficiently large $N$ by (\ref{muspower}), we have
\begin{align*}
&\sum_{j = k^* + 1}^J P \bigg( X_j(t) \notin \bigg[ \frac{(1 - \delta) s}{\mu} e^{\int_{\tau_{j+1}}^t G_j(v) \: dv}, \: \frac{(1 + \delta) s}{\mu} e^{\int_{\tau_{j+1}}^t G_j(v) \: dv} \bigg] \mbox{ for some }t \in [\tau_{j+1}, \rho_j'] \bigg) \nonumber \\
&\hspace{3in}\leq \sum_{j=k^* + 1}^J \bigg( \frac{756 \mu}{\delta^2 s^2 k_N} + P(F_j^c \cap \{\rho_j \geq \tau_{j+1}\}) \bigg).
\end{align*}
Because $\sum_{j=k^*+1}^J P(F_j^c \cap \{\rho_j \geq \tau_{j+1}\}) \leq \eps/25$ by Corollary \ref{pt2zeta} and $J \mu/(\delta^2 s^2 k_N) \rightarrow 0$ as $N \rightarrow \infty$ by (\ref{muspower}), the result follows.
\end{proof}

\subsection{Type $j$ individuals after time $\gamma_{j+K}$}

In this subsection, we show that the number of type $j$ individuals decreases rapidly after time $\gamma_{j+K}$.  More specifically, we show that the results of parts 4 and 5 of Proposition \ref{prop2} hold with high probability.  We will consider the event
$$H_j = \bigg\{\frac{(1 - \delta) s}{\mu} e^{\int_{\tau_{j+1}}^{\gamma_{j+K}} G_j(v) \: dv} \leq X_j(\gamma_{j + K}) \leq \frac{(1 + \delta) s}{\mu} e^{\int_{\tau_{j+1}}^{\gamma_{j+K}} G_j(v) \: dv} \bigg\}.$$  By Corollary \ref{pt3zeta} when $t = \gamma_{j+K}$, with probability at least $1 - \eps/24$, for all $j \in \{k^* + 1, \dots, J\}$ either $H_j$ occurs or $\gamma_{j+K} > \rho_j$.  Recall also the definition of the event $F_j$ from (\ref{Fjdef}).

\begin{Lemma}\label{ratlem}
Suppose $j \in \{k^*+1, \dots, J\}$.  For sufficiently large $N$, if $\ell \leq j-1$, then
\begin{equation}\label{ratb}
\frac{X_{\ell}(\gamma_{j+K})}{X_j(\gamma_{j+K})} \leq 3 k_N^2 \bigg( \frac{s}{\mu} \bigg)^{-1/13}
\end{equation}
on the event $F_j \cap H_j \cap \{\gamma_{j+K} < \rho_j\}$.
\end{Lemma}

\begin{proof}
We will assume throughout the proof that $\gamma_{j+K} < \rho_j$.  By Lemma \ref{lateimm}, if $k^* \leq \ell \leq j-1$ and $\gamma_{j+K} < \rho_j$, then
\begin{equation}\label{rat0}
X_{\ell}(\gamma_{j+K}) \leq \frac{(1 + 2 \delta) k_N^2 s}{\mu} e^{\int_{\tau_{\ell+1}}^{\gamma_{j+K}} G_{\ell}(v) \: dv}.
\end{equation}
Therefore, on the event $H_j$,
\begin{equation}\label{rat1}
\frac{X_{\ell}(\gamma_{j+K})}{X_j(\gamma_{j+K})} \leq \frac{(1 + 2 \delta) k_N^2}{1 - \delta} e^{\int_{\tau_{\ell+1}}^{\tau_{j+1}} G_{\ell}(v) \: dv} e^{\int_{\tau_{j+1}}^{\gamma_{j+K}} (G_{\ell}(v) - G_j(v)) \: dv}.
\end{equation}
Recall that on the event $F_j$, equation (\ref{expGjint}) holds, and therefore $e^{\int_{\tau_j}^{\tau_{j+1}} G_j(v) \: dv} \leq 2s/\mu$ for sufficiently large $N$.  Also, by Lemma \ref{smulem}, in view of the assumption that $\gamma_{j+K} < \rho_j$, 
the same result holds when $j$ is replaced by $h \in \{k^* + 1, \dots, j-1\}$.  Therefore, for sufficiently large $N$, we have
\begin{equation}\label{rat2}
e^{\int_{\tau_{\ell+1}}^{\tau_{j+1}} G_{\ell}(v) \: dv} \leq \prod_{h=\ell+1}^j e^{\int_{\tau_h}^{\tau_{h+1}} G_h(v) \: dv} \leq \bigg(\frac{2s}{\mu}\bigg)^{j-\ell}.
\end{equation}
Also, it follows from (\ref{tauspacing}) that $\gamma_{j + K} - \tau_{j+1} = a_N + \tau_{j+K} - \tau_{j+1} \geq a_N + a_N(K-1)/3k_N \geq 14a_N/13$ for sufficiently large $N$.  Therefore,
\begin{equation}\label{rat3}
\int_{\tau_{j+1}}^{\gamma_{j+K}} (G_{\ell}(v) - G_j(v)) \: dv = - s(j - \ell)(\gamma_{j+K} - \tau_{j+1}) \leq - \frac{14 (j - \ell) s a_N}{13} = - \frac{14(j-\ell)}{13} \log \bigg( \frac{s}{\mu} \bigg).
\end{equation}
By (\ref{rat1}), (\ref{rat2}), and (\ref{rat3}), for sufficiently large $N$, on the event $F_j \cap H_j \cap \{\gamma_{j+K} < \rho_j\}$, we have
$$\frac{X_{\ell}(\gamma_{j+K})}{X_j(\gamma_{j+K})} \leq \frac{(1 + 2 \delta) k_N^2}{1 - \delta} \bigg( 2 \bigg( \frac{s}{\mu} \bigg)^{-1/13} \bigg)^{j - \ell}.$$  Because $2(s/\mu)^{-1/13} \rightarrow 0$ as $N \rightarrow \infty$, for sufficiently large $N$ this expression is largest when $\ell = j-1$, and thus (\ref{ratb}) holds whenever $k^* \leq \ell \leq j-1$.

Next, suppose $0 \leq \ell \leq k^* - 1$.  Then by (\ref{prop11}), we have $$X_{\ell}(\gamma_{k^* + K}) \geq (1 - \delta) X_{\ell}(t^*) e^{\int_{t^*}^{\gamma_{k^* + K}} G_{\ell}(v) \: dv},$$ and by (\ref{SjgammaK}), we have $X_{\ell}(\gamma_{k^* + K}) \leq X_{k^*}(\gamma_{k^* + K})$.  Therefore, by (\ref{prop12}),
\begin{align*}
X_{\ell}(\gamma_{j+K}) &\leq k_N^2 X_{\ell}(t^*) e^{\int_{t^*}^{\gamma_{k^* + K}} G_{\ell}(v) \: dv} e^{\int_{\gamma_{k^* + K}}^{\gamma_{j+K}} G_{\ell}(v) \: dv} \\
&\leq \frac{k_N^2}{1 - \delta} X_{\ell}(\gamma_{k^* + K}) e^{\int_{\gamma_{k^* + K}}^{\gamma_{j+K}} G_{\ell}(v) \: dv} \\
&\leq \frac{k_N^2}{1 - \delta} X_{k^*}(\gamma_{k^* + K}) e^{\int_{\gamma_{k^* + K}}^{\gamma_{j+K}} G_{k^*}(v) \: dv}. 
\end{align*}
Now using Lemma \ref{imm},
$$X_{\ell}(\gamma_{j + K}) \leq \frac{(1 + 3 \delta) k_N^2 s}{(1 - \delta) \mu} e^{\int_{\tau_{k^*+1}}^{\gamma_{j+K}} G_{k^*}(v) \: dv},$$
which is the same as (\ref{rat0}) when $\ell = k^*$ except for the constant in front involving $\delta$.  Therefore, (\ref{ratb}) holds on $F_j \cap H_j \cap \{\gamma_{j+K} < \rho_j\}$ in this case as well.
\end{proof}

\begin{Prop}\label{pt4zeta}
For sufficiently large $N$,
\begin{align}\label{m4zeta}
&\sum_{j = k^*+1}^J P\bigg( \bigg\{ X_j(t) > \frac{k_N^2 s}{\mu} e^{\int_{\tau_{j+1}}^t G_j(v) \: dv} \mbox{ for some }t \in (\gamma_{j+K}, \rho_j] \bigg\} \nonumber \\
&\hspace{2in} \cup \big\{ X_j(t) > 0 \mbox{ for some }t \in [\gamma_{j+L}, \rho_j]  \big\} \bigg) < \frac{\eps}{12}.
\end{align}
\end{Prop}

\begin{proof}
Suppose $j \in \{k^*+1, \dots, J\}$.  Recall that $S_j(t) = X_0(t) + X_1(t) + \dots + X_j(t)$ for all $t \geq 0$.  By Proposition \ref{supprop} and Remark \ref{Gjtilde}, the process $$\big( e^{-\int_{\gamma_{j+K}}^{(\gamma_{j+K} + t) \wedge \rho_j} G_j(v) \: dv} S_j((\gamma_{j+K} + t) \wedge \rho_j), t \geq 0 \big)$$ is a nonnegative supermartingale.  Therefore,
\begin{equation}\label{p451}
P \bigg( \sup_{t \in [\gamma_{j+K}, \rho_j]} e^{-\int_{\gamma_{j+K}}^t G_j(v) \: dv} S_j(t) > \frac{k_N^2}{2} S_j(\gamma_{j + K}) \bigg| {\cal F}_{\gamma_{j + K}} \bigg) \leq \frac{2}{k_N^2}.
\end{equation}
Since $j \leq J \leq 4T k_N$ for sufficiently large $N$, on the event $F_j \cap H_j \cap \{\gamma_{j+K} < \rho_j\} \in {\cal F}_{\gamma_{j+K}}$, Lemma \ref{ratlem} implies that for sufficiently large $N$, $$S_j(\gamma_{j+K}) \leq \bigg(1 + 3 (j-1) k_N^2 \bigg( \frac{s}{\mu} \bigg)^{1/13} \bigg) X_j(\gamma_{j+K}) \leq \bigg(1 + 12 T k_N^3 \bigg( \frac{s}{\mu} \bigg)^{1/13} \bigg) X_j(\gamma_{j+K}).$$  Since $k_N^3 (s/\mu)^{-1/13} \rightarrow 0$ as $N \rightarrow \infty$, as can be seen by taking the logarithm and applying (\ref{A2prime}), for sufficiently large $N$ we have $S_j(\gamma_{j+K}) \leq (3/2) X_j(\gamma_{j+K})$ on $F_j \cap H_j \cap \{\gamma_{j+K} < \rho_j\}$.  Combining this observation with (\ref{p451}) gives that for sufficiently large $N$,
$$P \bigg( S_j(t) > \frac{3k_N^2}{4} e^{\int_{\gamma_{j+K}}^t G_j(v) \: dv} X_j(\gamma_{j+K}) \mbox{ for some }t \in (\gamma_{j+K}, \rho_j] \bigg| {\cal F}_{\gamma_{j+K}} \bigg) \leq \frac{2}{k_N^2}$$ on $F_j \cap H_j$.  Since $X_j(t) \leq S_j(t)$ for all $t \geq 0$ and $(3/4) X_j(\gamma_{j+K}) \leq (s/\mu) e^{\int_{\tau_{j+1}}^{\gamma_{j+K}} G_j(v) \: dv}$ on $H_j$, it follows that on $F_j \cap H_j$,
\begin{equation}\label{p452}
P \bigg( X_j(t) > \frac{k_N^2 s}{\mu} e^{\int_{\tau_{j+1}}^t G_j(v) \: dv} \mbox{ for some }t \in (\gamma_{j+K}, \rho_j] \bigg| {\cal F}_{\gamma_{j+K}} \bigg) \leq \frac{2}{k_N^2}.
\end{equation}

Also, on the complement of the event in (\ref{p451}), if $\rho_j \geq \gamma_{j+L}$ then
$$S_j(\gamma_{j+L}) \leq \frac{k_N^2}{2} e^{\int_{\gamma_{j+K}}^{\gamma_{j+L}} G_j(v) \: dv} S_j(\gamma_{j+K}).$$  Reasoning exactly as in (\ref{finSj1}), (\ref{finSj2}), (\ref{finSj3}), and (\ref{finSj4}) but with $j$ in place of $k^*$, we get that on the complement of the event in (\ref{p451}), if $\rho_j \geq \gamma_{j+L}$ then for sufficiently large $N$,
$$S_j(\gamma_{j+L}) \leq \frac{N k_N^2}{2} \bigg( \frac{s}{\mu} \bigg)^{-16 k_N/15}.$$  As in the discussion following (\ref{finSj4}), we see that the right-hand side tends to zero as $N \rightarrow \infty$ and thus must be less than one if $N$ is large enough.  Because $S_j(\gamma_{j + L})$ is an integer, it follows that $S_j(\gamma_{j + L}) = 0$, and therefore that $X_j(t) = S_j(t) = 0$ for all $t \geq \gamma_{j+L}$.
Combining this observation with (\ref{p452}), we get that the sum of the probabilities in (\ref{m4zeta}) is bounded above by $$\sum_{j=k^*+1}^J \bigg( \frac{2}{k_N^2} + P(F_j \cup H_j) \bigg).$$  By Corollaries \ref{pt2zeta} and \ref{pt3zeta}, this expression is at most $2J/k_N^2 + \eps/25 + \eps/24$, which is less than $\eps/12$ for sufficiently large $N$.
\end{proof}

We now combine the results of this section to complete the proof of Proposition \ref{zetaprop}.

\begin{proof}[Proof of part 3 of Proposition \ref{zetaprop}]
It follows from Corollary \ref{pt2zeta}, Proposition \ref{pt1zeta}, Proposition \ref{8155}, Corollary \ref{pt3zeta}, and Proposition \ref{pt4zeta} that
$$\sum_{j = k^* + 1}^J P(\{\zeta_0 = \infty\} \cap \{\zeta_{1,j} \leq \rho_j\}) \leq \frac{\eps}{25} + \frac{\eps}{4} + \frac{\eps}{48} + \frac{\eps}{24} + \frac{\eps}{12}$$
for sufficiently large $N$.  Also, Remark \ref{zeta1jstar} gives that for sufficiently large $N$,
$$\sum_{j=0}^J P(\{\zeta_0 = \infty\} \cap \{\zeta_{1,j} \leq \rho_j\}) < \frac{\eps}{16}.$$  Because $\eps/25 + \eps/4 + \eps/48 + \eps/24 + \eps/12 + \eps/16 < \eps/2$, it follows that (\ref{stszeta}) holds for sufficiently large $N$.  As noted at the beginning of section \ref{zetasec3}, this completes the proof of part 3 of Proposition \ref{zetaprop}.
\end{proof}

\section{Proof of Theorems \ref{Qthm}, \ref{speedthm}, and \ref{gaussthm}}\label{3propsec}

With Proposition \ref{zetaprop} having been established, in this section we use this result to prove Theorems \ref{Qthm}, \ref{speedthm}, and \ref{gaussthm}.  All of these theorems follow rather directly from Propositions \ref{prop1}, \ref{prop2}, \ref{meanprop}, and \ref{tauprop}, which, as noted in section \ref{strucsec}, all follow from Proposition \ref{zetaprop}.  We prove Theorem \ref{Qthm} in section \ref{tsec1}, Theorem \ref{speedthm} in section \ref{tsec2}, and Theorem \ref{gaussthm} in section \ref{tsec3}.

\subsection{The selective advantage of the fittest individuals}\label{tsec1}

Recall that $Q(t)$, defined in (\ref{Qdef}), is the difference between the number of mutations carried by the fittest individual and the mean number of mutations in the population.  Consequently, it is a measure of the selective advantage that the fittest individuals in the population have over typical individuals in the population.  Theorem \ref{Qthm} describes the asymptotic behavior of the process $(Q(t), t \geq 0)$ as the population size tends to infinity.

\begin{proof}[Proof of Theorem \ref{Qthm}]
It suffices to prove (\ref{mainQres}) for $S = [u, v]$, where either $0 < u < v < 1$ or $1 < u < v < \infty$.  In view of part 2 of Proposition \ref{tauprop}, it suffices to show that
\begin{equation}\label{th11}
\sup_{t \in S} \frac{|Q(a_N t) - R(a_N t)|}{k_N} \rightarrow_p 0,
\end{equation}
where $R(t)$ was defined in (\ref{Rdef}).  Throughout the proof, we fix $\eps > 0$, $\delta > 0$, and $T > \max\{1, v\}$.  

We assume that $N$ is large enough that the conclusions of Propositions \ref{prop1}, \ref{prop2}, \ref{meanprop}, and \ref{tauprop} hold with probability at least $1 - \eps$, and we work on the event that the conclusions of these propositions hold.  Suppose first that $0 < u < v < 1$.  By Proposition \ref{meanprop}, we have
\begin{equation}\label{th12}
\sup_{t \in S} M(a_N t) \leq 3
\end{equation}
Also, note that $a_N u > t^*$ for sufficiently large $N$, as can be seen from (\ref{A2prime}).  Therefore, by part 1 of Proposition \ref{tauprop}, for sufficiently large $N$ we have $\tau_{k^*+1} \leq a_N u$.  Therefore, recalling (\ref{Rdef}) and using either part 3 of Proposition \ref{tauprop} or Remark \ref{tauorder}, for sufficiently large $N$ we have $R(a_N t) = \max\{j: \tau_j \leq a_N t\}$ for all $t \in S$.  Suppose $R(a_N t) = i$, so $\tau_i \leq a_N t < \tau_{i+1}$.  By part 1 of Proposition \ref{prop2}, no type $i+2$ individual can appear before time $\tau_{i+1}$, which implies that $\max\{j: X_j(a_N t) > 0\} \leq i+1$.  By part 3 of Proposition \ref{prop2}, we have $X_{i-1}(a_N t) > 0$.  Therefore, for all $t \in S$,
\begin{equation}\label{th13}
R(a_N t) - 1 \leq \max\{j: X_j(a_N t) > 0\} \leq R(a_N t) + 1.
\end{equation}
Combining (\ref{th12}) and (\ref{th13}) gives
\begin{equation}\label{th16}
\sup_{t \in S} |Q(a_N t) - R(a_N t)| \leq 4.
\end{equation}

Now suppose instead $1 < u < v < \infty$.  Recall from (\ref{jtdef}) that $j(t) = \max\{j: \gamma_j \leq a_N t\}$, and write $h(t) = \max\{j: \tau_j \leq a_N t\}$.  Then, again recalling (\ref{Rdef}), for all $t \in S$ we have $R(a_N t) = h(t) - j(t)$.  By following again the derivation of (\ref{th13}), we get, for all $t \in S$,
\begin{equation}\label{th14}
h(t) - 1 \leq \max\{j: X_j(a_N t) > 0\} \leq h(t) + 1.
\end{equation}
Furthermore, we have $a_N t \in [\gamma_{j(t)}, \gamma_{j(t) + 1})$, which means $|M(a_N t) - j(t)| \leq 2C_5$ by part 3 of Proposition \ref{meanprop}.  Therefore,
\begin{equation}\label{th15}
\sup_{t \in S} |M(a_N t) - j(t)| \leq 2 C_5.
\end{equation}
It follows from (\ref{th14}) and (\ref{th15}) that
\begin{equation}\label{th17}
\sup_{t \in S} |Q(a_N t) - R(a_N t)| \leq 1 + 2 C_5.
\end{equation}

Because, for sufficiently large $N$, equation (\ref{th16}) holds with probability at least $1 - \eps$ if $0 < u < v < 1$ and equation (\ref{th17}) holds with probability at least $1 - \eps$ if $1 < u < v < \infty$, we may conclude (\ref{th11}).  Finally, the result (\ref{qlim}) was established as part of Lemma \ref{Qlem}.
\end{proof}

\subsection{The speed of evolution}\label{tsec2}

Here we obtain Theorem \ref{speedthm}, which gives the asymptotic behavior of the mean number of mutations in the population and therefore determines the speed of evolution.

\begin{proof}[Proof of Theorem \ref{speedthm}]
It suffices to prove (\ref{mainMres}) when $S = [u, v]$, where either $0 \leq u < v < 1$ or $1 < u < v < \infty$.  Suppose first that $0 \leq u < v < 1$.  Then $m(t) = 0$ for all $t \in S$.  By part 1 of Proposition \ref{meanprop}, for all $\eps > 0$, we have
$$P \bigg( \sup_{t \in S} M(a_N t) \leq 3 \bigg) > 1 - \eps$$ for sufficiently large $N$.  The result (\ref{mainMres}) follows immediately.

Suppose instead that $1 < u < v < \infty$.  We fix $\eps > 0$, $\delta > 0$, and $T > \max\{1, v\}$.  We assume for now that $N$ is large enough that the conclusions of Propositions \ref{meanprop} and \ref{tauprop} hold with probability at least $1 - \eps$, and we work on the event that the conclusions of these propositions hold.  Recall that $j(t) = \max\{j: \gamma_j \leq a_N t\} = \max\{j: \tau_j \leq a_N(t - 1)\}$, so $a_N t \in [\gamma_{j(t)}, \gamma_{j(t) + 1})$ for all $t \in S$.  By part 1 of Proposition \ref{tauprop} we have $j(t) \geq k^*+1$ for all $t \in S$ if $N$ is sufficiently large, so it follows from Proposition \ref{meanprop} that
\begin{equation}\label{Mjbound}
|M(a_N t) - j(t)| \leq 2 C_5 \hspace{.2in}\mbox{for all }t \in S.
\end{equation}
Now because $1 < u < v$, for all $t \in S$ we have
\begin{align}\label{qinteq}
m(t) &= 1 + \int_0^{t-1} q(u) \: du \nonumber \\
&= 1 + \int_0^{\tau_{k^*+1}/a_N} q(u) \: du + \sum_{j=k^*+1}^{j(t) - 1} \int_{\tau_j/a_N}^{\tau_{j+1}/a_N} q(u) \: du + \int_{\tau_{j(t)}/a_N}^{t-1} q(u) \: du.
\end{align}
We now obtain upper and lower bounds on the expression in (\ref{qinteq}).  For the upper bound, we use (\ref{tautight1}) along with part 1 of Proposition \ref{tauprop} and the fact that $q(t) \leq e$ for all $t$ by Lemma \ref{Qlem} to get
\begin{equation}\label{qintupper}
m(t) \leq 1 + \frac{e \tau_{k^* + 1}}{a_N} + (j(t) - k^*) \cdot \frac{1 + 2 \delta}{k_N} \leq \frac{j(t)(1 + 2 \delta)}{k_N} + \bigg(1 - \frac{k^*}{k_N} \bigg) + \frac{2e}{k_N}.
\end{equation}
For the lower bound, we use (\ref{tautight2}) and the fact that $(\gamma_{k^*+1}/a_N) - 1 \leq 2/k_N$ by part 1 of Proposition \ref{tauprop} to get
\begin{equation}\label{qintlower}
m(t) \geq 1 + (j(t) - k^* - 1) \cdot \frac{1 - 2 \delta}{k_N} - \frac{2}{k_N} \geq \frac{j(t)(1 - 2 \delta)}{k_N} + \bigg(1 - \frac{k^* + 1}{k_N} \bigg) - \frac{2}{k_N}.
\end{equation}
It follows from (\ref{Mjbound}), (\ref{qintupper}), and (\ref{qintlower}) that there exists a positive constant $C$, depending on $\eps$, $\delta$, and $T$, such that
\begin{equation}\label{Mprepf}
\sup_{t \in S} \bigg| \frac{M(a_N t)}{k_N} - m(t) \bigg| \leq \frac{2 \delta j(t) + C}{k_N}.
\end{equation}
Because Proposition \ref{tauprop} implies that for all $t \in S$, $$j(t) \leq k^* + \frac{3 k_N}{a_N} \cdot a_N v \leq k^* + 3v k_N,$$ and, by our assumptions, the event in (\ref{Mprepf}) holds with probability at least $1 - \eps$ for sufficiently large $N$, the result (\ref{mainMres}) follows. 
\end{proof}

\subsection{The distribution of fitnesses in the population}\label{tsec3}

In this subsection, we prove Theorem \ref{gaussthm}, which describes the distribution of the fitnesses of individuals in the population at time $a_N t$.  We begin with a lemma concerning the differences $\tau_{j+1} - \tau_j$.  Recall again the definition of $j(t)$ from (\ref{jtdef}).

\begin{Lemma}\label{thetatau}
For each $\eta > 0$ and $t \in (1, 2) \cup (2, \infty)$, there exists $\theta = \theta(\eta, t) > 0$ such that $$\lim_{N \rightarrow \infty} P \bigg(\frac{(1 - \eta/3) a_N}{q(t-1) k_N} \leq \tau_{j+1} - \tau_j \leq \frac{(1 + \eta/3) a_N}{q(t-1) k_N} \mbox{ for all }j \in [j(t) - \theta k_N, j(t) + \theta k_N] \cap \Z \bigg) = 1$$ 
and, for each fixed $\eta > 0$ and $a > 2$, 
\begin{equation}\label{inftheta2}
\inf_{t \in [a, \infty)} \theta(\eta, t) > 0.
\end{equation}
\end{Lemma}

\begin{proof}
Lemma \ref{Qlem} states that the function $q$ is continuous on $[0, 1) \cup (1, \infty)$.  Also, we can see from (\ref{qdef}) and the fact that $1 \leq q(t) \leq e$ for all $t \geq 0$ that $q$ is uniformly continuous on $[1, \infty)$.  Therefore, we may choose $\theta = \theta(\eta, t) > 0$ such that the following hold:
\begin{align*}
&1) \hspace{.1in} |q(t - 1) - q(u)| < \eta/7 \hspace{.1in}\mbox{for all }u \in [t - 1 - 3 \theta, t - 1 + 3 \theta], \\
&2) \hspace{.1in} [t - 1 - 3 \theta, t - 1 + 3 \theta] \subset (0, 1) \cup (1, \infty), \\
&3) \hspace{.1in} (\ref{inftheta2}) \mbox{ holds for each fixed }\eta > 0, a > 2.
\end{align*}

Fix $\eps > 0$, $\delta \in (0, \eta/14)$, and $T > t$.  
We may assume $N$ is large enough that the conclusions of Proposition \ref{tauprop} hold with probability at least $1 - \eps$.  For now, we will work on the event that the conclusions of Proposition \ref{tauprop} hold.  By (\ref{jtdef}), we have $\gamma_{j(t)} \leq a_N t < \gamma_{j(t)+1}$, and it follows that $\tau_{j(t)} \leq a_N (t - 1) < \tau_{j(t)+1}$.  Therefore, by (\ref{tauspacing}), for all $j \in [j(t) - \theta k_N, j(t) + \theta k_N] \cap \Z$ such that $j \geq k^* + 1$, we have $$a_N(t - 1) - (\theta k_N + 1) \cdot \frac{2a_N}{k_N} \leq \tau_j \leq \tau_{j+1} \leq a_N(t - 1) + \theta k_N \cdot \frac{2a_N}{k_N}.$$  It follows that for sufficiently large $N$,
\begin{equation}\label{taujint}
a_N(t - 1 - 3 \theta) \leq \tau_j \leq \tau_{j+1} \leq a_N(t - 1 + 3 \theta).
\end{equation}
Because $[t - 1 - 3 \theta, t - 1 + 3 \theta] \subset (0, 1) \cup (1, \infty)$, we can see from part 1 of Proposition \ref{tauprop} that for sufficiently large $N$, we have $j(t) - \theta k_N \geq k^* + 1$.  Also, in view of part 1 of Proposition \ref{tauprop}, for sufficiently large $N$ the interval $[t - 1 - 3 \theta, t - 1 + 3 \theta]$ will not intersect $[1, \gamma_{k^*+1}/a_N]$.  Therefore, for sufficiently large $N$, equations (\ref{tautight1}) and (\ref{tautight2}) imply that for all $j \in [j(t) - \theta k_N, j(t) + \theta k_N] \cap \Z$, we have
$$\frac{1 - 2 \delta}{k_N} \leq \int_{\tau_j/a_N}^{\tau_{j+1}/a_N} q(u) \: du \leq \frac{1 + 2 \delta}{k_N}.$$  Combining this result with (\ref{taujint}) and condition 1) above, we get that for sufficiently large $N$, 
$$\frac{(1 - 2 \delta) a_N}{(q(t-1) + \eta/7) k_N} \leq \tau_{j+1} - \tau_j \leq \frac{(1 + 2 \delta) a_N}{(q(t-1) - \eta/7) k_N}$$
for all $j \in [j(t) - \theta k_N, j(t) + \theta k_N] \cap \Z$.  Because $q(u) \geq 1$ for all $u \geq 0$ by Lemma \ref{Qlem} and $\delta < \eta/14$, we have $$1 - \frac{\eta}{3} \leq \frac{(1 - 2 \delta) q(t-1)}{q(t-1) + \eta/7} \leq \frac{(1 + 2 \delta)q(t-1)}{q(t-1) - \eta/7} \leq 1 + \frac{\eta}{3}$$ if $\eta$ is sufficiently small.  Because $\eps > 0$ is arbitrary, the result follows.
\end{proof}

\begin{proof}[Proof of Theorem \ref{gaussthm}]
Let $\eta > 0$ and $t \in (1, 2) \cup (2, \infty)$.  Choose $\theta = \theta(\eta, t)$ such that $0 < \theta < 1/4$ and the three conditions at the beginning of the proof of Lemma \ref{thetatau} are satisfied.  As in the proof of Lemma \ref{thetatau}, choose $\eps > 0$, $\delta \in (0, \eta/14)$, and $T > t$.  We may assume that $N$ is large enough that the conclusions of Propositions \ref{prop2} and \ref{tauprop} hold with probability at least $1 - \eps$.  For now we will suppose the conclusions of Propositions \ref{prop2} and \ref{tauprop} hold.

Suppose $\ell$ is an integer with $|\ell| \leq \theta k_N$.  As noted following (\ref{taujint}) in the proof of Lemma \ref{thetatau}, the fact that $[t - 1 - 3 \theta, t - 1 + 3 \theta] \subset (0, 1) \cup (1, \infty)$ implies, if $N$ is large enough, that $j(t) + \ell \geq k^* + 1$.  Furthermore, because $\theta < 1/4$, for sufficiently large $N$ we have
$\gamma_{j(t) + \ell + K} \geq \gamma_{j(t) + 1} \geq a_N t$ and, by (\ref{taujint}), $\tau_{j(t) + \ell + 1} \leq a_N(t - 1 + 3 \theta) \leq a_N t$.  Therefore, by (\ref{prop23}),
$$\frac{(1 - \delta) s}{\mu} e^{\int_{\tau_{j(t)+\ell+1}}^{a_N t} G_{j(t) + \ell}(v) \: dv} \leq X_{j(t)+\ell}(a_N t) \leq \frac{( 1 + \delta)s}{\mu} e^{\int_{\tau_{j(t)+\ell+1}}^{a_N t} G_{j(t) + \ell}(v) \: dv}.$$
Consider first the upper bound when $1 \leq \ell \leq \theta k_N$.  We have
\begin{align}\label{gt1}
\log \bigg( \frac{X_{j(t) + \ell}(a_N t)}{X_{j(t)}(a_N t)} \bigg) &\leq \log \bigg( \frac{1 + \delta}{1 - \delta} \bigg) + \int_{\tau_{j(t) + \ell + 1}}^{a_N t} G_{j(t) + \ell}(v) \: dv - \int_{\tau_{j(t)+1}}^{a_N t} G_{j(t)}(v) \: dv \nonumber \\
&= \log \bigg( \frac{1 + \delta}{1 - \delta} \bigg) - \int_{\tau_{j(t) + 1}}^{\tau_{j(t) + \ell + 1}} G_{j(t) + \ell}(v) \: dv + s \ell (a_N t - \tau_{j(t)+1}) \nonumber \\
&= \log \bigg( \frac{1 + \delta}{1 - \delta} \bigg) - \sum_{i=1}^{\ell} \int_{\tau_{j(t) + i}}^{\tau_{j(t) + i + 1}} G_{j(t) + i}(v) \: dv \nonumber \\
&\hspace{.5in}- \sum_{i=1}^{\ell} s(\ell - i)(\tau_{j(t) + i + 1} - \tau_{j(t) + i}) + s \ell (a_N t - \tau_{j(t)+1}).
\end{align}
By Lemma \ref{smulem},
\begin{equation}\label{gt2}
- \sum_{i=1}^{\ell} \int_{\tau_{j(t) + i}}^{\tau_{j(t) + i + 1}} G_{j(t) + i}(v) \: dv \leq - \ell \log \bigg( \frac{s}{C_6 \mu} \bigg).
\end{equation}
By Lemma \ref{thetatau}, with probability tending to one as $N \rightarrow \infty$, we have
\begin{align}\label{gt3}
- \sum_{i=1}^{\ell} s(\ell - i)(\tau_{j(t) + i + 1} - \tau_{j(t) + i}) &\leq - \frac{(1 - \eta/3) s a_N}{q(t-1) k_N} \sum_{i=1}^{\ell} (\ell - i) \nonumber \\
&= - \frac{(1 - \eta/3) \ell (\ell - 1)}{2 q(t-1) k_N} \log \bigg( \frac{s}{\mu} \bigg).
\end{align}
Also, by (\ref{ddef}) and Lemma \ref{thetatau}, with probability tending to one as $N \rightarrow \infty$ we have
$$\gamma_{j(t) + 1} - a_N t = (1/2 - d(t))(\gamma_{j(t) + 1} - \gamma_{j(t)}) \geq \frac{(1/2 - d(t))(1 - \eta/3) a_N}{q(t-1) k_N},$$
which leads to
\begin{align}\label{gt4}
s \ell (a_N t - \tau_{j(t) + 1}) &= s \ell (\gamma_{j(t) + 1} - \tau_{j(t) + 1}) - s \ell (\gamma_{j(t) + 1} - a_N t) \nonumber \\
&\leq \ell \log \bigg( \frac{s}{\mu} \bigg) - \frac{\ell (1/2 - d(t))(1 - \eta/3)}{q(t-1) k_N} \log \bigg( \frac{s}{\mu} \bigg).
\end{align}
Combining (\ref{gt1}), (\ref{gt2}), (\ref{gt3}), and (\ref{gt4}) and using that
\begin{equation}\label{ldeq}
\frac{\ell(\ell - 1)}{2} + \ell \bigg( \frac{1}{2} - d(t) \bigg) = \frac{\ell^2 - 2 d(t) \ell}{2},
\end{equation}
we get
\begin{equation}\label{logupper}
\log \bigg( \frac{X_{j(t) + \ell}(a_N t)}{X_{j(t)}(a_N t)} \bigg) \leq \log \bigg( \frac{1 + \delta}{1 - \delta} \bigg) + \ell \log C_6 - \frac{(1 - \eta/3)}{q(t-1) k_N} \bigg( \frac{\ell^2 - 2 d(t) \ell}{2} \bigg) \log \bigg( \frac{s}{\mu} \bigg).
\end{equation}
The argument for the lower bound follows the same steps.  From Lemma \ref{smulem}, we get $2$ in place of $1/C_6$ in (\ref{gt2}), and the result becomes
\begin{equation}\label{loglower}
\log \bigg( \frac{X_{j(t) + \ell}(a_N t)}{X_{j(t)}(a_N t)} \bigg) \geq \log \bigg( \frac{1 - \delta}{1 + \delta} \bigg) - \ell \log 2 - \frac{(1 + \eta/3)}{q(t-1) k_N} \bigg( \frac{\ell^2 - 2 d(t) \ell}{2} \bigg) \log \bigg( \frac{s}{\mu} \bigg).
\end{equation}
Since $|d(t)| \leq 1/2$, we have $(\ell^2 - 2 d(t) \ell)/2 \leq (\ell^2 + \ell)/2 \leq \ell^2$.  Since $q(t-1) \geq 1$ by Lemma \ref{Qlem} and $\log(s/\mu)/k_N \rightarrow \infty$ as $N \rightarrow \infty$ by (\ref{A2prime}), it follows that when (\ref{logupper}) and (\ref{loglower}) hold and $N$ is sufficiently large, we have
\begin{equation}\label{abslog}
\bigg| \log \bigg( \frac{X_{j(t) + \ell}(a_N t)}{X_{j(t)}(a_N t)} \bigg) + \frac{\ell^2 - 2 d(t) \ell}{2 q(t-1) k_N} \log \bigg( \frac{s}{\mu} \bigg) \bigg| \leq \frac{\eta \ell^2 \log(s/\mu)}{k_N}.
\end{equation}

Suppose now that $- \theta k_N \leq \ell \leq -1$.  The proof is similar to the case in which $\ell$ is positive.  For an upper bound, note that
\begin{align}\label{gt5}
\log \bigg( \frac{X_{j(t) + \ell}(a_N t)}{X_{j(t)}(a_N t)} \bigg) &\leq \log \bigg( \frac{1 + \delta}{1 - \delta} \bigg) + \int_{\tau_{j(t) + \ell + 1}}^{a_N t} G_{j(t) + \ell}(v) \: dv - \int_{\tau_{j(t)+1}}^{a_N t} G_{j(t)}(v) \: dv \nonumber \\
&= \log \bigg( \frac{1 + \delta}{1 - \delta} \bigg) + \int_{\tau_{j(t) + \ell + 1}}^{\tau_{j(t) + 1}} G_{j(t) + \ell}(v) \: dv + s \ell (a_N t - \tau_{j(t) + 1}) \nonumber \\
&= \log \bigg( \frac{1 + \delta}{1 - \delta} \bigg) + \sum_{i = \ell + 1}^0 \int_{\tau_{j(t) + i}}^{\tau_{j(t) + i + 1}} G_{j(t) + i}(v) \: dv \nonumber \\
&\hspace{.5in}+ \sum_{i = \ell+1}^0 s(\ell - i)(\tau_{j(t) + i + 1} - \tau_{j(t) + i}) + s \ell (a_N t - \tau_{j(t) + 1}).
\end{align}
Using Lemma \ref{smulem} again,
\begin{equation}\label{gt6}
\sum_{i=\ell+1}^0 \int_{\tau_{j(t) + i}}^{\tau_{j(t) + i + 1}} G_{j(t) + i}(v) \: dv \leq - \ell \log \bigg( \frac{2s}{\mu} \bigg).
\end{equation}
By Lemma \ref{thetatau}, with probability tending to one as $N \rightarrow \infty$, we have
\begin{align}\label{gt7}
\sum_{i = \ell + 1}^0 s(\ell - i)(\tau_{j(t) + i + 1} - \tau_{j(t) + i}) &\leq \frac{(1 - \eta/3) s a_N}{q(t-1) k_N} \sum_{i=\ell+1}^0 (\ell - i) \nonumber \\
&= -\frac{(1 - \eta/3) \ell (\ell - 1)}{2 q(t-1) k_N} \log \bigg( \frac{s}{\mu} \bigg).
\end{align}
Repeating the reasoning that leads to (\ref{gt4}) gives
\begin{equation}\label{gt8}
s \ell (a_N t - \tau_{j(t) + 1}) \leq \ell \log \bigg( \frac{s}{\mu} \bigg) - \frac{\ell (1/2 - d(t))(1 + \eta/3)}{q(t-1) k_N} \log \bigg (\frac{s}{\mu} \bigg).
\end{equation}
Combining (\ref{gt5}), (\ref{gt6}), (\ref{gt7}), and (\ref{gt8}), and then using (\ref{ldeq}) again along with the inequality $\ell(\ell - 1)/2 - \ell(1/2 - d(t)) \leq \ell^2/2 - 3 \ell/2 \leq 2 \ell^2$, we get
$$\log \bigg( \frac{X_{j(t) + \ell}(a_N t)}{X_{j(t)}(a_N t)} \bigg) \leq \log \bigg( \frac{1 + \delta}{1 - \delta} \bigg) - \ell \log 2 - \frac{\ell^2 - 2 d(t) \ell}{2 q(t-1) k_N} \log \bigg( \frac{s}{\mu} \bigg) + \frac{2 \eta \ell^2 \log(s/\mu)}{3 q(t-1) k_N}.$$
By following the same steps, we obtain the analogous lower bound
$$\log \bigg( \frac{X_{j(t) + \ell}(a_N t)}{X_{j(t)}(a_N t)} \bigg) \geq \log \bigg( \frac{1 - \delta}{1 + \delta} \bigg) + \ell \log C_6 - \frac{\ell^2 - 2 d(t) \ell}{2 q(t-1) k_N} \log \bigg( \frac{s}{\mu} \bigg) - \frac{2 \eta \ell^2 \log(s/\mu)}{3 q(t-1) k_N}.$$
Since $q(t-1) \geq 1$ by Lemma \ref{Qlem} and $\log(s/\mu)/k_N \rightarrow \infty$ as $N \rightarrow \infty$ by (\ref{A2prime}), it follows from these upper and lower bounds that (\ref{abslog}) holds for sufficiently large $N$ in this case as well.

Since (\ref{abslog}) is trivial when $\ell = 0$, equation (\ref{abslog}) holds for all $\ell \in [\theta k_N, \theta k_N] \cap \Z$ with probability at least $1 - \eps$, if $N$ is large enough.  Recalling (\ref{kNdef}), since $\eps > 0$ was arbitrary, Theorem \ref{gaussthm} follows.
\end{proof}


\begin{thebibliography}{99}
\bibitem{athney} K. B. Athreya and P. E. Ney (1972).  {\it Branching Processes}.  Springer-Verlag.

\bibitem{beer07} N. Beerenwinkel, T. Antal, D. Dingli, A. Traulsen, K. W. Kinzler, V. E. Velculescu, B. Vogelstein, and M. A. Nowak (2007).  Genetic progression and the waiting time to cancer.  {\it PLoS Comput. Biol.} {\bf 3}, 2239-2246.

\bibitem{brw08} \'E. Brunet, I. M. Rouzine, and C. O. Wilke (2008).  The stochastic edge in adaptive evolution.  {\it Genetics} {\bf 179}, 603-620.

\bibitem{df07} M. M. Desai and D. S. Fisher (2007).  Beneficial mutation-selection balance and the effect of linkage on positive selection.  {\it Genetics} {\bf 176}, 1759-1798.

\bibitem{dwf13} M. M. Desai, A. M. Walczak, and D. S. Fisher (2013).  Genetic diversity and the structure of genealogies in rapidly adapting populations.  {\it Genetics} {\bf 193}, 565-585.

\bibitem{durrett} R. Durrett (2008).  {\it Probability Models for DNA Sequence Evolution}.  2nd ed.  Springer.

\bibitem{dflmm10} R. Durrett, J. Foo, K. Leder, J. Mayberry, and F. Michor (2010).  Evolutionary dynamics of tumor progression with random fitness values.  {\it Theor. Popul. Biol.} {\bf 78}, 54-66.

\bibitem{dm11} R. Durrett and J. Mayberry (2011).  Traveling waves of selective sweeps.  {\it Ann. Appl. Probab.} {\bf 21}, 699-744.

\bibitem{ek86} S. N. Ethier and T. G. Kurtz (1986).  {\it Markov Processes: Characterization and Convergence}.  Wiley, New York.

\bibitem{feller} W. Feller (1941).  On the integral equation of renewal theory.  {\it Ann. Math. Statist.} {\bf 12}, 243-267.

\bibitem{gl98} P. J. Gerrish and R. E. Lenski (1998).  The fate of competing beneficial mutations in an asexual population.  {\it Genetica} {\bf 102/103}, 127-144.

\bibitem{table} I. S. Gradshteyn and I. M. Ryzhik (2007).  {\it Table of Integrals, Series, and Products}.  7th ed.  Elsevier.

\bibitem{karlin} S. Karlin and H. M. Taylor (1975).  {\it A First Course in Stochastic Processes}.  2nd ed.  Academic Press.

\bibitem{kelly} M. Kelly (2013).  Upper bound on the rate of adaptation in an asexual population.  {\it Ann. Appl. Probab.} {\bf 23}, 1377-1408.

\bibitem{kle} F. C. Klebaner (2005).  {\it Introduction to Stochastic Calculus with Applications}.  2nd ed.  Imperial College Press.

\bibitem{nh13} R. A. Neher and O. Hallatschek (2013).  Genealogies of rapidly adapting populations.  {\it Proc. Natl. Acad. Sci.} {\bf 110}, 437-442.

\bibitem{pk07} S.-C. Park and J. Krug (2007).  Clonal interference in large populations.  {\it Proc. Natl. Acad. Sci.} {\bf 104}, 18135-18140.

\bibitem{psk10} S.-C. Park, D. Simon, and J. Krug (2010).  The speed of evolution in large asexual populations.  {\it J. Stat. Phys.} {\bf 138}, 381-410.

\bibitem{rbw08} I. M. Rouzine, \'E. Brunet, and C. O. Wilke (2008).  The traveling-wave approach to asexual evolution: Muller's ratchet and speed of adaptation.  {\it Theor. Pop. Biol} {\bf 73}, 24-46.

\bibitem{rwc03} I. M. Rouzine, J. Wakeley, and J. M. Coffin (2003).  The solitary wave of asexual evolution.  {\it Proc. Natl. Acad. Sci.} {\bf 100}, 587-592.

\bibitem{schII} J. Schweinsberg (2015).  Rigorous results for a population model with selection II: genealogy of the population.  Preprint.

\bibitem{tlk96} L. S. Tsimring, H. Levine, and D. A. Kessler (1996).  RNA virus evolution via a fitness-space model.  {\it Phys. Rev. Lett.} {\bf 76}, 4440-4443.

\bibitem{wilke} C. O. Wilke (2004).  The speed of adaptation in large asexual populations.  {\it Genetics}, {\bf 167}, 2045-2053.

\bibitem{yec10} F. Yu, A. Etheridge, and C. Cuthbertson (2010).  Asymptotic behavior of the rate of adaptation.  {\it Ann. Appl. Probab.} {\bf 20}, 978-1004.
\end{thebibliography}
\end{document}